\documentclass[10pt, reqno]{amsart}
\usepackage[utf8]{inputenc}

\usepackage{graphicx}
\usepackage{comment}
\usepackage{fourier}
\usepackage[T1]{fontenc}
\usepackage[margin=0.75in]{geometry}
\usepackage{parskip}
\usepackage{amsthm}
\usepackage{amsmath}
\usepackage{amssymb}
\usepackage{mathtools}
\usepackage{apptools}
\usepackage{setspace}
\usepackage{url}
\usepackage{tabularx}
\usepackage[
backend=biber,
style=numeric,
sorting=nyt,
maxnames=10,
giveninits=true
]{biblatex}
\newcommand {\R}{\mathbb{R}}

\newcommand {\Z}{\mathbb{Z}}
\newcommand {\N}{\mathbb{N}}

\newcommand {\C}{\mathbb{C}}
\newcommand {\D}{\mathbb{D}}

\newcommand {\back}{\backslash}

\newcommand{\la}{\lambda}
\newcommand{\re}{\operatorname{Re}}
\newcommand{\im}{\operatorname{Im}}
\newcommand{\sumast}{\mathop{\sum\nolimits^{\mathrlap{\ast}}}}

\newcommand{\Mod}[1]{\ (\mathrm{mod}\ #1)}
\newtheorem{thm}{Theorem}[section]

\newtheorem{lemma}[thm]{Lemma}

\newtheoremstyle{named}{}{}{\itshape}{}{\bfseries}{.}{.5em}{\thmnote{#3}}
\theoremstyle{named}

\AtAppendix{\counterwithin{lemma}{section}}
\numberwithin{equation}{section}

\usepackage{multicol}
\usepackage{xcolor}
\usepackage{hyperref}
\hypersetup{
    colorlinks,
    linkcolor={red!50!black},
    citecolor={blue!50!black},
    urlcolor={blue!80!black}
}

\title{A GL(3) converse theorem via a ``beyond endoscopy'' approach}
\author{Valentin Blomer and Wing Hong Leung}
\date{}

\begin{document}

\begin{abstract} We give a new proof of the converse theorem for Maa{\ss} forms on ${\rm GL}(3)$ using a technique that is inspired by Langlands' philosophy of ``beyond endoscopy'', thereby implementing these ideas for the first time in a higher rank setting. 
    \end{abstract}


\thanks{First author supported by DFG through  EXC-2047/1 - 390685813 and BL 915/5-1 and by ERC
Advanced Grant 101054336. Second author acknowledges support and excellent working conditions by the Max-Planck Institute in Bonn. }

\maketitle

\section{Introduction}

\subsection{Converse theorems}

A converse theorem is a statement that the coefficients of a Dirichlet series come from an automorphic form provided that the Dirichlet series has sufficiently nice properties -- at least a functional equation, usually also functional equations for suitable twists, and sometimes an Euler product. The classical converse theorems for ${\rm GL}(2)$ go back to Hecke \cite{He} and Weil \cite{We}. The first converse theorem for ${\rm GL}(3)$ was proved by Jacquet, Piatetskii-Shapiro and Shalika \cite[Theorem (13.6)]{JPSS}: roughly speaking it says that if the coefficients $A(n, m)$ satisfy the Hecke relations and if for all Dirichlet twists the $L$-functions $\sum_{n} A(n, 1) \chi(n) n^{-s}$ satisfy the correct functional equations, then $A(n, 1)$ are the Hecke eigenvalues of an automorphic form on ${\rm GL}(3)$. A different proof was obtained by Miller and Schmid \cite[Theorem 7.10]{MS} and another proof can be found in \cite[Theorem 7.1.3]{goldfeld2006automorphic}. 

In this paper we present a new and structurally different proof which is inspired by  Langlands' ideas on ``beyond endoscopy''. This is the first implementation in a higher rank set-up. We now explain this connection.

In \cite{La}, Langlands proposed a method to understand functoriality not by endoscopic methods, but by analytic properties of $L$-functions. More precisely, functorial lifts should be witnessed by poles of certain associated $L$-functions which in turn can be detected by methods of analytic number theory. The prototypical examples are dihedral forms on ${\rm GL}(2)$ (automorphic induction from ${\rm GL}(1)$ Gr\"o\ss encharacters) which are distinguished by the fact that their symmetric square $L$-functions twisted by a suitable Hecke character have a pole at $s=1$. This was worked out by Venkatesh \cite{Ve}, see also \cite{Sa-letter}.  In his thesis \cite[Chapter 3]{Ve-thesis}, he observed that the same ideas can be used to prove a ${\rm GL}(2)$ converse theorem: suppose that a sequence $b(n)$ satisfies the Ramanujan conjecture (only for technical simplicity) and the Voronoi summation formula (this is essentially the same as the functional equation for Dirichlet twists), then it comes from a modular form. We will recall the proof in the next subsection. 

Langlands' ideas on ``beyond endoscopy'' have been applied in various other classical ${\rm GL}(2)$ situations, see e.g.\ \cite{Al2, GM, He1} for a Selberg-type bound towards the Ramanujan conjecture and the existence of poles at $s=1$ for Rankin-Selberg and Asai $L$-functions of ${\rm GL}(2)$ automorphic forms. An interesting and difficult application of these ideas to a new proof of Waldspurger’s theorem can be found in \cite{Sa1, Sa2}. 

However, this philosophy has resisted all attempts so far to be applicable to higher rank situations. 
The present paper is an extension to ${\rm GL}(3)$ of Venkatesh's version of the converse theorem using genuine higher rank tools.

\begin{thm}\label{thm1} Suppose that a sequence $B(n, m)$ of complex numbers satisfies the usual Hecke relations, the Ramanujan conjecture and the Voronoi summation formula. Then this sequence is the set of Fourier coefficients of an automorphic form on ${\rm GL}(3)$ whose archimedean Langlands parameter is determined by the gamma factors in the Voronoi summation formula. 
\end{thm}

The Hecke relations mean that $B(n, m)$ is multiplicative in both arguments (in particular $B(1, 1) = 1$) and locally a ${\rm PGL}(3)$ Schur polynomial:
\begin{equation}\label{sat}
B(p^{k}, p^{\ell}) =  \det \left(\begin{smallmatrix} \alpha_p^{k+\ell + 2} & \beta_p^{k+\ell + 2} & \gamma_p^{k + \ell + 2}\\ \alpha_p^{k+1} & \beta_p^{k+1} & \gamma_p^{k+1}\\ 1 & 1 & 1\end{smallmatrix}\right)/\det \left(\begin{smallmatrix} \alpha_p^{2} & \beta_p^{  2} & \gamma_p^{ 2}\\ \alpha_p  & \beta_p  & \gamma_p \\ 1 & 1 & 1\end{smallmatrix}\right)
\end{equation}
for primes $p$ and $k,\ell \in \Bbb{N}_0$ and ``Satake parameters'' $\alpha_p, \beta_p, \gamma_p \in \Bbb{C}$ satisfying $\alpha_p \beta_p\gamma_p = 1$ (trivial central character) 
and $|\alpha_p| =  |\beta_p| = | \gamma_p| = 1 $ (Ramanujan conjecture). The Voronoi formula with respect to a ``spectral parameter'' $\mu$ is spelled out explicitly in Lemma \ref{lem.Voronoi} below. It is equivalent to the functional equations of twisted $\mathrm{GL}(3)$ $L$-functions by the work of Kiral and Zhou \cite{kiral2016voronoi}. The assumption of the Ramanujan conjecture is only for technical convenience and can  most likely be relaxed to the combination of  
(a) 
unitarity, i.e.\  
$\{\alpha_p^{-1}, \beta_p^{-1}, \gamma_p^{-1}\} = \{\overline{\alpha_p}, \overline{\beta_p}, \overline{\gamma_p}\}$, (b)
 some suitable spectral gap at each place, e.g.\ the Kim-Sarnak bound \cite{KS}, 
and (c) a Rankin-Selberg bound $\sum_{nm^2 \leq x} |B(m, n)|^2 \ll x$. 
 On the other hand, we don't expect to find Fourier coefficients of automorphic forms on ${\rm GL}(3)$ \emph{not} satisfying the Ramanujan conjecture.

\subsection{The method}

We start by sketching the method in the case of ${\rm GL}(2)$ Maa{\ss} forms, ignoring all technicalities. Suppose that $b(n)$ is a sequence of numbers satisfying the usual Voronoi summation formula with spectral parameter $t_0$. Fix a weight function $h$ and an integer $m$, and for a parameter $X$ tending to infinity (while keeping $m$ and $h$ fixed) consider
\begin{equation}\label{sum}
\sum_{j}h(t_j)  \lambda_j(m) \sum_{n \asymp X} b(n) \lambda_j(n)
\end{equation}
where the $j$-sum runs over a basis of Maa{\ss} forms with Hecke eigenvalues $\lambda_j(n)$ and spectral parameter $t_j$. If the $b(n)$ come from an automorphic form, they should ``resonate'' with some element in the $j$-sum and produce a main term of size $X$. We sum over $j$ (ignoring the Eisenstein spectrum as well as the necessary harmonic weights)  with the Kuznetsov formula, which produces a diagonal contribution of size $O(1)$ (since $m$ is fixed) and an off-diagonal term roughly of the shape
$$\sum_{n \asymp X} b(n) \sum_{c} \frac{1}{c} S(n, m, c) \int \frac{h(t)}{\cosh(\pi t)} J_{2it}\Big( \frac{4 \pi \sqrt{nm}}{c}\Big) t\, \mathrm{d}t.$$
The effective length of the $c$-sum is $O(\sqrt{X})$. We can now open the Kloosterman sum and apply Voronoi summation in $n$ with parameter $t_0$. This will reduce the $n$-sum to length $O(1)$ and turn the Kloosterman sum into a Ramanujan sum, which we think of as a congruence $n \equiv m$ (mod $c$). So heuristically we obtain
$$  \sum_{c \asymp X^{1/2}}\frac{1}{c} \sum_{\substack{n \asymp 1\\ n \equiv m \, (\text{ mod } c)}} b(n) \int_{x \asymp X} \int \frac{h(t)}{\cosh(\pi t)\cosh(\pi t_0)} J_{2it}\Big( \frac{4 \pi \sqrt{xm}}{c}\Big) J_{2it_0}\Big( \frac{4 \pi \sqrt{xn}}{c}\Big)  t\, \mathrm{d}t\, \mathrm{d}x.$$
This restricts the $n$-sum to $n=m$ (the congruence becomes an equality!), in which case the double integral is an instance of the Sears-Titchmarsh transform, and we obtain asymptotically
\begin{equation}\label{eval}
X b(m) h(t_0).
\end{equation}
Since $m$ and $h$ are free, we can think of $h$ as approximating a delta function at $t_0$, and in the finite dimensional vector space of Maa{\ss} forms with spectral parameter $t_0$ we can vary $m$ to show the sequence $b(m)$ comes from a Maa{\ss} form with spectral parameter $t_0$.  

A few remarks are in order. First of all, the sum \eqref{sum} contains only elements $b(n)$ with $n \asymp X$, nevertheless the Voronoi formula contains enough ``glue'' so that the asymptotic evaluation \eqref{eval} features $b(m)$ with $m$ fixed, in particular far outside the range $[X, 2X]$. More importantly, while the argument is global in nature, there are some interesting local features. The archimedean part is powered by the fact that the Kuznetsov kernel is essentially the same as the Voronoi kernel, and that it is unitary and can be inverted (``Sears-Titchmarsh transform''). 

Let's try the same for ${\rm GL}(3)$ as a back-of-an-envelope computation. We consider
$$\sum_{j} A_j(m, 1) \sum_{n \asymp X} B(1, n) A_j(1, n)$$
Applying the ${\rm GL}(3)$ Kuznetsov formula and keeping only the long Weyl element contribution, we obtain something roughly of the shape
\begin{equation}\label{afterkuz}
\sum_{n \asymp X} B(1, n)\sum_{D_1 \asymp X^{2/3}} \sum_{D_2 \asymp X^{1/3}} \frac{S(n, D_2, D_1)S(m, D_1, D_2)}{D_1D_2},
\end{equation}
at least when $(D_1, D_2) = 1$, in which case the ${\rm GL}(3)$ Kloosterman sum factorizes nicely into two ${\rm GL}(2)$ Kloosterman sums. We pretend for simplicity that all pairs of numbers are coprime.  Next we open the first Kloosterman sum and apply the Voronoi formula to the $n$-sum getting something of the shape
$$\sum_{n \asymp X} B(n, 1)\sum_{D_1 \asymp X^{2/3}} \sum_{D_2 \asymp X^{1/3}} \frac{e(n \overline{D_2}/D_1) S(m, D_1, D_2)}{D_1^{1/2}D_2}.$$
(Here and henceforth the notation $\overline{x}$  denotes the multiplicative inverse of $x$ with respect to a modulus that is clear from the context.)  
This step has not reduced the length of the $n$-sum (unlike in the ${\rm GL}(2)$ case), but improved the shape. As $n \asymp D_1D_2$, we can apply additive reciprocity at no cost and write $e(n \overline{D_2}/D_1) \approx e(-n \overline{D_1}/D_2)$, which reduces the conductor of the character. Now we can apply Voronoi summation a \emph{second} time with respect to a much smaller modulus $D_2$ getting something roughly of the shape
$$X^{1/2} \sum_{n \asymp 1} B(1, n)\sum_{D_1 \asymp X^{2/3}} \sum_{D_2 \asymp X^{1/3}} \frac{S(n, D_1, D_2) S(m, D_1, D_2)}{D_1D_2}.$$
Poisson summation in $D_1$ (which is much longer than $D_2$) produces the congruence $n\equiv m$ (mod $D_2$) which has to be an equality due to the size condition, as in the ${\rm GL}(2)$ case. At this point  there is no oscillation anymore and we end up with $X B(1, m)$. Note that instead of a single Voronoi step we need the combination Voronoi-reciprocity-Voronoi. 

This global argument has an important local counterpart, and  it turns out that the  ${\rm GL}(3)$ Kuznetsov kernel is a convolution of two Voronoi kernels and a   ``diagonal'' exponential reflecting reciprocity, i.e.\ its double Mellin transform is  $$
\gamma(s_1 + s_2, \Bbb{1}) \gamma(1-s_1, \pi) \gamma(1 - s_2, \tilde{\pi})
$$  
with $\gamma(s, \pi) =  L_{\infty}(1-s, \tilde{\pi})/L_{\infty}(s, \pi)$. This identity has not yet appeared in the literature, and we will prove it in a precise form in Lemma \ref{kernelequal} below. There is a $p$-adic analogue of this formula of which we state the following special case: if $\chi, \psi$ are two non-trivial characters modulo $p$, then the finite double Mellin transform of the long Weyl element Kloosterman sum as in \eqref{KNDecomposition} satisfies
\begin{equation}\label{p-adic}
\underset{n, m\, (\text{mod } p)}{\sum\sum} S(n, 1, 1, m; p^2, p^2)\chi(-n)\psi(-m) =    \tau(\overline{\chi\psi}) \tau(\chi)^3\tau(\psi)^3 
\end{equation}
where $\tau$ denotes the Gau{\ss} sum. We prove and discuss this identity in Section \ref{sec23}. It is the analogue of the classical ${\rm GL}(2)$ formula $\sum_{n \, (\text{mod } p)} S(n, 1, p) \chi(n) = \tau(\chi)^2.$ 
The key point is that the kernels of the local orbital integrals in the ${\rm GL}(3)$ relative trace formula can be expressed in terms of concrete automorphic data.

This is in a nutshell the content of the present paper. Making this argument precise, however, takes the next 60 pages and involves a gigantic alphabet of summation variables. More precisely, we will prove the following theorem. Let $X>1$. Let $V\in C_c^\infty([1,2])$  and  $m$ be a fixed positive  integer
.  Let $h$ be a rapidly decaying, Weyl group invariant and holomorphic test function with prescribed zeros as specified at the beginning of Section \ref{sec35}. Let $\mathcal{B}$ be an orthonormal basis of Hecke-Maa{\ss} cusp forms with archimedean Langlands parameter $\mu = (\mu_1, \mu_2, \mu_3)$ for the group ${\rm SL}_3(\Bbb{Z})$. For $\pi\in\mathcal{B}$ let $A_\pi(m,n)$ be the corresponding Fourier coefficient. 
Consider the sum \begin{equation}\label{sm}
    S_m(X)=\sum_{\pi\in \mathcal{B}} \sum_n B(1,n)A_\pi(1,n)\overline{A_\pi(1,m)}V\left(\frac{n}{X}\right)\frac{h(\mu_\pi)}{L(1, \text{Ad}, \pi)}.
\end{equation}

\begin{thm}\label{thm2} Keep the assumptions on the sequence $B$ from Theorem \ref{thm1}. In particular, assume that the Voronoi summation formula of Lemma \ref{lem.Voronoi} holds with the kernel $\mathcal{G}_{\mu_0}^\pm$ defined in \eqref{defG} with $\mu_0 \in (i\Bbb{R})^3$ (i.e.\ satisfying the Ramanujan conjecture). Then there exist  constants $c \not= 0$, $\delta > 0$ such that for 
squarefree $m \in \Bbb{N}$ one has 
    $$S_m(X)=c B(1, m) \tilde{V}(1) h(\mu_0) X+O_{m, h, V}(X^{1-\delta})$$
    where $\tilde{V}$ denotes the  Mellin transform of $V$. 
 \end{thm}

In other words, the fact that the sequence $B$ satisfies a  Voronoi formula implies that it resonates with  some $A_{\pi}$, and we show this using the  Kuznetsov formula. 
The restriction to  squarefree $m$ is   convenient and the minimal requirement to deduce Theorem \ref{thm1}: From the relations $\overline{A(1, p)} = A(p, 1)$ and $A(1, p)^2 = A(1, p^2) + A(p, 1)$, we see that we can generate  all coefficients $A(p^n, p^m)$ from $A(1, p)$ and the Hecke relations. 


An obvious question is  whether this argument can also be applied for ${\rm GL}(4).$ The analogue of \eqref{afterkuz} is
$$\sum_{n \asymp X} B(1, 1, n) \sum_{D_1 \asymp X^{3/4}} \sum_{D_2 \asymp X^{1/2}} \sum_{D_3 \asymp X^{1/4}} \frac{S(n, D_2, D_1) S(D_1, D_3, D_2) S(m, D_2, D_3)}{D_1D_2D_3}.$$
Voronoi in $n$ now increases the length to $X^2$, but does not change the shape and keeps a ${\rm GL}(2)$ Kloosterman sum (since $2-4 = - 2$). So we run immediately out of moves. We could have guessed this before since it is expected that for a ${\rm GL}(4)$ converse theorem ${\rm GL}(1)$ twists alone (i.e.\ Voronoi summation) do not suffice, but also ${\rm GL}(2)$ twists are necessary which we cannot easily incorporate in this scheme. 

Let us end the introduction with a quick computation of what constant $c$ we expect (although this is not relevant for the argument). If the sequence $B(1, n)$ resonates with some $A_{\pi}(1, n)$ then it should come from the  contragredient representation $\tilde{\pi} \in \mathcal{B}$, and we should have
$$c =  \frac{1}{L(1, \text{Ad}, \pi)}\underset{s=1}{\text{res}}\sum_n  \frac{B(1, n)B(n, 1)}{n^s}.$$
A straightforward computation with \eqref{sat} and geometric series  shows 
$$\sum_n  \frac{B(1, n)B(n, 1)}{n^s} = L(s, \pi \times \tilde{\pi}) \prod_p\Big( 1 - \frac{1}{p^{2s}} - \frac{1}{p^{4s}} + \frac{1}{p^{6s}} + B(p, p)\Big( -\frac{1}{p^{2s}} + \frac{2}{p^{3s}} - \frac{1}{p^{4s}}\Big)\Big),$$
so that
\begin{equation}\label{euler-c}
c = \prod_p\Big( \Big(1 + \frac{1}{p^2}\Big)\Big(1 - \frac{1}{p^2}\Big)^2 - \frac{B(p, p)}{p^2} \Big(1 - \frac{1}{p}\Big)^2 \Big).
\end{equation}
We could have avoided this somewhat artificial Euler product by considering a true Rankin-Selberg situation in \eqref{sm} like
$$\sum_{\pi\in \mathcal{B}} \sum_{n_1, n_2} B(n_2,n_1)A_\pi(n_1, n_2)\overline{A_\pi(1,m)}V\left(\frac{n_1n_2^2}{X}\right)\frac{h(\mu_\pi)}{L(1, \text{Ad}, \pi)},$$
but this would have made the computations in Section \ref{sec5} even more complicated. The fact that in Subsection \ref{eulerprod} we arrive -- in a highly non-trivial way -- at this prediction is a strong indication that the long and subtle computations are correct. 

At some points in the argument we will use some support of a computer algebra system. These are all finite and straightforward computations, either  simplex algorithms or computations with multiple finite geometric sums. On the authors' notebooks they can be performed in a few seconds, but it would be extremely tedious to carry them out by hand. We provide the {\tt mathematica} code at the relevant places. 

\textbf{Acknowledgement:} We would like to thank Laurent Montaigu for pointing out an error in Lemma 2.5 in an earlier version. The present version was checked for correctness by ChatGPT. No AI was used for any of the mathematical content. 

\textbf{Changes with respect to previous versions:} Lemma \ref{lem.Phi6Truncation} has been corrected, while Lemma \ref{Kw6} has been corrected and upgraded to an asymptotic formula. This requires changes and additional arguments in the error term analysis in various parts of Section \ref{sec5}. Most importantly, in previous versions we started in Lemma 5.1 with a truncation of non-generic ranges for $D_0, D_1, D_2$. For technical reasons we have now moved this lemma \emph{after} the first Voronoi summation; in the present version it is Lemma \ref{newtruncation}. There are also some corresponding changes in Lemmas \ref{lemF1}, \ref{lemf2}, \ref{lemf3}, \ref{F4lem} estimating the weight functions, and in Lemmas \ref{prop.2ndPreBound},  \ref{lem55}, \ref{lem510} optimizing error exponents. 

\section{Automorphic forms on \texorpdfstring{${\rm GL}(3)$}{GL(3)}}

This section contains a number of new and known results on ${\rm GL}(3)$ automorphic forms. We point out in particular some useful Hecke relations (Lemma \ref{lem.Hecke}), a new identity for the Kuznetsov kernel (Lemma \ref{kernelequal}) and a proof of \eqref{p-adic}, as well as  sharp bounds of independent interest for certain integral transforms in the Kuznetsov formula (Lemma \ref{Kw6}).

\subsection{Hecke relations}\label{sect.Hecke}

Let $A(m,n)$ be the Fourier coefficients of an $\mathrm{SL}(3,\Z)$ Hecke Maass cusp form. They satisfy the following Hecke relations by \cite[Thm 6.4.11]{goldfeld2006automorphic}.
\begin{equation}\tag{$*$}
\begin{aligned}
& A\left(m_1 m_1^{\prime}, m_2 m_2^{\prime}\right)=A\left(m_1, m_2\right) \cdot A\left(m_1^{\prime}, m_2^{\prime}\right), \quad \text { if }\left(m_1 m_2, m_1^{\prime} m_2^{\prime}\right)=1, \\
& A(n, 1) A\left(m_1, m_2\right)=\sum_{\substack{d_0 d_1 d_2=n \\
d_1\mid m_1 \\
d_2\mid m_2}} A\left(\frac{m_1 d_0}{d_1}, \frac{m_2 d_1}{d_2}\right), \quad  A(1, n) A\left(m_1, m_2\right)=\sum_{\substack{d_0 d_1 d_2=n \\
d_1\mid m_1 \\
d_2\mid m_2}} A\left(\frac{m_1 d_2}{d_1}, \frac{m_2 d_0}{d_2}\right), \\
& A(m, n) = \overline{A(n, m)}. 
\end{aligned}
\end{equation}
Any sequence $B(m,n)$ satisfying $(*)$ is completely determined by $B(1,p), B(1,p^2)$ with   $p$ running through all the primes. By M\"obius inversion, any such sequence also satisfies \begin{equation}\label{mob1}
    B(m,n)=\sum_{d\mid (m,n)}\mu(d)B\left(\frac{m}{d},1\right)B\left(1,\frac{n}{d}\right),
    \end{equation}
    \begin{equation}\label{mob2}
     B(mn,1)=\sum_{\substack{abc=n\\b\mid c\mid m}}\mu(b)\mu(c)B\left(\frac{m}{c},\frac{c}{b}\right)B(a,1).
     \end{equation}
Moreover, by a more involved M\"obius inversion, we have the following lemma.

\begin{lemma}\label{lem.Hecke}
    Let $B(n,m)$ be a sequence satisfying the Hecke relations $(*)$. Then for any $n_1,n_2,m$, we have \begin{align*}
        B(n_1n_2,m)=\mathop{\sum\sum}_{\substack{abc=n_1\\b|mc, c|n_2}}\mu(b)\mu(c)B\left(\frac{n_2}{c},\frac{mc}{b}\right)B(a,1)
    \end{align*}
    as well as the corresponding dual relation with all entries of $B$ exchanged. 
\end{lemma}
\begin{proof}
    Applying M\"obius inversion, we have \begin{align*}
        B(n_1n_2,m)=&\sum_{a\mid n_1}\sum_{b\mid n_2}B\left(ab,\frac{mn_1}{a}\right)\delta\left(\frac{n_1}{a}=\frac{n_2}{b}=1\right) = \mathop{\sum\sum}_{\substack{a\mid n_1,b\mid n_2\\n_1/a=n_2/b}}B\left(ab,\frac{mn_1}{a}\right)\sum_{c\mid \frac{n_1}{a}}\mu(c)\\
        =&\mathop{\sum\sum\sum}_{\substack{acd_1=n_1\\bcd_1=n_2}}B\left(ab,\frac{mn_1}{a}\right)\mu(c) = \sum_{c\mid (n_1,n_2)}\mu(c)\sum_{d_1\mid \left(\frac{n_1}{c},\frac{n_2}{c}\right)}B\Big(\frac{n_1n_2}{c^2d_1^2},mcd_1\Big).
    \end{align*}
    Now we add an extra variable $\ell=1$ and detect it by M\"obius inversion. In this way we can recast $B(n_1n_2, m)$ as \begin{align*}
       &\sum_{c\mid (n_1,n_2)}\mu(c)\sum_{d_1\mid \left(\frac{n_1}{c},\frac{n_2}{c}\right)}\sum_{\ell\mid \left(\frac{n_1}{cd_1},mc\right)}B\left(\frac{n_1n_2}{c^2d_1^2\ell},\frac{mcd_1}{\ell}\right)\sum_{b\mid \ell}\mu(b)=\sum_{c\mid (n_1,n_2)}\mu(c)\sum_{d_1\mid \left(\frac{n_1}{c},\frac{n_2}{c}\right)}\mathop{\sum\sum}_{bd_2\mid \left(\frac{n_1}{cd_1},mc\right)}\mu(b)B\left(\frac{n_1n_2}{bc^2d_1^2d_2},\frac{mcd_1}{bd_2}\right)\\
        =&\sum_{c\mid (n_1,n_2)}\mu(c)\sum_{b\mid \left(\frac{n_1}{c},mc\right)}\mu(b)\sum_{d_1\mid \left(\frac{n_1}{c},\frac{n_2}{c}\right)}\sum_{d_2\mid \left(\frac{n_1}{bcd_1},\frac{mc}{b}\right)}B\left(\frac{n_1n_2}{bc^2d_1^2d_2},\frac{mcd_1}{bd_2}\right) =\mathop{\sum\sum}_{\substack{abc=n_1\\b\mid mc, c|n_2}}\mu(b)\mu(c)\mathop{\sum\sum}_{\substack{d_0d_1d_2=a\\d_1\mid \frac{n_2}{c},d_2\mid \frac{mc}{b}}}B\left(\frac{n_2a}{cd_1^2d_2},\frac{mcd_1}{bd_2}\right).
    \end{align*}
    Applying the Hecke relation in $(*)$, we conclude the claim. 
\end{proof}


\subsection{Voronoi summation}

Let $A(m,n)$ be the Fourier coefficients of an $\mathrm{SL}(3,\Z)$ Hecke Maass cusp form with Langlands parameter $-\mu \in (i\Bbb{R})^3$. They  satisfy the following Voronoi summation.

\begin{lemma}[{\cite[Thm 1.18]{MS}}]\label{lem.Voronoi}
    Let $m, a, q$ be integers such that $m, q>0$ and $(a,q)=1$. Let $g\in C_c^\infty(\R^+)$. Then for any $\sigma>-1$, we have
    \begin{align*}
        \sum_{n=1}^\infty A(m,n) e\Big(\frac{an}{q}\Big)g(n) = q\sum_\pm \sum_{n_0\mid mq}\sum_{n=1}^\infty \frac{A(n,n_0)}{n_0n} S\Big(m\overline{a}, \pm n; \frac{mq}{n_0}\Big) G_\pm \Big(\frac{n_0^2n}{mq^3}\Big),
    \end{align*}
    where \begin{align}\label{defG}
        G_\pm(y) = \frac{1}{2\pi i}\int_{(\sigma)} y^{-s}\mathcal{G}_\mu^{ \pm}(s+1)\tilde{g}(-s)\mathrm{d}s,
    \end{align}
    with \begin{align*}
        \mathcal{G}_\mu^{ \pm}(s)=4(2 \pi)^{-3 s} \prod_{j=1}^3 \Gamma\left(s+\mu_j\right)\left(\prod_{j=1}^3 \cos \left(\frac{\pi\left(s+\mu_j\right)}{2}\right) \pm \frac{1}{i} \prod_{j=1}^3 \sin \left(\frac{\pi\left(s+\mu_j\right)}{2}\right)\right)
    \end{align*}
    and $\tilde{g}(s)=\int_0^\infty g(x)x^{s-1}\mathrm{d} x$   the Mellin transform of $g$.
\end{lemma}

By \cite[Lemma 6]{Blo12}, we have the asymptotics \begin{align}\label{GL3KernelAsymp}
    G_\pm(x)=x^{2/3}\int_0^\infty \frac{g(y)}{y^{1/3}}W_\mu^{\pm}(xy)e(\pm 3(xy)^{1/3})\mathrm{d}y,
\end{align}
for some smooth function $W_\mu^\pm$ satisfying $x^j\frac{\mathrm{d}^j}{\mathrm{d}x^j}W_\mu^\pm(x)\ll_{\mu,j} 1$ for any $j\geq0$.







\subsection{Kloosterman sums}\label{sec23}

We now prepare for the Kuznetsov formula. Most notation is taken from \cite[Sections  3 and 4]{Blomer2015OnTS}. 

For   integers $n_1,n_2,m_1,m_2,D_1,D_2$, we define two types of Kloosterman sums as follows, \begin{align}\label{ShortKloostermanSum}
\tilde{S}\left(n_1, n_2, m_1 ; D_1, D_2\right):=\sum_{\substack{C_1\left(\bmod D_1\right), C_2\left(\bmod D_2\right) \\\left(C_1, D_1\right)=\left(C_2, D_2 / D_1\right)=1}} e\left(n_2 \frac{\bar{C}_1 C_2}{D_1}+m_1 \frac{\bar{C}_2}{D_2 / D_1}+n_1 \frac{C_1}{D_1}\right)
\end{align}
for $D_1 \mid D_2$, and \begin{align*}
& S\left(n_1, m_2, m_1, n_2 ; D_1, D_2\right) 
 =\sum_{\substack{B_1, C_1\left(\bmod D_1\right) \\
B_2, C_2\left(\bmod D_2\right) \\
D_1 C_2+B_1 B_2+D_2 C_1 \equiv 0\left(\bmod D_1 D_2\right) \\
\left(B_j, C_j, D_j\right)=1}} e\left(\frac{n_1 B_1+m_1\left(Y_1 D_2-Z_1 B_2\right)}{D_1}+\frac{m_2 B_2+n_2\left(Y_2 D_1-Z_2 B_1\right)}{D_2}\right), 
\end{align*}
where $Y_j B_j+Z_j C_j \equiv 1\left(\bmod D_j\right)$ for $j=1,2$.

For the second Kloosterman sum, we have the following decomposition into classical Kloosterman sums due to Kiral and Nakasuji \cite[Thm 5.10]{kiral2022parametrization}, \begin{align}\label{KNDecomposition}
    &S(n_1,m_2,m_1,n_2;D_1,D_2)\nonumber\\
    &=\sum_{D_0\mid (D_1,D_2)}D_0\sumast_{\substack{\alpha\Mod{D_0}\\m_1\frac{D_2}{D_0}+m_2\frac{D_1}{D_0}\alpha\equiv 0\Mod{D_0}}}S\left(n_1,\frac{m_1D_2+m_2D_1\alpha}{D_0^2};\frac{D_1}{D_0}\right)S\left(n_2,\frac{m_1D_2\overline{\alpha}+m_2D_1}{D_0^2};\frac{D_2}{D_0}\right).
\end{align}
(Note that $\bar{\alpha}$ in the second Kloosterman sum is well-defined by the congruence condition.)  
This is in some sense the analogue of the integral formulae \eqref{doubleBessel} below. We use this formula to quickly prove \eqref{p-adic}: for two non-trivial characters $\chi, \psi$ modulo $p$ we have
\begin{displaymath}
\begin{split}
\underset{n, m\, (\text{mod } p)}{\sum\sum} S(n, 1, 1, m; p^2, p^2)\chi(-n)\psi(-m) =\underset{n, m\, (\text{mod } p)}{\sum\sum} \sum_{j \leq 2} p^j \underset{\substack{\alpha \, (\text{mod } p^j)\\ p^{2j-2} \mid 1+\alpha}}{\left.\sum \right.^{\ast}} S(n, (1 + \alpha)p^{2-2j}, p^{2-j})S(m, (1 + \bar{\alpha})p^{2-2j}, p^{2-j})\chi(-n)\psi(-m).
\end{split}
\end{displaymath}
Note that this is well-defined. It is easy to see that the terms $j=0$ and $j=2$ vanish. For the $j=1$ term we open the Kloosterman sums and sum over $n, m$ getting
$$p\tau(\chi)\tau(\psi) \underset{\substack{\alpha, d, f\, (\text{mod } p)\\ (\alpha df, p) = 1}}{\sum\sum\sum}\bar{\chi}(-d) \bar{\psi}(-f) e\Big( \frac{(1+\alpha)\bar{d} + (1+ \bar{\alpha})\bar{f}}{p}\Big) = p\tau(\chi)^2\tau(\psi)^2 \underset{ \alpha \, (\text{mod } p)}{\left.\sum \right.^{\ast}}\bar{\chi}(-1 - \alpha) \bar{\psi}(-1 - \bar{\alpha}).$$
Changing variables $\beta = \overline{1 + \alpha}$ and using the identities \cite[(3.18), (3.20)]{IK} for the Jacobi sum depending on whether $\chi\psi$ is non-trivial or trivial, followed by \cite[(3.13), (3.15)]{IK},  we conclude the proof of \eqref{p-adic}.   

Although we will not use this formula directly in our paper, it is an interesting $p$-adic analogue of Lemma \ref{kernelequal} below that will be used explicitly in Subsection \ref{endgame}: the double Mellin transform of the Kuznetsov kernel associated with the long Weyl element can be expressed in terms of epsilon-factors (i.e.\ Gau{\ss} sums resp.\ Gamma quotients). The reader may wonder why we consider Kloosterman sums to moduli $(p^2, p^2)$ rather than the seemingly more natural choice of moduli $(p, p)$. The reason  {is} simply that 
$$\underset{n, m\, (\text{mod } p)}{\sum\sum} S(n, 1, 1, m; p^a, p^b)\chi(n)\psi(m) = 0$$
for $a=b=1$ and $\chi, \psi$ modulo $p$ not both trivial, and the only well-defined choice of $a, b$ where this does not vanish is $a=b=2$. 

\subsection{Integral kernels}\label{sect.IntegralKernels}

For $s\in\C$, $\mu\in\Lambda_0 := \{( {\mu}_1, \mu_2, \mu_3) \in (i\Bbb{R})^3 \mid \mu_1 + \mu_2 + \mu_3 = 0\}$, define the meromorphic function \begin{align*}
    \tilde{G}^{ \pm}(s, \mu):=\frac{\pi^{-3 s}}{12288 \pi^{7 / 2}}\left(\prod_{j=1}^3 \frac{\Gamma\left(\frac{1}{2}\left(s-\mu_j\right)\right)}{\Gamma\left(\frac{1}{2}\left(1-s+\mu_j\right)\right)} \pm i \prod_{j=1}^3 \frac{\Gamma\left(\frac{1}{2}\left(1+s-\mu_j\right)\right)}{\Gamma\left(\frac{1}{2}\left(2-s+\mu_j\right)\right)}\right).
\end{align*}
and for $\epsilon_1, \epsilon_2=\pm1$, $s=(s_1,s_2)\in\C^2$, $\mu\in\Lambda_0$, define \begin{align*}
    G_{\mathrm{sym}}^{\epsilon_1,\epsilon_2}(s,\mu)=\frac{1}{1024 \pi^{5 / 2}} \sum_{\substack{d_1,d_2,d_3 \in\{0,1\}\\ d_3\equiv d_1+d_2\Mod{2}}} \epsilon_1^{d_1} \epsilon_2^{d_2}(-1)^{d_1 d_2} \frac{\Gamma\left(\frac{1+d_3-s_1-s_2}{2}\right)}{\Gamma\left(\frac{d_3+s_1+s_2}{2}\right)} \prod_{i=1}^3 \frac{\Gamma\left(\frac{d_1+s_1-\mu_i}{2}\right) \Gamma\left(\frac{d_2+s_2+\mu_i}{2}\right)}{\Gamma\left(\frac{1+d_1-s_1+\mu_i}{2}\right) \Gamma\left(\frac{1+d_2-s_2-\mu_i}{2}\right)}.
\end{align*}
The latter is taken from \cite[(16)-(18)]{buttcane2020plancherel}, and it is the symmetrized version over the Weyl group of the function   defined in \cite[Section 3.3]{Blomer2015OnTS}. 
We can now define the following integral kernels. For $y\in\R\back\{0\}$ and $\mu\in\Lambda_0$, define
\begin{align}\label{K4Kernel}
    K_{\omega_4}(y;\mu)=\int_{-i\infty}^{i\infty}|y|^{-s}\tilde{G}^{\mathrm{sgn}(y)}(s,\mu)\frac{\mathrm{d}s}{2\pi i},
\end{align}
and 
for $y=(y_1,y_2)\in(\R\back\{0\})^2$, we define
\begin{align}\label{SymKernelMellin}
    K_{\omega_6}^{\mathrm{sym}}(y ; \mu)=\int_{-i\infty}^{i\infty} \int_{-i\infty}^{i\infty} \left(\pi^2 |y_1|\right)^{-s_1}\left(\pi^2 |y_2|\right)^{-s_2} G_{\mathrm{sym}}^{\mathrm{sgn}(y_1),\mathrm{sgn}(y_2)}(s, \mu) \frac{\mathrm{d} s_1 \mathrm{~d} s_2}{(2 \pi i)^2}.
\end{align}
 We generally follow the Barnes integral convention that the contour should pass to the right of all of the poles of the gamma functions in the form $\Gamma(s_j + a)$  and to the left of all of the poles of the gamma functions 
 in the form $\Gamma(a - s_j)$. 
 Moreover, we choose the contour such that all integrals are absolutely convergent, which can always be arranged by shifting the unbounded part appropriately.

 The function $ K_{\omega_6}^{\mathrm{sym}}(y ; \mu)$ is a linear combination (over permutations of the Weyl group) of the following functions; see e.g. \cite[Section 5]{Blomer2015OnTS}:
 \begin{equation}\label{doubleBessel}
\begin{split}
 &\mathcal{J}^{\pm}_{1}(y; \mu) = 
  \Bigl| \frac{y_1}{y_2}\Bigr|^{\frac{1}{2}\mu_2}      \int_0^\infty J^{\pm}_{3\nu_3 }\left(2\pi |y_1|^{1/2}\sqrt{1+u^2}\right) J^{\pm}_{3\nu_3}\left(2\pi |y_2|^{1/2}\sqrt{1+u^{-2}}\right) u^{3\mu_2} \frac{\mathrm{d}u}{u},\\
 &\mathcal{J}_{2}(y; \mu) = 
 \Bigl| \frac{y_1}{y_2}\Bigr|^{\frac{1}{2}\mu_2}      \int_1^\infty J^{-}_{3\nu_3 }\left(2\pi |y_1|^{1/2}\sqrt{u^2-1}\right) J^{-}_{3\nu_3}\left(2\pi |y_2|^{1/2}\sqrt{1-u^{-2}}\right) u^{3\mu_2} \frac{\mathrm{d}u}{u},\\
 &  \mathcal{J}_{3}(y; \mu) = 
 \Bigl| \frac{y_1}{y_2}\Bigr|^{\frac{1}{2}\mu_2}      \int_0^\infty \tilde{K}_{3\nu_3 }\left(2\pi |y_1|^{1/2}\sqrt{1+u^2}\right) J^{-}_{3\nu_3}\left(2\pi |y_2|^{1/2}\sqrt{1+u^{-2}}\right) u^{3\mu_2} \frac{\mathrm{d}u}{u},
 \\
& \mathcal{J}_{4}(y; \mu) = 
  \Bigl| \frac{y_1}{y_2}\Bigr|^{\frac{1}{2}\mu_2}      \int_0^1 \tilde{K}_{3\nu_3 }\left(2\pi |y_1|^{1/2}\sqrt{1-u^2}\right) \tilde{K}_{3\nu_3}\left(2\pi |y_2|^{1/2}\sqrt{u^{-2}-1}\right) u^{3\mu_2} \frac{\mathrm{d}u}{u},
  \\
     \end{split}
\end{equation}
\begin{equation*}
\begin{split}
 &\mathcal{J}_{5}(y;\mu) =  \Bigl|\frac{y_1 }{y_2 }\Bigr|^{\frac{1}{2}\mu_2}   \int_0^\infty \tilde{K}_{3\nu_3 }\left(2\pi |y_1|^{1/2}\sqrt{1+u^2}\right) \tilde{K}_{3\nu_3}\left(2\pi |y_2|^{1/2}\sqrt{1+u^{-2}}\right) u^{3\mu_2} \frac{\mathrm{d}u}{u},
\end{split}
\end{equation*}
where
\begin{equation*}
J_{\alpha}^+(x) := \frac{\pi}{2} \frac{J_{-\alpha}(2x) + J_{\alpha}(2x)}{\cos(\pi \alpha/2)}, \quad J_{\alpha}^-(x) := \frac{\pi}{2} \frac{J_{-\alpha}(2x) - J_{\alpha}(2x)}{\sin(\pi \alpha/2)}, \quad \tilde{K}_{\alpha}(x) = 2\cos\left(\frac{\pi}{2} \alpha\right) K_{\alpha}(2x)
\end{equation*}
 { and  $\nu_1 = \frac{1}{3}(\mu_1 - \mu_2), \quad \nu_2 = \frac{1}{3}(\mu_2 - \mu_3),\quad  \nu_3 = \frac{1}{3}(\mu_3 - \mu_1)$. }
 
 We note that the kernels $\tilde{G}^{\pm}(s, \mu)$ and the Voronoi kernel $\mathcal{G}_{\mu}^{\pm}(s)$ are essentially the same, and by Stirling's formula both are bounded by $(1 + |\Im s|)^{3\Re s - 3/2}$ away from poles. We will use this bound frequently below. With later applications in mind, for $\epsilon_{1, 2} \in \{\pm 1\}$ we define
\begin{equation}\label{vrvkernel}
    G_{\mu_0}^{\epsilon_1,\epsilon_2}(s_1,s_2):=\sum_{\eta_1\eta_2=\epsilon_1\epsilon_2}\Gamma (1-s_1-s_2 )e^{\frac{\pi i}{2}\epsilon_2\eta_1(s_1+s_2-1)}\mathcal{G}_{\mu_0}^{\eta_1}(s_1)\mathcal{G}_{-\mu_0}^{\eta_2}(s_2),
\end{equation}

\begin{lemma}\label{kernelequal}
 For $\epsilon_{1, 2} \in \{\pm 1\}$,  $\mu_0 \in (i\Bbb{R})^3$ and $s = (s_1, s_2) \in \Bbb{C}^2$ we have an equality of meromorphic functions   $$ G_{\mu_0}^{\epsilon_1,\epsilon_2}(s) = 512\pi^{3(2-s_1-s_2)} 2^{-s_1-s_2}G_{\text{{\rm sym}}}^{-\epsilon_2, -\epsilon_1}(s, -\mu_0).$$
\end{lemma}

\begin{proof} This is a direct, but non-trivial computation using properties of trigonometric functions and the gamma function. By the uniqueness of the meromorphic extension, it suffices to prove the equality when $s=(s_1,s_2)\in (i\R)^2$. We recall that  \begin{align*}
    \frac{\Gamma(z+1/2)}{\Gamma(\overline{z})}=\frac{2^{1-2z}\sqrt{\pi}\Gamma(2z)}{|\Gamma(z)|^2} \quad \text{ and } \quad \frac{\Gamma(z)}{\Gamma(\overline{z}+1/2)}=\frac{2^{1-2z}\sqrt{\pi}\Gamma(2z)}{|\Gamma(z+1/2)|^2}
\end{align*}
for any $z\in\C$, and hence \begin{align*}
    4^{s_1+s_2}G_{\mathrm{sym}}^{-\epsilon_2,-\epsilon_1}(s,-\mu_0)=&\frac{\pi}{8}\sum_{\substack{d_1,d_2,d_3\in\{0,1\}\\d_3\equiv d_1+d_2\Mod{2}}}\frac{(-\epsilon_2)^{d_1}(-\epsilon_1)^{d_2}(-1)^{d_1d_2}}{2^{3d_1+3d_2+d_3}}\frac{\Gamma\left(d_3-s_1-s_2\right)}{\big|\Gamma\left(\frac{d_3-s_1-s_2}{2}\right)\big|^2} \prod_{j=1}^3\frac{\Gamma\left(d_1+s_1+\mu_{0,j}\right)\Gamma\left(d_2+s_2-\mu_{0,j}\right)}{\big|\Gamma\left(\frac{1+d_1+s_1+\mu_{0,j}}{2}\right) \Gamma\left(\frac{1+d_2+s_2-\mu_{0,j}}{2}\right)\big|^2}.
\end{align*}
Using \begin{align*}
    \left|\Gamma(t)\right|^2=-\frac{\pi}{t\sin(\pi t)} \quad \text{ and } \quad \left|\Gamma(1/2+t)\right|^2=\frac{\pi}{\cos(\pi t)}
\end{align*}
for any $t\in i\R$ together with $\Gamma(z+1)=z\Gamma(z)$, we have   \begin{align*}
   4^{s_1+s_2}G_{\mathrm{sym}}^{-\epsilon_2,-\epsilon_1}(s,-\mu_0) =  &\frac{1}{16\pi^6}\Gamma(1-s_1-s_2)\prod_{j=1}^3\Gamma(s_1+\mu_{0,j})\Gamma(s_2-\mu_{0,j})\Bigg\{\sin\left(\frac{\pi}{2}(s_1+s_2)\right)\prod_{j=1}^3\cos\left(\frac{\pi}{2}(s_1+\mu_{0,j})\right)\cos\left(\frac{\pi}{2}(s_2-\mu_{0,j})\right)\nonumber\\
    &-\epsilon_2\cos\left(\frac{\pi}{2}(s_1+s_2)\right)\prod_{j=1}^3\sin\left(\frac{\pi}{2}(s_1+\mu_{0,j})\right)\cos\left(\frac{\pi}{2}(s_2-\mu_{0,j})\right)\nonumber\\
    &-\epsilon_1\cos\left(\frac{\pi}{2}(s_1+s_2)\right)\prod_{j=1}^3\cos\left(\frac{\pi}{2}(s_1+\mu_{0,j})\right)\sin\left(\frac{\pi}{2}(s_2-\mu_{0,j})\right)\nonumber\\
    &-\epsilon_1\epsilon_2\sin\left(\frac{\pi}{2}(s_1+s_2)\right)\prod_{j=1}^3\sin\left(\frac{\pi}{2}(s_1+\mu_{0,j})\right)\sin\left(\frac{\pi}{2}(s_2-\mu_{0,j})\right)\Bigg\}.
\end{align*}

On the other hand, we have
\begin{align*}
    G_{\mu_0}^{\epsilon_1,\epsilon_2}(s_1,s_2)=&16(2\pi)^{-3s_1-3s_2}\Gamma(1-s_1-s_2)\prod_{j=1}^3\Gamma(s_1+\mu_{0,j})\Gamma(s_2-\mu_{0,j})\sum_{\eta_1\eta_2=\epsilon_1\epsilon_2}e^{\epsilon_2\eta_1\pi i\frac{s_1+s_2-1}{2}}S_G^{\eta_1,\eta_2}(s;\mu_0),
\end{align*}
where $S_G^{\eta_1,\eta_2}(s;\mu_0)$ is given by  \begin{align*}
    \left(\prod_{j=1}^3\cos\left(\frac{\pi}{2}(s_1+\mu_{0,j})\right)+\frac{\eta_1}{i}\prod_{j=1}^3\sin\left(\frac{\pi}{2}(s_1+\mu_{0,j})\right)\right)  \left(\prod_{j=1}^3\cos\left(\frac{\pi}{2}(s_2-\mu_{0,j})\right)+\frac{\eta_2}{i}\prod_{j=1}^3\sin\left(\frac{\pi}{2}(s_2-\mu_{0,j})\right)\right).
\end{align*}
Now we simplify the $\eta$-sum by \begin{align*}
    &\sum_{\eta_1\eta_2=\epsilon_1\epsilon_2}e^{\epsilon_2\eta_1\pi i\frac{s_1+s_2-1}{2}}S_G^{\eta_1,\eta_2}(s;\mu_0)\\
    =&\sum_{\eta=\pm1}e^{\epsilon_2\eta\pi i\frac{s_1+s_2-1}{2}}\left(\prod_{j=1}^3\cos\left(\frac{\pi}{2}(s_1+\mu_{0,j})\right)-i\eta\prod_{j=1}^3\sin\left(\frac{\pi}{2}(s_1+\mu_{0,j})\right)\right)\left(\prod_{j=1}^3\cos\left(\frac{\pi}{2}(s_2-\mu_{0,j})\right)-i\eta\epsilon_1\epsilon_2\prod_{j=1}^3\sin\left(\frac{\pi}{2}(s_2-\mu_{0,j})\right)\right)\\
     =&\left(e^{\epsilon_2\pi i\frac{s_1+s_2-1}{2}}+e^{-\epsilon_2\pi i\frac{s_1+s_2-1}{2}}\right)\left(\prod_{j=1}^3\cos\left(\frac{\pi}{2}(s_1+\mu_{0,j})\right)\cos\left(\frac{\pi}{2}(s_2-\mu_{0,j})\right)-\epsilon_1\epsilon_2\prod_{j=1}^3\sin\left(\frac{\pi}{2}(s_1+\mu_{0,j})\right)\sin\left(\frac{\pi}{2}(s_2-\mu_{0,j})\right)\right)\\
    &- i\left(e^{\epsilon_2\pi i\frac{s_1+s_2-1}{2}}-e^{-\epsilon_2\pi i\frac{s_1+s_2-1}{2}}\right)\left(\prod_{j=1}^3\sin\left(\frac{\pi}{2}(s_1+\mu_{0,j})\right)\cos\left(\frac{\pi}{2}(s_2-\mu_{0,j})\right)+ {\epsilon_1\epsilon_2}\prod_{j=1}^3\cos\left(\frac{\pi}{2}(s_1+\mu_{0,j})\right)\sin\left(\frac{\pi}{2}(s_2-\mu_{0,j})\right)\right)\\
    =&2\sin\left(\frac{\pi}{2}(s_1+s_2)\right)\left(\prod_{j=1}^3\cos\left(\frac{\pi}{2}(s_1+\mu_{0,j})\right)\cos\left(\frac{\pi}{2}(s_2-\mu_{0,j})\right)-\epsilon_1\epsilon_2\prod_{j=1}^3\sin\left(\frac{\pi}{2}(s_1+\mu_{0,j})\right)\sin\left(\frac{\pi}{2}(s_2-\mu_{0,j})\right)\right)\\
    &-2\epsilon_2\cos\left(\frac{\pi}{2}(s_1+s_2)\right)\left(\prod_{j=1}^3\sin\left(\frac{\pi}{2}(s_1+\mu_{0,j})\right)\cos\left(\frac{\pi}{2}(s_2-\mu_{0,j})\right)+\epsilon_1\epsilon_2\prod_{j=1}^3\cos\left(\frac{\pi}{2}(s_1+\mu_{0,j})\right)\sin\left(\frac{\pi}{2}(s_2-\mu_{0,j})\right)\right).
\end{align*}
Comparing the two expressions above concludes the proof.
\end{proof}

 \subsection{Choice of the test function and properties of its integral transform}\label{sec35}

We will choose our test function $h$ to be  invariant under the Weyl group $W$, holomorphic, rapidly decaying as $\| \mu \| \rightarrow \infty$ in strips of fixed real part, and we impose that for some sufficiently large $A > 0$ it has zeros at 
\begin{equation}\label{zeros}
    \mu_i - \mu_j = 2n+1 \in \Bbb{Z}, \quad |n| \leq A, \quad 1 \leq i < j \leq 3.
    \end{equation}

Let $N : \Bbb{C}^3 \rightarrow \Bbb{C}$, $(z_1, z_2, z_3) \mapsto z_1^2 + z_2^2 + z_3^2$. Let $S$ be a set of representatives of each $W$-orbit of the finite set of cuspidal spectral parameters, excluding $W\mu_0$, whose distance from $W\mu_0$ is at most 10.  Choose a linear form $\ell$ such that 
$$\ell(w\mu_0 - \nu) \not = 0, \quad \nu \in S, w \in W.$$
Such an $\ell$ exists because the previous condition excludes at most finitely many hyperplanes. 

We will later  choose, for a given  spectral parameter $\mu_0 \in (i\Bbb{R})^3$,  the following family of test functions
\begin{equation}\label{hz}
h_{Z}(\mu) = \frac{\sum_{w \in W} \exp\big( N( \mu - w(\mu_0)) Z\big)}{\sum_{w \in W} \exp\big( N( \mu_0 - w(\mu_0)) Z\big)}  \frac{P(\mu)}{P(\mu_0)} \frac{Q(\mu)}{Q(\mu_0)}, \quad P(\mu) = \prod_{|n| \leq A} \prod_{1 \leq i  {\not=} j \leq 3} (\mu_i - \mu_j - (2n+1))
\end{equation}
for a large fixed number $A$ as above, a parameter $Z$ tending to infinity and
$$Q(\mu) = \prod_{w \in W} \prod_{\nu \in S} \ell(w(\mu) - \nu) .$$
This polynomial is Weyl group invariant, vanishes on every orbit in $S$ and satisfies $Q(\mu_0) \not= 0$.
 
The function $h_Z$ satisfies all the required properties and also  $h_Z(\mu_0)=1$. For increasing $Z$ it localizes at the $W$-orbit of $\mu_0$. Note that
for $\mu = a + ib$  (not necessarily satisfying Ramanujan) and $\mu_0 = ib_0$ we have
$$\Re N(\mu - w(\mu_0)) = \| a \|^2 - \| b - w(b_0) \|^2 = 2\| a \|^2 - \| \mu - w(\mu_0)\|^2 < 0 $$
for cuspidal parameters $\mu \not \in W\mu_0$ with $Q(\mu) \not= 0$ using the Jacquet-Shalika bounds, hence $h_Z$ decays exponentially in $Z$ on the cuspidal spectrum away from $W\mu_0$. 
 
We define the integral transforms
\begin{align*}
    \begin{aligned}
& \Phi_{\omega_4}(y)=\int_{\Re \mu=0} h(\mu) K_{\omega_4}(y ; \mu) \operatorname{spec}(\mu) \mathrm{d} \mu, \\
& \Phi_{\omega_5}(y)=\int_{\Re \mu=0} h(\mu) K_{\omega_4}(-y ;-\mu) \operatorname{spec}(\mu) \mathrm{d} \mu, \\
& \Phi_{\omega_6}\left(y_1, y_2\right)=\int_{\Re \mu=0} h(\mu) K^{\text{sym}}_{w_6}\left(\left(y_1, y_2\right) ; \mu\right) \operatorname{spec}(\mu) \mathrm{d} \mu .
\end{aligned}
\end{align*}
where
$$\text{spec}(\mu) = (\mu_1 - \mu_2)(\mu_2 - \mu_3)(\mu_3 - \mu_1)  \tan \Big(\frac{\pi}{2}(\mu_1 - \mu_2)\Big)\tan \Big(\frac{\pi}{2}(\mu_2 - \mu_3)\Big)\tan \Big(\frac{\pi}{2}(\mu_3 - \mu_1)\Big).$$
Note that this has polar divisors whenever $\mu_i - \mu_j$ ($i \not= j$) is an odd integer and zeros whenever $\mu_i - \mu_j$ ($i \not= j$) is an even integer. 




\begin{lemma}\label{lem.Phi4Truncation} 
     For any $A, \varepsilon > 0$, $j \in \Bbb{N}_0$ 
     we have \begin{align*}
       |y|^j \Phi^{(j)}_{\omega_4}(y), \,\, |y|^j\Phi^{(j)}_{\omega_5}(y)\ll_{A, \varepsilon, j, h}   |y|^{j/3} \min\big(|y|^A, |y|^{-1/6+\varepsilon}\big).
    \end{align*}
    If $h = h_Z$ as in \eqref{hz}, then the dependence on $Z$ in the implied constant is of the form $\exp(BZ)$ for some $B = B(A, j) > 0$. 
\end{lemma}

\begin{proof}
To show rapid decay near zero, we shift the contour in \eqref{K4Kernel} to the far left. The remaining integral is of size $  |y|^A$. This picks up various residues at $s = \mu_j - n$, $n \in \Bbb{N}_0$, $j = 1, 2, 3$. In the residues we shift the $\mu_j$-contour appropriately to get the desired decay. Note that the poles of the gamma functions in these residues are cancelled by the zeros of the spectral measure, and 
the poles of the spectral measure are cancelled by the zeros of $h$ imposed in \eqref{zeros}. 

For $y \geq 1$, we differentiate $j$ times with respect to $y$ under the integral sign and shift the contour to $\Re s = 1/6 - j/3 - \varepsilon$ where the $s$-integral is just absolutely convergent. This gives the desired bound $  |y|^{j/3 - 1/6 + \varepsilon}$. The residues are of smaller order of magnitude. Note that the contour shifts introduce a dependence of the form $\exp(BZ)$ if $h=h_Z$. 
\end{proof}


    

\begin{lemma}\label{lem.Phi6Truncation}
    For any $j_1, j_2\geq 0$, we have 
   \begin{displaymath}
   \begin{split}
        &y_1^{j_1}y_2^{j_2}\frac{\mathrm{d}^{j_1+j_2}}{\mathrm{d} y_1^{j_1}\mathrm{d} y_2^{j_2}}\Phi_{\omega_6}(y_1,y_2)\ll_{j_1,j_2, h} (1+|y_1|^{1/2}+|y_1|^{1/3}|y_2|^{1/6})^{j_1}(1+|y_2|^{1/2}+|y_2|^{1/3}|y_1|^{1/6})^{j_2}.
        \end{split}
    \end{displaymath}
 {If $j_1 = j_2 = 0$, then for any $A, \varepsilon > 0$ we have 
$$ \Phi_{\omega_6}(y_1,y_2)
         \ll_{A, \varepsilon, h} \min\big(|y_1y_2^2|,|y_1^2y_2|\big)^A(1+|y_1|+|y_2|)^{\varepsilon}.$$
}
     If $h = h_Z$ as in \eqref{hz}, then the dependence on $Z$ in the implied constants is of the form $\exp(BZ)$ for some $B = B(A, j_1,j_2) > 0$. 
\end{lemma}

\emph{Remark:} The parameter $A$ in this lemma is not necessarily the same as in \eqref{zeros}, but for each $A$ there is a corresponding choice in \eqref{zeros}, and correspondingly a choice of $h = h_A$, that provides the bounds claimed in this lemma.

\begin{proof}  {The first bound} follows with $T=1$ from \cite[Lemma 9]{Blomer2015OnTS}. To show rapid decay if (say) $ {|y_1^2y_2|}$ is small, we proceed as in the previous proof and shift  {the contours to the left. This can become extremely messy, unless one chooses the right set-up, so we describe the details. Recall from Lemma \ref{kernelequal} and the definition \eqref{SymKernelMellin} that 
\begin{displaymath}
    \begin{split}
K^{\text{sym}}_{w_6}\Big(\frac{1}{4\pi^2} y; -\mu\Big) &= \frac{3}{32\pi^6} \int_{-i\infty}^{i\infty}  \int_{-i\infty}^{i\infty}|y_1|^{-2u_1-u_2} |y_2|^{-2u_2 - u_1} \Gamma(1 - 3u_1 - 3u_2)\sum_{\eta_1\eta_2 = \epsilon_1\epsilon_2}e^{-\frac{\pi i}{2} \epsilon_1\eta_1(3u_1+3u_2 - 1)}\\
&
\prod_{j=1}^3 \Gamma\left(2u_1 + u_2+\mu_j\right)\left(\prod_{j=1}^3 \cos \left(\frac{\pi\left(2u_1 + u_2+\mu_j\right)}{2}\right) + \frac{\eta_1}{i} \prod_{j=1}^3 \sin \left(\frac{\pi\left(2u_1 + u_2+\mu_j\right)}{2}\right)\right)\\
&\prod_{j=1}^3 \Gamma\left(2u_2 + u_1 { -} \mu_j\right)\left(\prod_{j=1}^3 \cos \left(\frac{\pi\left(2u_2 + u_1 { -} \mu_j\right)}{2}\right) +\frac{\eta_2}{i} \prod_{j=1}^3 \sin \left(\frac{\pi\left(2u_2 + u_1 { -} \mu_j\right)}{2}\right)\right)  \frac{\mathrm{d}u_1\, \mathrm{d}u_2}{(2\pi i)^2}. 
\end{split}
\end{displaymath}
where $\epsilon_j = \text{sgn}(y_j)$. 
Recall that initially both contours have an unbounded part which is slightly negative, but pass the poles to the right. We now shift the $u_1$ contour to $\Re u_1 = -A$. We may assume that $2A$ is not an integer.  The remaining integral satisfies the desired bound, and we have to analyze the residues that we pick up on the way. Without loss of generality we can assume that $\mu_1, \mu_2, \mu_3$ are pairwise distinct, so that we only pick up simple poles. (Once we have proved this case, the case where one or more values coincide follows from Cauchy's integral theorem.) The poles come from two sources, namely
\begin{itemize}
    \item $u_1 =- \frac{1}{2}(\mu_j + u_2+ n_0)$, $n_0 \in \{0, 1,  \ldots, [2A]\}$, $j = 1, 2, 3$;
    \item $u_1 = \mu_j - 2u_2 - n_0$, $n_0 \in \{0, 1,  \ldots, [A]\}$, $j = 1, 2, 3$. 
\end{itemize}
In both cases we compute the residue explicitly. For concreteness let us choose $j = 2$, say, and let us choose coordinates $\mu_2$ and $\nu_3$ on the spectral plane $\mathfrak{a}$; i.e.\ $$(\mu_1, \mu_2, \mu_3) = \Big(\frac{1}{2}(-\mu_2-3\nu_3), \mu_2, \frac{1}{2}(-\mu_2 + 3\nu_3)\Big).$$}

 {In these coordinates,  the first class of residues equals
\begin{displaymath}
   \begin{split}
        \frac{3}{32\pi^6}&\int_{-i\infty}^{i\infty}|y_1^2y_2|^{\frac{1}{2}(\mu_2 + u_2 + n_0)} |y_2^2 y_1|^{-u_2} \frac{(-1)^{n_0}}{2n_0!} \prod_{\pm} \Gamma\Big(-\frac{3}{2} (\mu_2 \pm  \nu_3) - n_0\Big) \Gamma\Big(
  \frac{3}{2} (u_2 \pm \nu_3) - \frac{n_0}{2}\Big)\cos\Big(\frac{3\pi}{4} (\mu_2 \pm \nu_3)\Big)  \\
  & \Gamma\Big(-\frac{n_0}{2} + \frac{3}{2} (u_2 - \mu_2)\Big) \Gamma\Big(
  \frac{3}{2} n_0 + 1 + \frac{3}{2} (\mu_2 - u_2)\Big)  \sin\Big(\frac{3\pi}{2}(\mu_2 - u_2) + \frac{\pi}{2}n_0\Big)\begin{cases}\cos(\frac{3\pi}{2}u_2+ \frac{\pi}{2}n_0), &y_1 < 0 \\ (-1)\cos(\frac{3\pi}{2}\nu_{ 3}), & y_1 > 0\end{cases} \,\,\frac{\mathrm{d} u_2}{2\pi i}.
    \end{split}
\end{displaymath}
(This is actually a finite computation, since the trigonometric functions are periodic in $n_0$ modulo 4.) The first important observation is that 
$$\Gamma\Big(-\frac{n_0}{2} + \frac{3}{2} (u_2 - \mu_2)\Big) \Gamma\Big(
  \frac{3}{2} n_0 + 1 + \frac{3}{2} (\mu_2 - u_2)\Big)  \sin\Big(\frac{3\pi}{2}(\mu_2 - u_2) + \frac{\pi}{2}n_0\Big)$$
is a polynomial in $u_2 - \mu_2$ and hence holomorphic. 
The second observation is that
$$\text{spec}\Big(\frac{1}{2}(-\mu_2-3\nu_3), \mu_2, \frac{1}{2}(-\mu_2 + 3\nu_3)\Big)h\Big(\frac{1}{2}(-\mu_2-3\nu_3), \mu_2, \frac{1}{2}(-\mu_2 + 3\nu_3)\Big)\prod_{\pm} \Gamma\Big(-\frac{3}{2} (\mu_2 \pm  \nu_3) - n_0\Big) \cos\Big(\frac{3\pi}{4} (\mu_2 \pm \nu_3)\Big)$$
is also holomorphic by \eqref{zeros}, noting that $\frac{3}{2}(\mu_2 + \nu_3) =  \mu_2 - \mu_1$  and $\frac{3}{2}(\mu_2 - \nu_3) =  \mu_2 - \mu_3$. We can now shift the $\mu_2$-contour to $\Re \mu_2 = 2A - n_0$ without crossing any poles, getting an acceptable bound. 
}

 {The second class of residues equals 
\begin{displaymath}
   \begin{split}
        \frac{3}{32\pi^6}&\int_{-i\infty}^{i\infty}|y_1^2y_2|^{n_0 - \mu_2 + 2u_2} |y_2^2 y_1|^{-u_2} \frac{(-1)^{n_0}}{n_0!} \prod_{\pm} \Gamma\Big(\frac{3}{2} (\mu_2 \pm  \nu_3) - n_0\Big) \Gamma\Big(
  \frac{3}{2} (\mu_2 \pm \nu_3) - 3u_2- 2n_0\Big)\cos\Big(\frac{3\pi}{4} (\mu_2 \pm \nu_3)\Big)  \\
  & \Gamma \big(-2n_0 - 3 (u_2 - \mu_2) \big) \Gamma \big(
  3n_0 + 1 - 3(\mu_2 - u_2) \big)  \sin\big(3\pi(\mu_2 - u_2) \big)\begin{cases}\cos (\frac{3\pi}{2}(\mu_2 - 2 u_2) ), & y_2 < 0\\ \cos (\frac{3\pi}{2} \nu_3 ), & y_2 > 0\end{cases} \,\,  \frac{\mathrm{d} u_2}{2\pi i}.
    \end{split}
\end{displaymath}}
 {Again the product $\Gamma \big(-2n_0 - 3 (u_2 - \mu_2) \big) \Gamma \big(
  3n_0 + 1 - 3(\mu_2 - u_2) \big)  \sin\big(3\pi(\mu_2 - u_2) \big)$ is a polynomial in $\mu_2 - u_2$. 
 As above the product of  $\prod_{\pm} \Gamma (\frac{3}{2} (\mu_2 \pm  \nu_3) - n_0 )  \cos (\frac{3\pi}{4} (\mu_2 \pm \nu_3) ) $ times the  function $h$ and the spectral measure is holomorphic. Hence  we can shift $\mu_2$ to the left (with an acceptable bound), picking up poles only at 
$$\mu_2 = 2u_2 \mp \nu_3  - \frac{2}{3} m_0, \quad m_0 = 0, 1,  \ldots$$
with residues 
\begin{displaymath}
   \begin{split} 
   \frac{3}{32\pi^6}\int_{-i\infty}^{i\infty}&|y_1^2y_2|^{\pm\nu_3 + n_0 + \frac{2}{3}m_0}  |y_2^2 y_1|^{-u_2} \frac{(-1)^{n_0}}{n_0!} \frac{2}{3} \frac{1}{(2n_0+m_0)!} \Gamma(-m_0 - 2 n_0 \mp 3 \nu_3) \Gamma(-m_0 - n_0 + 3 u_2)\Gamma(-m_0 - n_0 \mp 3 \nu_3 + 3 u_2) \\
  &\Big(\Gamma(-2 m_0 - 
   2 n_0 \mp 3 \nu_3 + 3 u_2)\Gamma(
  1 + 2 m_0 + 3 n_0 \pm 3 \nu_3 - 3 u_2) \sin(3\pi(u_2 \mp \nu_3)) \Big) \\
  &\cos\Big(\frac{3\pi}{2} \nu_3 \Big)\cos\Big(\frac{3\pi}{2}u_2 - \frac{\pi}{2}  m_0 \Big) \cos\Big(\frac{3\pi}{2}(u_2 \mp \nu_3) - \frac{\pi}{2}m_0\Big) \begin{cases} (-1)^{m_0}, & y_2 > 0\\ 1, & y_2 < 0 \end{cases}  \frac{\mathrm{d} u_2}{2\pi i}.
     \end{split}
\end{displaymath}
 times the contribution of $h$ and the spectral measure 
\begin{equation}\label{extra}
\text{spec}\Big(\frac{1}{3}m_0 \mp \nu_3 - u_2, -\frac{2}{3}m_0 \mp \nu_3 + 2u_2, \frac{1}{3}m_0 \pm 2\nu_3 - u_2\Big)h\Big(\frac{1}{3}m_0 \mp \nu_3 - u_2, -\frac{2}{3}m_0 \mp \nu_3 + 2u_2, \frac{1}{3}m_0 \pm 2\nu_3 - u_2\Big).
\end{equation}
Again the big parenthesis in the second line of the $u$-integral is a polynomial in $u_2 \mp \nu_3$, and 
$$ \Gamma(-m_0 - 2 n_0 \mp 3 \nu_3) \Gamma(-m_0 - n_0 \mp 3 \nu_3 + 3 u_2)\cos\Big(\frac{3\pi}{2} \nu_3 \Big)  \cos\Big(\frac{3\pi}{2}(u_2 \mp \nu_3) - \frac{\pi}{2}m_0\Big) $$
times \eqref{extra} is also holomorphic by \eqref{zeros}. }
 {Hence, in these residues, we can shift $\nu_3$ to $\pm (A - n_0 - \frac{2}{3}m_0)$, getting an acceptable bound.}

 {As in the previous proof,}  the contour shifts introduce a dependence of the form $\exp(BZ)$ if $h=h_Z$.  
    \end{proof}

We complement this with the following stronger bound for larger $y_1, y_2$. This bound is likely sharp and might be useful in other situations, too. We need the full strength of this result to estimate certain error terms in the course of the proof of Theorem \ref{thm2}. 

\begin{lemma}\label{Kw6} a) We have
  $$K^{\text{sym}}_{w_6}(y; \mu) \ll_{\mu} (2 + |\log |y_1|| + |\log|y_2||)^3 (1+ |y_1|)^{-1/4}(1 + |y_2|)^{-1/4} $$
and a fortiori the same bound for $\Phi_{\omega_6}(y_1, y_2)$.    

b) Let $Z > 1$ be a  parameter. Suppose that $\max(|y_1|, |y_2|) \geq c\min(|y_1|, |y_2|)$ for some sufficiently large constant $c > 0$, i.e.\ $|y_1|$ and $|y_2|$ are of different order of magnitude. Then
\begin{equation}\label{Kasymp}
\Phi_{\omega_6}(y_1, y_2) {= \sum_{\sigma_1, \sigma_2 \in \{\pm 1\} } \frac{e(2\sigma_1 ||y_1|^{1/3} + \sigma_2 |y_2|^{1/3}|^{3/2})}{|y_1|^{5/12} |y_2|^{1/3} + |y_1|^{1/3} |y_2|^{5/12}} \Lambda_{\sigma_1,\sigma_2}(y_1, y_2) +  O\Big(\frac{Z}{|y_1| + |y_2|}\Big)    + O(Z^{-A}) }
\end{equation}
for any $A, \varepsilon > 0$ {where
$$|y_1|^{j_1} |y_2|^{j_2}\Lambda_{\sigma_1,\sigma_2}^{(j_1, j_2)}(y_1, y_2)  \ll_{j_1, j_2} {Z^{1+j_1+j_2}}$$
for all $j_1, j_2\in \Bbb{N}_0$. 

All implied constants may depend on $\varepsilon, A, h$ and $\mu$.}
\end{lemma}

\emph{Remarks {and notation}:} It is easy to see that for both parts the dependence on $\mu$ is polynomial. We will apply part b) later with $Z = X^{\varepsilon}$. It also holds (with a different function $\Lambda$ depending on $\mu$) when $\Phi_{\omega_6}$ is replaced with $K_{\omega_6}$ and we assume $\min(|y_1|^2 |y_2|, |y_1| |y_2|^2) \geq Z$. The integration over $\mu$ is needed only in the first sentence of the proof where we refer to Lemma \ref{lem.Phi6Truncation}. {Corresponding to the decomposition in \eqref{Kasymp} we write
\begin{equation}\label{phidecomp}
\Phi_{\omega_6}(y_1, y_2) = \Phi^{\sharp}_{\omega_6}(y_1, y_2) + \Phi^{\flat}_{\omega_6}(y_1, y_2),
\end{equation}
provided that  the assumptions on $y_1, y_2$ of part b) of the previous lemma are satisfied (i.e.\ they are of different order of magnitude), where $\Phi^{\sharp}_{\omega_6}$ corresponds to the main term in \eqref{Kasymp} and $\Phi^{\flat}_{\omega_6}(y_1, y_2) \ll Z(|y_1| + |y_2|)^{-1} + Z^{-A}$. }

\begin{proof} 

Part a) is  \cite[Lemma 22]{Buttcane2022}. 

Part b) {{First note that Lemma \ref{lem.Phi6Truncation} implies that the statement is trivially true if $\min(|y_1|^{1/2}|y_2|, |y_2|^{1/2}|y_1|)< Z^{-1/20}$ or $\max(|y_1|,|y_2|)<Z$. We will henceforth assume $\min(|y_1|^{1/2}|y_2|, |y_2|^{1/2}|y_1|)\geq Z^{-1/20}$ and $\max(|y_1|,|y_2|)\geq Z$ for the rest of the proof. The remaining cases}} follow from the integral representations \eqref{doubleBessel} and a careful stationary phase analysis, based on Lemma 8.1 and Proposition 8.2 in \cite{BKY}. We write  ${\tt R}, {\tt U}, {\tt Q}, {\tt Y}, {\tt V}$ for the parameters in these results. We will frequently use the well-known asymptotic
\begin{equation}\label{Bes}
J^{\pm}_{\mu}(y) = \sum_{\sigma \in \{\pm 1\}} e^{2\sigma iy} W_{\pm, \sigma, \mu}(y)
\end{equation}
with $$y^j {W_{\pm, \sigma, \mu}^{(j)}(y) \ll_{\mu, j} \begin{cases}  1 + |\log y|, & y \leq 1,\\y^{-1/2}, & y \geq 1\end{cases}}$$ for all $j \in \Bbb{N}_0$; see e.g.\  \cite[Lemma 15]{BlMi}. 
We write $Y_j = |y_j|^{1/2}${, $M=\min(Y_1Y_2^2, Y_2Y_1^2)\geq Z^{-1/20}$,} and treat each case in \eqref{doubleBessel} separately. Note that we have $\max(Y_1,Y_2)\geq Z^{1/2}$. {By a negligible error we mean a quantity of size $Z^{-A}$.}

1) {Analysis of} $\mathcal{J}_1$: We restrict the integral defining $\mathcal{J}_1$ smoothly to $u \asymp U$ and assume without loss of generality $U \geq 1$ (otherwise exchange $y_1, y_2$). {Using  \eqref{Bes}, we write it as a finite linear combination of oscillatory integrals of the form
\begin{equation}\label{int}
\int_{u \asymp U}  w(u) e^{4 \sigma \pi i \phi_{\pm}(u)} du
\end{equation}
with $\sigma \in \{\pm 1\}$, the phase
\begin{equation*}
\phi_{\pm}(u) = Y_1 \sqrt{ 1+ u^2} \pm Y_2 \sqrt{1 + 1/u^2}
\end{equation*}
and a weight function $w$ satisfying
\begin{align}\label{wDerBound}
    U^n w^{(n)}(u) \ll_n \frac{(1 + \delta_{Y_1 U \leq 1}|\log Y_1U|)(1 + \delta_{Y_2 \leq 1}|\log Y_2|)}{(1 + Y_1U )^{1/2} (1+Y_2)^{1/2}U}
\end{align}
for all $n \in \Bbb{N}_0$. } Then
$$\phi_{{\pm}}'(u) = \frac{u^3 Y_1 \mp Y_2}{u^2 \sqrt{1 + u^2}} 
, \quad \phi_{{\pm}}^{(n)}(u) \ll \frac{Y_1}{U^{n+1}} + \frac{Y_2}{U^{n+2}}\,\,\,  (n \geq 2).$$

Let us first assume 
$U \not \asymp (Y_2/Y_1)^{1/3}$. More precisely, write $A= c_0 + Z^{1/20}M^{-1/3}$ for some sufficiently large constant $c_0$ and let us assume   $$U \geq (Y_2/Y_1)^{1/3} A \quad \text{or} \quad U \leq (Y_2/Y_1)^{1/3} A^{-1/2}.$$  
Then $\phi_{\pm}'(u) \gg Y_1 + Y_2/U^3$, and we can now apply \cite[Lemma 8.1]{BKY} with 
$${\tt U} = U, \quad {\tt R} = Y_1 + \frac{Y_2}{U^3}, \quad {\tt Q} = U,\quad {\tt Y} = \frac{Y_1}{U} + \frac{Y_2}{U^2}$$
so that ${\tt R}{\tt U} \gg UY_1  {+ Y_2/U^2 \geq Y_1^{2/3} {Y_2}^{1/3}A \geq M^{1/3}A }  \geq Z^{1/20}$ and $$\frac{{\tt Q}{\tt R}}{{\tt Y}^{1/2}} \asymp \frac{U^2Y_1 + Y_2/U}{(Y_1U)^{1/2} + Y^{1/2}_2} \gg \begin{cases}   U^{3/2} Y_1^{1/2} \geq Y_2^{1/3}Y_1^{1/6} AU^{1/2} , & Y_1 U \geq Y_2, Y_1U^3 \geq Y_2A^3\\  Y_2 Y_1^{-1/2} U^{-2}  \geq Y_1^{1/6} Y_2^{1/3} A, &Y_1 U \geq Y_2, Y_1U^3 \leq Y_2A^3 \\ U^2 Y_1Y_2^{-1/2} \geq A^2 Y_1^{1/3} Y_2^{1/6}, & Y_1 U \leq Y_2, Y_1U^3 \geq Y_2A^3 \\ Y_2^{1/2} U^{-1} \geq A^{1/2} Y_1^{1/3} Y_2^{1/6}, & Y_1 U \leq Y_2, Y_1U^3 \leq Y_2A^3\end{cases} \,\,\,\geq M^{1/6} A^{1/2} \geq Z^{1/40}.$$
Hence \eqref{int} is negligible.   {The same argument holds for $\phi_-$ if $M>Z$ without any assumption on $U$: in this case we always have  $\phi_-'(u) \gg Y_1 + Y_2/U^3$, and we can run the same argument by formally putting $A = 1.$}

Assume next that {$M>Z$,} $U \asymp (Y_2/Y_1)^{1/3}$ {and consider $\phi_+$.} Note that this can only happen if $Y_2 \gg Y_1$.  Then there is  one  stationary point  at ${u_0} = (Y_2/Y_1)^{1/3}$, and $$\phi_{{+}}''(u) = {\frac{u^3 Y_1 + 2 Y_2 + 3 u^2 Y_2}{u^3 (1 + u^2)^{3/2}}}\gg \frac{Y_1}{U^3} + \frac{Y_2}{U^4}  {\asymp} \frac{Y_1^2}{Y_2} + \frac{Y_1^{4/3}}{Y_2^{1/3}} \asymp \frac{Y_1^{4/3}}{Y_2^{1/3}}.$$ Hence we can apply \cite[Proposition 8.2]{BKY}  with
$${\tt V} = U, \quad  {\tt Q} = U,\quad {\tt Y} = \frac{Y_1}{U} + \frac{Y_2}{U^2} \asymp \frac{Y_1^{4/3}}{Y_2^{1/3}} + Y_1^{2/3}  {Y_2}^{1/3} \gg  {{Z}^{1/3}} \geq 1.$$
{In this way we see that \eqref{int} is, up to a negligible error, of the form
$$\frac{e^{4\pi \sigma i \phi_+(u_0)}}{(\phi''_+(u_0))^{1/2}} \sum_{n \leq A'} p_n(u_0)$$
for some $A'$ depending on $A$ and functions $p_n$ of the form $(\phi''(u_0))^{-n}$ times a linear combination of products  of $w^{(j)}(u_0)$ and $\phi^{(j)}(u_0)$, bounded by $w(u_0)$. We compute 
$$\phi_+(u_0) = (Y_1^{2/3} + Y_2^{2/3})^{3/2}, \quad \phi''(u_0) = \frac{3Y_1^{4/3}}{\sqrt{Y_1^{2/3} + Y_2^{2/3}}}$$
and
$$\frac{w(u_0)}{(\phi_+''(u_0))^{1/2}} \ll \frac{1}{{(Y_1U)^{1/2} Y_2^{1/2} }   U}\Big(  \frac{Y_1^{4/3}}{Y_2^{1/3}}\Big)^{-1/2} \asymp  \frac{1}{ {Y_1^{2/3} Y_2^{5/6}}}.$$}

{Finally assume that $Z^{-1/20}\leq M\leq Z$ and $(Y_2/Y_1)^{1/3}A^{-1/2}\ll U \ll (Y_2/Y_1)^{1/3}A$. Notice that since $U\geq1$, we have $Y_2\gg Y_1A^{-3}$. If $Y_1\geq Y_2$, then $Y_1 = \max(Y_1, Y_2) \geq \sqrt{Z}$, and hence $M=Y_1Y_2^2\gg Y_1^3A^{-6}\gg Z^{3/2 - 6/15} > Z$,  which is a contradiction. Hence we have $Y_1<Y_2$, $Y_2 = \max(Y_1, Y_2)\geq\sqrt{Z}$ and $M=Y_1^2Y_2 < Z$, and so all in all $$Y_1\leq Z^{1/4} \leq Z^{1/2} \leq Y_2.$$ Then we have the coarse bound  \begin{align}\label{coarse}
    |y_1|^{j_1} |y_2|^{j_2}\frac{\mathrm{d}^{j_1+j_2}}{\mathrm{d}y_1^{j_1}\mathrm{d}y_2^{j_2}}e\left(2\sigma(\phi_\pm(u)\mp |Y_1^{2/3}\pm Y_2^{2/3}|^{3/2})\right)\ll_{j_1,j_2}(Y_1^{2/3} Y_2^{1/3}A)^{j_1+j_2} \ll Z^{j_1+j_2}.
\end{align} 
Hence we can write artificially \begin{align*}
    \int_{u \asymp U}  w(u) e^{4 \sigma \pi i \phi_{\pm}(u)} du=\frac{e(\pm 2\sigma_1 |Y_1^{2/3} \pm  Y_2^{2/3}|^{3/2})}{Y_1^{5/6} Y_2^{2/3} + Y_1^{2/3} Y_2^{5/6}}  \Big[(Y_1^{5/6} Y_2^{2/3} + Y_1^{2/3} Y_2^{5/6}) \int_{u \asymp U}  w(u) e\left(2\sigma(\phi_\pm(u)\mp |Y_1^{2/3}\pm Y_2^{2/3}|^{3/2})\right) du\Big].
\end{align*}
Note that \eqref{wDerBound} and our assumptions on $Y_1, Y_2, U$ imply
$Y_1U\gg Y_1^{2/3}Y_2^{1/3}A^{-1/2}\gg Z^{-1/20}$, which  gives us \begin{align*}
    (Y_1^{5/6} Y_2^{2/3} + Y_1^{2/3} Y_2^{5/6})w(u)\ll Z^{\varepsilon}\frac{ Y_1^{2/3} Y_2^{5/6}}{Y_1^{1/2} Y_2^{1/2} U^{3/2}} \ll Z^{\varepsilon}Y_1^{1/3} Y_2^{1/6}\frac{A^{1/4}}{U} \ll Z^{1/6+\varepsilon} \frac{A^{1/4}}{U} \ll \frac{Z^{1/4+\varepsilon}}{U}
\end{align*}
Hence   $\Lambda_{\pm\sigma_1,\pm}(y_1,y_2)$ defined to be the square brackets satisfies the  desired derivative bounds (in slightly stronger form). }{This completes the analysis of $\mathcal{J}_1$.}

2) {Analysis of} $\mathcal{J}_2$: We restrict smoothly to $u \asymp U$ where we first assume $u \geq 2$. Then the analysis is exactly as in the case of $\mathcal{J}_1$ {with different signs: with the phase
$$\phi_{\pm}(u) = Y_1 \sqrt{ u^2-1} \pm Y_2 \sqrt{1- 1/u^2}$$ 
we have
$$\phi_{\pm}'(u) =     \frac{u^3 Y_1 \pm Y_2}{u^2 \sqrt{  u^2-1}}, $$ so that a stationary point at $u_0 = (Y_2/Y_1)^{1/3}$ can only occur for $\phi_-$, and we have
$$\phi_-(u_0) =- (Y_2^{2/3} - Y_1^{2/3})^{3/2} .$$}

Next, we consider the portion $1 \leq u \leq 2$ and restrict the integral smoothly to $u = 1 + v$, $v \asymp V \ll 1$. Here (and only here) we impose the additional condition $Y_1 \not\asymp Y_2$ as in the lemma. With the phase
$$\phi(v) = Y_1 \sqrt{  (1+v)^2-1} \pm Y_2 \sqrt{1 - 1/(1+v)^2}$$ 
we then have
$$\phi'(v) = \frac{(1+v)^3 Y_1 \pm Y_2}{(1 + v)^2 \sqrt{v(2+v)}} \asymp \frac{Y_1 + Y_2}{V^{1/2}}, \quad \phi^{(n)}(v) \ll \frac{Y_1+Y_2}{V^{-1/2 + n}}.$$
Applying \cite[Lemma 8.1]{BKY} with
$${\tt R}  = \frac{Y_1+Y_2}{V^{1/2}}, \quad  {\tt Y} = V^{1/2}(Y_1+ Y_2),  \quad {\tt U} = {\tt Q} = V$$
we have ${\tt RU} = V^{1/2}(Y_1 + Y_2) {\geq Z^{\varepsilon/2} }\geq 1$ and ${\tt QR}/{\tt Y}^{1/2}  =  V^{1/4}(Y_1 + Y_2)^{1/2} {\geq Z^{\varepsilon/4}}\geq 1$ provided that $V \geq {Z^{\varepsilon}}/(Y_1 + Y_2)^2$. If this holds, the integral is negligible, otherwise we estimate trivially  
$$ \int_{v \ll Z^\varepsilon (Y_1 + Y_2)^{{-2}} } {(1+ |\log Y_1 v|)(1+ |\log Y_2 v|)}  \mathrm{d}v  \ll \frac{{Z^{\varepsilon}}}{{(Y_1 + Y_2)^2}}.$$

3)  {B}ound for $\mathcal{J}_3$: by the decay of the Bessel $K$-function we  {can assume} $Y_1 \leq  {Z}^{\varepsilon}$, $u \asymp U \leq  {Z}^{\varepsilon}/Y_1$, {up to a negligible error}. Then we have $Y_2\geq \sqrt{Z}$. We now consider the phase
 $$\phi(u) = \pm Y_2 \sqrt{1 + u^{-2}}$$
 with {$\phi'(u) \asymp Y_2/(U^2(U+1))$ and  $\phi^{(n)}(u)\ll Y_2/(U^{n+1}(U+1))$. {If $U\leq \sqrt{Y_2}Z^{-1/8}$,} we apply \cite[Lemma 8.1]{BKY}  with
 $${\tt R} = \frac{Y_2}{U^2(U+1)}, \quad {\tt Y} = \frac{Y_2}{U(U+1)}, \quad {\tt Q} = {\tt U} = U.$$}
Since $${\tt RU} = {\frac{Y_2}{U(U+1)} \geq \frac{1}{2}{Z}^{1/4}},  \quad \frac{{\tt QR}}{{\tt Y}^{1/2} }= {\Big(\frac{Y_2}{U(U+1)}\Big)^{1/2} \geq  \frac{1}{2}{Z}^{1/8}},$$  the integral is negligible. {If $\sqrt{Y_2}Z^{-1/8}<U\leq Z^\varepsilon/Y_1$, we conclude as in \eqref{coarse} that \begin{align*}
    |y_1|^{j_1} |y_2|^{j_2}\frac{\mathrm{d}^{j_1+j_2}}{\mathrm{d}y_1^{j_1}\mathrm{d}y_2^{j_2}}e\left(2\sigma(\phi(u)\mp |Y_1^{2/3}\pm Y_2^{2/3}|^{3/2})\right)\ll_{j_1,j_2, \varepsilon}Z^{j_1+j_2}.
\end{align*}
Then the same analysis as in Case 1 yields the desired result.}

4)  {B}ound for $\mathcal{J}_4$: By the decay of the Bessel $K$-function, the integral is  restricted to $\sqrt{1 - u} \ll  {Z}^{\varepsilon}\min(1/Y_1, 1/Y_2)$, up to a negligible error, hence we obtain trivially the bound $ Z/(Y_1 + Y_2)^2$.

5)  {B}ound for $\mathcal{J}_5$: by the decay of the Bessel $K$-function, this is trivially  $\ll {(Y_1+  Y_2)^{-A}}$. {(Note that by assumption $Y_1+Y_2\gg  Z^{1/2}$ is large.)}
\end{proof}


The kernel $K^{\text{sym}}_{w_6}$ satisfies the following orthogonality property:
\begin{lemma}\label{ortho} With $h$ as above and $\mu \in (i\Bbb{R})^3$, $\mu_1 + \mu_2 + \mu_3 = 0$, we have
$$\sum_{\epsilon_1, \epsilon_2 = \pm 1}\int_{[0, \infty)^2} \Phi_{\omega_6}( {\epsilon_1}y_1,  {\epsilon_2}y_2) K_{\omega_6}^{\text{{\rm sym}}}(( {\epsilon_1}y_1,  {\epsilon_2}y_2), -\mu)  \frac{\mathrm{d}y_1\, \mathrm{d}y_2}{y_1y_2} = c_0 h(\mu)$$
for some constant $0 \not= c_0 \in \Bbb{R}$. 
    \end{lemma}

\begin{proof}
This is essentially \cite[(122)]{Buttcane2022}. Note that $\textbf{sin}^{d_1}_{2}$ should be $\textbf{sin}^{d_1\ast}_{2}$ as defined in \cite[(69)]{Buttcane2022} which equals, up to a constant, our definition of $\text{spec}(\mu)$. Since the normalizations in the various papers on the ${\rm GL}(3)$ Kuznetsov formula are different and apparently not entirely consistent, we did not put any effort in tracing back the value of the constant $c_0$ (which is irrelevant for our purpose). 
\end{proof}

\subsection{The Kuznetsov formula}

With the above setup, we can now state the Kuznetsov formula.

\begin{lemma}\label{lem.Kuznetsov}
    Let $m_1,m_2,n_1,n_2\in\N$ and let $h$ be as described at the beginning of Section \ref{sec35}. Then there are absolute constants $c_{\text{{\rm cusp}}}, c_{\text{{\rm min}}}, c_{\text{{\rm max}}} > 0$ such that  \begin{align*}
        \mathcal{C}+E_{\min}+E_{\max}=\Delta+S_4+S_5+S_6,
    \end{align*}
    where $$
\begin{aligned}
& \mathcal{C}=c_{\text{{\rm cusp}}}\sum_{\pi_j\in \mathcal{B}} \frac{h\left(\mu_j\right)}{L(1, \text{{\rm Ad}}, \pi_j)} \overline{A_j\left(m_1, m_2\right)}A_j\left(n_1, n_2\right), \\
& E_{\min }=c_{\text{{\rm min}}} \mathop{\int\int}_{\Re(\mu)=0} \frac{h(-\mu)}{\prod_{1 \leq i < j \leq 3}|\zeta(1 + \mu_i - \mu_j)|^2} \overline{A_\mu\left(m_1, m_2\right)}A_\mu\left(n_1, n_2\right) \mathrm{d} \mu_1 \mathrm{d} \mu_2, \quad (\mu_3 = - \mu_1 - \mu_2) \\
& E_{\max }=c_{\text{{\rm max}}} \sum_g \int_{\Re(\mu)=0} \frac{h\left(\mu+\mu(g), \mu-\mu(g),-2 \mu\right)}{L(1, \text{{\rm Ad}}, g)|L(1 + 3\mu, g)|^2} \overline{A_{\mu, g}\left(m_1, m_2\right)} A_{\mu, g}\left(n_1, n_2\right) \mathrm{d} \mu,\\
& \Delta=\delta(m_1=n_1, m_2=n_2)  {\frac{1}{192\pi^5}}\int_{\Re(\mu)=0} h(\mu) \operatorname{spec}(\mu) \mathrm{d} \mu,\\
& S_4=\sum_{\epsilon= \pm 1} \sum_{\substack{D_2 \mid D_1 \\
m_2 D_1=n_1 D_2^2}} \frac{\tilde{S}\left(-\epsilon n_2, m_2, m_1 {;} D_2, D_1\right)}{D_1 D_2} \Phi_{\omega_4}\left(\frac{\epsilon m_1 m_2 n_2}{D_1 D_2}\right), \\
\end{aligned}$$
$$\begin{aligned}
& S_5=\sum_{\epsilon= \pm 1} \sum_{\substack{D_1 \mid D_2 \\
m_1 D_2=n_2 D_1^2}} \frac{\tilde{S}\left(\epsilon n_1, m_1, m_2 {;} D_1, D_2\right)}{D_1 D_2} \Phi_{\omega_5}\left(\frac{\epsilon n_1 m_1 m_2}{D_1 D_2}\right), \\
& S_6=\sum_{\epsilon_1, \epsilon_2= \pm 1} \sum_{D_1, D_2} \frac{S\left(\epsilon_2 n_2, \epsilon_1 n_1, m_1, m_2 ; D_1, D_2\right)}{D_1 D_2} \Phi_{\omega_6}\left(-\frac{\epsilon_2 m_1 n_2 D_2}{D_1^2},-\frac{\epsilon_1 m_2 n_1 D_1}{D_2^2}\right).
\end{aligned}
$$
Here $g$ runs over an orthonormal basis of cusp forms for the group ${\rm SL}_2(\Bbb{Z})$ and the Fourier coefficients of the minimal and maximal Eisenstein series are given as in \cite[Section 5]{Blomer2013}:
$$A_{\mu}(1, m) = \sum_{d_1d_2d_3 = m} d_1^{\mu_1} d_2^{\mu_2} d_3^{\mu_3}, \quad A_{\mu, g}(1, m) = \sum_{d_1d_2 = m} \lambda_g(d_1) d_1^{-\mu} d_2^{2\mu}$$
along with the usual Hecke relations. 
\end{lemma}




\section{Beginning of the proof of Theorem \ref{thm2} -- the easy cases}

We now start with the proof of Theorem \ref{thm2}. The starting point is an application of   the Kuznetsov formula (Lemma \ref{lem.Kuznetsov}), which yields \begin{align}\label{KuznetsovSplit}
    c_{\text{cusp}} S_m(X)+\mathcal{E}_{\min}+\mathcal{E}_{\max}=\Delta + \Sigma_4+\Sigma_5+\Sigma_6,
\end{align}
where \begin{align*}
\begin{aligned}
    &\mathcal{E}_{\min}= c_{\text{min}}\mathop{\int\int}_{\Re(\mu)=0} \frac{h(-\mu)}{\prod_{1 \leq i < j \leq 3}|\zeta(1 + \mu_i - \mu_j)|^2} \sum_n B(1, n) A_{\mu}(1, n) \overline{A_{\mu}(1, m)} V\Big(\frac{n}{X}\Big)  \mathrm{d}\mu_1\, \mathrm{d}\mu_2 ,\\
    &\mathcal{E}_{\max}= c_{\text{{\rm max}}} \sum_g \int_{\Re(\mu)=0} \frac{h\left(\mu+\mu(g), \mu-\mu(g),-2 \mu\right)}{L(1, \text{{\rm Ad}}, g)|L(1 + 3\mu, g)|^2}\sum_n B(1, n) A_{\mu, g}(1, n) \overline{A_{\mu, g}(1, m)} V\Big(\frac{n}{X}\Big) \mathrm{d}\mu ,\\
    &\Delta=\frac{B(1,m)}{192\pi^5}V\left(\frac{m}{X}\right)\int_{\re(\mu)=0}h(\mu)\mathrm{spec}(\mu)\mathrm{d}\mu,\\
     &\Sigma_4=\sum_{\epsilon=\pm1}\sum_n B(1,n)V\left(\frac{n}{X}\right)\sum_{m\mid D}\frac{\Tilde{S}(-\epsilon n,m,1;D,D^2/m)}{D^3/m}\Phi_{\omega_4}\left(\frac{\epsilon m^2n}{D^3}\right),\\
            &\Sigma_5=\sum_{\epsilon=\pm1}\sum_n B(1,n)V\left(\frac{n}{X}\right)\sum_D\frac{\Tilde{S}(\epsilon,1,m;D,nD^2)}{nD^3}\Phi_{\omega_5}\left(\frac{\epsilon m}{nD^3}\right),\\
    &\Sigma_6=\sum_{\epsilon_1,\epsilon_2=\pm1}\sum_n B(1,n)V\left(\frac{n}{X}\right)\mathop{\sum\sum}_{D_1,D_2}\frac{S(\epsilon_2 n,\epsilon_1,1,m;D_1,D_2)}{D_1D_2}\Phi_{\omega_6}\left(-\frac{\epsilon_2 nD_2}{D_1^2},-\frac{\epsilon_1 mD_1}{D_2^2}\right).
\end{aligned}
\end{align*}


We start by showing that on the right hand side only the contribution $\Sigma_6$ of the long Weyl element is relevant in the subsequent analysis. At this point we introduce the following notational convention: we will write henceforth
$$A \preccurlyeq B$$
to mean $A \ll_{\varepsilon} X^{\varepsilon} B + O(X^{-A})$ for every $A, \varepsilon > 0$. We will also frequently use Kronecker deltas in the form $\delta(...)$ to denote a condition that some variables have to satisfy. We call a quantity of size $O_A(X^{-A})$ for any $A > 0$ negligible. 

Since $m, V$ are fixed and the $\mu$-integral converges by the properties of $h$, we have \begin{align}\label{DiagonalBound}
    \Delta\ll 1.
\end{align}

By Lemma \ref{lem.Phi4Truncation}, we see immediately 
\begin{align}\label{Sum5Bound}
    \Sigma_5\ll X^{-1000} {.}
\end{align}

The treatment of $\Sigma_4$ requires slightly more work. Again by Lemma \ref{lem.Phi4Truncation} we can truncate the $D$-sum getting
 \begin{align*}
    \Sigma_4=\sum_{\epsilon=\pm1}\sum_n B(1,n)V\left(\frac{n}{X}\right)\sum_{D\preccurlyeq X^{1/3}}\frac{\Tilde{S}(-\epsilon n,m,1;mD,mD^2)}{m^2D^3}\Phi_{\omega_4}\left(\frac{\epsilon n}{mD^3}\right)+O\left(X^{-1000}\right).
\end{align*}
Recalling the definition of $\Tilde{S}$ in \eqref{ShortKloostermanSum}, we have \begin{align*}
    \Tilde{S}(-\epsilon n,m,1;mD,mD^2)=\sumast_{C_1\Mod{mD}}\sum_{\substack{C_2\Mod{mD^2}\\(C_2,D)=1}}e\left(\frac{\overline{C_1}C_2+\overline{C_2}}{D}-\frac{\epsilon nC_1}{mD}\right).
\end{align*}
In order to apply the Voronoi summation formula to the $n$-sum, we need to understand the Mellin transform of the function
$$g(n) = V\left(\frac{n}{X}\right)\Phi_{\omega_4}\left(\frac{\epsilon n}{mD^3}\right).$$
Lemma \ref{lem.Phi4Truncation} implies easily
$$\tilde{g}(\sigma + it) \ll_{j, \sigma} X^{\sigma} \Big(\frac{X}{D^3}\Big)^{j/3 - 1/6} \frac{1}{(1+ |t|)^j}$$
for any $j \in \Bbb{N}_0$, and then by interpolation for any real $j \geq0$. We use this with $j = 3\sigma + 5/2 + \varepsilon$ in the formula \eqref{defG} (where  {the} contour runs over $\Re s = \sigma$) to make the integral absolutely convergent. Shifting $\sigma$ to the right when $n\gg 1$ and $0$ for the bound, we deduce
$$G_{\pm}\Big( \frac{n}{D^3}\Big) \ll_A \Big(\frac{X}{D^3}\Big)^{2/3 } (1 + n)^{-A}$$
for $n> 0$. For fixed $m$, we conclude from   Lemma \ref{lem.Voronoi} that 
\begin{align*}
    &\sum_n B(1,n)V\left(\frac{n}{X}\right)e\left(-\frac{\epsilon nC_1}{mD}\right)\Phi_{\omega_4}\left(\frac{\epsilon n}{mD^3}\right)
    \preccurlyeq \frac{X^{2/3}}{D} \underset{n_0n \preccurlyeq 1}{\sum_{n_0\mid mD}\sum_n} \Big|B(n,n_0) S\left(-\epsilon\overline{C_1},\eta n;\frac{mD}{n_0}\right)\Big| {.} 
\end{align*}
Using Weil's bound and otherwise trivial bounds for the sums over $C_1, C_2$, we conclude
\begin{align}
\label{Sum4Bound}
    \Sigma_4\preccurlyeq   X^{2/3 } \sum_{D \preccurlyeq X^{1/3 }} D^{-1/2} \preccurlyeq X^{5/6 } . 
\end{align}
This could be improved considerably, but suffices for our purpose. 


Substituting the bounds (\ref{DiagonalBound}), (\ref{Sum5Bound}) and (\ref{Sum4Bound}) into \eqref{KuznetsovSplit}, we get \begin{align}
    c_{\text{cusp}}S_m(X)+\mathcal{E}_{\min}+\mathcal{E}_{\max}= \Sigma_6+O\left(X^{5/6+\varepsilon}\right). \nonumber
\end{align}

\section{The Eisenstein contribution}

In this section we show that $|\mathcal{E}_{\min}|+|\mathcal{E}_{\max}| \preccurlyeq X^{12/13}$. The underlying reason is that the Voronoi summation formula has no main term, and hence the sequence $B(1, n)$ remembers that it is orthogonal to the coefficients of Eisenstein series.

\subsection{Minimal Eisenstein series}

 In order to bound the term $\mathcal{E}_{\text{min}}$ we only consider the $n$-sum
 $$\sum_n B(1, n) A_{\mu}(1, n) V\Big(\frac{n}{X}\Big) = \mathop{\sum\sum\sum}_{d_1, d_2, d_3} B(1, d_1d_2d_3) d_1^{\mu_1} d_2^{\mu_2} d_3^{\mu_3} V\Big( \frac{d_1d_2d_3}{X}\Big).$$
We apply smooth dyadic partitions of unity for each of the three variables and localize $d_j \asymp D_j$ with $D_1D_2D_3 \asymp X$. Suppose without loss of generality that $D_1$ is the largest, so that $D_1 \gg X^{1/3}$. We write
$$d_1 = \delta_{1} d_1', \quad d_2 = \delta_{1} d_2', \quad (d_1', d_2') = 1$$
and then
$$d_1 {'} = \delta_{2}d_1'', \quad d_3 = \delta_{2} d_3', \quad (d_1'', d_3') = 1$$
so that (after obvious relabeling) we are left with 
$$ \mathop{\sum\sum\sum\sum\sum}_{\substack{\delta_1, \delta_2, d_1, d_2, d_3\\ (d_1, d_3) = (d_1\delta_2, d_2) = 1}} B(1, \delta_1^2\delta_2^2 d_1d_2d_3) (\delta_1\delta_2d_1)^{\mu_1} (\delta_1d_2)^{\mu_2} (\delta_2d_3)^{\mu_3} V\Big( \frac{\delta_1^2\delta_2^2d_1d_2d_3}{X}\Big)W\Big( \frac{\delta_1\delta_2d_1}{D_1}\Big) W\Big( \frac{\delta_1d_2}{D_2}\Big) W\Big( \frac{\delta_2d_3}{D_3}\Big)  .$$
We apply   the Hecke relation \eqref{mob2}, obtaining 
\begin{displaymath}
\begin{split}
   & \mathop{\sum\sum\sum\sum\sum\sum\sum}_{\substack{\delta_1, \delta_2,   d_2, d_3, a, b, c\\ (abc, d_3) = (abc\delta_2, d_2) = 1\\ b \mid c\mid \delta_1^2 \delta_2^2 
  }}  \mu(b)\mu(c) B\Big(\frac{c}{b}, \frac{\delta_1^2\delta_2^2d_2d_3}{c}\Big) B(1, a) (\delta_1\delta_2abc)^{\mu_1} (\delta_1d_2)^{\mu_2} (\delta_2d_3)^{\mu_3} \\
 &\quad\quad\quad\quad\quad\quad\times V\Big( \frac{\delta_1^2\delta_2^2abcd_2d_3}{X}\Big)W\Big( \frac{\delta_1\delta_2abc}{D_1}\Big) W\Big( \frac{\delta_1d_2}{D_2}\Big) W\Big( \frac{\delta_2d_3}{D_3}\Big) . 
 \end{split}
 \end{displaymath}
 We detect the condition $(a, d_2d_3) = 1$ by M\"obius inversion getting
\begin{displaymath}
\begin{split}
  & \mathop{\sum\sum\sum\sum\sum\sum\sum\sum}_{\substack{\delta_1, \delta_2,   d_2, d_3, a, b, c, f\\ (bc, d_3) = (bc\delta_2, d_2) = 1\\ b \mid c\mid \delta_1^2 \delta_2^2,  
  f\mid d_2d_3}} \mu(f)  \mu(b)\mu(c) B\Big(\frac{c}{b}, \frac{\delta_1^2\delta_2^2d_2d_3}{c}\Big) B(1, af) (\delta_1\delta_2afbc)^{\mu_1} (\delta_1d_2)^{\mu_2} (\delta_2d_3)^{\mu_3} \\
 &\quad\quad\quad\quad\quad\quad\times V\Big( \frac{\delta_1^2\delta_2^2afbcd_2d_3}{X}\Big)W\Big( \frac{\delta_1\delta_2afbc}{D_1}\Big) W\Big( \frac{\delta_1d_2}{D_2}\Big) W\Big( \frac{\delta_2d_3}{D_3}\Big) . 
 \end{split}
 \end{displaymath}
Applying \eqref{mob2} once again, this is
\begin{displaymath}
\begin{split}
  & \mathop{\sum\sum\sum\sum\sum\sum\sum}_{\substack{\delta_1, \delta_2,   d_2, d_3, b, c, f\\ (bc, d_3) = (bc\delta_2, d_2) = 1\\ b \mid c\mid \delta_1^2 \delta_2^2,  
  f\mid d_2d_3}}  \mathop{\sum\sum\sum}_{\substack{r, s, t \\ s\mid t\mid f}}\mu(f) \mu(s)\mu(t) B\Big(\frac{t}{s}, \frac{f}{t}\Big) B(1, r)  \mu(b)\mu(c) B\Big(\frac{c}{b}, \frac{\delta_1^2\delta_2^2d_2d_3}{c}\Big)   (\delta_1\delta_2rstfbc)^{\mu_1} \\
 &\quad\quad\quad\quad\quad\quad\times (\delta_1d_2)^{\mu_2} (\delta_2d_3)^{\mu_3} V\Big( \frac{\delta_1^2\delta_2^2rstfbcd_2d_3}{X}\Big)W\Big( \frac{\delta_1\delta_2rstfbc}{D_1}\Big) W\Big( \frac{\delta_1d_2}{D_2}\Big) W\Big( \frac{\delta_2d_3}{D_3}\Big) . 
 \end{split}
 \end{displaymath}
 
 After splitting in dyadic segments $r\asymp R$, by Voronoi summation the $r$-sum is easily seen to be of size $\| \mu \|^{O(1)}$, uniformly in all other parameters, so that we obtain
 \begin{equation*}
 \mathcal{E}_{\text{min}} \preccurlyeq \mathop{\sum\sum\sum\sum  }_{ \delta_1d_2 \leq D_2, \delta_2 d_3 \leq D_3}1 \preccurlyeq D_2D_3 \ll X^{2/3}. 
 \end{equation*}


\subsection{Maximal Eisenstein series}

In order to bound the term $\mathcal{E}_{\text{max}}$ we again consider only the $n$-sum
 \begin{align*}
     \mathcal{E}_{\max}\ll& \sup_{|\mu|, |\mu(g)|\preccurlyeq 1}\Big|\sum_n B(1, n) A_{\mu, g}(1, n) V\Big(\frac{n}{X}\Big)\Big| + O\left(X^{-99}\right)=\sup_{|\mu|, |\mu(g)|\preccurlyeq 1}\Big|\mathop{\sum\sum}_{d_1, d_2 } B(1, d_1d_2) \lambda_g(d_1) d_1^{-\mu} d_2^{2\mu}   V\Big( \frac{d_1d_2}{X}\Big)\Big|+O\left(X^{-99}\right).
 \end{align*}
Here we used that the decay of the function $h$ allows us to restrict $|\mu|, |\mu(g)|\preccurlyeq 1$ at the cost of a negligible error, with $\mu(g)$ being the spectral parameter of $g$. In the following we denote by $\theta \leq 7/64$ an admissible exponent for the Ramanujan conjecture for the ${\rm SL}(2)$ cusp form $g$ and recall the Rankin-Selberg bound $\sum_{n \ll X} |\lambda_g(n)|^2  \ll X$ (with polynomial dependence on $\mu(g)$).

We first pull out the divisors of $(d_1,d_2)$. Write $$\delta=(d_1,d_2), \quad  \delta_{1}=(d_1/\delta,\delta^\infty), \quad \delta_{2}=(d_2/\delta,\delta^\infty), \quad n_1=\frac{d_1}{\delta\delta_{1}}, \quad  n_2=\frac{d_2}{\delta\delta_{2}}.$$
Here and in the following we write $a\mid b^{\infty}$ to mean that $a$ has only prime divisors occurring in $b$ and we write
$$(a, b^{\infty}) := \max_n (a, b^n).$$

Again we partition dyadically $n_1 \asymp D_1$, $n_2 \asymp D_2$ with $D_1D_2 \ll X$ and apply \eqref{mob1} and the GL(2) Hecke relations. It thus suffices to bound \begin{align}\label{EmaxS}
    & \sup_{D_1D_2\ll X}\sum_{\delta}\mathop{\sum\sum}_{\delta_{1},\delta_{2}\mid \delta^\infty}(\delta\delta_1)^{\theta} 
    \big|\mathcal{S}\big|,
\end{align}
where \begin{align}\label{start2}
    \mathcal{S}=\mathop{\sum\sum}_{\substack{n_1, n_2\\ (n_1, n_2) = (n_1n_2,\delta) = 1}} B(1, n_1)B(1,n_2) \lambda_g(n_1) F_\mu\left(\frac{n_1}{D_1},\frac{n_2}{D_2}\right),
\end{align}
with \begin{align*}
    F_\mu\left(\frac{n_1}{D_1},\frac{n_2}{D_2}\right)=n_1^{-\mu} n_2^{2\mu}   V\Big( \frac{\delta^2\delta_{1}\delta_{2}n_1n_2}{X}\Big)W\Big(\frac{n_1}{D_1}\Big)W\Big(\frac{n_2}{D_2}\Big)
\end{align*}
for some fixed $W\in C_c^\infty([1/2,5/2])$. 

We are going to bound $\mathcal{S}$ in two ways. Applying M\"obius inversion 
and then \eqref{mob2} on the $n_2$-sum gives us \begin{align*}
    \sum_{(n_2,\delta n_1)=1}B(1,n_2)F_\mu\left(\frac{n_1}{D_1},\frac{n_2}{D_2}\right)=&\sum_{d\mid n_1\delta}\mu(d)\sum_{n_2}B(1,dn_2)F_\mu\left(\frac{n_1}{D_1},\frac{dn_2}{D_2}\right)\\
    =&\sum_{d\mid n_1\delta}\mu(d)\sum_{b\mid c\mid d}\mu(b)\mu(c)B\left(\frac{c}{b},\frac{d}{c}\right)\sum_{r}B(1,r)F_\mu\left(\frac{n_1}{D_1},\frac{bcdr}{D_2}\right).
\end{align*}
As before, Voronoi summation implies the $r$-sum is bounded by $\preccurlyeq 1$, so that we easily find  
\begin{align}\label{MaxEFirstBound}
  \mathcal{S}\preccurlyeq \sum_{n_1\ll D_1} |B(1, n_1)\lambda_g(n_1)| \preccurlyeq  D_1.
\end{align}


We now return to \eqref{start2} and focus on the $n_1$-sum. In view of the previous bound we may assume that $D_1 \gg X^{\varepsilon}$. The $n_1$-sum  is morally a ${\rm GL}(3) \times {\rm GL}(2)$ ``shifted convolution sum without shift'', and we obtain in \eqref{S2ndBound} below a second bound for $\mathcal{S}$. 

We use   Jutila's circle method
, which is set up as follows. Let $U\in C_c^\infty([1/4,3])$ be fixed such that $U(x)=1$ for $1/2\leq x\leq 5/2$. Then we have \begin{align*}
    \mathcal{R}:=&\sum_{(n_1,n_2\delta)=1}B(1,n_1)\la_g(n_1)F_\mu\left(\frac{n {_1}}{D_1},\frac{n {_2}}{D_2}\right)=\sum_{n_1} B(1,n_1)F_\mu\left(\frac{n_1}{D_1},\frac{n_2}{D_2}\right)\sum_{(r,n_2\delta)=1} \la_g(r)U\left(\frac{r}{D_1}\right)\int_0^1e\left((n_1-r)x\right)\mathrm{d}x.
\end{align*}
Let $Q, u>0$ such that $Q\ll D_1\ll Q^2$ and $Q^{-2}\ll u\ll Q^{-1}$. Let $\mathcal{Q}=\{q \,\,\text{prime} : Q\leq q \leq 2Q, q\nmid n_2\delta\}$. Then $|\mathcal{Q}|\gg Q^{1-\varepsilon}$ and $H:=\sum_{q\in\mathcal{Q}}\phi(q)\gg Q^{2-\varepsilon}$. Define \begin{align*}
    \Tilde{\mathcal{R}}:=\sum_{n_1} B(1,n_1)F_\mu\left(\frac{n_1}{D_1},\frac{n_2}{D_2}\right)\sum_{(r,n_2\delta)=1} \la_g(r)U\left(\frac{r}{D_1}\right)\int_0^1\tilde{I}_{\mathcal{Q},u}(x)e\left((n_1-r)x\right)\mathrm{d}x,
\end{align*}
with \begin{align*}
    \Tilde{I}_{\mathcal{Q},u}(x)=\frac{1}{2u H}\sum_{q\in\mathcal{Q}}\sumast_{\alpha\Mod{q}} \delta\Big( x \in  \Big[\frac{\alpha}{q}-u,\frac{\alpha}{q}+u\Big] \Big).
\end{align*}
Notice that \begin{align*}
    \int_0^1\tilde{I}_{\mathcal{Q},u}(x)e(bx)dx=\frac{1}{2u H}\sum_{q\in\mathcal{Q}}\sumast_{\alpha\Mod{q}}e\left(\frac{\alpha b}{q}\right)\int_{-u}^u e(bx)\mathrm{d}x
\end{align*}
for any $b\in\R$. Hence we have \begin{align}\label{TildeRDef}
    \Tilde{\mathcal{R}}=\frac{1}{H}\sum_{q\in\mathcal{Q}}\sum_{n_1} B(1,n_1)\sum_{(r,n_2\delta)=1} \la_g(r)\sumast_{\alpha\Mod{q}}e\left(\frac{\alpha(n_1-r)}{q}\right)\mathcal{W}\left(\frac{n_1}{D_1},\frac{n_2}{D_2},\frac{r}{D_1}\right),
\end{align}
where \begin{align*}
    \mathcal{W}\left(a,b,c\right)=F_\mu\left(a,b\right)U\left(c\right)\frac{1}{2u}\int_{-u}^u e\left((a-c)D_1x\right)\mathrm{d}x.
\end{align*}
Choosing $u=D_1^{-1}$ yields $$x^{j_1}z^{j_2}\frac{\mathrm{d}^{j_1+j_2}}{\mathrm{d}x^{j_1}\mathrm{d}z^{j_2}}\mathcal{W}(x,y,z)\ll_{j_1,j_2}(1+|\mu|)^{j_1}\delta(x,y,z\asymp 1)$$ for any $j_1,j_2\geq0$.

We first show that $\mathcal{R}$ is strongly approximated by $\tilde{\mathcal{R}}$, using the estimate \cite[Lemma 1]{Jut}
\begin{align*}
        \int_0^1\left|1-\Tilde{I}_{\mathcal{Q},u}(x)\right|^2dx\ll_\varepsilon\frac{Q^{2+\varepsilon}}{u H^2}.
\end{align*}
Note that M\"obius inversion and the GL(2) Hecke relations give us \begin{align*}
    \sum_{(r,n_2\delta)=1}\la_g(r)e\left(-rx\right)U\left(\frac{r}{D_1}\right)=\sum_{df|n_2\delta}\mu(df)\mu(f)\lambda_g(d)\sum_r\lambda_g(r)e(-df^2rx)U\left(\frac{df^2r}{D_1}\right),
\end{align*}
which by the Wilton-type bound and the Kim-Sarnak bound \cite{KS}, is bounded by \begin{align*}
    \preccurlyeq \sum_{df|n_2\delta}d^{\theta} \sqrt{\frac{D_1}{df^2}}\preccurlyeq \sqrt{D_1}.
\end{align*}
Hence we have 
\begin{displaymath}
    \begin{split}
   \left|\mathcal{R}-\Tilde{\mathcal{R}}\right|&\ll\int_0^1\left|\sum_{n_1}B(1,n_1)e\left(n_1x\right)F_\mu\left(\frac{n_1}{D_1}  {,\frac{n_2}{D_2}}\right)\right|\left|\sum_{(r,n_2\delta)=1}\la_g(r)e\left(-rx\right)U\left(\frac{r}{D_1}\right)\right|\left|1-\tilde{I}_{\mathcal{Q},u}(x)\right|\mathrm{d}x\\
    & \preccurlyeq \sqrt{D_1}\int_0^1\left|\sum_{n_1}B(1,n_1)e\left(n_1x\right)F_\mu\left(\frac{n_1}{D_1},\frac{n_2}{D_2}\right)\right|\left|1-\tilde{I}_{\mathcal{Q},u}(x)\right|\mathrm{d}x.
\end{split}
\end{displaymath}
Together with the Cauchy-Schwarz inequality  we obtain  \begin{align*}
    \left|\mathcal{R}-\Tilde{\mathcal{R}}\right|
    \preccurlyeq & \sqrt{D_1}\frac{Q}{\sqrt{u}H}\Big(\sum_{n_1\sim D_1}\left|B(1,n_1)\right|^2\Big)^{1/2}, 
\end{align*}
so that  \begin{align}\label{JutilaError}
    \tilde{\mathcal{R}}=\mathcal{R}+O\big({D_1^{3/2}}{Q^{-1}}X^\varepsilon\big).
\end{align}

Now we analyze $\tilde{\mathcal{R}}$ in \eqref{TildeRDef}. Applying Voronoi summation (Lemma \ref{lem.Voronoi} with \eqref{GL3KernelAsymp}) on the $n_1$-sum, we have \begin{align*}
    \sum_{n_1}B(1,n_1)e\left(\frac{\alpha n_1}{q}\right)\mathcal{W}\left(\frac{n_1}{D_1},\frac{n_2}{D_2},\frac{r_1}{D_1}\right)=\frac{D_1^{2/3}}{q}\sum_{\epsilon_1=\pm}\sum_{n_0\mid q}\sum_{n_1}B(n_1,n_0)\left(\frac{n_0}{n_1}\right)^{1/3}S\left(\overline{\alpha},\epsilon_1n_1;\frac{q}{n_0}\right)\mathcal{W}_1^{\epsilon_1}\left(\frac{n_0^2n_1D_1}{q^3},\frac{n_2}{D_2},\frac{r}{D_1}\right),
\end{align*}
where \begin{align*}
    \mathcal{W}_1^{\epsilon_1}\left(a,b,c\right)=\int_0^\infty \mathcal{W}\left(x,b,c\right)W_{\mu_0}^{\epsilon_1}(ax)x^{-1/3}e\left(\epsilon_13(ax)^{1/3}\right)\mathrm{d}x,
\end{align*}
with some function $W_{\mu_0}^{\epsilon_1}$ satisfying $x^j\frac{\mathrm{d}^j}{\mathrm{d}x^j}W_{\mu_0}^{\epsilon_1}(x) \ll_{j} 1$ for $j\geq0$. Repeated integration by parts on the $x$-integral gives us arbitrary savings unless $n_0^2n_1\preccurlyeq_\mu  Q^3 D_1^{-1}$. Applying M\"obius inversion and the GL(2) Hecke relations on the $r$-sum, we have \begin{align*}
    &\sum_{(r,n_2\delta)=1}\la_g(r)e\left(-\frac{\alpha r}{q}\right)\mathcal{W}_1^{\epsilon_1}\left(\frac{n_0^2n_1D_1}{q^3},\frac{n_2}{D_2},\frac{r}{D_1}\right)=\mathop{\sum\sum}_{df\mid n_2\delta}\mu(df)\mu(f)\la_g\left(d\right)\sum_r\la_g(r)e\left(-\frac{\alpha df^2r}{q}\right)\mathcal{W}_1^{\epsilon_1}\left(\frac{n_0^2n_1D_1}{q^3},\frac{n_2}{D_2},\frac{df^2r}{D_1}\right).
\end{align*}
Since $(q,n_2\delta)=1$, applying the GL(2) Voronoi summation formula (\cite[Thm A.4]{KMV}) gives us \begin{align*}
    &\sum_{\epsilon_2=\pm}\mathop{\sum\sum}_{df\mid n_2\delta}\mu(df)\mu(f)\frac{D_1}{df^2q}\la_g\left(d\right)\sum_r\overline{\la_g(r)}e\left(\epsilon_2\frac{\overline{\alpha df^2}r}{q}\right)\mathcal{W}_2^{\epsilon_1,\epsilon_2}\left(\frac{n_0^2n_1D_1}{q^3},\frac{n_2}{D_2},\frac{rD_1}{df^2q^2}\right),
\end{align*}
where \begin{align*}
    \mathcal{W}_2^{\epsilon_1,\epsilon_2}(a,b,c)=&\int_0^\infty \mathcal{W}_1^{\epsilon_1}(a,b,y)\mathcal{J}_{\mu(g)}^{\epsilon_2}\left(4\pi\sqrt{cy}\right)\mathrm{d}y
\end{align*}
with $\mathcal{J}_{\mu(g)}^{\epsilon_2}$ being some normalized Bessel function satisfying the asymptotics \begin{align*}
    \mathcal{J}_{\mu(g)}^{\epsilon_2}(4\pi x)=\sum_\pm W_{\mu(g)}^{\epsilon_2,\pm}(x)e\left(\pm 2x\right)
\end{align*}
for some functions $W_{\mu(g)}^{\epsilon_2,\pm}$ satisfying $x^j\frac{\mathrm{d}^j}{\mathrm{d}x^j}W_{\mu(g)}^{\epsilon_2,\pm}(x) \ll_{j, \mu(g)} 1$ for $j\geq0$. Repeated integration by parts on the $y$-integral gives us arbitrary savings unless $r\preccurlyeq _{\mu} \ df^2Q^2D_1^{-1}$. Combining the above analysis with the Kim-Sarnak bound \cite{KS} we obtain \begin{align*}
    \Tilde{\mathcal{R}}\preccurlyeq  & 
   \frac{D_1^{5/3}}{Q^2}\sum_{\epsilon_1,\epsilon_2=\pm}\sum_{q\in\mathcal{Q}}\sum_{df\mid n_2\delta}\sum_{r\preccurlyeq  \frac{df^2Q^2}{D_1}}\frac{|\la_g(r)|}{d^{1-\theta}f^2q^2}\Bigg|\mathop{\sum_{n_0\mid q}\sum_{n_1}}_{n_0^2n_1\preccurlyeq  Q^3/D_1 } B(n_1,n_0)\left(\frac{n_0}{n_1}\right)^{1/3} \\ &\times 
  \sumast_{\alpha\Mod{q}}S\left(\overline{\alpha},\epsilon_1n_1, \frac{q}{n_0}\right)e\left(\epsilon_2\frac{\overline{\alpha df^2}r}{q}\right)\mathcal{W}_2^{\epsilon_1,\epsilon_2}\left(\frac{n_0^2n_1D_1}{q^3},\frac{n_2}{D_2},\frac{rD_1}{df^2q^2}\right)\Bigg|+X^{-99}.
\end{align*}

Since $q$ is prime, we have $n_0=1$ or $q$. Choosing $$Q\ll D_1X^{-\varepsilon},$$ we can see from the restriction of the $n_1$-sum that the contribution of $n_0=q$ is negligible. Together with a smooth dyadic subdivision on the $n_1$-sum, we have \begin{align}\label{tildeRN}
    \Tilde{\mathcal{R}}\preccurlyeq &\frac{D_1^{5/3}}{Q^2}\sum_{\epsilon_1,\epsilon_2=\pm}\sum_{q\in\mathcal{Q}}\sum_{df\mid n_2\delta}\sum_{r\preccurlyeq  \frac{df^2Q^2}{D_1}}\frac{|\la_g(r)|}{d^{1-\theta}f^2q^2}\sup_{N\preccurlyeq \frac{Q^3}{D_1}}|\mathcal{N}|+X^{-99},
\end{align}
where \begin{align*}
    \mathcal{N}=\sum_{n_1} \frac{B(n_1,1)}{n_1^{1/3}}\sumast_{\alpha\Mod{q}}S\left(\overline{\alpha},\epsilon_1n_1, q\right)e\left(\epsilon_2\frac{\overline{\alpha df^2}r}{q}\right)\varphi\left(\frac{n_1}{N}\right)\mathcal{W}_2^{\epsilon_1,\epsilon_2}\left(\frac{n_1D_1}{q^3},\frac{n_2}{D_2},\frac{rD_1}{df^2q^2}\right)
\end{align*}
for some fixed $\varphi\in C_c^\infty([1/2,5/2])$. Opening the Kloosterman sum and summing over $\alpha$, we have \begin{align*}
    \sumast_{\alpha\Mod{q}}S\left(\overline{\alpha},\epsilon_1n_1, q\right)e\left(\epsilon_2\frac{\overline{\alpha df^2}r}{q}\right)=&\sumast_{\beta\Mod{q}}e\left(\epsilon_1\frac{n_1\beta}{q}\right)\left(q \cdot \delta\big(q\nmid r, \beta\equiv -\epsilon_2df^2\overline{r}\Mod{q}\big)-1\right)\\
    =&q\delta(q\nmid r)e\left(-\epsilon_1\epsilon_2\frac{df^2n_1\overline{r}}{q}\right)-q\delta(q\mid n {_1})+1.
\end{align*}
Hence \begin{align*}
    \mathcal{N}=\sum_{n_1} \frac{B(n_1,1)}{n_1^{1/3}}\left(q\delta(q\nmid r)e\left(-\epsilon_1\epsilon_2\frac{df^2n_1\overline{r}}{q}\right)-q\delta(q\mid n_1)+1\right)\varphi\left(\frac{n_1}{N}\right)\mathcal{W}_2^{\epsilon_1,\epsilon_2}\left(\frac{n_1D_1}{q^3},\frac{n_2}{D_2},\frac{rD_1}{df^2q^2}\right).
\end{align*}

We estimate the three terms in the previous display separately and in reverse order. First, 
by Voronoi summation (Lemma \ref{lem.Voronoi}) we can easily see that \begin{equation}\label{term1}
    \sum_{n_1} \frac{B(n_1,1)}{n_1^{1/3}}\varphi\left(\frac{n_1}{N}\right)\mathcal{W}_2^{\epsilon_1,\epsilon_2}\left(\frac{n_1D_1}{q^3},\frac{n_2}{D_2},\frac{rD_1}{df^2q^2}\right)\preccurlyeq 1.
\end{equation}
Similarly, since $q^2\nmid n_1$ and assuming for now only $Q\ll D_1X^{-\varepsilon}$,  we have \begin{equation}\label{term2}
\begin{split}
    &q\sum_{q \mid n_1} \frac{B(n_1,1)}{n_1^{1/3}}\varphi\left(\frac{n_1}{N}\right)\mathcal{W}_2^{\epsilon_1,\epsilon_2}\left(\frac{n_1D_1}{q^3},\frac{n_2}{D_2},\frac{rD_1}{df^2q^2}\right)
    = q^{2/3}B(q,1)\sum_{n_1} \frac{B(n_1,1)}{n_1^{1/3}}\varphi\left(\frac{qn_1}{N}\right)\mathcal{W}_2^{\epsilon_1,\epsilon_2}\left(\frac{n_1D_1}{q^2},\frac{n_2}{D_2},\frac{rD_1}{df^2q^2}\right)\preccurlyeq   Q^{2/3} .
    \end{split}
\end{equation}
Finally, applying additive reciprocity with Voronoi summation (Lemma \ref{lem.Voronoi} with \eqref{GL3KernelAsymp}), we obtain \begin{align*}
    &q\sum_{n_1} \frac{B(n_1,1)}{n_1^{1/3}}e\left(-\epsilon_1\epsilon_2\frac{df^2n_1\overline{r}}{q}\right)\varphi\left(\frac{n_1}{N}\right)\mathcal{W}_2^{\epsilon_1,\epsilon_2}\left(\frac{n_1D_1}{q^3},\frac{n_2}{D_2},\frac{rD_1}{df^2q^2}\right)\\
    =&q\sum_{n_1} \frac{B(n_1,1)}{n_1^{1/3}}e\left(\epsilon_1\epsilon_2\frac{df^2n_1\overline{q}}{r}-\epsilon_1\epsilon_2\frac{df^2n_1}{qr}\right)\varphi\left(\frac{n_1}{N}\right)\mathcal{W}_2^{\epsilon_1,\epsilon_2}\left(\frac{n_1D_1}{q^3},\frac{n_2}{D_2},\frac{rD_1}{df^2q^2}\right)\\
    =&\frac{(df^2,r)qN^{1/3}}{r}\sum_{\epsilon_3=\pm}\sum_{n_0\mid \frac{r}{(df^2,r)}}\sum_{n_1} B(n_0,n_1)\left(\frac{n_0}{n_1}\right)^{1/3}
   S\left(\epsilon_1\epsilon_2q\overline{\frac{df^2}{(df^2,r)}},\epsilon_3n_1;\frac{r}{(df^2,r)n_0}\right)\mathcal{W}_3^{\epsilon_1,\epsilon_2,\epsilon_3}\left(\frac{n_0^2n_1(df^2,r)^3N}{r^3},\frac{n_2}{D_2},\frac{rD_1}{df^2q^2}\right),
\end{align*}
where \begin{align*}
    \mathcal{W}_3^{\epsilon_1,\epsilon_2,\epsilon_3}\left(a,b,c\right)=\int_0^\infty \varphi\left(z\right)e\left(-\epsilon_1\epsilon_2\frac{D_1Nz}{cq^3}+\epsilon_33(az)^{1/3}\right)\mathcal{W}_2^{\epsilon_1,\epsilon_2}\left(\frac{D_1Nz}{q^3},b,c\right)W_{-\mu_0}^{\epsilon_3}(az)z^{-1/3}\mathrm{d}z.
\end{align*}
Repeated integration by parts gives us arbitrary savings unless $$n_0^2n_1\preccurlyeq  \left(\frac{df^2}{(r,df^2)}\right)^3\frac{N^2}{Q^3}.$$ Bounding everything trivially with the Weil bound for Kloosterman sums and  the Cauchy-Schwarz inequality, we  finally get the bound  \begin{align*}
    \preccurlyeq & \sqrt{\frac{(df^2,r)}{r}}qN^{1/3}\sum_{n_0\mid \frac{r}{(df^2,r)}} \sum_{ n_0^2n_1\preccurlyeq  \left(\frac{df^2}{(r,df^2)}\right)^3\frac{N^2}{Q^3}}
    \frac{1}{n_0^{1/6}n_1^{1/3}}
   \preccurlyeq \frac{d^2f^4N^{5/3}}{(df^2,r)^{3/2}\sqrt{r}Q}.
\end{align*}
Combining this with \eqref{term1} and \eqref{term2}, we see that  \begin{align*}
    \mathcal{N}\preccurlyeq 1 + Q^{2/3} + \frac{d^2f^4N^{5/3}}{\sqrt{r}Q}\preccurlyeq   Q^{2/3} + \frac{d^2f^4Q^4}{\sqrt{r}D_1^{5/3}}
\end{align*}
since $N \preccurlyeq Q^3/D_1$.  Inserting this back into \eqref{tildeRN} and \eqref{JutilaError}, we have \begin{align*}
    \mathcal{R}\preccurlyeq   \frac{D_1^{5/3}}{Q^2}\sum_{q\in\mathcal{Q}}\sum_{df\mid n_2\delta}\sum_{r\ll \frac{df^2Q^2}{D_1}X^\varepsilon}\frac{|\la_g(r)|}{d^{1-\theta}f^2q^2}\left(Q^{2/3} +\frac{d^2f^4Q^4}{\sqrt{r}D_1^{5/3}}\right)+\frac{D_1^{3/2}}{Q} . 
\end{align*}
Applying  Cauchy-Schwarz yields \begin{align*}
    \mathcal{R}\preccurlyeq   \frac{D_1^{3/2}}{Q}+\frac{D_1^{2/3}}{Q^{1/3}}(n_2\delta)^\theta+\frac{Q^2}{\sqrt{D_1}}(n_2\delta)^3 .
\end{align*}
Choosing $\sqrt{D_1}\ll Q\ll D_1X^{-\varepsilon}$ optimally with $$Q=\frac{D_1^{2/3}}{n_2\delta}+\sqrt{D_1},$$
we finally obtain 
    $\mathcal{R}\preccurlyeq    n_2\delta D_1^{5/6}+(n_2\delta)^3\sqrt{D_1}. $
Recall that $\mathcal{R}$ is the $n_1$-sum in $\mathcal{S}$ in \eqref{start2}.  Hence we have the bound \begin{align}\label{S2ndBound}
    \mathcal{S}\ll& \sum_{n_2\ll D_2}|B(1,n_2)||\mathcal{R}|\preccurlyeq    \sum_{n_2\ll D_2} \left(n_2\delta D_1^{5/6}+(n_2\delta)^3\sqrt{D_1}\right) 
    \preccurlyeq   \delta D_1^{5/6}D_2^2+\delta^3\sqrt{D_1}D_2^4 .
\end{align}

Inserting \eqref{MaxEFirstBound} and \eqref{S2ndBound} into \eqref{EmaxS}, we finally obtain \begin{align*}
    \mathcal{E}_{\max}\preccurlyeq &
    \mathop{\sum_{\delta}\mathop{\sum\sum}_{\delta_{1 },\delta_2 \mid \delta^\infty}}_{\delta^2\delta_{1}\delta_{2}D_1D_2\ll X}\delta^{ \theta}\delta_{1 }^{ \theta} \min\left\{D_1 , \delta D_1^{5/6}D_2^2+\delta^3\sqrt{D_1}D_2^4\right\}. 
\end{align*}
We estimate the minimum by a suitable geometric mean to obtain
\begin{displaymath}
\begin{split}
\mathcal{E}_{\max}& \preccurlyeq 
\mathop{\sum_{\delta}\mathop{\sum\sum}_{\delta_{1 },\delta_2 \mid \delta^\infty}}_{\delta^2\delta_{1}\delta_{2}D_1D_2\ll X}\delta^{ \theta}\delta_{1 }^{ \theta}  \Big( D_1^{7/13}  (\delta D_1^{5/6}D_2^2)^{6/13} + D_1^{7/9}  (\delta^3\sqrt{D_1}D_2^4)^{2/9}\Big)  \\
&\preccurlyeq \sum_{\delta}\mathop{\sum\sum}_{\delta_{1 },\delta_2 \mid \delta^\infty} \Big( \frac{X^{12/13} }{ \delta^{18/13 - \theta} \delta_1^{12/13 - \theta}\delta_2^{12/13}} +  \frac{X^{8/9}}{\delta^{10/9 - \theta} \delta_1^{8/9 - \theta} \delta_2^{8/9}}\Big) \preccurlyeq X^{12/13}
\end{split} 
\end{displaymath}
since $\theta <1/9$.  This completes the analysis of the Eisenstein contribution.

\section{\texorpdfstring{$\Sigma_6$}{Sum6} contribution}\label{sec5}

So far we have shown $c_{\text{cusp}} S_m(X) = \Sigma_6 + O(X^{12/13 + \varepsilon})$. We now come to the central task of analyzing $\Sigma_6$.  Recall that $\Sigma_6$ is given by \begin{align*}
    \Sigma_6=\sum_{\epsilon_1,\epsilon_2=\pm1}\sum_n B(1,n)V\left(\frac{n}{X}\right)\mathop{\sum\sum}_{D_1,D_2}\frac{S(\epsilon_2 n,\epsilon_1,1,m;D_1,D_2)}{D_1D_2}\Phi_{\omega_6}\left(-\frac{\epsilon_2 nD_2}{D_1^2},-\frac{\epsilon_1 mD_1}{D_2^2}\right).
\end{align*}
Applying the decomposition of $S(\epsilon_2 n,\epsilon_1,1,m;D_1,D_2)$ due to Kiral and Nakasuji  as stated in \eqref{KNDecomposition}, we have \begin{align}\label{sigma6}
    \Sigma_6=&\sum_{\epsilon_1,\epsilon_2=\pm1}\sum_n B(1,n)V\left(\frac{n}{X}\right)\sum_{D_0}\sum_{D_1}\sum_{D_2}\frac{1}{D_0D_1D_2}\Phi_{\omega_6}\left(-\frac{\epsilon_2 nD_2}{D_0D_1^2},-\frac{\epsilon_1m D_1}{D_0D_2^2}\right)\\ \nonumber
    &\times\sumast_{\substack{\alpha\Mod{D_0}\\D_2+\epsilon_1D_1\alpha\equiv 0\Mod{D_0}}}S\left(\epsilon_2 n,\frac{D_2+\epsilon_1D_1\alpha}{D_0};D_1\right)S\left(m,\frac{D_2\overline{\alpha}+\epsilon_1D_1}{D_0};D_2\right).
\end{align}
 {By the decay properties of $\Phi_{\omega_6}$ in Lemma \ref{lem.Phi6Truncation}} it is clear that this expression is absolutely convergent. 

Following the sketch in the introduction, we will use a three-step approach Voronoi-reciprocity-Voronoi, followed by a final Poisson summation. Since Poisson in one variable commutes with Voronoi in another variable, it is technically simpler to first execute the easier Poisson step and postpone the second Voronoi step to the very end of the argument. We will frequently have to pause in order to estimate various error terms that we ignored in the sketch in the introduction.


Every summation step introduces an integral transform of the previous weight function. In this way we obtain a sequence of different weight functions, and for the reader's convenience we summarize at this point the various definitions and bounds:
\begin{itemize}
\item $F_1$ is defined in \eqref{F1Def} and bounded in Lemma \ref{lemF1};
\item the correction factor $G$ is defined in \eqref{defGaux};
\item $F_2$ is defined in \eqref{F2Def} and bounded in Lemma \ref{lemf2}; the variations $\tilde{F}_2$ and $\breve{F}_2$ are defined in  \eqref{FTilde2Def}  and \eqref{F2Defnew} respectively, the variation $\acute{F}_2 = \tilde{F}_2 + \breve{F}_2$ is defined in \eqref{Facute2Def};
\item $F_3$ is defined in \eqref{deff3} and bounded in Lemma \ref{lemf3}, the variation $\mathcal{F}_3$ is defined in \eqref{curlyf3}; 
\item $F_4$ is defined in \eqref{F4Def} and bounded in Lemma \ref{F4lem}.
\end{itemize}

As a general piece of notation, $V$ will always denote a smooth function with support in $[1, 2]$, while $U$ will denote a smooth function with compact support in $[-3, 3]$ that is 1 on $[-2, 2]$.

\subsection{Step 1: Voronoi summation}

We apply Voronoi summation (Lemma \ref{lem.Voronoi}) on the $n$-sum in \eqref{sigma6}, giving us \begin{align}\label{afterVor}
    &\sum_n B(1,n)V\left(\frac{n}{X}\right)\Phi_{\omega_6}\left(-\frac{\epsilon_2 nD_2}{D_0D_1^2},-\frac{\epsilon_1 mD_1}{D_0D_2^2}\right)e\left(\frac{\epsilon_2\overline{\beta} n}{D_1}\right)\\ \nonumber
    =&D_1\sum_{\eta_1=\pm1}\sum_{n_0\mid D_1}\sum_n\frac{B(n,n_0)}{n_0n}S\left(\epsilon_2\beta,\eta_1 n;\frac{D_1}{n_0}\right)F_1^{\epsilon_1,\epsilon_2,\eta_1}\left(\frac{n_0^2nX}{D_1^3},\frac{XD_2}{D_0D_1^2},\frac{D_1}{D_0D_2^2}\right),
\end{align}
for $\beta \in (\Bbb{Z}/D_1\Bbb{Z})^{\ast}$, where \begin{align}\label{F1Def}
    F_1^{\epsilon_1,\epsilon_2,\eta_1}\left(a,b,c\right)=\int_{(0)}a^{-s} \mathcal{G}_{\mu_0}^{\eta_1}(s+1)\int_0^\infty V(x)\Phi_{\omega_6}\left(-\epsilon_2bx,-\epsilon_1mc\right)x^{-s-1}\mathrm{d} x\frac{\mathrm{d}s}{2\pi i}
\end{align}
with $\mathcal{G}_{\mu_0}^\pm(s)$ as defined in Lemma \ref{lem.Voronoi}. {Using the decomposition \eqref{phidecomp}, if $b \geq (c+1)X^{\varepsilon}$, we decompose
\begin{align}\label{F1AsympDecomposition}
    F_1^{\epsilon_1,\epsilon_2,\eta_1}\left(a,b,c\right) = F_1^{\epsilon_1,\epsilon_2,\eta_1; \sharp}\left(a,b,c\right) + F_1^{\epsilon_1,\epsilon_2,\eta_1; \flat}\left(a,b,c\right).
\end{align}}

\begin{lemma}
\label{lemF1} {a)} For any $A > 0$, we have
\begin{align*}
    F_1^{\epsilon_1,\epsilon_2,\eta_1}\left(a,b,c\right)\preccurlyeq_{A 
    }&
   \min(a, a^{2/3})\Big(1 +  {\frac{1}{b^2c} + \frac{1}{bc^2}}+ \frac{a^{1/3}}{\sqrt{b} + b^{1/3}c^{1/6}}\Big)^{-A}  {(bc)^{-1/4} \begin{cases} 1, & b \asymp c,\\   {(b^2c+bc^2)^{-1/12}}, &\text{otherwise.}\end{cases}} 
\end{align*}

  { b) If $ b>(c+1)X^\varepsilon$, we have \begin{align*}
    F_1^{\epsilon_1,\epsilon_2,\eta_1; \sharp}\left(a,b,c\right)
    =&\ \left(\frac{b}{c}\right)^{1/3}
    \sum_{\sigma_1,\sigma_2=\pm1}G_{\sigma_1,\sigma_2}\left(x_0,\frac{a}{b^{3/2}},b,c\right)e\left(-\sigma_1\frac{a}{b}+3\sigma_1\sigma_2\left(\frac{mac}{b}\right)^{1/3}\right)
\end{align*}
up to a negligible error, 
where $G_{\sigma_1,\sigma_2}$ are smooth functions (depending also on $\epsilon_1, \epsilon_2, \eta_1$) satisfying 
\begin{align*}
    x^{j_1}y^{j_2}z^{j_3}w^{j_4}G_{\sigma_1,\sigma_2}^{(j_1,j_2,j_3,j_4)}(x,y,z,w)\ll_{j_1,j_2,j_3, j_4} \delta(1\leq x\leq 2, y\asymp 1) 
\end{align*}
for any $j_1, j_2, j_3, j_4\in \N_0$, and \begin{align*}
    x_0=\frac{1}{b}\left(\left(\frac{a}{b}\right)^{2/3}-\sigma_2(mc)^{1/3}\right)^3.
\end{align*}
}
\end{lemma}

\begin{proof} {a)} Repeated integration by parts on the $x$-integral with Lemma \ref{lem.Phi6Truncation} gives us arbitrary savings unless 
\begin{equation}\label{boundT}
    \Im(s)\preccurlyeq_{\Re(s)}  \sqrt{b}+b^{1/3}c^{1/6}
 \end{equation}  (where we can assume that the right hand side is $\gg X^{-\varepsilon}$).   Hence together with Stirling's approximation on $\mathcal{G}_{\mu_0}^{\eta_1}(s+1)$, shifting the contour to the line $\Re(s)=A$ for sufficiently large $A>0$ gives us arbitrary savings unless $a\preccurlyeq   (\sqrt{b}+b^{1/3}c^{1/6})^3$. If $a \leq 1$, we shift the contour to $\Re s = -1+\varepsilon$, getting a bound $\preccurlyeq a$. If $a \geq 1$, we shift the contour to $\Re s = 0$, say, and apply stationary phase on the $s$-integral with the phase
$${g(t)} = t \log  \frac{|t|^3}{(e^3 xa)},$$
derived from Stirling's formula for $\mathcal{G}_{\mu_0}^{\eta_1}(s+1) (xa)^{-s}$. 
Restricting the integral smoothly to $t \asymp T$, we have a stationary point at $|t| = (xa)^{1/3}$. If this point is not in the range of integration, i.e.\ $T \not\asymp a^{1/3}$, we apply \cite[Lemma 8.1]{BKY} with $${\tt R} = 1, \quad {\tt Y} = {\tt Q} = {\tt U} = T$$ to see that this portion is negligible if $T\geq X^{\varepsilon}$. For $T \asymp a^{1/3}$ we apply \cite[Proposition 8.2]{BKY} instead with $${\tt Y} = {\tt Q} = {\tt V} = T$$
getting a bound $T^2 \asymp a^{2/3}$
provided that $T \geq X^{\varepsilon}.$ For $T \ll X^{\varepsilon}$ we estimate trivially, obtaining the final bound 
\begin{equation}\label{statphs}\int_{(0)}a^{-s} \mathcal{G}_{\mu_0}^{\eta_1}(s+1) x^{-s } \frac{\mathrm{d}s}{2\pi i}\preccurlyeq a^{2/3}.
\end{equation}
Together with Lemma \ref{Kw6} we complete the proof {of part a)}. 

{
b) By \eqref{boundT} and our present assumption we can restrict to  $\Im(s)\preccurlyeq   \sqrt{b}$ and $a\preccurlyeq b^{3/2}$. 
Shifting the contour to $\Re(s)=-1/2$ for convenience, this leads us to 
\begin{align*}
    F_1^{\epsilon_1,\epsilon_2,\eta_1; \sharp}\left(a,b,c\right)=\frac{\sqrt{a}}{b^{5/12}c^{1/3}}\sum_{\sigma_1,\sigma_2=\pm1}\int_{|t|\preccurlyeq \sqrt{b}}\int_0^\infty R(x, t)e\left(\phi(x,t)\right)\mathrm{d} x\frac{\mathrm{d}t}{2\pi}
\end{align*}
(up to a negligible error) where \begin{align*}
    \phi(x,t)=2\sigma_1\left((bx)^{1/3}+\sigma_2(mc)^{1/3}\right)^{3/2}+t \log \frac{|t|^3}{e^3 xa},
\end{align*} 
and $R(x,t)=R(x,t,b,mc)$ is supported on $[1,2]\times [-\sqrt{b}X^{\varepsilon},\sqrt{b}X^\varepsilon]$ in the first two arguments, and \begin{align*}
    x^{j_1}t^{j_2}y^{j_3}z^{j_4}R^{(j_1,j_2,j_3,j_4)}(x,t,y,z)\ll 1.
\end{align*}

Note that \begin{align*}
    \frac{\mathrm{d}}{\mathrm{d}x}\phi(x,t)=\sigma_1\left((bx)^{1/3}+\sigma_2(mc)^{1/3}\right)^{1/2}\left(\frac{b}{x^2}\right)^{1/3}-\frac{t}{x} \quad \text{ and } \quad \frac{\mathrm{d}^j}{\mathrm{d}x^j}\phi(x,t)\preccurlyeq_j \sqrt{b}
\end{align*}
for any $j\in \N_0$.  Applying \cite[Lemma 8.1]{BKY} with
$$ Q=U = 1, \quad  R = b^{1/2 }, \quad Y = b^{1/2 + \varepsilon}$$
we get arbitrary savings unless $|t|\asymp \sqrt{b}$ and $\sigma_1t>0$ (in which case the bound $R \gg b^{1/2}$ may fail). When $|t|\asymp \sqrt{b}$ (and in particular $\gg X^{\varepsilon}$) and $\sigma_1t>0$, we apply the stationary phase analysis for the $t$-integral as before to get \begin{align*}
    F_1^{\epsilon_1,\epsilon_2,\eta_1;\sharp}\left(a,b,c\right)=\frac{\sqrt{a}}{b^{1/6}c^{1/3}}\sum_{\sigma_1,\sigma_2=\pm1}\int_0^\infty W(x, \sigma_1(ax)^{1/3})e\left(\phi(x)\right)\mathrm{d} x
\end{align*}
where $W(x,t)=W(x,t,y,z)$ satisfies the same support and derivative conditions as $R$ and
\begin{align*}
    \phi(x)=2\sigma_1\left((bx)^{1/3}+\sigma_2(mc)^{1/3}\right)^{3/2}-3\sigma_1(ax)^{1/3}.
\end{align*}
We further apply stationary phase analysis on the $x$-integral. Note that 
\begin{align*}
    \phi'(x)=\sigma_1\left((bx)^{1/3}+\sigma_2(mc)^{1/3}\right)^{1/2}\left(\frac{b}{x^2}\right)^{1/3}-\sigma_1\frac{a^{1/3}}{x^{2/3}},  
\end{align*}
\begin{align*}
    \phi''(x)=-\sigma_1\frac{
b^{1/3}\left(3(bx)^{1/3}+4\sigma_2(mc)^{1/3}\right)
}{
6x^{5/3}
\left((bx)^{1/3}+\sigma_2(mc)^{1/3}\right)^{1/2}
}+\sigma_1\frac{2a^{1/3}}{3x^{5/3}}.
\end{align*}
and $\phi^{(j)}\ll_j \sqrt{b}$ for any $j\in\N_0$. The stationary point is given by \begin{align*}
    \phi'(x_0)=0 \quad \Longleftrightarrow \quad x_0=\frac{1}{b}\left(\left(\frac{a}{b}\right)^{2/3}-\sigma_2(mc)^{1/3}\right)^3.
\end{align*}
Since $b>(c+1)X^\varepsilon$, $x_0\asymp1$ implies that $a\asymp b^{3/2}$. We compute 
\begin{align*}
    \phi(x_0)=-\sigma_1\frac{a}{b}+3\sigma_1\sigma_2\left(\frac{mac}{b}\right)^{1/3} \quad \text{and} \quad 
    \phi''(x_0)=\sigma_1\frac{a^{1/3}}{6x_0^{5/3}}-\frac{\sigma_1\sigma_2}{6x_0^{5/3}}\left(\frac{mb^2c}{a}\right)^{1/3}\asymp a^{1/3}\asymp \sqrt{b}.
\end{align*}
 Applying \cite[Prop 8.2]{BKY} with $V=Q=1$, $Y = b^{1/2}$, we conclude the proof of part b).
}
\end{proof}

Inserting \eqref{afterVor} back into \eqref{sigma6} gives us \begin{align*}
    \Sigma_6=&\sum_{\epsilon_1,\epsilon_2, \eta_1=\pm1}\sum_{D_0}\sum_{D_1}\sum_{D_2}\frac{1}{D_0D_2}\sum_{n_0\mid D_1}\sum_n\frac{B(n,n_0)}{n_0n}F_1^{\epsilon_1,\epsilon_2,\eta_1}\left(\frac{n_0^2nX}{D_1^3},\frac{XD_2}{D_0D_1^2},\frac{D_1}{D_0D_2^2}\right)
    \\
    &\times\sumast_{\substack{\alpha\Mod{D_0}\\D_2+\epsilon_1D_1\alpha\equiv 0\Mod{D_0}}}\sumast_{\beta\Mod{D_1}}S\left(\epsilon_2\beta,\eta_1 n;\frac{D_1}{n_0}\right)e\left(\frac{\beta(D_2+\epsilon_1D_1\alpha)/D_0}{D_1}\right)S\left(m,\frac{D_2\overline{\alpha}+\epsilon_1D_1}{D_0};D_2\right).
\end{align*}
 {By the decay properties of $\Phi_{\omega_6}$},  the $D_0, D_1, D_2$-sums are rapidly convergent, and since  the first entry of $F_1$ can be bounded in terms of the other two, the same holds for the $n$-sum, so the entire expression is absolutely convergent. 

In the following sections we will make extensive changes of variables. This will inevitably introduce an increasing and perhaps seemingly random alphabet. 

By the congruence condition satisfied by $\alpha$, we have $D_{0,0}:=(D_1,D_0)=(D_2,D_0)$. Rewrite $$D_0 \leftarrow D_{0,0}D_{0,1}D_0 , \quad D_1 \leftarrow D_{0,0}D_{1,0}D_1, \quad  D_2 \leftarrow D_{0,0}D_{2,0}D_2$$  with $D_{0,1}D_{1,0}D_{2,0}\mid D_{0,0}^\infty$ and $(D_0D_1D_2,D_{0,0})=1$. Rewriting $$n_0 \leftarrow n_{0,0}n_{1,0}n_0$$ with $n_{0,0}\mid D_{0,0}$, $n_{1,0}\mid D_{1,0}$, $(n_{0,0},D_{1,0}/n_{1,0})=1$ and $n_0\mid D_1$, rewriting $$D_{0,0} \leftarrow n_{0,0}D_{0,0}, \quad D_{1,0} \leftarrow n_{1,0}D_{1,0}, \quad  D_1 \leftarrow n_0D_1$$ as well, we have \begin{align}\label{afterVorfull}
    \Sigma_6=&\sum_{\epsilon_1,\epsilon_2,\eta_1=\pm1} \sum_{n_{0,0}} \sum_{D_{0,0}}\mathop{\sum\sum\sum\sum}_{\substack{n_{1,0}D_{0,1}D_{1,0}D_{2,0}\mid (n_{0,0}D_{0,0})^\infty\\(n_{1,0}D_{1,0}D_{2,0},D_{0,1})=1\\(D_{1,0},n_{0,0})=1}}\mathop{\sum\sum\sum\sum}_{(n_0D_0D_1D_2,n_{0,0}D_{0,0})=(n_0D_1D_2,D_0)=1} \frac{1 
    }{n_{0,0}^2D_{0,0}^2D_{0,1}D_{2,0}D_0D_2}\\
    & \times 
    \sum_n \frac{B(n,n_{0,0}n_{1,0}n_0)}{n_{0,0}n_{1,0}n_0n}\mathcal{C}_1
   F_1^{\epsilon_1,\epsilon_2,\eta_1}\left(\frac{nX}{n_{0,0}n_{1,0}n_0(D_{0,0}D_{1,0}D_1)^3},\frac{XD_{2,0}D_2}{n_{0,0}^2n_{1,0}^2n_0^2D_{0,0}^2D_{0,1}D_{1,0}^2D_0D_1^2},\frac{n_{1,0}n_0D_{1,0}D_1}{n_{0,0}^2D_{0,0}^2D_{0,1}D_{2,0}^2D_0D_2^2}\right), \nonumber
\end{align}
where \begin{align}\label{defC1}
    \mathcal{C}_1
    =&\sumast_{\substack{\alpha\Mod{n_{0,0}D_{0,0}D_{0,1}D_0}\\Y_1\equiv 0\Mod{D_{0,1}D_0}}}\sumast_{\beta\Mod{n_{0,0}n_{1,0}n_0D_{0,0}D_{1,0}D_1}}e\left(\frac{\beta Y_1/(D_{0,1}D_0)}{n_{0,0}n_{1,0}n_0D_{0,0}D_{1,0}D_1}\right) \\& \times S\left(\epsilon_2\beta,\eta_1 n;D_{0,0}D_{1,0}D_1\right)S\left(m,\frac{Y_1 \overline{\alpha}}{D_{0,1}D_0};n_{0,0}D_{0,0}D_{2,0}D_2\right). \nonumber
\end{align}
with
$$Y_1 = D_{2,0}D_2+\epsilon_1n_{1,0}n_0D_{1,0}D_1\alpha.$$

\subsection{Interlude I: A truncation}\label{smalld1}

The generic range for $\mathcal{D}_1 := n_{0, 0} n_{1, 0} n_0 D_{0,0}D_{1,0}D_1$, formerly simply $D_1$, is $X^{2/3}$. For the subsequent analysis it is important that this quantity is not too small.   The idea is that for small $\mathcal{D}_1$, the character sum has smaller conductor, and hence it is beneficial to apply Cauchy-Schwarz, since the resulting sum over $n$ is now amenable to a relatively powerful Poisson summation. 

Write $n_0=n_1n_2$ with $n_2\mid D_2$ and $(n_1,D_2/n_2)=1$, and write $D_2=n_2d_2$. For notational simplicity, we temporarily write $$\nu_0=n_{0,0}n_{1,0}n_1n_2, \quad \mathcal{D}_0=n_{0,0}D_{0,0}D_{0,1}D_0, \quad \mathcal{D}_1=\nu_0D_{0,0}D_{1,0}D_1, \quad \mathcal{D}_2=n_{0,0}n_2D_{0,0}D_{2,0}d_2.$$  They represent the original $n_0, D_0, D_1, D_2$ before the change of variables. We restrict the $\mathcal{D}_0, \mathcal{D}_1, \mathcal{D}_2$-summation to dyadic intervals and define for $\Delta = (\Delta_0, \Delta_1,\Delta_2)$ the subsum $\Sigma_6(\Delta)$ where
\begin{align}\label{S6D0D1D2}
    \Sigma_6(\Delta)=&\sum_{\epsilon_1,\epsilon_2,\eta_1=\pm1} \sum_{n_{0,0}} \sum_{D_{0,0}}\mathop{\sum\sum\sum\sum}_{\substack{n_{1,0}D_{0,1}D_{1,0}D_{2,0}\mid (n_{0,0}D_{0,0})^\infty\\(n_{1,0}D_{1,0}D_{2,0},D_{0,1})=1\\(D_{1,0},n_{0,0})=1}}\mathop{\sum\sum\sum\sum\sum}_{\substack{(n_1n_2D_0D_1d_2,n_{0,0}D_{0,0})=1\\
    (n_1n_2D_1d_2,D_0)=(n_1,d_2)=1\\ \mathcal{D}_0\asymp \Delta_0, \mathcal{D}_1\asymp \Delta_1, \mathcal{D}_2\asymp \Delta_2}}  
   \sum_n \frac{B(n,\nu_0)}{\nu_0\mathcal{D}_0\mathcal{D}_2n}\mathcal{C}_1
   F_1^{\epsilon_1,\epsilon_2,\eta_1}\left(\frac{\nu_0^2nX}{\mathcal{D}_1^3},\frac{X\mathcal{D}_2}{\mathcal{D}_0\mathcal{D}_1^2},\frac{\mathcal{D}_1}{\mathcal{D}_0\mathcal{D}_2^2}\right).
\end{align}

\begin{lemma}\label{newtruncation}
If $\Delta_0\Delta_1 \leq X^{2/3 - \varepsilon}$, then we have 
$$\Sigma_6(\Delta)\preccurlyeq X^{1/3}\Delta_1^{5/4}+X^{2/3}\Delta_1^{1/2}+X^{5/6}.$$
\end{lemma}
\begin{proof}
Since the proof is relatively long, we divide it into several steps. \\

\textbf{1) Some preparatory moves.}     
By Lemma \ref{lem.Phi6Truncation}, we may restrict ourselves to 
\begin{align}\label{Delta1Delta2Truncation}
    \Delta_0\Delta_1\preccurlyeq X^{2/3} \quad \text{ and } \quad \Delta_0\Delta_2\preccurlyeq X^{1/3}.
\end{align}

We recall from Lemma \ref{lemF1}a) and in particular the condition $a \preccurlyeq b^{3/2} + b c^{1/2}$ that we may assume 
\begin{align*}
    n\ll N_0:=\frac{\Delta_1^3}{\nu_0^2X}\left(\left(\frac{X\Delta_2}{\Delta_0\Delta_1^2}\right)^{3/2}+\frac{X}{(\Delta_0\Delta_1)^{3/2}}\right)X^\varepsilon=\frac{\sqrt{X}\Delta_2^{3/2}+\Delta_1^{3/2}}{\nu_0^2\Delta_0^{3/2}}X^\varepsilon
\end{align*}
at the cost of a negligible error. 

On the other hand, we can slightly simplify $\mathcal{C}_1$, defined in \eqref{defC1}, by evaluating the $\beta$-sum. Precisely, we have 
\begin{align*}
    \mathcal{C}_1=&\ \sumast_{\substack{\alpha\Mod{\mathcal{D}_0}\\ Y_1\equiv 0\Mod{D_{0,1}D_0}}}S\left(m,\frac{Y_1\overline{\alpha}}{D_{0,1}D_0};\mathcal{D}_2\right)\sumast_{\gamma\Mod{D_{0,0}D_{1,0}D_1}}e\left(\frac{\eta_1n\gamma}{D_{0,0}D_{1,0}D_1}\right)\sumast_{\beta\Mod{\mathcal{D}_1}}e\left(\frac{\epsilon_2\beta\overline{\gamma}}{D_{0,0}D_{1,0}D_1}+\frac{\beta Y_1/(D_{0,1}D_0)}{\mathcal{D}_1}\right)\nonumber\\
    =&\ \sumast_{\substack{\alpha\Mod{\mathcal{D}_0}\\ Y_1\equiv 0\Mod{D_{0,1}D_0}}}S\left(m,\frac{Y_1\overline{\alpha}}{D_{0,1}D_0};\mathcal{D}_2\right)\sum_{ab=\mathcal{D}_1}\mu(a)b\sumast_{\substack{\gamma\Mod{D_{0,0}D_{1,0}D_1}\\ \overline{\gamma}\nu_0\equiv -\epsilon_2Y_1/(D_{0,1}D_0)\Mod{b}}}e\left(\frac{\eta_1n\gamma}{D_{0,0}D_{1,0}D_1}\right)\nonumber\\
    =&\ \sumast_{\substack{\alpha\Mod{\mathcal{D}_0}\\ Y_1\equiv 0\Mod{D_{0,1}D_0}}}S\left(m,\frac{Y_1\overline{\alpha}}{D_{0,1}D_0};\mathcal{D}_2\right)\sum_{\substack{ab_0c=\mathcal{D}_1\\ b_0\mid (\nu_0,Y_1/(D_{0,1}D_0))\\ (c,\nu_0/b_0)=1}}\mu(a)b_0c\sumast_{\substack{\gamma\Mod{D_{0,0}D_{1,0}D_1}\\ \gamma\equiv \overline{\frac{Y_1}{b_0D_{0,1}D_0}}\frac{\nu_0}{b_0}\Mod{c}}}e\left(\frac{-\epsilon_2\eta_1n\gamma}{D_{0,0}D_{1,0}D_1}\right)\nonumber
\end{align*}
with $Y_1 = n_2D_{2,0}d_2+\epsilon_1n_{1,0}n_1n_2D_{1,0}ac\alpha$. Now we further evaluate the $\gamma$-sum by \begin{align*}
    \sumast_{\substack{\gamma\Mod{D_{0,0}D_{1,0}ac}\\ \gamma\equiv \overline{\frac{Y_1}{b_0D_{0,1}D_0}}\frac{\nu_0}{b_0}\Mod{c_0c}}}e\left(\frac{-\epsilon_2\eta_1n\gamma}{D_{0,0}D_{1,0}ac}\right)=&\ \sumast_{\substack{\gamma_0\Mod{D_{0,0}D_{1,0}}\\ \gamma_0\equiv \overline{\frac{Y_1}{b_0D_{0,1}D_0}}\frac{\nu_0}{b_0}\Mod{c_0}}}e\left(\frac{-\epsilon_2\eta_1\overline{ac}n\gamma_0}{D_{0,0}D_{1,0}}\right)\sumast_{\substack{\gamma_1\Mod{ac}\\ \gamma_1\equiv \overline{D_{2,0}d_2}n_{0,0}n_{1,0}n_1D_{0,1}D_0\Mod{c}}}e\left(\frac{-\epsilon_2\eta_1\overline{D_{0,0}D_{1,0}}n\gamma_1}{ac}\right).
\end{align*}
Note that $(c,n_1)=1$ as $n_1\mid \nu_0/b_0$. Writing $a=a_0'a'$ with $a_0'=(a,c^\infty)$ and $(a',c)=1$, the $\gamma_1$-sum further decomposes into the product of \begin{align*}
    \sumast_{\gamma\Mod{a'}}e\left(\frac{-\epsilon_2\eta_1\overline{D_{0,0}D_{1,0}a_0'c}n\gamma}{a'}\right)=\sum_{a_1a_2=a'}\mu(a_1)a_2 \delta(a_2|n)
\end{align*}
and \begin{align*}
    \sumast_{\substack{\gamma\Mod{a_0'c}\\ \gamma\equiv \overline{D_{2,0}d_2}n_{0,0}n_{1,0}n_1D_{0,1}D_0\Mod{c}}}\hspace{-1cm} e\left(\frac{-\epsilon_2\eta_1\overline{D_{0,0}D_{1,0}a'}n\gamma}{a_0'c}\right)
  & =e\left(-\epsilon_2\eta_1\frac{\overline{D_{0,0}D_{1,0}D_{2,0}a'd_2}n_{0,0}n_{1,0}n_1D_{0,1}D_0n}{a_0'c}\right)\sum_{\beta\Mod{a_0'}} e\left(-\epsilon_2\eta_1\frac{\overline{D_{0,0}D_{1,0}a'}n\beta}{a_0'}\right)\nonumber\\
    &\ = a_0'\delta(a_0'\mid n)e\left(-\epsilon_2\eta_1\frac{\overline{D_{0,0}D_{1,0}D_{2,0}a'd_2}n_{0,0}n_{1,0}n_1D_{0,1}D_0n/a_0'}{c}\right).
\end{align*}
Hence combining $a_0'$ and $a_2$, we have \begin{align}\label{C0C10}
    \mathcal{C}_1=&\ \mathop{\sum\sum}_{\substack{a_0b_0c_0=\nu_0D_{0,0}D_{1,0}\\ b_0\mid (\nu_0,\mathcal{D}_2)\\ (c_0,\nu_0/b_0)=1}}\mathop{\sum\sum}_{\substack{a_1a_2c=D_1\\ (a_1a_2,n_1)=(c,a_1d_2\nu_0/b_0)=1\\a_2|n}}\mu(a_0a_1a_2)\mu(a_1)a_2b_0c_0c\mathcal{C}_{1,0},
\end{align}
where \begin{align}\label{C10Def}
    \mathcal{C}_{1,0}=&\ \sumast_{\substack{\alpha\Mod{\mathcal{D}_0}\\ Y_1\equiv 0\Mod{b_0D_{0,1}D_0}}}S\left(m,\frac{Y_1\overline{\alpha}}{D_{0,1}D_0};\mathcal{D}_2\right)\sumast_{\substack{\gamma\Mod{D_{0,0}D_{1,0}}\\ \gamma\equiv \overline{\frac{Y_1}{b_0D_{0,1}D_0}}\frac{\nu_0}{b_0}\Mod{c_0}}}e\left(\frac{-\epsilon_2\eta_1\overline{a_1c}\gamma n/a_2}{D_{0,0}D_{1,0}}\right)\nonumber\\
    &\ \times e\left(-\epsilon_2\eta_1\frac{\overline{D_{0,0}D_{1,0}D_{2,0}a_1d_2}n_{0,0}n_{1,0}n_1D_{0,1}D_0n/a_2}{c}\right).
\end{align}
Note that $n_1\mid a_0$ by the conditions of $a_0, b_0, c_0$. Since $(a_1a_2,a_0)=1$, we have $(a_1a_2,n_1)=1$. Using $(\nu_0,\mathcal{D}_2)=n_{0,0}n_2(n_{1,0},D_{0,0}D_{2,0})$, this gives us \begin{align}\label{5.2StartingPoint}
    \Sigma_{6}(\Delta)=&\ \sum_{\epsilon_1,\epsilon_2,\eta_1=\pm1} \sum_{n_{0,0}} \sum_{D_{0,0}}\mathop{\sum\sum\sum\sum}_{\substack{n_{1,0}D_{0,1}D_{1,0}D_{2,0}\mid (n_{0,0}D_{0,0})^\infty\\(n_{1,0}D_{1,0}D_{2,0},D_{0,1})=1\\(D_{1,0},n_{0,0})=1}}\mathop{\sum\sum\sum}_{\substack{(n_1n_2D_0,n_{0,0}D_{0,0})=1\\
    (n_1n_2,D_0)=(c_0,\nu_0/b_0)=1\\  \mathcal{D}_0\asymp \Delta_0}}\mathop{\sum\sum}_{\substack{a_0b_0c_0=\nu_0D_{0,0}D_{1,0}\\ b_0\mid n_{0,0}n_2(n_{1,0},D_{0,0}D_{2,0})}}\mathop{\sum\sum\sum\sum}_{\substack{(a_1a_2cd_2,n_{0,0}D_{0,0}D_0)=1\\
    (a_1a_2d_2,n_1)=1\\
    (c,a_1d_2\nu_0/b_0)=1 \\ \mathcal{D}_1\asymp \Delta_1, \mathcal{D}_2\asymp \Delta_2}}  
   \nonumber\\
   &\ \times \frac{\mu(a_0a_1a_2)\mu(a_1)b_0c_0c}{\nu_0\mathcal{D}_0\mathcal{D}_2}\sum_{n'\ll N_0/a_2} \frac{B(a_2n',\nu_0)}{n'} \mathcal{C}_{1,0}
   F^{\epsilon_1, \epsilon_2, \eta_1}_{1}\left(\frac{\nu_0^2a_2n'X}{\mathcal{D}_1^3},\frac{X\mathcal{D}_2}{\mathcal{D}_0\mathcal{D}_1^2},\frac{\mathcal{D}_1}{\mathcal{D}_0\mathcal{D}_2^2}\right),
\end{align}
with $\nu_0=n_{0,0}n_{1,0}n_1n_2, \quad \mathcal{D}_0=n_{0,0}D_{0,0}D_{0,1}D_0, \quad \mathcal{D}_1=\nu_0D_{0,0}D_{1,0}a_1a_2c, \quad \mathcal{D}_2=n_{0,0}n_2D_{0,0}D_{2,0}d_2.$

Next 
we insert the definition \eqref{F1Def} of $F_1^{\epsilon_1,\epsilon_2,\eta_1}$ and shift the contour  to $\Re(s)=-1/2$. From the proof of Lemma \ref{lemF1}a) we know that up to a negligible error we may restrict to   \begin{align*}
    \im(s)\ll T:=\left(\sqrt{\frac{X\Delta_2}{\Delta_0\Delta_1^2}}+\frac{X^{1/3}}{\sqrt{\Delta_0\Delta_1}}\right)X^\varepsilon.
\end{align*}
This gives us \begin{align*}
\Sigma_6(\Delta)=&\sqrt{X}\sum_{\epsilon_1,\epsilon_2,\eta_1=\pm1} \sum_{n_{0,0}} \sum_{D_{0,0}}\mathop{\sum\sum\sum\sum}_{\substack{n_{1,0}D_{0,1}D_{1,0}D_{2,0}\mid (n_{0,0}D_{0,0})^\infty\\(n_{1,0}D_{1,0}D_{2,0},D_{0,1})=1\\(D_{1,0},n_{0,0})=1}}\mathop{\sum\sum\sum}_{\substack{(n_1n_2D_0,n_{0,0}D_{0,0})=1\\
    (n_1n_2,D_0)=(c_0,\nu_0/b_0)=1\\  \mathcal{D}_0\asymp \Delta_0}}\mathop{\sum\sum}_{\substack{a_0b_0c_0=\nu_0D_{0,0}D_{1,0}\\ b_0\mid n_{0,0}n_2(n_{1,0},D_{0,0}D_{2,0})}}\mathop{\sum\sum\sum\sum}_{\substack{(a_1a_2cd_2,n_{0,0}D_{0,0}D_0)=1\\
    (a_1a_2d_2,n_1)=1\\
    (c,a_1d_2\nu_0/b_0)=1 \\ \mathcal{D}_1\asymp \Delta_1, \mathcal{D}_2\asymp \Delta_2}}  
   \nonumber\\
   &\ \times \frac{\mu(a_0a_1a_2)\mu(a_1)\sqrt{a_2}b_0c_0c}{\mathcal{D}_0\mathcal{D}_1^{3/2}\mathcal{D}_2}\sum_{n'\ll N_0/a_2} \frac{B(a_2n',\nu_0)}{\sqrt{n'}} \mathcal{C}_{1,0}  \nonumber\\
    & \times 
   \int_{|t|\ll T}\mathcal{G}_{\mu_0}^{\eta_1}\left(\frac{1}{2}+it\right)\int_0^\infty V(x)\Phi_{\omega_6}\left(-\epsilon_2\frac{X\mathcal{D}_2x}{\mathcal{D}_0\mathcal{D}_1^2},-\epsilon_1\frac{m\mathcal{D}_1}{\mathcal{D}_0\mathcal{D}_2^2}\right)\left(\frac{\nu_0^2a_2n'Xx}{\mathcal{D}_1^3}\right)^{-it}\frac{\mathrm{d}x}{\sqrt{x}}\frac{\mathrm{d}t}{2\pi}+O(X^{-99}).
\end{align*}

\textbf{2. Applying Cauchy-Schwarz.} Let $U$ be a fixed smooth function as described at the beginning of this section. Write $n=a_2n'$.  
Applying the Cauchy-Schwarz inequality to take out the $a_0, a_1, a_2, b_0, c_0, n_{0,0}, n_{1,0}, n_1, n_2,  D_{0,0}, D_{0,1}, D_{1,0}, D_{2,0}, D_0, n'$-sum and the $s$-integral along with Stirling's formula to bound $\mathcal{G}_{\mu_0}^{\eta_1}\left(1/2+it\right)\ll 1$, we have
\begin{align*}
    \Sigma_6(\Delta)^2\preccurlyeq &\ X\sum_{\epsilon_1,\epsilon_2,\eta_1=\pm1} \sum_{n_{0,0}} \sum_{D_{0,0}}\mathop{\sum\sum\sum\sum}_{\substack{n_{1,0}D_{0,1}D_{1,0}D_{2,0}\mid (n_{0,0}D_{0,0})^\infty\\(n_{1,0}D_{1,0}D_{2,0},D_{0,1})=1\\(D_{1,0},n_{0,0})=1}}\mathop{\sum\sum\sum}_{\substack{(n_1n_2D_0,n_{0,0}D_{0,0})=(n_1n_2,D_0)=1\\ \nu_0=n_{0,0}n_{1,0}n_1n_2\\ \mathcal{D}_0=n_{0,0}D_{0,0}D_{0,1}D_0\asymp \Delta_0}}\mathop{\sum\sum}_{\substack{(a_1a_2,n_1\mathcal{D}_0)=1}}\mathop{\sum\sum}_{\substack{b_0\mid n_{0,0}n_2(n_{1,0},D_{0,0}D_{2,0})\\ c_0\mid D_{0,0}D_{1,0}}}\frac{a_1(a_2b_0c_0)^2n_1n_2T}{D_{0,1}\mathcal{D}_0} \nonumber\\
    &\ \times \int_\R U\left(\frac{t}{T}\right)\sum_{n'} U\left(\frac{a_2n'}{N_0}\right) \left|\mathop{\sum\sum}_{\substack{(cd_2,\mathcal{D}_0)=(c,a_1d_2\nu_0/b_0)=(d_2,n_1)=1\\ \nu_0D_{0,0}D_{1,0}a_1a_2c=\mathcal{D}_1\asymp \Delta_1\\ n_{0,0}n_2D_{0,0}D_{2,0}d_2=\mathcal{D}_2\asymp \Delta_2}}\frac{c}{\mathcal{D}_1^{3/2}\mathcal{D}_2}\mathcal{C}_{1,0}
   \int_0^\infty V(x)\Phi_{\omega_6}\left(-\epsilon_2\frac{X\mathcal{D}_2x}{\mathcal{D}_0\mathcal{D}_1^2},-\epsilon_1\frac{m\mathcal{D}_1}{\mathcal{D}_0\mathcal{D}_2^2}\right)\left(\frac{x}{\mathcal{D}_1^3}\right)^{-it}\frac{\mathrm{d}x}{\sqrt{x}}\right|^2\mathrm{d}t.
\end{align*}
Opening the square, we write $\nu_0D_{0,0}D_{1,0}a_1a_2c'=\mathcal{D}_1'$ and $n_{0,0}n_2D_{0,0}D_{2,0}d_2'=\mathcal{D}_2'$ as the second copy of $\nu_0D_{0,0}D_{1,0}a_1a_2c=\mathcal{D}_1$ and $n_{0,0}n_2D_{0,0}D_{2,0}d_2=\mathcal{D}_2$ respectively, while the second copy of the summation variable $\gamma$ in $\mathcal{C}_{1, 0}$ is called $\gamma'$. 

\textbf{3. Applying Poisson.} Applying Poisson summation, the $n'$-sum is of the form \begin{align*}
    & \sum_{n'} e\left(-\epsilon_2\eta_1\frac{\overline{a_1} n'(\overline{c}\gamma-\overline{c'}\gamma')}{D_{0,0}D_{1,0}}-\epsilon_2\eta_1\frac{\overline{D_{0,0}D_{1,0}D_{2,0}a_1d_2}n_{0,0}n_{1,0}n_1D_{0,1}D_0n'}{c}+\epsilon_2\eta_1\frac{\overline{D_{0,0}D_{1,0}D_{2,0}a_1d_2'}n_{0,0}n_{1,0}n_1D_{0,1}D_0n'}{c'}\right) U\left(\frac{a_2n'}{N_0}\right)\\
    &\ = \frac{N_0}{D_{0,0}D_{1,0}a_2cc'}\sum_n\sum_{\sigma\Mod{D_{0,0}D_{1,0}cc'}}e\Big(-\epsilon_2\eta_1\frac{\overline{a_1}(\overline{c}\gamma-\overline{c'}\gamma')\sigma}{D_{0,0}D_{1,0}}-\epsilon_2\eta_1\frac{\overline{D_{0,0}D_{1,0}D_{2,0}a_1}n_{0,0}n_{1,0}n_1D_{0,1}D_0(c'\overline{d_2}-c\overline{d_2'})\sigma}{cc'}\Big)\\
    &\ \quad \times e\Big(\frac{n\sigma}{D_{0,0}D_{1,0}cc'}\Big) \hat{U}\left(\frac{nN_0}{D_{0,0}D_{1,0}a_2cc'}\right)\\
    &\ = \frac{N_0}{a_2} \sum_n \delta\left(\gamma c'-\gamma'c\equiv \epsilon_2\eta_1 a_1n\Mod{D_{0,0}D_{1,0}}, n_{0,0}n_{1,0}n_1D_{0,1}D_0(c'\overline{d_2}-c\overline{d_2'})\equiv \epsilon_2\eta_1a_1D_{2,0}n\Mod{cc'}\right) \hat{U}\left(\frac{nN_0}{D_{0,0}D_{1,0}a_2cc'}\right),
\end{align*}
where $\hat{U}(y)=\int_\R U(z)e(-yz)dz$ is the Fourier transform of $U$. Repeated integration by parts gives us arbitrary savings unless \begin{align*}
    |n|\ll N_1:=\frac{\Delta_1^2}{a_1^2a_2\nu_0^2D_{0,0}D_{1,0}N_0}X^\varepsilon= \frac{\Delta_0^{3/2}\Delta_1^2}{a_1^2a_2D_{0,0}D_{1,0}(\sqrt{X}\Delta_2^{3/2}+\Delta_1^{3/2})}.
\end{align*}
Hence we arrive at
\begin{align}\label{Sigma6delta1delta2AfterPoissoninn}
    \Sigma_6(\Delta)^2\preccurlyeq &\ X\sum_{\epsilon_1,\epsilon_2,\eta_1=\pm1} \sum_{n_{0,0}} \sum_{D_{0,0}}\mathop{\sum\sum\sum\sum}_{\substack{n_{1,0}D_{0,1}D_{1,0}D_{2,0}\mid (n_{0,0}D_{0,0})^\infty\\(n_{1,0}D_{1,0}D_{2,0},D_{0,1})=1\\(D_{1,0},n_{0,0})=1}}\mathop{\sum\sum\sum}_{\substack{(n_1n_2D_0,n_{0,0}D_{0,0})=(n_1n_2,D_0)=1\\ \nu_0=n_{0,0}n_{1,0}n_1n_2\\ \mathcal{D}_0=n_{0,0}D_{0,0}D_{0,1}D_0\asymp\Delta_0}}\mathop{\sum\sum}_{\substack{(a_1a_2,n_1\mathcal{D}_0)=1}}\mathop{\sum\sum}_{\substack{b_0\mid n_{0,0}n_2(n_{1,0},D_{0,0}D_{2,0})\\ c_0\mid D_{0,0}D_{1,0}}}\nonumber\\
    &\ \times \frac{a_1a_2(b_0c_0)^2n_1n_2TN_0}{D_{0,1}\mathcal{D}_0}\sum_{|n|\ll N_1}\mathop{\sum\sum\sum\sum}_{\substack{(cc'd_2d_2',\mathcal{D}_0)=(cc',a_1\nu_0/b_0)=1\\
    (c,d_2)=(c',d_2')=(d_2d_2',n_1)=1\\ \nu_0D_{0,0}D_{1,0}a_1a_2c=\mathcal{D}_1, \nu_0D_{0,0}D_{1,0}a_1a_2c'=\mathcal{D}_1'\asymp \Delta_1\\ n_{0,0}n_2D_{0,0}D_{2,0}d_2=\mathcal{D}_2, n_{0,0}n_2D_{0,0}D_{2,0}d_2'=\mathcal{D}_2'\asymp \Delta_2}}\frac{cc'}{\mathcal{D}_1^{3/2}\mathcal{D}_1'^{3/2}\mathcal{D}_2\mathcal{D}_2'}\mathcal{C}_{1,1}\mathcal{I}_0,
\end{align}
where \begin{align}\label{C11Def}
    \mathcal{C}_{1,1}=&\ \mathop{\sumast\quad\sumast}_{\substack{\alpha, \alpha'\Mod{\mathcal{D}_0}\\Y_1, Y_1'\equiv 0\Mod{b_0D_{0,1}D_0}}}S\left(m,\frac{Y_1 \overline{\alpha}}{D_{0,1}D_0};\mathcal{D}_2\right)S\left(m,\frac{Y_1' \overline{\alpha'}}{D_{0,1}D_0};\mathcal{D}_2'\right)\sumast_{\substack{\gamma\Mod{D_{0,0}D_{1,0}}\\ \gamma\equiv \overline{\frac{Y_1}{b_0D_{0,1}D_0}}\frac{\nu_0}{b_0}\Mod{c_0} }}\quad \sumast_{\substack{\gamma'\Mod{D_{0,0}D_{1,0}}\\ \gamma'\equiv \overline{\frac{Y_1'}{b_0D_{0,1}D_0}}\frac{\nu_0}{b_0}\Mod{c_0} }}\nonumber \\
    &\ \times \delta\left(\gamma c'-\gamma'c\equiv \epsilon_2\eta_1 a_1n\Mod{D_{0,0}D_{1,0}}, n_{0,0}n_{1,0}n_1D_{0,1}D_0(c'\overline{d_2}-c\overline{d_2'})\equiv \epsilon_2\eta_1a_1D_{2,0}n\Mod{cc'}\right)
\end{align}
with \begin{align}\label{Y1Y1'DefInPreliminaryBound}
    Y_1=n_2D_{2,0}d_2+\epsilon_1n_{1,0}n_1n_2D_{1,0}a_1a_2c\alpha, \quad Y_1'=n_2D_{2,0}d_2'+\epsilon_1n_{1,0}n_1n_2D_{1,0}a_1a_2c'\alpha', 
\end{align}
and \begin{align*}
    \mathcal{I}_0=&\ \hat{U}\left(\frac{nN_0}{D_{0,0}D_{1,0}a_2cc'}\right)\int_\R U\left(\frac{t}{T}\right)\int_0^\infty\int_0^\infty V(x)V(\tilde{x})\nonumber\\
    &\ \times \Phi_{\omega_6}\left(-\epsilon_2\frac{X\mathcal{D}_2x}{\mathcal{D}_0\mathcal{D}_1^2},-\epsilon_1\frac{m\mathcal{D}_1}{\mathcal{D}_0\mathcal{D}_2^2}\right)\overline{\Phi_{\omega_6}\left(-\epsilon_2\frac{X\mathcal{D}_2'\tilde{x}}{\mathcal{D}_0\mathcal{D}_1'^2},-\epsilon_1\frac{m\mathcal{D}_1'}{\mathcal{D}_0\mathcal{D}_2'^2}\right)}\left(\frac{\mathcal{D}_1'^3x}{\mathcal{D}_1^3\tilde{x}}\right)^{-it}\frac{\mathrm{d}x}{\sqrt{x}}\frac{\mathrm{d}\tilde{x}}{\sqrt{\tilde{x}}}\mathrm{d}t.
\end{align*}
Integrating by parts repeatedly on the $t$-integral gives us arbitrary savings unless 
\begin{align*}
    \left|\tilde{x}-\frac{\mathcal{D}_1'^3}{\mathcal{D}_1^3}x\right|\preccurlyeq T^{-1}.
\end{align*}
Using the inequality $|AB| \leq A^2 + B^2$, 
it suffices to bound 
\begin{align}\label{I0Bound}
    \mathcal{I}_0
    \preccurlyeq \int_0^\infty \left|W(x)\Phi_{\omega_6}\left(-\epsilon_2\frac{X\mathcal{D}_2x}{\mathcal{D}_0\mathcal{D}_1^2},-\epsilon_1\frac{m\mathcal{D}_1}{\mathcal{D}_0\mathcal{D}_2^2}\right)\right|^2\frac{\mathrm{d}x}{\sqrt{x}}
   \preccurlyeq \begin{cases}
        \sqrt{\frac{\Delta_0^2\Delta_1\Delta_2}{X}}\preccurlyeq \frac{\Delta_0\Delta_1}{X^{2/3}} & \text{ if } X^{1/3}\Delta_2\preccurlyeq\Delta_1\\
        \frac{\Delta_0^{3/2}\Delta_1\sqrt{\Delta_2}}{X^{5/6}} & \text{ if } X^{1/3-\varepsilon}\Delta_2>\Delta_1
    \end{cases}\preccurlyeq \frac{\Delta_0\Delta_1}{X^{2/3}}
\end{align}
where we used Lemma \ref{Kw6} and \eqref{Delta1Delta2Truncation}.

\textbf{4) Character sum analysis.} Now we simplify the character sum $\mathcal{C}_{1,1}$, defined in \eqref{C11Def}. For $n=0$, the congruence condition implies that \begin{align*}
    c'=c, \quad d_2'\equiv d_2\Mod{c}, \quad \gamma'\equiv \gamma\Mod{D_{0,0}D_{1,0}}.
\end{align*}
Hence together with the Weil bound on the Kloosterman sum, we have \begin{align}\label{C110Bound}
    \mathcal{C}_{1,1}\preccurlyeq \left(\frac{n_{0,0}D_{0,0}}{(b_0,n_{0,0}D_{0,0})}\right)^2\Delta_2 \frac{D_{0,0}D_{1,0}}{c_0}\delta(c'=c, \,\, d_2'\equiv d_2\Mod{c}).
\end{align}

For $n\neq0$, note that $(\nu_0\mathcal{D}_0,a_1n\mathcal{D}_2)=(\nu_0\mathcal{D}_0,a_1n\mathcal{D}_2')=n_{0,0}n_2D_{0,0}(n_{0,0}n_{1,0}n_1D_{0,1}D_0,D_{2,0}n)$. As a result, the congruence condition implies 
\begin{align*}
    c'\frac{\nu_0\mathcal{D}_0}{(\nu_0\mathcal{D}_0,n\mathcal{D}_2)} \equiv \frac{\epsilon_2\eta_1 a_1n\mathcal{D}_2}{(\nu_0\mathcal{D}_0,n\mathcal{D}_2)} \Mod{c} \quad \text{ and } \quad c\frac{\nu_0\mathcal{D}_0}{(\nu_0\mathcal{D}_0,n\mathcal{D}_2)} \equiv -\frac{\epsilon_2\eta_1 a_1n\mathcal{D}_2'}{(\nu_0\mathcal{D}_0,n\mathcal{D}_2)}\Mod{c'}.
\end{align*}
Hence there exist integers $|k|, |k'|\ll K:=\frac{\nu_0\Delta_0}{(\nu_0\mathcal{D}_0,n\mathcal{D}_2)}+\frac{a_1^2a_2\nu_0\mid n\mid D_{0,0}D_{1,0}\Delta_2}{(\nu_0\mathcal{D}_0,n\mathcal{D}_2)\Delta_1}$ such that 
\begin{align*}
    c'\frac{\nu_0\mathcal{D}_0}{(\nu_0\mathcal{D}_0,n\mathcal{D}_2)} = \frac{\epsilon_2\eta_1 a_1n\mathcal{D}_2}{(\nu_0\mathcal{D}_0,n\mathcal{D}_2)} -ck\quad \text{ and } \quad c\frac{\nu_0\mathcal{D}_0}{(\nu_0\mathcal{D}_0,n\mathcal{D}_2)} = -\frac{\epsilon_2\eta_1 a_1n\mathcal{D}_2'}{(\nu_0\mathcal{D}_0,n\mathcal{D}_2)}-c'k',
\end{align*}
which is equivalent to \begin{align*}
    \left(\begin{matrix}
        (\nu_0\mathcal{D}_0,n\mathcal{D}_2)k & \nu_0\mathcal{D}_0 \\ \nu_0\mathcal{D}_0 & (\nu_0\mathcal{D}_0,n\mathcal{D}_2)k' 
    \end{matrix}\right)\left(\begin{matrix}
        c \\ c'
    \end{matrix}\right) = \epsilon_2\eta_1 a_1n\left(\begin{matrix}
        \mathcal{D}_2 \\ -\mathcal{D}_2'
    \end{matrix}\right).
\end{align*}
This gives us two cases. If $(\nu_0\mathcal{D}_0,n\mathcal{D}_2)^2kk'\neq (\nu_0\mathcal{D}_0)^2$, then $c$ and $c'$ are uniquely determined by all the other variables. Moreover, the matrix equality with all entries being integers also implies that \begin{align}\label{kk'DivisionCondition}
    kk'-\left(\frac{\nu_0\mathcal{D}_0}{(\nu_0\mathcal{D}_0,n\mathcal{D}_2)}\right)^2\Big| k'\frac{a_1n\mathcal{D}_2}{(\nu_0\mathcal{D}_0,n\mathcal{D}_2)}+\frac{\nu_0\mathcal{D}_0}{(\nu_0\mathcal{D}_0,n\mathcal{D}_2)}\frac{a_1n\mathcal{D}_2'}{(\nu_0\mathcal{D}_0,n\mathcal{D}_2)}
\end{align}
This contribution of $ \mathcal{C}_{1,1}$, say $\mathcal{C}_{1, 1}^{\flat}$, is bounded by  \begin{align}\label{C11Bound}
    \mathcal{C}^{\flat}_{1,1}\preccurlyeq \left(\frac{n_{0,0}D_{0,0}}{(b_0,n_{0,0}D_{0,0})}\right)^2\Delta_2 \frac{D_{0,0}D_{1,0}}{c_0}\mathop{\sum\sum}_{\substack{|k|, |k'|\ll K\\ (\nu_0\mathcal{D}_0,n\mathcal{D}_2)^2kk'\neq (\nu_0\mathcal{D}_0)^2}}\delta\left(\left(\begin{matrix}
        (\nu_0\mathcal{D}_0,n\mathcal{D}_2)k & \nu_0\mathcal{D}_0 \\ \nu_0\mathcal{D}_0 & (\nu_0\mathcal{D}_0,n\mathcal{D}_2)k' 
    \end{matrix}\right)\left(\begin{matrix}
        c \\ c'
    \end{matrix}\right) = \epsilon_2\eta_1 a_1n\left(\begin{matrix}
        \mathcal{D}_2 \\ -\mathcal{D}_2'
    \end{matrix}\right)\right)
\end{align}
using the Weil bound $S(m, *; \mathcal{D}_2)S(m, *; \mathcal{D}'_2) \preccurlyeq \Delta_2$.

On the other hand, if $(\nu_0\mathcal{D}_0,n\mathcal{D}_2)^2kk'=(\nu_0\mathcal{D}_0)^2$, the matrix equality implies that
\begin{align*}
    \mathcal{D}_2'=-\frac{\nu_0\mathcal{D}_0\mathcal{D}_2}{(\nu_0\mathcal{D}_0,n\mathcal{D}_2)k} \quad \Longleftrightarrow \quad d_2'=-\frac{n_{0,0}n_{1,0}n_1D_{0,1}D_0d_2}{(n_{0,0}n_{1,0}n_1D_{0,1}D_0,D_{2,0}n)k}.
\end{align*}
Since $(d_2d_2',n_{0,0}n_{1,0}n_1D_{0,1}D_0)=1$, this implies that $\frac{n_{0,0}n_{1,0}n_1D_{0,1}D_0}{(n_{0,0}n_{1,0}n_1D_{0,1}D_0,D_{2,0}n)}=\frac{\nu_0\mathcal{D}_0}{(\nu_0\mathcal{D}_0,n\mathcal{D}_2)}\mid k$. Similarly, the matrix equality implies that $\mathcal{D}_2=-\frac{\nu_0\mathcal{D}_0\mathcal{D}_2'}{(\nu_0\mathcal{D}_0,n\mathcal{D}_2)k'}$ and yields $\frac{\nu_0\mathcal{D}_0}{(\nu_0\mathcal{D}_0,n\mathcal{D}_2)}\mid k'$. Hence the equality condition of $kk'$ forces \begin{align*}
    k=k'=-\frac{\nu_0\mathcal{D}_0}{(\nu_0\mathcal{D}_0,n\mathcal{D}_2)}.
\end{align*}
Note that the matrix equality then implies that $\nu_0\mathcal{D}_0\mid n\mathcal{D}_2$. Hence the matrix equality is equivalent to 
\begin{align*}
    k=k'=-1, \quad \mathcal{D}_2=\mathcal{D}_2', \quad \nu_0\mathcal{D}_0\mid n\mathcal{D}_2, \quad \text{ and } \quad c'=c+\epsilon_2\eta_1 \frac{a_1n\mathcal{D}_2}{\nu_0\mathcal{D}_0}. 
\end{align*}
Combining with \eqref{C11Bound}, we obtain $\mathcal{C}_{1, 1} = \mathcal{C}_{1, 1}^{\sharp} + \mathcal{C}_{1, 1}^{\flat}$ where
\begin{align}\label{C11Asymp}
    \mathcal{C}^\sharp_{1,1}=&\ \delta\left(\mathcal{D}_2=\mathcal{D}_2', \nu_0\mathcal{D}_0\mid n\mathcal{D}_2, c'=c+\epsilon_2\eta_1 \frac{a_1n\mathcal{D}_2}{\nu_0\mathcal{D}_0}\right)\mathop{\sumast\quad\sumast}_{\substack{\alpha, \alpha'\Mod{\mathcal{D}_0}\\Y_1, Y_1'\equiv 0\Mod{b_0D_{0,1}D_0}}}S\left(m,\frac{Y_1 \overline{\alpha}}{D_{0,1}D_0};\mathcal{D}_2\right)S\left(m,\frac{Y_1' \overline{\alpha'}}{D_{0,1}D_0};\mathcal{D}_2\right)\mathcal{C}_{1,2}(c)
\end{align}
with \begin{align*}
    \mathcal{C}_{1,2}(c)=\sumast_{\substack{\gamma\Mod{D_{0,0}D_{1,0}}\\ \gamma\equiv \overline{\frac{Y_1}{b_0D_{0,1}D_0}}\frac{\nu_0}{b_0}\Mod{c_0} }}\quad \sumast_{\substack{\gamma'\Mod{D_{0,0}D_{1,0}}\\ \gamma'\equiv \overline{\frac{Y_1'}{b_0D_{0,1}D_0}}\frac{\nu_0}{b_0}\Mod{c_0} }}\delta(\gamma c'-\gamma'c\equiv \epsilon_2\eta_1 a_1n\Mod{D_{0,0}D_{1,0}}).
\end{align*}


\textbf{5) Easy bounds.} We are now in a position to bound $\Sigma_6(\Delta)^2$ in two easy cases. 
Suppose that $\Upsilon$ is a bound for $\mathcal{I}_0$  in terms of $\Delta_0, \Delta_1, \Delta_2, X$. We have such a bound in \eqref{I0Bound}.

By \eqref{C110Bound}, the contribution of $n=0$ to $\Sigma_6(\Delta)^2$ in \eqref{Sigma6delta1delta2AfterPoissoninn} is bounded by \begin{align}\label{contr10}
    \preccurlyeq &\ X\sum_{n_{0,0}} \sum_{D_{0,0}}\mathop{\sum\sum\sum\sum}_{\substack{n_{1,0}D_{0,1}D_{1,0}D_{2,0}\mid (n_{0,0}D_{0,0})^\infty\\(n_{1,0}D_{1,0}D_{2,0},D_{0,1})=1\\(D_{1,0},n_{0,0})=1}}\mathop{\sum\sum\sum}_{\substack{(n_1n_2D_0,n_{0,0}D_{0,0})=(n_1n_2,D_0)=1\\ \nu_0=n_{0,0}n_{1,0}n_1n_2\\ \mathcal{D}_0=n_{0,0}D_{0,0}D_{0,1}D_0\asymp \Delta_0}}\mathop{\sum\sum}_{\substack{(a_1a_2,n_1\mathcal{D}_0)=1}}\mathop{\sum\sum}_{\substack{b_0\mid n_{0,0}n_2(n_{1,0},D_{0,0}D_{2,0})\\ c_0\mid D_{0,0}D_{1,0}}}\frac{a_1a_2(b_0c_0)^2n_1n_2TN_0}{D_{0,1}\Delta_0}\nonumber\\
       &\ \times \sum_{\substack{(c,a_1\mathcal{D}_0\nu_0/b_0)=1\\ \nu_0D_{0,0}D_{1,0}a_1a_2c=\mathcal{D}_1\asymp \Delta_1}}\mathop{\sum\sum}_{\substack{(d_2d_2',cn_1\mathcal{D}_0)=1\\ d_2\equiv d_2'\Mod{c}\\ n_{0,0}n_2D_{0,0}D_{2,0}d_2=\mathcal{D}_2\asymp \Delta_2\\ n_{0,0}n_2D_{0,0}D_{2,0}d_2'=\mathcal{D}_2'\asymp\Delta_2}}\frac{(\Delta_1/(a_1a_2\nu_0D_{0,0}D_{1,0}))^2}{\Delta_1^{3}\Delta_2^2}\left(\frac{n_{0,0}D_{0,0}}{(b_0,n_{0,0}D_{0,0})}\right)^2\Delta_2 \frac{D_{0,0}D_{1,0}}{c_0}\Upsilon\nonumber\\
    \preccurlyeq &\ X\frac{\sqrt{X}\Delta_2^{3/2}+\Delta_1^{3/2}}{\Delta_0^{3/2}}\left(\sqrt{\frac{X\Delta_2}{\Delta_0\Delta_1^2}}+\frac{X^{1/3}}{\sqrt{\Delta_0\Delta_1}}\right)\left(1+\frac{\Delta_2}{\Delta_1}\right)\Upsilon.
\end{align}

Now we handle the contribution of $n\neq0$. By \eqref{C11Bound} and the surrounding discussion, the contribution of the ``error term''   $\mathcal{C}^\flat_{1,1}$   to $\Sigma_6(\Delta)^2$ in \eqref{Sigma6delta1delta2AfterPoissoninn} is bounded by 
\begin{align}\label{contr1}
    \preccurlyeq &\ X\sum_{n_{0,0}} \sum_{D_{0,0}}\mathop{\sum\sum\sum\sum}_{\substack{n_{1,0}D_{0,1}D_{1,0}D_{2,0}\mid (n_{0,0}D_{0,0})^\infty\\(n_{1,0}D_{1,0}D_{2,0},D_{0,1})=1\\(D_{1,0},n_{0,0})=1}}\mathop{\sum\sum\sum}_{\substack{(n_1n_2D_0,n_{0,0}D_{0,0})=(n_1n_2,D_0)=1\\ \nu_0=n_{0,0}n_{1,0}n_1n_2\\ \mathcal{D}_0=n_{0,0}D_{0,0}D_{0,1}D_0\asymp \Delta_0}}\mathop{\sum\sum}_{\substack{(a_1a_2,n_1\mathcal{D}_0)=1\\ \nu_0D_{0,0}D_{1,0}a_1a_2\ll \Delta_1}}\mathop{\sum\sum}_{\substack{b_0\mid n_{0,0}n_2(n_{1,0},D_{0,0}D_{2,0})\\ c_0\mid D_{0,0}D_{1,0}}}\frac{a_1a_2(b_0c_0)^2n_1n_2TN_0}{D_{0,1}\Delta_0}\sum_{|n|\ll N_1}\nonumber\\
       &\ \times \mathop{\sum\sum}_{\substack{(d_2d_2',n_1\mathcal{D}_0)=1\\ n_{0,0}n_2D_{0,0}D_{2,0}d_2=\mathcal{D}_2\asymp \Delta_2\\ n_{0,0}n_2D_{0,0}D_{2,0}d_2'=\mathcal{D}_2'\asymp\Delta_2}}\mathop{\sum\sum}_{\substack{|k|,|k'|\ll K\\ (\nu_0\mathcal{D}_0,n\mathcal{D}_2)^2kk'\neq (\nu_0\mathcal{D}_0)^2\\ \eqref{kk'DivisionCondition} \text{ holds}}}\frac{(\Delta_1/(a_1a_2\nu_0D_{0,0}D_{1,0}))^2}{\Delta_1^{3}\Delta_2^2}\left(\frac{n_{0,0}D_{0,0}}{(b_0,n_{0,0}D_{0,0})}\right)^2\Delta_2 \frac{D_{0,0}D_{1,0}}{c_0}\Upsilon\nonumber\\
    \preccurlyeq &\ X\Delta_1\Delta_2\left(1+\frac{\sqrt{X}\Delta_2^{3/2}+\Delta_1^{3/2}}{\Delta_0^{3/2}\Delta_1^2}\right)\left( \Delta_0+\frac{\Delta_0^{3/2}\Delta_1\Delta_2}{\sqrt{X}\Delta_2^{3/2}+\Delta_1^{3/2}} \right)\left(\sqrt{\frac{X\Delta_2}{\Delta_0\Delta_1^2}}+\frac{X^{1/3}}{\sqrt{\Delta_0\Delta_1}}\right)\Upsilon, 
\end{align}
and by \eqref{C11Asymp} the contribution of the ``main term'' $\mathcal{C}_{1, 1 }^{\sharp}$ is \begin{align*}
     \preccurlyeq &\ X\sum_{\epsilon_1,\epsilon_2,\eta_1=\pm1} \sum_{n_{0,0}} \sum_{D_{0,0}}\mathop{\sum\sum\sum\sum}_{\substack{n_{1,0}D_{0,1}D_{1,0}D_{2,0}\mid (n_{0,0}D_{0,0})^\infty\\(n_{1,0}D_{1,0}D_{2,0},D_{0,1})=1\\(D_{1,0},n_{0,0})=1}}\mathop{\sum\sum\sum}_{\substack{(n_1n_2D_0,n_{0,0}D_{0,0})=(n_1n_2,D_0)=1\\ \nu_0=n_{0,0}n_{1,0}n_1n_2\\ \mathcal{D}_0=n_{0,0}D_{0,0}D_{0,1}D_0\asymp\Delta_0}}\mathop{\sum\sum}_{\substack{(a_1a_2,n_1\mathcal{D}_0)=1}}\mathop{\sum\sum}_{\substack{b_0\mid n_{0,0}n_2(n_{1,0},D_{0,0}D_{2,0})\\ c_0\mid D_{0,0}D_{1,0}}}\frac{a_1a_2(b_0c_0)^2n_1n_2TN_0}{D_{0,1}\mathcal{D}_0}\nonumber\\
    & \times \sum_{|n|\ll N_1}\mathop{\sum\sum}_{\substack{c, d_2\\ n_{0,0}n_2D_{0,0}D_{2,0}d_2=\mathcal{D}_2=\mathcal{D}_2'\asymp \Delta_2\\ \nu_0\mathcal{D}_0\mid n\mathcal{D}_2\\ c'= c+\epsilon_2\eta_1 \frac{a_1n\mathcal{D}_2}{\nu_0\mathcal{D}_0} \\ \nu_0D_{0,0}D_{1,0}a_1a_2c=\mathcal{D}_1\asymp \Delta_1\\ \nu_0D_{0,0}D_{1,0}a_1a_2c'=\mathcal{D}_1'\asymp \Delta_1}}\frac{cc'}{\mathcal{D}_1^{3/2}\mathcal{D}_1'^{3/2}\mathcal{D}_2^2}\mathop{\sumast\quad\sumast}_{\substack{\alpha, \alpha'\Mod{\mathcal{D}_0}\\Y_1, Y_1'\equiv 0\Mod{b_0D_{0,1}D_0}}}S\left(m,\frac{Y_1 \overline{\alpha}}{D_{0,1}D_0};\mathcal{D}_2\right)S\left(m,\frac{Y_1' \overline{\alpha'}}{D_{0,1}D_0};\mathcal{D}_2\right)  \mathcal{C}_{1,2}(c)\mathcal{I}_0
\end{align*}
with $Y_1$, $Y_1'$ as in \eqref{Y1Y1'DefInPreliminaryBound}. Applying the Weil bound and $\mathcal{C}_{1,2}(c)\preccurlyeq D_{0,0}D_{1,0}/c_0$, this is at most 
\begin{align}\label{Sigma6DeltaCase1Bound}
    \preccurlyeq&\ X\sum_{n_{0,0}} \sum_{D_{0,0}}\mathop{\sum\sum\sum\sum}_{\substack{n_{1,0}D_{0,1}D_{1,0}D_{2,0}\mid (n_{0,0}D_{0,0})^\infty\\(n_{1,0}D_{1,0}D_{2,0},D_{0,1})=1\\(D_{1,0},n_{0,0})=1}}\mathop{\sum\sum\sum}_{\substack{(n_1n_2D_0,n_{0,0}D_{0,0})=(n_1n_2,D_0)=1\\ \nu_0=n_{0,0}n_{1,0}n_1n_2\\ \mathcal{D}_0=n_{0,0}D_{0,0}D_{0,1}D_0\asymp\Delta_0}}\mathop{\sum\sum}_{\substack{(a_1a_2,n_1\mathcal{D}_0)=1}}\mathop{\sum\sum}_{\substack{b_0\mid n_{0,0}n_2(n_{1,0},D_{0,0}D_{2,0})\\ c_0\mid D_{0,0}D_{1,0}}}\frac{a_1a_2(b_0c_0)^2n_1n_2TN_0}{D_{0,1}\Delta_0}\nonumber\\
    &\ \times \frac{\Delta_2}{n_{0,0}D_{0,0}D_{2,0}}\left(1+\frac{N_1}{D_0}\right)\left(\frac{\Delta_1}{a_1a_2\nu_0D_{0,0}D_{1,0}}\right)^3\frac{1}{\Delta_1^3\Delta_2^2}\left(\frac{n_{0,0}D_{0,0}}{(b_0,n_{0,0}D_{0,0})}\right)^2\Delta_2 \frac{D_{0,0}D_{1,0}}{c_0}\Upsilon \nonumber\\
    \preccurlyeq &\ X\left(\sqrt{\frac{X\Delta_2}{\Delta_0\Delta_1^2}}+\frac{X^{1/3}}{\sqrt{\Delta_0\Delta_1}}\right)\left(\frac{\sqrt{X}\Delta_2^{3/2}+\Delta_1^{3/2}}{\Delta_0^{3/2}}+\frac{\Delta_1^2}{\Delta_0}\right)\Upsilon 
\end{align}
Let us now first assume 
$$\Delta_2\preccurlyeq\Delta_1X^{-1/3}.$$
Then using \eqref{Delta1Delta2Truncation} and \eqref{I0Bound}, the combined contribution of \eqref{contr10}, \eqref{contr1}, and \eqref{Sigma6DeltaCase1Bound} is
\begin{equation}\label{CASE1}
\Sigma_6(\Delta)^2 \preccurlyeq X^{2/3} (\Delta_0\Delta_1)^{3/2} \Delta_2 + X^{2/3} \Delta_1^{5/2} \preccurlyeq X^{4/3} \Delta_1 + X^{2/3} \Delta_1^{5/2}. 
\end{equation}

{
From now on we assume $\Delta_2\geq \Delta_1X^{-1/3+\varepsilon}$.  Together with our general assumption $\Delta_0\Delta_1 \leq X^{2/3 - \varepsilon}$, we are now in a position to decompose $\Phi_{\omega_6} = \Phi^{\sharp}_{\omega_6} + \Phi_{\omega_6}^{\flat}$ as in \eqref{phidecomp}  to use Lemma \ref{lemF1}b). We first bound the contribution  of $\Phi_{\omega_6}^{\flat}$. This allows us to replace the bound \eqref{I0Bound} with  
\begin{align*}
    \preccurlyeq \left( \frac{\Delta_0\Delta_1^2}{X\Delta_2}\right)^2.
\end{align*}
Hence the same analysis yields that the contribution $\Sigma_{6; \flat}(\Delta)^2$, say, of   $\Phi^{\flat}_{\omega_6}$ to $\Sigma_6(\Delta)^2$ is bounded by 
\begin{align}\label{Sigma6DeltaErrorBound}
  \Sigma_{6; \flat}(\Delta)^2  \preccurlyeq &\ X\Delta_1\Delta_2\left(1+\frac{\sqrt{X}\Delta_2^{3/2}+\Delta_1^{3/2}}{\Delta_0^{3/2}\Delta_1^2}\right)\left( \Delta_0+\frac{\Delta_0^{3/2}\Delta_1\Delta_2}{\sqrt{X}\Delta_2^{3/2}+\Delta_1^{3/2}} \right)\left(\sqrt{\frac{X\Delta_2}{\Delta_0\Delta_1^2}}+\frac{X^{1/3}}{\sqrt{\Delta_0\Delta_1}}\right)\left(\frac{\Delta_0\Delta_1^2}{X\Delta_2}\right)^2\nonumber\\
    &\ +X\left(\sqrt{\frac{X\Delta_2}{\Delta_0\Delta_1^2}}+\frac{X^{1/3}}{\sqrt{\Delta_0\Delta_1}}\right)\left(\frac{\sqrt{X}\Delta_2^{3/2}+\Delta_1^{3/2}}{\Delta_0^{3/2}}+\frac{\Delta_1^2}{\Delta_0}\right)\left(1+\frac{\Delta_2}{\Delta_1}\right)\left(\frac{\Delta_0\Delta_1^2}{X\Delta_2}\right)^2 \nonumber\\
    \preccurlyeq&\ \frac{\Delta_0^{5/2}\Delta_1^4}{\sqrt{X\Delta_2}}+\Delta_0\Delta_1^2\Delta_2+\sqrt{\Delta_0}\Delta_1^{7/2}\preccurlyeq X^{4/3}\Delta_1+X^{1/3}\Delta_1^3\preccurlyeq X^{4/3}\Delta_1+X^{2/3}\Delta_1^{5/2}.
\end{align}

\textbf{6) Cauchy-Schwarz and Poisson II.} Now we analyse the second case, where $\Delta_2\geq \Delta_1 X^{-1/3+\varepsilon}$  
and $\Phi_{\omega_6}$ is replaced with the ``main term'' $\Phi^{\sharp}_{\omega_6}$. 
The analysis in the present case differs in that we keep the   $t$-integral inside the Cauchy-Schwarz step.  With this in mind, we return to \eqref{5.2StartingPoint} in this proof, perform a dyadic subdivision on $a_1 \asymp A_1$, fix a sign $\sigma_1 \in \{\pm 1\}$, use a smooth function $V$ to restrict $n'$ to a dyadic interval $n' \asymp N_0'$, say, and call the corresponding contribution before Cauchy-Schwarz $\Sigma_{6; \sharp}^{\sigma_1, A_1}(\Delta, N_0').$ 
Writing \begin{align}\label{F1Signsum}
    F^{\epsilon_1, \epsilon_2, \eta_1; \sharp}_1(a,b,c)=\sum_{\sigma_1,\sigma_2=\pm1}F^{\epsilon_1, \epsilon_2, \eta_1; \sharp}_{1,\sigma_1,\sigma_2}(a,b,c)
\end{align}
according to the decomposition in Lemma \ref{lemF1}, we use the Cauchy-Schwarz inequality to obtain the bound 
 \begin{align}\label{Sigma6Sharp5.2Case1}
    \Sigma_{6; \sharp}^{\sigma_1, A_1}&(\Delta, N_0')^2\preccurlyeq  
  \sum_{\substack{\epsilon_1,\epsilon_2,\eta_1, \sigma_2=\pm1}} \sum_{n_{0,0}} \sum_{D_{0,0}}\mathop{\sum\sum\sum\sum}_{\substack{n_{1,0}D_{0,1}D_{1,0}D_{2,0}\mid (n_{0,0}D_{0,0})^\infty\\(n_{1,0}D_{1,0}D_{2,0},D_{0,1})=1\\(D_{1,0},n_{0,0})=1}}\mathop{\sum\sum\sum}_{\substack{(n_1n_2D_0,n_{0,0}D_{0,0})=(n_1n_2,D_0)=1\\ \nu_0=n_{0,0}n_{1,0}n_1n_2\\ \mathcal{D}_0=n_{0,0}D_{0,0}D_{0,1}D_0\asymp \Delta_0}}\mathop{\sum\sum}_{\substack{(a_1a_2,n_1\mathcal{D}_0)=1\\ a_1\asymp A_1}}\mathop{\sum\sum}_{\substack{b_0\mid n_{0,0}n_2(n_{1,0},D_{0,0}D_{2,0})\\ c_0\mid D_{0,0}D_{1,0}}}\nonumber\\
    & \times \frac{a_1a_2(b_0c_0)^2n_1n_2}{\nu_0^2D_{0,1}\mathcal{D}_0}\sum_{n'} \frac{1}{n'}V\left(\frac{n'}{N_0'}\right) \left|\mathop{\sum\sum}_{\substack{(cd_2,\mathcal{D}_0)=(c,a_1\nu_0/b_0)=(d_2,cn_1)=1\\ \nu_0D_{0,0}D_{1,0}a_1a_2c=\mathcal{D}_1\asymp \Delta_1\\ n_{0,0}n_2D_{0,0}D_{2,0}d_2=\mathcal{D}_2\asymp \Delta_2}}\frac{c}{\mathcal{D}_2}\mathcal{C}_{1,0}F^{\epsilon_1, \epsilon_2, \eta_1; \sharp}_{1,\sigma_1,\sigma_2}\left(\frac{\nu_0^2a_2n'X}{\mathcal{D}_1^3},\frac{X\mathcal{D}_2}{\mathcal{D}_0\mathcal{D}_1^2},\frac{\mathcal{D}_1}{\mathcal{D}_0\mathcal{D}_2^2}\right)
   \right|^2,
\end{align}
where $\mathcal{C}_{1,0}$ is as in \eqref{C10Def}. Opening the square and applying Poisson summation on the $n'$-sum in the same manner as the prior case, we obtain
\begin{align}\label{SDeltaSecondCaseAfterPoisson}
   \Sigma_{6; \sharp}^{\sigma_1, A_1}(\Delta, N_0')^2
    \preccurlyeq &\ \sum_{\epsilon_1,\epsilon_2,\eta_1, \sigma_2=\pm1} \sum_{n_{0,0}} \sum_{D_{0,0}}\mathop{\sum\sum\sum\sum}_{\substack{n_{1,0}D_{0,1}D_{1,0}D_{2,0}\mid (n_{0,0}D_{0,0})^\infty\\(n_{1,0}D_{1,0}D_{2,0},D_{0,1})=1\\(D_{1,0},n_{0,0})=1}}\mathop{\sum\sum\sum}_{\substack{(n_1n_2D_0,n_{0,0}D_{0,0})=(n_1n_2,D_0)=1\\ \nu_0=n_{0,0}n_{1,0}n_1n_2\\ \mathcal{D}_0=n_{0,0}D_{0,0}D_{0,1}D_0\asymp\Delta_0}}\mathop{\sum\sum}_{\substack{(a_1a_2,n_1\mathcal{D}_0)=1\\ a_1\asymp A_1}}\mathop{\sum\sum}_{\substack{b_0\mid n_{0,0}n_2(n_{1,0},D_{0,0}D_{2,0})\\ c_0\mid D_{0,0}D_{1,0}}}\nonumber\\
    &\ \times \frac{a_1a_2(b_0c_0)^2n_1n_2}{\nu_0^2D_{0,1}\mathcal{D}_0}\sum_{n}\mathop{\sum\sum\sum\sum}_{\substack{(cc'd_2d_2',\mathcal{D}_0)=(cc',a_1\nu_0/b_0)=1\\ (c,d_2)=(c',d_2')=(d_2d_2',n_1)=1\\ \nu_0D_{0,0}D_{1,0}a_1a_2c=\mathcal{D}_1, \nu_0D_{0,0}D_{1,0}a_1a_2c'=\mathcal{D}_1'\asymp \Delta_1\\ n_{0,0}n_2D_{0,0}D_{2,0}d_2=\mathcal{D}_2, n_{0,0}n_2D_{0,0}D_{2,0}d_2'=\mathcal{D}_2'\asymp \Delta_2}}\frac{cc'}{\mathcal{D}_2\mathcal{D}_2'}\mathcal{C}_{1,1}\mathcal{J}_0,
\end{align}
where \begin{align*}
    \mathcal{J}_0=\int_0^\infty V(y)F^{\epsilon_1, \epsilon_2, \eta_1; \sharp}_{1,\sigma_1,\sigma_2}\left(\frac{\nu_0^2a_2N_0'Xy}{\mathcal{D}_1^3},\frac{X\mathcal{D}_2}{\mathcal{D}_0\mathcal{D}_1^2},\frac{\mathcal{D}_1}{\mathcal{D}_0\mathcal{D}_2^2}\right)\overline{F^{\epsilon_1, \epsilon_2, \eta_1; \sharp}_{1,\sigma_1,\sigma_2}\left(\frac{\nu_0^2a_2N_0'Xy}{\mathcal{D}_1'^3},\frac{X\mathcal{D}_2'}{\mathcal{D}_0\mathcal{D}_1'^2},\frac{\mathcal{D}_1'}{\mathcal{D}_0\mathcal{D}_2'^2}\right)}e\left(-\frac{nN_0'y}{D_{0,0}D_{1,0}cc'}\right)\frac{\mathrm{d}y}{y}.
\end{align*}
 with $\mathcal{C}_{1,1}$ defined in \eqref{C11Def}. 

Lemma \ref{lemF1}b) tells us under the present assumptions
$$a^j \frac{\partial^j}{\partial a^j}F^{\epsilon_1, \epsilon_2, \eta_1; \sharp}_{1,\sigma_1,\sigma_2}(a, b, c) \ll_j \Big(\frac{b}{c}\Big)^{1/3} b^{j/2}$$
for any $j \geq 0$. Hence repeated integration by parts in $y$ gives us arbitrary savings unless \begin{align*}
    |n|\preccurlyeq N_1':=\frac{\Delta_1^2}{(a_1a_2\nu_0)^2D_{0,0}D_{1,0}N_0'}\sqrt{\frac{X\Delta_2}{\Delta_0\Delta_1^2}}.
\end{align*}
In addition, we have the bound \begin{align*}
    \mathcal{J}_0\preccurlyeq \frac{X^{2/3}\Delta_2^2}{\Delta_1^2}\delta\left(N_0'\asymp \frac{\sqrt{X}\Delta_2^{3/2}}{a_2\Delta_0^{3/2}\nu_0^2}\right).
\end{align*}

We first bound the contribution of $n=0$. Using \eqref{C110Bound}, the contribution of $n=0$ to \eqref{SDeltaSecondCaseAfterPoisson} is bounded by (recall \eqref{Delta1Delta2Truncation})\begin{align}\label{Sigma6Deltan0Bound}
    \preccurlyeq &\ \sum_{\epsilon_1,\epsilon_2,\eta_1=\pm1} \sum_{n_{0,0}} \sum_{D_{0,0}}\mathop{\sum\sum\sum\sum}_{\substack{n_{1,0}D_{0,1}D_{1,0}D_{2,0}\mid (n_{0,0}D_{0,0})^\infty\\(n_{1,0}D_{1,0}D_{2,0},D_{0,1})=1\\(D_{1,0},n_{0,0})=1}}\mathop{\sum\sum\sum}_{\substack{(n_1n_2D_0,n_{0,0}D_{0,0})=(n_1n_2,D_0)=1\\ \nu_0=n_{0,0}n_{1,0}n_1n_2\\ \mathcal{D}_0=n_{0,0}D_{0,0}D_{0,1}D_0\asymp\Delta_0}}\mathop{\sum\sum}_{\substack{(a_1a_2,n_1\mathcal{D}_0)=1}}\mathop{\sum\sum}_{\substack{b_0\mid n_{0,0}n_2(n_{1,0},D_{0,0}D_{2,0})\\ c_0\mid D_{0,0}D_{1,0}}}\frac{a_1a_2(b_0c_0)^2n_1n_2}{\nu_0^2D_{0,1}\mathcal{D}_0}\nonumber\\
    &\ \times \mathop{\sum\sum\sum}_{\substack{c, d_2, d_2'\\ d_2\equiv d_2'\Mod{c}\\ n_{0,0}n_2D_{0,0}D_{2,0}d_2=\mathcal{D}_2, n_{0,0}n_2D_{0,0}D_{2,0}d_2'=\mathcal{D}_2'\asymp \Delta_2\\ \nu_0D_{0,0}D_{1,0}a_1a_2c=\mathcal{D}_1\asymp \Delta_1}}\frac{(\Delta_1/(a_1a_2\nu_0D_{0,0}D_{1,0}))^2}{\Delta_2^2}\left(\frac{n_{0,0}D_{0,0}}{(b_0,n_{0,0}D_{0,0})}\right)^2\Delta_2 \frac{D_{0,0}D_{1,0}}{c_0}\frac{X^{2/3}\Delta_2^2}{\Delta_1^2}\nonumber\\
    \preccurlyeq &\ X^{2/3}\Delta_1\Delta_2^2(1+\Delta_2/\Delta_1)\preccurlyeq X^{4/3}\Delta_1+X^{5/3}.
\end{align}

\textbf{7) The non-zero frequencies.} For $n\neq0$, we have $\mathcal{C}_{1,1} = \mathcal{C}_{1, 1}^{\sharp} + \mathcal{C}_{1, 1 }^{\flat}$ as before, and analysed in \eqref{C11Bound} and \eqref{C11Asymp}. Applying \eqref{C11Bound}, the contribution of $\mathcal{C}_{1, 1 }^{\flat}$ to \eqref{SDeltaSecondCaseAfterPoisson} is bounded by (recall \eqref{Delta1Delta2Truncation})
\begin{align*}
    \preccurlyeq &\ \sum_{n_{0,0}} \sum_{D_{0,0}}\mathop{\sum\sum\sum\sum}_{\substack{n_{1,0}D_{0,1}D_{1,0}D_{2,0}\mid (n_{0,0}D_{0,0})^\infty\\(n_{1,0}D_{1,0}D_{2,0},D_{0,1})=1\\(D_{1,0},n_{0,0})=1}}\mathop{\sum\sum\sum}_{\substack{(n_1n_2D_0,n_{0,0}D_{0,0})=(n_1n_2,D_0)=1\\ \nu_0=n_{0,0}n_{1,0}n_1n_2\\ \mathcal{D}_0=n_{0,0}D_{0,0}D_{0,1}D_0\asymp \Delta_0}}\mathop{\sum\sum}_{\substack{(a_1a_2,n_1\mathcal{D}_0)=1\\ \nu_0D_{0,0}D_{1,0}a_1a_2\ll \Delta_1}}\mathop{\sum\sum}_{\substack{b_0\mid n_{0,0}n_2(n_{1,0},D_{0,0}D_{2,0})\\ c_0\mid D_{0,0}D_{1,0}}}\frac{a_1a_2(b_0c_0)^2n_1n_2}{\nu_0^2D_{0,1}\Delta_0}\sum_{0\neq |n|\preccurlyeq N_1'}\nonumber\\
    &\ \times \mathop{\sum\sum}_{\substack{(d_2d_2',\mathcal{D}_0)=1\\ n_{0,0}n_2D_{0,0}D_{2,0}d_2=\mathcal{D}_2\asymp \Delta_2\\ n_{0,0}n_2D_{0,0}D_{2,0}d_2'=\mathcal{D}_2'\asymp\Delta_2}}\mathop{\sum\sum}_{\substack{|k|,|k'|\ll K\\ (\nu_0\mathcal{D}_0,n\mathcal{D}_2)^2kk'\neq (\nu_0\mathcal{D}_0)^2\\ \eqref{kk'DivisionCondition} \text{ holds}}}\frac{(\Delta_1/(a_1a_2\nu_0D_{0,0}D_{1,0}))^2}{\Delta_2^2}\left(\frac{n_{0,0}D_{0,0}}{(b_0,n_{0,0}D_{0,0})}\right)^2\Delta_2 \frac{D_{0,0}D_{1,0}}{c_0}\frac{X^{2/3}\Delta_2^2}{\Delta_1^2}\delta\left(N_0'\asymp \frac{\sqrt{X}\Delta_2^{3/2}}{a_2\Delta_0^{3/2}\nu_0^2}\right)\nonumber\\
   \preccurlyeq &\ X^{2/3}\Delta_0\Delta_2^3\left(1+ {\Delta_0\Delta_1}/{\Delta_2}\right)\preccurlyeq X^{5/3}+X^{4/3}\Delta_1, 
\end{align*}
which is already contained in \eqref{Sigma6Deltan0Bound}. Applying \eqref{C11Asymp}, the contribution of $\mathcal{C}_{1, 1 }^{\sharp}$ to \eqref{SDeltaSecondCaseAfterPoisson} is bounded by
\begin{align}\label{Sigma6Deltannon0BoundBigA1}
    \preccurlyeq &\ \sum_{\epsilon_1,\epsilon_2,\eta_1=\pm1} \sum_{n_{0,0}} \sum_{D_{0,0}}\mathop{\sum\sum\sum\sum}_{\substack{n_{1,0}D_{0,1}D_{1,0}D_{2,0}\mid (n_{0,0}D_{0,0})^\infty\\(n_{1,0}D_{1,0}D_{2,0},D_{0,1})=1\\(D_{1,0},n_{0,0})=1}}\mathop{\sum\sum\sum}_{\substack{(n_1n_2D_0,n_{0,0}D_{0,0})=(n_1n_2,D_0)=1\\ \nu_0=n_{0,0}n_{1,0}n_1n_2\\ \mathcal{D}_0=n_{0,0}D_{0,0}D_{0,1}D_0\asymp\Delta_0}}\mathop{\sum\sum}_{\substack{(a_1a_2,n_1\mathcal{D}_0)=1\\ a_1\asymp A_1}}\mathop{\sum\sum}_{\substack{b_0\mid n_{0,0}n_2(n_{1,0},D_{0,0}D_{2,0})\\ c_0\mid D_{0,0}D_{1,0}}}\frac{a_1a_2(b_0c_0)^2n_1n_2}{\nu_0^2D_{0,1}\mathcal{D}_0}\nonumber\\
    &\ \times \sum_{0\neq |n|\preccurlyeq N_1'}\mathop{\sum\sum}_{\substack{c, d_2\\ n_{0,0}n_2D_{0,0}D_{2,0}d_2=\mathcal{D}_2=\mathcal{D}_2'\asymp \Delta_2\\ \nu_0\mathcal{D}_0\mid n\mathcal{D}_2\\ c'= c+\epsilon_2\eta_1 \frac{a_1n\mathcal{D}_2}{\nu_0\mathcal{D}_0} \\ \nu_0D_{0,0}D_{1,0}a_1a_2c=\mathcal{D}_1, \nu_0D_{0,0}D_{1,0}a_1a_2c'=\mathcal{D}_1'\asymp \Delta_1}}\frac{(\Delta_1/(a_1a_2\nu_0D_{0,0}D_{1,0}))^2}{\Delta_2^2}\left(\frac{n_{0,0}D_{0,0}}{(b_0,n_{0,0}D_{0,0})}\right)^2\Delta_2 \frac{D_{0,0}D_{1,0}}{c_0}\frac{X^{2/3}\Delta_2^2}{\Delta_1^2}\nonumber\\
    \preccurlyeq &\ X^{2/3}\Delta_1^2\Delta_2A_1^{-3}\preccurlyeq X\Delta_1^2A_1^{-3}.
\end{align}

This bound is useful if $A_1$ is big. If $A_1$ is small, we argue differently and treat first the case $\sigma_1=\epsilon_2\eta_1$. We estimate $\mathcal{J}_0$ in a better manner given $\mathcal{D}_2=\mathcal{D}_2'$ and $c'= c+\epsilon_2\eta_1 \frac{a_1n\mathcal{D}_2}{\nu_0\mathcal{D}_0}$. We temporarily write \begin{align*}
    A=\nu_0^2a_2N_0'X, \quad B=\frac{X\mathcal{D}_2}{\mathcal{D}_0}, \quad C=\frac{1}{\mathcal{D}_0\mathcal{D}_2^2}, \quad M=\frac{nN_0'}{D_{0,0}D_{1,0}cc'}.
\end{align*}
Then Lemma \ref{lemF1}b) gives us, up to a negligible error, 
\begin{align*}
    \mathcal{J}_0=\frac{X^{2/3}\Delta_2^2}{\Delta_1^2}\delta\left(N_0'\asymp \frac{\sqrt{X}\Delta_2^{3/2}}{a_2\Delta_0^{3/2}\nu_0^2}\right)\int_0^\infty H_{\sigma_1,\sigma_2}(y)e\left(\phi(y)\right)\mathrm{d}y,
\end{align*}
where 
\begin{align*}
    \phi(y)=-\sigma_1\frac{Ay}{B}\left(\frac{1}{\mathcal{D}_1}-\frac{1}{\mathcal{D}_1'}\right)-My=-(1+\sigma_1\epsilon_2\eta_1)My
\end{align*}
with $H_{\sigma_1\sigma_2}(x_0,x_0',y)$ being a smooth function supported on $[1,2]^3$ satisfying 
   $ x^{j_1}x'^{j_2}y^{j_3}H_{\sigma_1,\sigma_2}^{j_1,j_2,j_3}(x,x',y)\ll_{j_1,j_2,j_3} 1.$
If $\sigma_1=\epsilon_2\eta_1$, repeated integration by parts then gives us arbitrary savings unless \begin{align*}
    M \preccurlyeq 1 \quad \Longleftrightarrow \quad n \preccurlyeq \frac{D_{0,0}D_{1,0}c^2}{N_0'}\preccurlyeq \frac{\Delta_0^{3/2}\Delta_1^2}{a_1^2a_2D_{0,0}D_{1,0} \sqrt{X}\Delta_2^{3/2}}.
\end{align*} 
Hence for $\sigma_1=\epsilon_2\eta_1$, the contribution of $n\neq0$ to \eqref{SDeltaSecondCaseAfterPoisson} is bounded by  (recall \eqref{Delta1Delta2Truncation})
\begin{align}\label{Sigma6Deltannon0Bound}
    \preccurlyeq &\ \sum_{\epsilon_1,\epsilon_2,\eta_1=\pm1} \sum_{n_{0,0}} \sum_{D_{0,0}}\mathop{\sum\sum\sum\sum}_{\substack{n_{1,0}D_{0,1}D_{1,0}D_{2,0}\mid (n_{0,0}D_{0,0})^\infty\\(n_{1,0}D_{1,0}D_{2,0},D_{0,1})=1\\(D_{1,0},n_{0,0})=1}}\mathop{\sum\sum\sum}_{\substack{(n_1n_2D_0,n_{0,0}D_{0,0})=(n_1n_2,D_0)=1\\ \nu_0=n_{0,0}n_{1,0}n_1n_2\\ \mathcal{D}_0=n_{0,0}D_{0,0}D_{0,1}D_0\asymp\Delta_0}}\mathop{\sum\sum}_{\substack{(a_1a_2,n_1\mathcal{D}_0)=1}}\mathop{\sum\sum}_{\substack{b_0\mid n_{0,0}n_2(n_{1,0},D_{0,0}D_{2,0})\\ c_0\mid D_{0,0}D_{1,0}}}\frac{a_1a_2(b_0c_0)^2n_1n_2}{\nu_0^2D_{0,1}\mathcal{D}_0}\nonumber\\
    &\ \times \sum_{0\neq |n|\preccurlyeq \frac{\Delta_0^{3/2}\Delta_1^2}{a_1^2a_2D_{0,0}D_{1,0} \sqrt{X}\Delta_2^{3/2}}}\mathop{\sum\sum}_{\substack{c, d_2\\ n_{0,0}n_2D_{0,0}D_{2,0}d_2=\mathcal{D}_2=\mathcal{D}_2'\asymp \Delta_2\\ \nu_0\mathcal{D}_0\mid n\mathcal{D}_2\\ c'= c+\epsilon_2\eta_1 \frac{a_1n\mathcal{D}_2}{\nu_0\mathcal{D}_0}\\ \nu_0D_{0,0}D_{1,0}a_1a_2c=\mathcal{D}_1, \nu_0D_{0,0}D_{1,0}a_1a_2c'=\mathcal{D}_1'\asymp \Delta_1}}\frac{(\Delta_1/(a_1a_2\nu_0D_{0,0}D_{1,0}))^2}{\Delta_2^2}\left(\frac{n_{0,0}D_{0,0}}{(b_0,n_{0,0}D_{0,0})}\right)^2\Delta_2 \frac{D_{0,0}D_{1,0}}{c_0}\frac{X^{2/3}\Delta_2^2}{\Delta_1^2}\nonumber\\
        \preccurlyeq &\ X^{1/6}\Delta_1^3\sqrt{\Delta_0\Delta_2}\preccurlyeq X^{1/3}\Delta_1^3.
\end{align}

\textbf{8) Final case: Reciprocity and Voronoi summation.} Finally, we handle the remaining case, where $\Delta_2>\Delta_1 X^{-1/3+\varepsilon}$, $a_1\asymp A_1$ is small,  
and $\Phi_{\omega_6}$ is replaced with the ``main term'' $\Phi^{\sharp}_{\omega_6}$ and $\sigma_1=-\epsilon_2\eta_1$. This final case more closely resembles the generic setting, in which a main term appears. The smaller size of $\Delta_1$, however, allows this contribution to be absorbed into the error term. Its treatment therefore differs from that of the preceding two cases and follows a simplified version of the argument used later in the main part of the proof.

Similarly as in the previous case, recalling the basic definition \eqref{5.2StartingPoint} and the notation \eqref{F1Signsum} (see also \eqref{Sigma6Sharp5.2Case1})
we have \begin{align}\label{Sigma6sharp5.2Case2}
    \Sigma_{6; \sharp}^{-\epsilon_2\eta_1, A_1}(\Delta, N_0')=&\ \sum_{\epsilon_1,\epsilon_2,\eta_1, \sigma_2=\pm1} \sum_{n_{0,0}} \sum_{D_{0,0}}\mathop{\sum\sum\sum\sum}_{\substack{n_{1,0}D_{0,1}D_{1,0}D_{2,0}\mid (n_{0,0}D_{0,0})^\infty\\(n_{1,0}D_{1,0}D_{2,0},D_{0,1})=1\\(D_{1,0},n_{0,0})=1}}\mathop{\sum\sum\sum}_{\substack{(n_1n_2D_0,n_{0,0}D_{0,0})=1\\
    (n_1n_2,D_0)=(c_0,\nu_0/b_0)=1\\  \mathcal{D}_0\asymp \Delta_0}}\mathop{\sum\sum}_{\substack{a_0b_0c_0=\nu_0D_{0,0}D_{1,0}\\ b_0\mid n_{0,0}n_2(n_{1,0},D_{0,0}D_{2,0})}}\mathop{\sum\sum\sum\sum}_{\substack{(a_1a_2cd_2,n_{0,0}D_{0,0}D_0)=1\\
    (a_1a_2d_2,n_1)=1\\
    (c,a_1d_2\nu_0/b_0)=1 \\ a_1\asymp A_1, \mathcal{D}_1\asymp \Delta_1, \mathcal{D}_2\asymp \Delta_2}}  
   \nonumber\\
   &\ \times \frac{\mu(a_0a_1a_2)\mu(a_1)b_0c_0c}{\nu_0\mathcal{D}_0\mathcal{D}_2}\sum_{n'} \frac{B(a_2n',\nu_0)}{n'}V\left(\frac{n'}{N_0'}\right) \mathcal{C}_{1,0}
   F^{\epsilon_1, \epsilon_2, \eta_1; \sharp}_{1,-\epsilon_2\eta_1,\sigma_2}\left(\frac{\nu_0^2a_2n'X}{\mathcal{D}_1^3},\frac{X\mathcal{D}_2}{\mathcal{D}_0\mathcal{D}_1^2},\frac{\mathcal{D}_1}{\mathcal{D}_0\mathcal{D}_2^2}\right).
\end{align}
with $\mathcal{C}_{1,0}$ in \eqref{C10Def}, and $$\nu_0=n_{0,0}n_{1,0}n_1n_2, \quad \mathcal{D}_0=n_{0,0}D_{0,0}D_{0,1}D_0, \quad \mathcal{D}_1=\nu_0D_{0,0}D_{1,0}a_1a_2c, \quad \mathcal{D}_2=n_{0,0}n_2D_{0,0}D_{2,0}d_2.$$

We first apply reciprocity to get \begin{align*}
    \mathcal{C}_{1,0}=&\ \sumast_{\substack{\alpha\Mod{\mathcal{D}_0}\\ Y_1\equiv 0\Mod{b_0D_{0,1}D_0}}}S\left(m,\frac{Y_1\overline{\alpha}}{D_{0,1}D_0};\mathcal{D}_2\right)\sumast_{\substack{\gamma\Mod{D_{0,0}D_{1,0}}\\ \gamma\equiv \overline{\frac{Y_1}{b_0D_{0,1}D_0}}\frac{\nu_0}{b_0}\Mod{c_0}}}e\left(\frac{-\epsilon_2\eta_1\overline{a_1c}\gamma n'}{D_{0,0}D_{1,0}}\right)\nonumber\\
    &\ \times e\left(\epsilon_2\eta_1\frac{\overline{c}n_{0,0}n_{1,0}n_1D_{0,1}D_0n'}{D_{0,0}D_{1,0}D_{2,0}a_1d_2}-\epsilon_2\eta_1\frac{\nu_0^2\mathcal{D}_0a_2n'}{\mathcal{D}_1\mathcal{D}_2}\right).
\end{align*}

The $n'$-sum is now of the form \begin{align*}
    \sum_{n'} \frac{B(a_2n',\nu_0)}{n'}V\left(\frac{n'}{N_0'}\right)e\left(\epsilon_2\eta_1\frac{\overline{c}(n_{0,0}n_{1,0}n_1D_{0,1}D_0-\overline{a_1}D_{2,0}a_1d_2\gamma )n'}{D_{0,0}D_{1,0}D_{2,0}a_1d_2}-\epsilon_2\eta_1\frac{\nu_0^2\mathcal{D}_0a_2n'}{\mathcal{D}_1\mathcal{D}_2}\right)F^{\epsilon_1, \epsilon_2, \eta_1; \sharp}_{1,-\epsilon_2\eta_1,\sigma_2}\left(\frac{\nu_0^2a_2n'X}{\mathcal{D}_1^3},\frac{X\mathcal{D}_2}{\mathcal{D}_0\mathcal{D}_1^2},\frac{\mathcal{D}_1}{\mathcal{D}_0\mathcal{D}_2^2}\right).
\end{align*}
Using the Hecke relations in Lemma \ref{lem.Hecke}, we have \begin{align*}
    B(a_2n',\nu_0)=\mathop{\sum\sum}_{\substack{a_3a_4n''=n'\\a_4\mid \nu_0a_3, a_3|a_2}}\mu(a_3)\mu(a_4)B\left(\frac{a_2}{a_3},\frac{\nu_0a_3}{a_4}\right)B(n'',1).
\end{align*}
Write $q=(a_3a_4(n_{0,0}n_{1,0}n_1D_{0,1}D_0-\overline{a_1}D_{2,0}a_1d_2\gamma), D_{0,0}D_{1,0}D_{2,0}a_1d_2)$. Since $a_0$ is squarefree and $D_{1,0}\mid D_{0,0}^\infty$, $a_0b_0c_0=\nu_0D_{0,0}D_{1,0}$ implies that $D_{1,0}\mid b_0c_0$. The congruence on $\gamma$ gives $\gamma Y_1\equiv\nu_0D_{0,1}D_0\Mod{b_0c_0}$. Since $Y_1\equiv n_2D_{2,0}d_2\Mod{D_{1,0}}$ and $(n_2,D_{1,0})=(a_1,D_{1,0})=1$, we have
$\overline{a_1}\gamma D_{2,0}a_1d_2
\equiv n_{0,0}n_{1,0}n_1D_{0,1}D_0\Mod{D_{1,0}}.$ Hence we have $D_{1,0}\mid q$.

Applying Voronoi summation (Lemma \ref{lem.Voronoi}) with \eqref{GL3KernelAsymp} on the $n''$-sum, the original $n'$-sum is equal to \begin{align*}
    &\mathop{\sum\sum}_{\substack{a_3|a_2, a_4|\nu_0a_3}}B\left(\frac{a_2}{a_3},\frac{\nu_0a_3}{a_4}\right)\sum_\pm \sum_{r| D_{0,0}D_{1,0}D_{2,0}a_1d_2/q}\sum_n \frac{\mu(a_3)\mu(a_4)qr^{1/3}B(r,n)}{D_{0,0}D_{1,0}D_{2,0}a_1d_2(a_3a_4)^{2/3}(N_0'n)^{1/3}}\\
    &\ \times S\left(\epsilon_2\eta_1c\overline{a_3a_4(n_{0,0}n_{1,0}n_1D_{0,1}D_0-\overline{a_1}a_1D_{2,0}d_2\gamma)/q},\pm n;\frac{D_{0,0}D_{1,0}D_{2,0}a_1d_2}{qr}\right)\\ &\ \times G_\pm\left(\frac{\nu_0^2Xa_2N_0'}{\mathcal{D}_1^3},\frac{X\mathcal{D}_2}{\mathcal{D}_0\mathcal{D}_1^2},\frac{\mathcal{D}_1}{\mathcal{D}_0\mathcal{D}_2^2}, \frac{q^3r^2nN_0'}{a_3a_4(D_{0,0}D_{1,0}D_{2,0}a_1d_2)^3}\right),
\end{align*}
where \begin{align*}
    G_\pm(A,B,C,D)=\int_0^\infty\frac{V(y)}{y^{4/3}}W_{-\mu_0}^\pm(Dy)e\left(\pm3(Dy)^{1/3}-\epsilon_2\eta_1\frac{Ay}{B}\right)F^{\epsilon_1, \epsilon_2, \eta_1; \sharp}_{1,-\epsilon_2\eta_1,\sigma_2}(Ay,B,C)\mathrm{d}y
\end{align*}
with $W_\mu^\pm$ satisfying $x^j\frac{\mathrm{d}^j}{\mathrm{d}x^j}W_\mu^\pm(x)\ll_{\mu,j} 1$ for any $j\geq0$.

Using the expression of $F^{\epsilon_1, \epsilon_2, \eta_1; \sharp}_{1,-\epsilon_2\eta_1,\sigma_2}$ from Lemma \ref{lemF1}b), we have up to a negligible error \begin{align*}
    G_\pm(A,B,C,D)= \left(\frac{B}{C}\right)^{1/3}\delta\left(A\asymp B^{3/2}\right)\int_0^\infty H(x_0,y)e\left(\phi(y)\right)\mathrm{d}y
\end{align*}
where \begin{align*}
    x_0=\frac{1}{B}\left(\left(\frac{Ay}{B}\right)^{2/3}-\sigma_2(mC)^{1/3}\right)^3\asymp 1,
\end{align*}
with $H(x_0,y)$ being a smooth function supported on $[1,2]^2$ satisfying
   $ x^{j_1}y^{j_2}H^{j_1,j_2}(x,y)\ll_{j_1,j_2} 1$,
and the phase is given by
\begin{align*}
    \phi(y)=\pm3(Dy)^{1/3}-\epsilon_2\eta_1\frac{Ay}{B}+\epsilon_2\eta_1\frac{Ay}{B}-3\epsilon_2\eta_1\sigma_2\left(\frac{mACy}{B}\right)^{1/3}=\pm3(Dy)^{1/3}-3\epsilon_2\eta_1\sigma_2\left(\frac{mACy}{B}\right)^{1/3}.
\end{align*}
The cancellation in part of the above phase shows that the reciprocity step is natural. Repeated integration by parts implies that up to a negligible error, \begin{align*}
    G_\pm(A,B,C,D)\preccurlyeq \left(\frac{B}{C}\right)^{1/3}\delta\left(A\asymp B^{3/2}, \left|D^{1/3}\mp \epsilon_2\eta_1\sigma_2\left(\frac{mAC}{B}\right)^{1/3}\right|\preccurlyeq1\right).
\end{align*}

Inserting the above analysis back into \eqref{Sigma6sharp5.2Case2} and applying the Weil bound on the Kloosterman sums, we have \begin{align}\label{SmallA1CaseBound}
    \Sigma_{6; \sharp}^{-\epsilon_2\eta_1, A_1}(\Delta, N_0')\preccurlyeq &\ \sum_{n_{0,0}} \sum_{D_{0,0}}\mathop{\sum\sum\sum\sum}_{\substack{n_{1,0}D_{0,1}D_{1,0}D_{2,0}\mid (n_{0,0}D_{0,0})^\infty\\(n_{1,0}D_{1,0}D_{2,0},D_{0,1})=1\\(D_{1,0},n_{0,0})=1}}\mathop{\sum\sum\sum}_{\substack{(n_1n_2D_0,n_{0,0}D_{0,0})=1\\\nu_0=n_{0,0}n_{1,0}n_1n_2\\
    (n_1n_2,D_0)=(c_0,\nu_0/b_0)=1\\  \mathcal{D}_0=n_{0,0}D_{0,0}D_{0,1}D_0\asymp \Delta_0}}\mathop{\sum\sum}_{\substack{a_0b_0c_0=\nu_0D_{0,0}D_{1,0}\\ b_0\mid n_{0,0}n_2(n_{1,0},D_{0,0}D_{2,0})}}\mathop{\sum\sum\sum\sum}_{\substack{(a_1a_2cd_2,n_{0,0}D_{0,0}D_0)=1\\
    a_1\asymp A_1\\ \mathcal{D}_1=\nu_0D_{0,0}D_{1,0}a_1a_2c\asymp \Delta_1\\ \mathcal{D}_2=n_{0,0}n_2D_{0,0}D_{2,0}d_2\asymp \Delta_2}}
   \nonumber\\
   &\ \times \mathop{\sum\sum}_{\substack{a_3|a_2, a_4|\nu_0a_3}}\left|B\left(\frac{a_2}{a_3},\frac{\nu_0a_3}{a_4}\right)\right|\mathop{\sum\sum}_{\substack{qr| D_{0,0}D_{1,0}D_{2,0}a_1d_2\\ D_{1,0}|q}}\sum_n \frac{b_0c_0cqr^{1/3}|B(r,n)|}{\nu_0\Delta_0\Delta_2D_{0,0}D_{1,0}D_{2,0}a_1d_2(a_3a_4)^{2/3}(N_0'n)^{1/3}}\nonumber\\
   &\ \times \frac{n_{0,0}D_{0,0}}{(b_0,n_{0,0}D_{0,0})}\sqrt{\Delta_2}\frac{D_{0,0}D_{1,0}}{c_0}\sqrt{\frac{D_{0,0}D_{1,0}D_{2,0}a_1d_2}{r}}\frac{X^{1/3}\Delta_2}{\Delta_1}\delta\left(N_0'\asymp \frac{\sqrt{X}\Delta_2^{3/2}}{a_2\nu_0^2\Delta_0^{3/2}}, n\preccurlyeq \frac{a_1^3a_2a_3a_4\nu_0^2D_{1,0}^3}{r^2n_{0,0}^3n_2^3q^3}\right)\nonumber\\
    \preccurlyeq&\ X^{1/6}\sqrt{A_1}\Delta_0^{1/2}\Delta_1\Delta_2^{1/2}\preccurlyeq X^{1/3}\Delta_1\sqrt{A_1}.
\end{align} 

Combining \eqref{CASE1}, \eqref{Sigma6DeltaErrorBound}, \eqref{Sigma6Deltan0Bound}, \eqref{Sigma6Deltannon0BoundBigA1}, \eqref{Sigma6Deltannon0Bound}, and \eqref{SmallA1CaseBound} with 
\begin{align*}
    \min(X\Delta_1^2A_1^{-3}, (X^{1/3}\Delta_1\sqrt{A_1})^2) \preccurlyeq X^{3/4}\Delta_1^2
\end{align*}
yields \begin{align}
    \Sigma_6(\Delta)^2\preccurlyeq X^{4/3}\Delta_1+X^{2/3}\Delta_1^{5/2}+X^{5/3}+X^{3/4}\Delta_1^2+X^{1/3}\Delta_1^3,
\end{align}
which in turn gives the desired result.}
\end{proof}

As mentioned before, this lemma allows us to restrict to $\mathcal{D}_0\mathcal D_1 \geq \Delta^{\ast}$ for some 
$\Delta^{\ast}\leq X^{8/15 - \varepsilon}$ at the cost of an acceptable error. In the following we will choose $\Delta^{\ast} = X^{11/21}$, although we have some flexibility here. 

With this in mind, we insert a factor $(1 - U(\mathcal{D}_0 \mathcal{D}_1/\Delta^{\ast}))$ into the expression \eqref{afterVorfull} and analyse from now on
\begin{align*} 
    \Sigma_6(\Delta^{\ast})=& 
    \sum_{\epsilon_1,\epsilon_2,\eta_1=\pm1} \sum_{n_{0,0}} \sum_{D_{0,0}}\mathop{\sum\sum\sum\sum}_{\substack{n_{1,0}D_{0,1}D_{1,0}D_{2,0}\mid (n_{0,0}D_{0,0})^\infty\\(n_{1,0}D_{1,0}D_{2,0},D_{0,1})=1\\(D_{1,0},n_{0,0})=1}}\mathop{\sum\sum\sum\sum}_{(n_0D_0D_1D_2,n_{0,0}D_{0,0})=(n_0D_1D_2,D_0)=1} \frac{1  - U(   \frac{n_{0, 0} D_{0, 0} D_{0, 1} D_{0} \cdot n_{0,0}n_{1,0}n_0D_{0,0}D_{1,0}D_1 }{\Delta^{\ast}})
    }{n_{0,0}^2D_{0,0}^2D_{0,1}D_{2,0}D_0D_2}\\
    & \times 
    \sum_n \frac{B(n,n_{0,0}n_{1,0}n_0)}{n_{0,0}n_{1,0}n_0n}\mathcal{C}_1
   F_1^{\epsilon_1,\epsilon_2,\eta_1}\left(\frac{nX}{n_{0,0}n_{1,0}n_0(D_{0,0}D_{1,0}D_1)^3},\frac{XD_{2,0}D_2}{n_{0,0}^2n_{1,0}^2n_0^2D_{0,0}^2D_{0,1}D_{1,0}^2D_0D_1^2},\frac{n_{1,0}n_0D_{1,0}D_1}{n_{0,0}^2D_{0,0}^2D_{0,1}D_{2,0}^2D_0D_2^2}\right). \nonumber
\end{align*}

\color{black}


\subsection{Interlude II: character sum analysis}
We spend some time to transform the character sum $\mathcal{C}_1$  defined in \eqref{defC1} and occurring in the previous display, before we can apply additive reciprocity. Generically we would like to think of the important moduli being $D_1, D_2$, so that the $\alpha$-sum is essentially irrelevant and the $\beta$-sum looks like 
$$\sumast_{\beta\Mod{ D_1}} e\Big( \frac{\beta D_2}{D_1}\Big) S(\beta, n, D_1) \approx D_1 e\Big( \frac{\overline{D}_2 n}{D_1}\Big).$$
Making this precise requires some work to which this subsection is devoted.  The final expression for $\Sigma_6(\Delta^{\ast})$ is \eqref{S6BeforePoisson} below. 

Since $(D_0,n_{0,0}n_0D_{0,0}D_1D_2)=1$ and $n_{1,0}D_{0,1}D_{1,0}D_{2,0}\mid (n_{0,0}D_{0,0})^\infty$, we can split the $\alpha$-sum into   moduli $n_{0,0}D_{0,0}D_{0,1}$ and $D_0$ and evaluate to write $\mathcal{C}_1$  as \begin{align*}
    &\sumast_{\substack{\alpha\Mod{n_{0,0}D_{0,0}D_{0,1}}\\Y_1\equiv 0\Mod{D_{0,1}}}}\sumast_{\beta\Mod{n_{0,0}n_{1,0}n_0D_{0,0}D_{1,0}D_1}}e\left(\frac{\beta \overline{D_0}Y_1/D_{0,1}}{n_{0,0}n_{1,0}n_0D_{0,0}D_{1,0}D_1}\right)  S\left(\epsilon_2\beta,\eta_1 n;D_{0,0}D_{1,0}D_1\right)S\left(m,\frac{Y_1\overline{\alpha D_0}}{D_{0,1}};n_{0,0}D_{0,0}D_{2,0}D_2\right).
\end{align*}
 Consider now the $\beta$-sum. Since $(n_{0,0}n_{1,0}D_{0,0}D_{1,0},n_0D_1)=1$, for $\gamma \in (\Bbb{Z}/D_{0, 0}D_{1, 0} D_1\Bbb{Z})^{\ast}$ we have \begin{align*}
    &\sumast_{\beta \Mod{n_{0,0}n_{1,0}n_0D_{0,0}D_{1,0}D_1}}e\left(\frac{\beta \overline{D_0}Y_1/D_{0,1}}{n_{0,0}n_{1,0}n_0D_{0,0}D_{1,0}D_1}+\frac{\epsilon_2\beta\gamma}{D_{0,0}D_{1,0}D_1}\right)\\
    =&\sumast_{\beta_1\Mod{n_0D_1}}e\left(\frac{\beta_1 D_{2,0}D_2\overline{n_{0,0}n_{1,0}D_{0,0}D_{1,0}D_{0,1}D_0}}{n_0D_1}+\frac{\epsilon_2\beta_1\gamma\overline{D_{0,0}D_{1,0}}}{D_1}\right) \sumast_{\beta_2\Mod{n_{0,0}n_{1,0}D_{0,0}D_{1,0}}}e\left(\frac{\beta_2\overline{n_0D_0D_1} Y_1/D_{0,1}}{n_{0,0}n_{1,0}D_{0,0}D_{1,0}}+\frac{\epsilon_2\beta_2\gamma\overline{D_1}}{D_{0,0}D_{1,0}}\right)\\
    =&\sum_{f\mid (n_0D_1,D_{2,0}D_2+\epsilon_2n_{0,0}n_{1,0}n_0D_{0,1}D_0\gamma)}f\mu\left(\frac{n_0D_1}{f}\right) \sum_{f_0\mid (n_{0,0}n_{1,0}D_{0,0}D_{1,0},\frac{Y_1}{D_{0,1}}+\epsilon_2n_{0,0}n_{1,0}n_0D_0\gamma)}f_0\mu\left(\frac{n_{0,0}n_{1,0}D_{0,0}D_{1,0}}{f_0}\right).
\end{align*}
Notice that the first sum vanishes unless $$\nu:=(f,n_0)\mid D_2.$$ The presence of $\mu\left(n_{0,0}n_{1,0}D_{0,0}D_{1,0}/f_0\right)$ and $n_{1,0}D_{1,0}\mid (n_{0,0}D_{0,0})^\infty$ in the second sum implies that it vanishes unless $n_{1,0}D_{1,0}\mid f_0$. This in turn implies $n_{1,0}\mid D_{2,0}$ and $D_{1,0}\mid D_{2,0}D_2/n_{1,0}+\epsilon_2n_{0,0}n_0D_{0,1}D_0\gamma$. Moreover, we have \begin{align*}
    \nu_0:=\left(\frac{f_0}{n_{1,0}D_{1,0}},n_{0,0}\right)\,\Big|\,  \frac{Y_1}{n_{1, 0} D_{0, 1}} = \frac{D_{2,0}D_2/n_{1,0}+\epsilon_1n_0D_{1,0}D_1\alpha}{D_{0,1}}.
\end{align*}
This also implies $\left(\nu_0,D_{2,0}/n_{1,0}\right)=1$. Writing $$D_2=\nu D_2', \quad f=\nu f', \quad n_0=\nu n_0', \quad  f_0=\nu_0n_{1,0}D_{1,0}f_0', \quad n_{0,0}=\nu_0 n_{0,0}', \quad  D_{2,0}=n_{1,0}D_{2,0}',$$ we have  \begin{align*}
    \Sigma_6(\Delta^{\ast})=&\sum_{\epsilon_1,\epsilon_2,\eta_1=\pm1} \sum_{\nu_0}\sum_{n_{0,0}'} \sum_{D_{0,0}}\mathop{\sum\sum\sum\sum}_{\substack{n_{1,0}D_{0,1}D_{1,0}D_{2,0}'\mid (\nu_0n_{0,0}'D_{0,0})^\infty\\(n_{1,0}D_{1,0}D_{2,0}',D_{0,1})=1\\(D_{1,0},\nu_0n_{0,0}')=(D_{2,0}',\nu_0)=1}}\mathop{\sum\sum\sum\sum\sum}_{(\nu n_0'D_0D_1D_2',\nu_0n_{0,0}'D_{0,0})=(\nu n_0'D_1D_2',D_0)=1}\frac{D_{1,0}}{ \nu_0n_{0,0}'^2D_{0,0}^2D_{0,1}D_{2,0}'D_0D_2'}\\
    &\times \Big(1 - U\Big( \frac{\nu_0 n_{0,0}'D_{0,0} D_{0,1}D_0 \cdot \nu_0\nu n'_{0,0}n_{1,0}n'_0D_{0,0}D_{1,0}D_1}{\Delta^{\ast}}\Big)\Big) 
    \sum_n \frac{B(n,\nu_0\nu n_{0,0}'n_{1,0}n_0')}{\nu_0\nu n_{0,0}'n_{1,0}n_0'n}\mathcal{C}_2\\ 
    &\times F_1^{\epsilon_1,\epsilon_2,\eta_1}\left(\frac{nX}{\nu_0\nu n_{0,0}'n_{1,0}n_0'(D_{0,0}D_{1,0}D_1)^3},\frac{XD_{2,0}'D_2'}{\nu_0^2\nu n_{0,0}'^2n_{1,0}n_0'^2D_{0,0}^2D_{0,1}D_{1,0}^2D_0D_1^2},\frac{n_0'D_{1,0}D_1}{\nu_0^2\nu n_{0,0}'^2n_{1,0}D_{0,0}^2D_{0,1}D_{2,0}'^2D_0D_2'^2}\right),
\end{align*}
with \begin{align*}
    \mathcal{C}_2
    =&\sumast_{\substack{\alpha\Mod{\nu_0n_{0,0}'D_{0,0}D_{0,1}}\\Y_2 \equiv 0\Mod{\nu_0D_{0,1}}}}\sumast_{\substack{\gamma\Mod{D_{0,0}D_{1,0}D_1}\\Z_2 \equiv 0 \Mod{D_{1,0}}}}\sum_{\substack{f'\mid (D_1,Z_2)\\(f',n_0')=1}}f'\mu\left(\frac{n_0'D_1}{f'}\right)  \sum_{\substack{f_0'\mid \left(D_{0,0},\frac{1}{D_{1, 0}}  ( \frac{Y_2}{\nu_0D_{0,1}}+\epsilon_2n_{0,0}'n_0'D_0\gamma) \right)\\(f_0',n_{0,0}')=1}}f_0'\mu\left(\frac{n_{0,0}'D_{0,0}}{f_0'}\right) \\
    & \times e\left(\frac{\eta_1\overline{\gamma}n}{D_{0,0}D_{1,0}D_1}\right)S\left(m,\nu n_{1,0}\overline{D_0}\frac{Y_2\overline{\alpha}}{D_{0,1}};\nu_0\nu n_{0,0}'n_{1,0}D_{0,0}D_{2,0}'D_2'\right).
\end{align*}
where $$Y_2 = D_{2,0}'D_2'+\epsilon_1n_0'D_{1,0}D_1\alpha, \quad Z_2 = D_{2,0}'D_2'+\epsilon_2\nu_0n_{0,0}'n_0'D_{0,1}D_0\gamma.$$
Splitting the $\gamma$-sum into modulus $D_1$ and $D_{0,0}D_{1,0}$ and reordering the sums, we rewrite  $\mathcal{C}_2$ as
\begin{align*}
    &\sum_{\substack{f_0'\mid D_{0,0}\\(f_0',n_{0,0}')=1}}\sum_{\substack{f'\mid D_1\\(f',n_0')=1}}f_0'f'\mu\left(n_{0,0}'n_0'\frac{D_{0,0}}{f_0'}\frac{D_1}{f'}\right)\sumast_{\substack{\alpha\Mod{\nu_0n_{0,0}'D_{0,0}D_{0,1}}\\Y_2 \equiv 0\Mod{\nu_0D_{0,1}}}} \sumast_{\substack{\gamma_1\Mod{D_1}\\D_{2,0}'D_2'\equiv -\epsilon_2\nu_0n_{0,0}'n_0'D_{0,1}D_0\gamma_1\Mod{f'}}}e\left(\frac{\eta_1\overline{\gamma_1 D_{0,0}D_{1,0}}n}{D_1}\right)\\
    &\times \sumast_{\substack{\gamma_2\Mod{D_{0,0}D_{1,0}}\\\frac{Y_2}{\nu_0D_{0,1}}\equiv-\epsilon_2n_{0,0}'n_0'D_0\gamma_2\Mod{f_0'D_{1,0}}}}e\left(\frac{\eta_1\overline{\gamma_2 D_1}n}{D_{0,0}D_{1,0}}\right) S\left(m,\nu n_{1,0}\overline{D_0}\frac{Y_2 \overline{\alpha}}{D_{0,1}};\nu_0\nu n_{0,0}'n_{1,0}D_{0,0}D_{2,0}'D_2'\right).
\end{align*}
Now the $\gamma$-sums evaluate as follows. For the $\gamma_1$-sum, the congruence condition implies $(D_2',f')=1$ and we have \begin{align*}
    \sumast_{\substack{\gamma_1\Mod{D_1}\\D_{2,0}'D_2'\equiv -\epsilon_2\nu_0n_{0,0}'n_0'D_{0,1}D_0\gamma_1\Mod{f'}}}&e\left(\frac{\eta_1\overline{\gamma_1 D_{0,0}D_{1,0}}n}{D_1}\right)=\sum_{g\mid D_1}\mu(g)\sum_{\substack{\gamma_1\Mod{D_1/g}\\D_{2,0}'D_2'g\gamma_1\equiv -\epsilon_2\nu_0n_{0,0}'n_0'D_{0,1}D_0\Mod{f'}}}e\left(\frac{\eta_1\gamma_1\overline{ D_{0,0}D_{1,0}}n}{D_1/g}\right)\\
    =&\sum_{\substack{g\mid D_1/f'\\(g,f')=1}}\mu(g)\frac{D_1}{f'g}e\left(-\epsilon_2\eta_1\frac{\nu_0n_{0,0}'n_0'D_{0,1}D_0\overline{D_{0,0}D_{1,0}D_{2,0}'D_2'g}n}{D_1/g}\right)\delta\left(\frac{D_1}{f'g}\big| n\right).
\end{align*}
Similarly, the $\gamma_2$-sum gives us the condition \begin{align*}
(D_{2,0}',D_{1,0})=\left(\frac{D_{2,0}'D_2'+\epsilon_1n_0'D_{1,0}D_1\alpha}{\nu_0D_{0,1}},f_0'\right)=1
\end{align*}
and \begin{align*}
    \sumast_{\substack{\gamma_2\Mod{D_{0,0}D_{1,0}}\\\frac{Y_2}{\nu_0D_{0,1}}\equiv-\epsilon_2n_{0,0}'n_0'D_0\gamma_2\Mod{f_0'D_{1,0}}}}  e\left(\frac{\eta_1\overline{\gamma_2 D_1}n}{D_{0,0}D_{1,0}}\right) &= \sum_{g_0\mid D_{0,0}D_{1,0}}\mu(g_0)\sum_{\substack{\gamma_2\Mod{D_{0,0}D_{1,0}/g_0}\\\frac{Y_2}{\nu_0D_{0,1}}g_0\gamma_2\equiv-\epsilon_2n_{0,0}'n_0'D_0\Mod{f_0'D_{1,0}}}}e\left(\frac{\eta_1\gamma_2\overline{ D_1}n}{D_{0,0}D_{1,0}/g_0}\right)\\
  &  =\sum_{\substack{g_0\mid D_{0,0}/f_0'\\(g_0,D_{1,0}f_0')=1}}\mu(g_0)\frac{D_{0,0}}{f_0'g_0}e\Bigg(-\epsilon_2\eta_1\frac{n_{0,0}'n_0'D_0\overline{\frac{Y_2}{\nu_0D_{0,1}}g_0D_1}n}{D_{0,0}D_{1,0}/g_0}\Bigg)\delta\left(\frac{D_{0,0}}{f_0'g_0}\big| n\right).
\end{align*}

Writing $$D_{0,0} \leftarrow f_0'g_0D_{0,0}, \quad D_1 \leftarrow f'gD_1, \quad n=D_{0,0}D_1n',$$ we obtain  \begin{align}\label{S6BeforePoisson}
 \Sigma_6(\Delta^{\ast}) =   &\sum_{\epsilon_1,\epsilon_2,\eta_1=\pm1} \sum_{\nu_0}\sum_{g_0}\sum_{n_{0,0}'}\sum_{(f_0',n_{0,0}'g_0)=1} \sum_{D_{0,0}}\mathop{\sum\sum\sum\sum}_{\substack{n_{1,0}D_{0,1}D_{1,0}D_{2,0}'\mid (\nu_0n_{0,0}'D_{0,0}f_0'g_0)^\infty\\(n_{1,0}D_{1,0}D_{2,0}',D_{0,1})=1\\(D_{1,0},\nu_0n_{0,0}'D_{2,0}'g_0)=(D_{2,0}',\nu_0)=1}}\mathop{\sum\sum\sum\sum\sum\sum\sum}_{\substack{(\nu n_0'D_0D_1D_2'f'g,\nu_0n_{0,0}'D_{0,0}f_0'g_0)=1\\
    (f',n_0'D_0D_2'g)=(\nu n_0'D_1D_2'g,D_0)=1}}\nonumber\\
    &\times \frac{D_{1,0}f' \mu(n_{0,0}'n_0'D_{0,0}D_1g_0g)\mu(g_0g)}{ \nu_0n_{0,0}'^2D_{0,0}^2D_{0,1}D_{2,0}'D_0D_2'f_0'g_0^2}\sum_{n'} \frac{B(D_{0,0}D_1n',\nu_0\nu n_{0,0}'n_{1,0}n_0')}{\nu_0\nu n_{0,0}'n_{1,0}n_0'n'}\mathcal{C}_3 \nonumber \\ 
    \end{align}
    \begin{align*}
    & \times \Big(1-U\Big( \frac{  \nu_0^2\nu (n'_{0,0})^2n_{1,0}n'_0(f_0')^2 f'g_0^2gD^2_{0,0}D_{0, 1}D_{1,0}D_0D_1 }{\Delta^{\ast}}\Big)\Big)   F_1^{\epsilon_1,\epsilon_2,\eta_1}\Bigg(\frac{D_{0,0}D_1n'X}{\nu_0\nu n_{0,0}'n_{1,0}n_0'(D_{0,0}D_{1,0}D_1f_0'f'g_0g)^3}, \nonumber\\ 
    & \quad\quad\quad  \frac{XD_{2,0}'D_2'}{\nu_0^2\nu n_{0,0}'^2n_{1,0}n_0'^2D_{0,0}^2D_{0,1}D_{1,0}^2D_0D_1^2f_0'^2f'^2g_0^2g^2},
     \frac{n_0'D_{1,0}D_1f'g}{\nu_0^2\nu n_{0,0}'^2n_{1,0}D_{0,0}^2D_{0,1}D_{2,0}'^2D_0D_2'^2f_0'^2g_0^2}\Bigg)\\ \nonumber
\end{align*}
where 
\begin{align}\label{C3Def}
   \mathcal{C}_3 = & \sumast_{\substack{\alpha\Mod{\nu_0n_{0,0}'D_{0,0}D_{0,1}f_0'g_0}\\Y_3  \equiv 0\Mod{\nu_0D_{0,1}}\\\left(\frac{Y_3}{\nu_0D_{0,1}},f_0'\right)=1}} e\left(-\epsilon_2\eta_1\frac{\nu_0n_{0,0}'n_0'D_{0,1}D_0\overline{D_{1,0}D_{2,0}'D_2'f_0'g_0g}n'}{f'}-\epsilon_2\eta_1\frac{n_{0,0}'n_0'D_0\overline{\frac{Y_3}{\nu_0D_{0,1}}f'g_0g} n'}{D_{1,0}f_0'}\right)\nonumber\\
    &\times S\left(m,\nu n_{1,0}\overline{D_0}\frac{Y_3\overline{\alpha}}{D_{0, 1}};\nu_0\nu n_{0,0}'n_{1,0}D_{0,0}D_{2,0}'D_2'f_0'g_0\right), \quad\quad Y_3 = D_{2,0}'D_2'+\epsilon_1n_0'D_{1,0}D_1f'g\alpha.
\end{align}
Although not necessary, we confirm once again the absolute convergence of this expression: the  {decay properties} of $F_1$ near zero in the last two entries imply that the $\nu_0, \nu, n_{0, 0}', n_{1, 0}, n_0', D_{0, 0}, D_{0, 1}, D_{1, 0}, D_{2, 0}, D_0, D_1, D_2, f_0', f', g_0, g$-sums are rapidly convergent. The last remaining $n'$-sum converges rapidly by bounding the first entry of $F_1$ in terms of the other two. 

\subsection{Step 2: Additive reciprocity}

Applying additive reciprocity, we have \begin{equation*}
\begin{split}
    \mathcal{C}_3&
    =e\left(\epsilon_2\eta_1\frac{\nu_0n_{0,0}'n_0'D_{0,1}D_0\overline{f'}n'}{D_{1,0}D_{2,0}'D_2'f_0'g_0g}-\epsilon_2\eta_1\frac{\nu_0n_{0,0}'n_0'D_{0,1}D_0n'}{D_{1,0}D_{2,0}'D_2'f_0'f'g_0g}\right)\\
    &\times \sumast_{\substack{\alpha\Mod{\nu_0n_{0,0}'D_{0,0}D_{0,1}f_0'g_0}\\Y_3 \equiv 0\Mod{\nu_0D_{0,1}}\\\left(\frac{Y_3}{\nu_0D_{0,1}},f_0'\right)=1}} e\left(-\epsilon_2\eta_1\frac{n_{0,0}'n_0'D_0\overline{\frac{Y_3}{\nu_0D_{0,1}}f'g_0g} n'}{D_{1,0}f_0'}\right)  S\left(m,\nu n_{1,0}\overline{D_0}\frac{Y_3\overline{\alpha}}{D_{0,1}};\nu_0\nu n_{0,0}'n_{1,0}D_{0,0}D_{2,0}'D_2'f_0'g_0\right).
    \end{split}
\end{equation*}
We confirm that $\mathcal{C}_3$ is well-defined (which is clear from the derivation, but it might be useful to double-check again): the only point that is not directly visible from the summation conditions is that
$$\Big(\frac{Y_3}{\nu_0 D_{0, 1}}, D_{1, 0}\Big) = \Big(\frac{D_{2,0}'D_2'+\epsilon_1n_0'D_{1,0}D_1f'g\alpha}{\nu_0D_{0,1}}, D_{1, 0}\Big) = 1.$$
Indeed, since $(D_{2, 0}'D_2', \nu_0D_{0, 1}) = 1$, the congruence condition in the $\alpha$-sum implies that $(D_{1, 0}, \nu_0D_{0, 1}) = 1$, so the claim is equivalent to
$$ ( D_{2,0}'D_2'+\epsilon_1n_0'D_{1,0}D_1f'g\alpha , D_{1, 0}\Big) = 1$$
and hence to $( D_{2,0}'D_2' , D_{1, 0}) = 1$, which is true.

To prepare for Poisson summation, we note that $(n_0',D_0)=1$. By M\"obius inversion we can remove the coprimality condition $(f',n_0'D_0)$ by inserting the factor \begin{align*}
    \sum_{h_1\mid (n_0',f')}\sum_{h_2\mid (D_0,f')}\mu(h_1h_2).
\end{align*} 
Writing 
\begin{equation}\label{h1h2}
f'=h_1h_2f, \quad n_0'=h_1n_0, \quad D_0=h_2D_0',
\end{equation}we have \begin{align}\label{s6delta}
    \Sigma_6(\Delta^{\ast})=&\sum_{\epsilon_1,\epsilon_2,\eta_1=\pm1} \sum_{\nu_0}\sum_{g_0}\sum_{n_{0,0}'}\sum_{(f_0',n_{0,0}'g_0)=1} \sum_{D_{0,0}}\mathop{\sum\sum\sum\sum}_{\substack{n_{1,0}D_{0,1}D_{1,0}D_{2,0}'\mid (\nu_0n_{0,0}'D_{0,0}f_0'g_0)^\infty\\(n_{1,0}D_{1,0}D_{2,0}',D_{0,1})=1\\(D_{1,0},\nu_0n_{0,0}'D_{2,0}'g_0)=(D_{2,0}',\nu_0)=1}}\mathop{\sum\sum\sum\sum\sum\sum\sum\sum\sum}_{\substack{(\nu n_0D_0'D_1D_2'fgh_1h_2,\nu_0n_{0,0}'D_{0,0}f_0'g_0)=1 \nonumber\\
    (fh_1h_2,D_2'g)=(\nu n_0D_1D_2'gh_1,D_0'h_2)=1}}\\
    &\times \frac{D_{1,0}f}{\nu_0^2\nu n_{0,0}'^3n_{1,0}n_0D_{0,0}^2D_{0,1}D_{2,0}'D_0'D_2'f_0'g_0^2}\mu\left(n_{0,0}'n_0D_{0,0}D_1g_0gh_1\right)\mu(g_0gh_1h_2)\left(1-U\left( \frac{ E f }{\Delta^*}\right)\right) \nonumber \\
    &\times \sum_{n'} \frac{B(D_{0,0}D_1n',\nu_0\nu n_{0,0}'n_{1,0}n_0h_1)}{n'}\mathcal{C}_4
    e\left(-\epsilon_2\eta_1\frac{An'}{f}\right)F_1^{\epsilon_1,\epsilon_2,\eta_1}\Bigg(\frac{Bn'}{f^3},\frac{C}{f^2},Df\Bigg)   
\end{align}
where \begin{align}\label{abcde}
    \begin{matrix*}[l]
     \displaystyle   A=\frac{\nu_0n_{0,0}'n_0D_{0,1}D_0'}{D_{1,0}D_{2,0}'D_2'f_0'g_0g}, & \displaystyle  B=\frac{D_{0,0}D_1X}{\nu_0\nu n_{0,0}'n_{1,0}n_0h_1(D_{0,0}D_{1,0}D_1f_0'g_0gh_1h_2)^3},\\
       \displaystyle  C=\frac{XD_{2,0}'D_2'}{\nu_0^2\nu n_{0,0}'^2n_{1,0}n_0^2D_{0,0}^2D_{0,1}D_{1,0}^2D_0'D_1^2f_0'^2g_0^2g^2h_1^4h_2^3}, & \displaystyle D=\frac{n_0D_{1,0}D_1gh_1^2}{\nu_0^2\nu n_{0,0}'^2n_{1,0}D_{0,0}^2D_{0,1}D_{2,0}'^2D_0'D_2'^2f_0'^2g_0^2},\\
       {E}= \nu_0^2\nu (n_{0,0}')^2n_{1,0}n_0D^2_{0,0}D_{1,0}D_{0, 1}D_0'D_1(f_0')^2g^2_0gh_1^2h^2_2
    \end{matrix*}
\end{align}
and $\mathcal{C}_4
    e\left(-\epsilon_2\eta_1\frac{An'}{f}\right)$ equals $\mathcal{C}_3$ 
    after the change of variables \eqref{h1h2}, i.e.\
\begin{equation}\label{C4Def}
\begin{split}
    \mathcal{C}_4 = \mathcal{C}_4(f) 
    =&\sumast_{\substack{\alpha\Mod{\nu_0n_{0,0}'D_{0,0}D_{0,1}f_0'g_0}\\Y_4 \equiv 0 \Mod{\nu_0D_{0,1}}\\\left(\frac{Y_4}{\nu_0D_{0,1}},f_0'\right)=1}} e\left(-\epsilon_2\eta_1\frac{n_{0,0}'n_0D_0'\overline{\frac{Y_4}{\nu_0D_{0,1}}fg_0g} n'}{D_{1,0}f_0'}\right)\nonumber\\
    &\times S\left(m,\nu n_{1,0}\overline{D_0'h_2}\frac{\overline{\alpha} Y_4}{D_{0,1}};\nu_0\nu n_{0,0}'n_{1,0}D_{0,0}D_{2,0}'D_2'f_0'g_0\right)e\left(\epsilon_2\eta_1\frac{\nu_0n_{0,0}'n_0D_{0,1}D_0'\overline{f}n'}{D_{1,0}D_{2,0}'D_2'f_0'g_0g}\right)
    \end{split}
\end{equation}
with
$$Y_4 =  D_{2,0}'D_2' +\epsilon_1 n_0D_{1,0}D_1fgh_1^2h_2\alpha.$$

\subsection{Interlude III: A second truncation}\label{secondprelim}

We recall the definition of $F_1^{\epsilon_1, \epsilon_2, \eta_1}(a, b, c)$ in \eqref{F1Def} and attach a redundant factor
\begin{equation}\label{defGaux}
G(b, c) + (1 - G(b, c)), \quad G(b, c) = U\Bigg( \frac{X^{\varepsilon}}{D_{1,0}A^{\ast}} \Big( b^{7/6}c^{1/3} + bQ+ b^{2/3}c^{5/6} \Big)\Bigg)
\end{equation}
where $U$ is the same smooth function with support in $[-3, 3]$ as before and
$$A^{\ast} = \frac{X^{2/3}}{\nu_0^3\nu n_{0,0}'^3n_{1,0}{n_0}D_{0,0}^3D_{0,1}^2D_{1,0}^2D_{2,0}'D_0'D_1D_2'f_0'^3g_0^3g^2(h_1h_2)^2}, \quad Q := \frac{h_1^2 h_2D'_0 D_{0,1} D_1 n' n_0^2 n'_{0,0} \nu_0}{D_2'^2 D_{2,0}'^2 f'_0 g_0 X^{1/3}}.$$
The reason for this particular weight will become clear later. It prepares the stage to eliminate the non-zero frequencies in the upcoming Poisson summation.  Consequently we   decompose $\Sigma_6(\Delta^{\ast})$, given in \eqref{S6BeforePoisson}, as 
$$\Sigma_6(\Delta^{\ast}) = \Sigma^{(2)}_6 + \Sigma_6^{\flat}$$
(dropping the parameter $\Delta^{\ast}$ from the notation) where $\Sigma_6^{(2)}$ contains the $G$-term and $\Sigma_6^{\flat}$ contains the $(1-G)$-term. We will remove the correction factor $G(b, c)$ in $\Sigma_6^{(2)}$ directly after Poisson summation in Subsection \ref{sec57}. 

Using a suitable Cauchy-Schwarz argument, we now show that the contribution of $\Sigma^{\flat}_6$ is negligible. 

\begin{lemma}\label{prop.2ndPreBound}
    {For $\Delta^{\ast} \geq X^{11/21}$} we have  
       $ \Sigma_6^{\flat}\preccurlyeq  {X^{41/42}}.$
\end{lemma}

\begin{proof} We shift the contour in the definition of $F_1$ to $\Re s = {-1/2}$ and write {
$$F_1^{\epsilon_1, \epsilon_2, \eta_1}(a, b, c) (1 - G(b, c)) = \frac{1}{2\pi} \int_{\Bbb{R}} a^{1/2 - it} \mathcal{G}^{\eta_1}_{\mu_0}(1/2 + it)G_1(b, c, t)\, \mathrm{d}t$$
where
$$G_1(b, c, t) = (1 - G(b, c))\int_0^{\infty} V(x) \Phi_{\omega_6} (-\epsilon_2 b x, - \epsilon_1 m c) x^{1/2-it } \frac{\mathrm{d} x}{x} .$$
By Lemma \ref{lem.Phi6Truncation}, repeated integration by parts on the $x$-integral gives us arbitrary savings unless $|t|\preccurlyeq  \sqrt{b}+b^{1/3}c^{1/6}$.}
Note that $(n_0,D_0')=1$. We further   pull out the gcd of $(f,n_0D_0')$ by writing $f=f_1f_2f_3$, $n_0=f_1n_1$, and $D_0'=f_2d_0$ with $(f_3,n_1d_0)=1$. We split all sums into dyadic ranges:
\begin{displaymath}
\begin{split}
   & g_0 \asymp G_0, \quad n'_{0, 0} \asymp N_{0, 0}, \quad f_0' \asymp F_0, \quad D_{0, 0} \asymp \Delta_{0, 0}, \quad n_{1, 0} \asymp N_{1, 0}, \quad D_{0, 1} \asymp \Delta_{0, 1}, \quad D_{1, 0} \asymp \Delta_{1, 0}, \quad D_{2, 0}' \asymp \Delta_{2, 0}, \quad D_1 \asymp \Delta_1,\\ 
   &n' \asymp N,  \quad n_1 \asymp N_1, \quad d_0 \asymp \Delta_0, \quad D'_2  \asymp \Delta_2, \quad f_1 \asymp F_1, \quad f_2 \asymp F_2, \quad f_3 \asymp F_3, \quad {g} \asymp G, \quad \nu \asymp M, \quad \nu_0 \asymp M_0, \quad h_1 \asymp H_1, \quad h_2 \asymp H_2
    \end{split}
\end{displaymath}
and we remember that $F_1^{\epsilon_1, \epsilon_2, \eta_1}(a, b, c)$ is negligible unless $a^{1/3} \preccurlyeq b^{1/2} + b^{1/3} c^{1/6}$. By Cauchy-Schwarz we are left with estimating \begin{align*}
     &{\sqrt{TX}}
\Bigg(\sum_{\nu_0}\sum_{g_0}\sum_{n_{0,0}'}\sum_{(f_0',n_{0,0}'g_0)=1} \sum_{D_{0,0}}\mathop{\sum\sum\sum\sum}_{\substack{n_{1,0}D_{0,1}D_{1,0}D_{2,0}'\mid(\nu_0n_{0,0}'D_{0,0}f_0'g_0)^\infty\\(n_{1,0}D_{1,0}D_{2,0}',D_{0,1})=1\\(D_{1,0},\nu_0n_{0,0}'D_{2,0}'g_0)=(D_{2,0}',\nu_0)=1}}\mathop{\sum\sum\sum\sum\sum\sum}_{(\nu D_1f_1f_2h_1h_2,\nu_0n_{0,0}'D_{0,0}f_0'g_0)=1} \int_\R { U\left(\frac{t}{T}\right)}\sum_{n' }\Bigg|\mathop{\sum\sum\sum\sum\sum}_{\substack{(n_1d_0D_2'f_3g,\nu_0n_{0,0}'D_{0,0}f_0'g_0)=1\\
    (f_1f_2f_3h_1h_2,D_2'g)=(f_3,n_1d_0)=1\\
    (\nu n_1D_1D_2'f_1g,d_0f_2)=1\\   \mathbf{A}^{1/3} \preccurlyeq \mathbf{B}^{1/2} + \mathbf{B}^{1/3} \mathbf{C}^{1/6}}}\\
    &   
    {\frac{\mu(f_1h_1n_1g)\mu(gh_2)(n_1f_3^3g^3)^{it}\mathcal{C}_3}{ \nu_0^2\nu n_{0,0}'^3n_{1,0}^{3/2}n_1^{3/2}D_{0,0}^{5/2}D_{0,1}D_{1,0}^{1/2}D_{2,0}'d_0D_1^{1/2}D_2'f_0'^2g_0^3f_1^{3/2}f_2f_3^{1/2}g^{3/2}h_1^2 h_2^{3/2}}}
   G_1\Big(\frac{C}{(f_1f_2f_3)^2},Df_1f_2f_3,t\Big)\left(1-U\Big(\frac{Ef_1f_2f_3  }{\Delta^{\ast}}\Big) 
   \right)\Bigg|^2\mathrm{d}t\Bigg)^{1/2},
\end{align*}
where for notational simplicity we returned momentarily to $\mathcal{C}_3=\mathcal{C}_4e\left(-\epsilon_2\eta_1\frac{An'}{f_1f_2f_3}\right)$,
$$\mathbf{A} = \frac{BN}{(F_1F_2F_3)^3}, \quad \mathbf{B} = \frac{C}{(F_1F_2F_3)^2}, \quad \mathbf{C} = DF_1F_2F_3, \quad T= ( \sqrt{\mathbf{B}}+\mathbf{B}^{1/3}\mathbf{C}^{1/6})X^\varepsilon$$
with $A, B, C, D, E$ as in \eqref{abcde} under the notation $n_0=f_1n_1$ and $D_0'=f_2d_0$. 
Let $L\geq 1$ be a parameter. We use this to enlarge the $n'$-sum and insert a factor $U(n'/LN)$ into the $n'$-sum. This is costly in the diagonal term, but helpful for the off-diagonal term in the upcoming Poisson summation. 

We open the square and introduce a second set of variables $\tilde{n}_1, \tilde{d}_0, \tilde{D}_2', \tilde{f}_3, \tilde{g}$. Correspondingly we define $\tilde{A}, \tilde{B}, \tilde{C}$ analogously to \eqref{abcde}. Recalling the definition \eqref{C3Def} of $\mathcal{C}_3$ with \begin{align*}
    f'=f_1f_2f_3h_1h_2, \quad n_0'=f_1h_1n_1, \quad \text{ and } \quad D_0=f_2h_2d_0,
\end{align*}
the innermost $n'$-sum becomes
\begin{align*}
    &\sum_{n'}U\left(\frac{n'}{LN}\right)e(\epsilon_2 \eta_1 Zn'),
\end{align*}
where $Z$ is given by \begin{align*}
    - \frac{\nu_0n_{0,0}'n_1D_{0,1}d_0\overline{D_{1,0}D_{2,0}'D_2'f_0'g_0g} }{f_3}- \frac{n_{0,0}'n_1d_0\overline{\frac{Y_3}{\nu_0D_{0,1}}f_3g_0g}  }{D_{1,0}f_0'}  + \frac{\nu_0n_{0,0}'\tilde{n}_1D_{0,1}\tilde{d}_0\overline{D_{1,0}D_{2,0}'\tilde{D_2'}f_0'g_0\tilde{g}} }{\tilde{f}_3}+ \frac{n_{0,0}'\tilde{n}_1\tilde{d}_0\overline{\frac{\tilde{Y_3}}{\nu_0D_{0,1}}\tilde{f}_3g_0\tilde{g}}  }{D_{1,0}f_0'} .
\end{align*}
We choose $L = (1 + D_{1, 0} f_0' F_3^2/N)X^{\varepsilon}$. Recalling that $(f_3, \nu_0n'_{0, 0}n_1 D_{0, 1} D_{1, 0} D'_{2, 0} d_0 f_0' g_0) = 1$, it is easy to see by Poisson summation that the $n'$-sum is negligible unless 
$$f_3 = \tilde{f}_3, \quad n_1d_0\tilde{D}_2' \tilde{g}  \equiv \tilde{n}_1 \tilde{d}_0 D_2' g \, (\text{mod } f_3)$$
(and another condition modulo $D_{1, 0}f'_0$ which we drop) in which case it is bounded by $LN$.
{ On the other hand, the $t$-integral with $f_3=\tilde{f}_3$ is of the form
\begin{align*}
    \int_{\R} U\left(\frac{t}{T}\right)G_1(B,C,t)\overline{G_1(\tilde{B},\tilde{C},t)}dt=&\ (1 - G(B,C))\overline{(1-G(\tilde{B},\tilde{C}))}\int_0^{\infty}\int_0^{\infty} V(x)V(\tilde{x}) \Phi_{\omega_6} (-\epsilon_2 B x, - \epsilon_1 m C)
    \\
    &\ \times \overline{\Phi_{\omega_6} (-\epsilon_2 \tilde{B}\tilde{x}, - \epsilon_1 m \tilde{C})} \int_\R U\left(\frac{t}{T}\right)\left(\frac{n_1g^3\tilde{x}}{\tilde{n}_1\tilde{g}^3x}\right)^{it} \mathrm{d}t\frac{\mathrm{d} x}{\sqrt{x}}\frac{\mathrm{d} \tilde{x}}{\sqrt{\tilde{x}}}.
\end{align*}
The $t$-integral then gives arbitrary savings unless $|\tilde{x}-\frac{\tilde{n}_1\tilde{g}^3}{n_1g^3}x|\preccurlyeq 1/T$. Lemma \ref{Kw6} shows that the $t$-integral is bounded by $$ \preccurlyeq \begin{cases} (BC)^{-1/2}, & B \asymp C,\\  {(BC)^{-1/2} (B^2C+C^2B)^{-1/6}}, &\text{otherwise.}\end{cases} $$}
Estimating the remaining part of $\mathcal{C}_3$ by Weil's bound and everything else trivially, we obtain after careful bookkeeping the bound
{\begin{align*}
    \frac{\sqrt{NX}\mathbf{C}^{-1/6}(\mathbf{B}^{-1/12}+\mathbf{C}^{-1/12})}{M_0N_{0,0}N_{1,0}N_1F_0F_1G_0G\sqrt{\Delta_{0,0}\Delta_{1,0}\Delta_{2,0}\Delta_0F_2}}\left(1+\sqrt{\frac{\Delta_{1,0}F_0F_3^2}{N}}\right)\left(1+\sqrt{\frac{N_1\Delta_0\Delta_2G}{F_3}}\right)\begin{cases} 1, & \mathbf{B} \asymp \mathbf{C},\\  {\max(\mathbf{B}^2\mathbf{C}, \mathbf{B}\mathbf{C}^2)^{-1/12}} , & \text{otherwise,}\end{cases}
\end{align*}}
subject to the condition{s}
\begin{equation}\label{inequ}
\textbf{A}^{1/3} \preccurlyeq \mathbf{B}^{1/2} + \mathbf{B}^{1/3} \mathbf{C}^{1/6}, \quad  {\mathbf{B}^2\mathbf{C}, \mathbf{C}^2\mathbf{B}} \succcurlyeq 1, \quad  \mathbf{B}^{7/6}\mathbf{C}^{1/3} + \mathbf{B}Q + \mathbf{B}^{2/3}\mathbf{C}^{5/6}\succcurlyeq \Delta_{1, 0} A^{\ast}, \quad  EF_1F_2F_3  \gg \Delta^{\ast} 
\end{equation}
where the variables in $  A^{\ast}, Q$ have to be restricted to the respective dyadic intervals and undergo the change of variables $n_0=f_1n_1$ and $D_0'=f_2d_0$. To arrive at this bound we recall for $X \geq 1$ and $y\in \Bbb{N}$ that 
$$\sum_{\substack{x \asymp X \\ x \mid y^{\infty}}} 1 \ll (Xy)^{\varepsilon}.$$

Bounding the previous expression by hand appears to be rather complicated, but this can be viewed as a linear  program that can be easily solved by a machine, e.g.\ by {\tt mathematica}. We write every variable as a power of $X$ and obtain linear inequalities from \eqref{inequ}. For clarity, we distinguish the two cases $B\asymp C$ and $B \not\asymp C$, obtaining the bounds $X^{5/6}$ in the former and $X^{41/42}$ in the latter:

\noindent {\tt a := m0 + n00 + f1 + n1 + d01 + f2 + d0 - d10 - d20 - d2 - f0 - g0 - g;\\
b := d00 + d1 + 1 - m0 - m - n00 - f1 - n10  - n1 - h1 - 
   3 (d00 + d10 +d1 + f0 + g0 + g + h1 + h2);\\
c := 1 + d20 + d2 - 2 m0 - m - 2 n00 - n10 -  2 f1 - 2 n1 - 2 d00 - 
   d01 - 2 d10 - d0 - f2 - 2 d1 - 2 f0 - 2 g0 - 2 g - 4 h1 - 3 h2;\\
d := f1 + n1 + d10 + d1 + g + 2 h1 - 2 m0 - m - 2 n00 - n10 - 2 d00 - 
   d01 - 2 d20 - d0 - f2 - 2 d2 - 2 f0 - 2 g0;\\
ee := 2 m0 + m + 2 n00 + n10 + f1 + n1 + 2 d00 + d10 + d01 + d0 + f2 + 
   d1 + 2 f0 + 2 g0 + g + 2 h1 + 2 h2;\\
A := b + n - 3 (f1 + f2 + f3);
B := c - 2 (f1 + f2 + f3);
CC := d + f1 + f2 + f3;\\
astar := 2/3 - 3 m0 - m - 3 n00 - n10 - f1 - n1 - 3 d00 - 2 d01 - 
   2 d10 - d20 - d0 - f2 - d1 - d2 - 3 f0 - 3 g0 - 2 g - 2h1 - 2h2;\\
q := 2h1 + h2 + d0 + f2 + d01 + d1 + n + 2 n1 + 2 f1 + n00 + m0 - 2 d2 - 2 d20 - 
   f0 - g0 - 1/3;}

\noindent {\tt Maximize[\{1/2 n + 1/2 - 1/6 CC + Max[-1/12 B, - 1/12 CC] - m0 - n00 - 
   n10 - n1 - f0 - f1 - g0 - g - 1/2 (d00 + d10 + d20 + d0 + f2) + 
   Max[0, 1/2 (d10 + f0 + 2 f3 - n) ] + 
   Max[0, 1/2 (n1 + d0 + d2 + g - f3)], B == CC,  
  A/3 <= Max[1/2 B, 1/3 B + 1/6 CC], 2 B + CC >= 0, B + 2 CC >= 0, 
  Max[7/6 B + 1/3 CC, B + q, 2/3 B + 5/6 CC] >= d10 + astar, 
  ee + f1 + f2 + f3 >= 11/21, m0 >= 0, m >= 0, n00 >= 0, n10 >= 0, 
  n1 >= 0, d00 >= 0, d01 >= 0, d10 >= 0, d20 >= 0, d0 >= 0, d1 >= 0, 
  d2 >= 0, f0 >= 0, g0 >= 0, g >= 0, n >= 0, f1 >= 0, f2 >= 0, 
  f3 >= 0, h1 >= 0, h2 >= 0\}, \{m0, m, n00, n10, n1,   d00, d01, d10, 
  d20, d0, d1, d2, f0, g0, g, n, f1, f2, f3, h1, h2\}]}

\noindent {\tt   \{5/6, \{m0 -> 0, m -> 0, n00 -> 0, n10 -> 0, n1 -> 0, d00 -> 0, d01 -> 0, d10 \
-> 0, d20 -> 0, d0 -> 0, d1 -> 0, d2 -> 1/3, f0 -> 0, g0 -> 0, g -> \
0, n -> 2/3, f1 -> 0, f2 -> 0, f3 -> 1/3, h1 -> 1/6, h2 -> 0\}\}}

\noindent {\tt Maximize[\{1/2 n + 1/2 - 1/6 CC + Max[-1/12 B, - 1/12 CC] - m0 - n00 - 
   n10 - n1 - f0 - f1 - g0 - g - 1/2 (d00 + d10 + d20 + d0 + f2) + 
   Max[0, 1/2 (d10 + f0 + 2 f3 - n) ] + 
   Max[0, 1/2 (n1 + d0 + d2 + g - f3)] - 1/12 Max[2 B + CC, B + 2 CC], 
  A/3 <= Max[1/2 B, 1/3 B + 1/6 CC], 2 B + CC >= 0, B + 2 CC >= 0, 
  Max[7/6 B + 1/3 CC, B + q, 2/3 B + 5/6 CC] >= d10 + astar, 
  ee + f1 + f2 + f3 >= 11/21, m0 >= 0, m >= 0, n00 >= 0, n10 >= 0, 
  n1 >= 0, d00 >= 0, d01 >= 0, d10 >= 0, d20 >= 0, d0 >= 0, d1 >= 0, 
  d2 >= 0, f0 >= 0, g0 >= 0, g >= 0, n >= 0, f1 >= 0, f2 >= 0, 
  f3 >= 0, h1 >= 0, h2 >= 0\}, \{m0, m, n00, n10, n1,   d00, d01, d10, 
  d20, d0, d1, d2, f0, g0, g, n, f1, f2, f3, h1, h2\}]}

\noindent {\tt \{41/42, \{m0 -> 0, m -> 0, n00 -> 0, n10 -> 0, n1 -> 0, d00 -> 0, d01 -> 0, \
d10 -> 0, d20 -> 0, d0 -> 0, d1 -> 0, d2 -> 1/3, f0 -> 0, g0 -> 0, g \
-> 1/42, n -> 1, f1 -> 0, f2 -> 0, f3 -> 1/2, h1 -> 0, h2 -> 0\}\}}

This completes the proof.     \end{proof}

\subsection{Step 3: Poisson summation}

We continue with the analysis of $\Sigma_6^{(2)}$, which we recall to be given by \eqref{s6delta} with an additional $G$-truncation factor. As mentioned before, it is easier to apply Poisson summation in $f$ first and postpone the second application of Voronoi summation; in particular  we shall see in a moment that only the central term will survive. The final expression is given in \eqref{sigmasharp} at the end of this subsection.

To ensure absolute convergence in the subsequent expression, we need to apply a very mild truncation on the $g$-variable, and for convenience we also truncate the $D_{1, 0}$-variable. 
With this in mind we let $W_0$ be a fixed smooth function that is 1 on $[0, 1]$ and 0 on $[2, \infty)$ and let $$\Xi = X^{10^3}.$$ 
Since the previous expression for $\Sigma_6^{(2)}$ is convergent in every variable on a polynomial scale and every variable is at least 1, we can freely attach a factor $W_0(D_{1, 0}g/\Xi)$ at the cost of a small error, say $O(X^{-1})$. (A truncation of $g$ alone would suffice, if we estimate carefully.) We continue to call this slightly modified expression $\Sigma_6^{(2)}$. It is convenient to define
\begin{align*}
    F=\frac{X^{2/3}}{E} = \frac{X^{2/3}}{\nu_0^2\nu n_{0,0}'^2n_{1,0}n_0D_{0,0}^2D_{0,1}D_{1,0}D_0'D_1f_0'^2g_0^2gh_1^2h_2^2}, \quad M=\nu_0n_{0,0}'D_{0,0}D_{0,1}D_{1,0}D_{2,0}'D_2'f_0'g_0g. 
\end{align*}
Note that the variables are not independent, but satisfy
\begin{equation}\label{var}
C = B/A, \quad C^2 D = F^3, \quad F/M = A^{\ast}.
\end{equation}


Applying Poisson summation on the $f$-sum with modulus $M$, we have \begin{equation}\label{afterpois}
\begin{split}
    &\sum_{(f,M)=1}f\mathcal{C}_4
    F_1^{\epsilon_1,\epsilon_2,\eta_1}\left(\frac{Bn'}{f^3},\frac{C}{f^2},Df\right)G\Big( \frac{C}{f^2}, Df\Big)\left(1-U\left( {\frac{  Ef }{\Delta^*}}\right)\right)
    e\left(-\epsilon_2\eta_1\frac{An'}{f}\right)\\
    &= \frac{F^2}{M}\sum_f\mathcal{C}_5(f)
    F_2^{\epsilon_1,\epsilon_2,\eta_1}\Bigg(\frac{Ff}{M},\frac{An'}{F},\frac{Bn'}{F^3},\frac{C}{F^2}  \Bigg),
    \end{split}
\end{equation}
where \begin{align}\label{C5Def}
    \mathcal{C}_5(f)
    =\sumast_{k\Mod{M}}\mathcal{C}_4(k) 
    e\left(\frac{fk}{M}\right)
\end{align}
and \begin{align}\label{F2Def}
    F_2^{\epsilon_1,\epsilon_2,\eta_1}(a,b,c,d )& = \int_0^{\infty} yF_1^{\epsilon_1,\epsilon_2,\eta_1}\Big( \frac{c}{y^3}, \frac{d}{y^2}, \frac{y}{d^2}\Big) e\left(-\epsilon_2\eta_1by^{-1}-ay\right)U \Big(\frac{fX^{\varepsilon}}{D_{1, 0}a} \Big(\frac{\sqrt{d}}{y^2} + \frac{b}{y^2} + \frac{1}{d\sqrt{y}}\Big)\Big)\left(1- {U\left(\frac{  X^{2/3}y }{ \Delta^*}\right)}\right)\, \mathrm{d} y  \nonumber \\
    & =     \int_0^\infty \int_{(0)}\int_0^\infty x^{-s-1}c^{-s}y^{3s+1} \mathcal{G}_{\mu_0}^{\eta_1}(s+1)U \Big(\frac{fX^{\varepsilon}}{D_{1, 0}a} \Big(\frac{\sqrt{d}}{y^2} + \frac{b}{y^2} + \frac{1}{d\sqrt{y}}\Big)\Big)\nonumber\\
    &\quad\quad\quad \times \Phi_{\omega_6}\left(-\epsilon_2\frac{dx}{y^2},-\epsilon_1\frac{my}{d^2}\right)V(x)\left(1- {U\left(\frac{  X^{2/3}y }{ \Delta^*}\right)}\right)e\left(-\epsilon_2\eta_1by^{-1}-ay\right)\mathrm{d} x\, \frac{\mathrm{d}s}{2\pi i} \mathrm{d} y
\end{align}
(which also depends on $D_{1, 0}$ and $f$). Here we used the second condition in \eqref{var}. In a moment we will also use the first condition to eliminate one argument in $F_2$.

We analyze the function $F_2$. The triple integral is not absolutely convergent, but we verify that it makes sense as a conditionally convergent integral and is uniformly bounded.    To this end we first insert a smooth dyadic partition of unity that localizes $y \asymp Y$ and call the corresponding piece $F^{\epsilon_1, \epsilon_2, \eta_1}_{2, Y}(a, b, c, d)$. 

\begin{lemma}\label{lemf2} We have
\begin{equation}\label{boundf2}
\begin{split}
&F^{\epsilon_1, \epsilon_2, \eta_1}_{2, Y}(a, b, c, d) \\
&\preccurlyeq \min\Big(\frac{c}{Y},  c^{2/3} \Big)    {\big(\min({d^{1/4}\sqrt{Y}, \sqrt{d}Y^{1/4}}) + \delta(d \asymp Y) (dY)^{1/4}\big)} \Big(1 +  {Y + d} + \frac{c}{d^{3/2} + Y^{3/2}} + \frac{a}{\frac{\sqrt{d}+b}{Y^2}+\frac{1}{d\sqrt{Y}}} +  \frac{\Delta^*}{X^{2/3}Y}\Big)^{-A}.
\end{split}
\end{equation}
\end{lemma}

\begin{proof} The triple integral restricted to $y \asymp Y$ is absolutely convergent if we shift the contour to $\Re s < - 5/6$. We may then interchange freely the three integrals and perform suitable contour shifts and integration by parts.

\begin{itemize}
    \item The decay of $\Phi_{\omega_6}$ given in Lemma \ref{lem.Phi6Truncation} yields  {$d, Y \preccurlyeq 1$};
    \item integration by parts in $x$ gives $\Im s \preccurlyeq \frac{\sqrt{d}}{Y} + \frac{1}{\sqrt{Y}}$; hence shifting the $s$-contour to the right gives $c \preccurlyeq Y^3(\frac{\sqrt{d}}{Y} + \frac{1}{\sqrt{Y}})^3 \ll d^{3/2} + Y^{3/2}$, in particular $c \preccurlyeq 1$;
    \item integration by parts in $y$ gives $a\preccurlyeq \frac{\sqrt{d}}{Y^2}+\frac{1}{Y^{3/2}} +\frac{1}{d\sqrt{Y}}+ \frac{b}{Y^2}\preccurlyeq \frac{\sqrt{d}+b}{Y^2}+\frac{1}{d\sqrt{Y}}$; 
    \item the factor  {$ 1- {U\left(\frac{X^{2/3}y}{\Delta^*}\right)} $ implies $X^{2/3}Y \gg\Delta^*$}.  
    \end{itemize}
    The same stationary phase analysis as in \eqref{statphs} along with Lemma \ref{Kw6} shows that the portion $y \asymp Y$ of \eqref{F2Def} is bounded by
    $$Y \min\Big(\frac{c}{Y^3}, \frac{c^{2/3}}{Y^2}\Big)\int_{x \asymp 1} \int_{y \asymp Y} \Big|\Phi_{ {\omega_6}}\Big(\pm \frac{dx}{y^2}, \pm \frac{my}{d^2}\Big) \Big|\mathrm{d} x\,   \mathrm{d} y \preccurlyeq \min\Big(\frac{c}{Y},  c^{2/3} \Big)   {\big(\min({d^{1/4}\sqrt{Y}, \sqrt{d}Y^{1/4}}) + \delta(d \asymp Y) (dY)^{1/4}\big)}.$$ 
This completes the proof. 
\end{proof}

In particular, at this point we can sum over all pieces of the dyadic decomposition and drop it. We now interpret $F_2$ in the above way as a smooth function in 4 variables satisfying
\begin{equation*}
   F_2^{\epsilon_1,\epsilon_2,\eta_1}(a,b,c,d) \ll c^{2/3}  {d^{1/4}}  (1 + c + d  )^{-A} 
   \end{equation*}
with support such that necessarily 
\begin{equation}\label{fbound}
f \ll \frac{D_{1, 0}}{X^{\varepsilon}}.
\end{equation}



Putting \eqref{afterpois} back into $\Sigma^{(2)}_6$, we have (up to a small error from inserting $W_0$ which we agreed to ignore in the notation) \begin{align*}
    \Sigma^{(2)}_6 =& X^{4/3}\sum_{\epsilon_1,\epsilon_2,\eta_1=\pm1} \sum_{\nu_0}\sum_{g_0}\sum_{n_{0,0}'}\sum_{(f_0',n_{0,0}'g_0)=1} \sum_{D_{0,0}}\mathop{\sum\sum\sum\sum}_{\substack{n_{1,0}D_{0,1}D_{1,0}D_{2,0}'\mid (\nu_0n_{0,0}'D_{0,0}f_0'g_0)^\infty\\(n_{1,0}D_{1,0}D_{2,0}',D_{0,1})=1\\(D_{1,0},\nu_0n_{0,0}'D_{2,0}'g_0)=(D_{2,0}',\nu_0)=1}}\mathop{\sum\sum\sum\sum\sum\sum\sum\sum}_{\substack{(\nu n_0D_0'D_1D_2'gh_1h_2,\nu_0n_{0,0}'D_{0,0}f_0'g_0)=1\\
    (h_1h_2,D_2'g)=(\nu n_0D_1D_2'gh_1,D_0'h_2)=1}}\\
    &\times \frac{\mu(n_{0,0}'n_0D_{0,0}D_1g_0gh_1)\mu(g_0gh_1h_2) }{\nu_0^7\nu^3 n_{0,0}'^8n_{1,0}^3n_0^3D_{0,0}^7D_{0,1}^4D_{1,0}^2D_{2,0}'^2D_0'^3D_1^2D_2'^2f_0'^6g_0^7g^3h_1^4h_2^4}  W_0\Big(\frac{D_{1, 0}g}{\Xi}\Big) 
    \sum_{n'} \frac{B(D_{0,0}D_1n',\nu_0\nu n_{0,0}'n_{1,0}n_0h_1)}{n'}\sum_{f\in \Bbb{Z}}\mathcal{C}_5(f)\\
      &\times F_2^{\epsilon_1,\epsilon_2,\eta_1}\Bigg(\frac{X^{2/3}f}{\nu_0^3\nu n_{0,0}'^3n_{1,0}n_0D_{0,0}^3D_{0,1}^2D_{1,0}^2D_{2,0}'D_0'D_1D_2'f_0'^3g_0^3(gh_1h_2)^2},\frac{\nu_0^3\nu n_{0,0}'^3n_{1,0}n_0^2D_{0,0}^2D_{0,1}^2D_0'^2D_1f_0'g_0h_1^2h_2^2n'}{D_{2,0}'D_2'X^{2/3}},\\
    &\quad \frac{\nu_0^5\nu^2n_{0,0}'^5n_{1,0}^2n_0^2D_{0,0}^4D_{0,1}^3D_0'^3D_1f_0'^3g_0^3h_1^2h_2^3n'}{X},\frac{\nu_0^2\nu n_{0,0}'^2n_{1,0}D_{0,0}^2D_{0,1}D_{2,0}'D_0'D_2'f_0'^2g_0^2h_2}{X^{1/3}}  
    \Bigg).
\end{align*}

We note that this expression is absolutely convergent:  the rapid decay of $F_2$ near $\infty$ in the third and fourth entry along with \eqref{fbound} and the truncation factor $W_0(D_{1, 0}g/\Xi)$ effectively limit all variables.

We now argue that $\mathcal{C}_5(f)$ 
vanishes unless $D_{1, 0} \mid f$ which in view of \eqref{fbound} restricts to the central term $f=0$. Indeed, let us write $$d_1=(D_{0,0},D_{1,0}^\infty),  \quad f_1=(f_0',D_{1,0}^\infty),  
\quad D_{0,0}=d_1d_2, \quad f_0'=f_1f_2
$$ and $M_0=\nu_0n_{0,0}'d_2D_{0,1}D_{2,0}'f_2g_0$. Note that $(\nu D_2'g,n_{1,0}d_1D_{1,0}f_1M_0)
=1$, so we can factor off a character sum  
\begin{align*}
    \mathcal{C}_{5,1}=&\sumast_{k\Mod{d_1D_{1,0}f_1}}\sumast_{\alpha\Mod{d_1f_1}} e\left(\epsilon_2\eta_1\frac{\nu_0n_{0,0}'n_0D_{0,1}D_0'n'\overline{f_2g_0gk}\left(\overline{D_{2,0}'D_2'}-\overline{Y_4}\right)}{D_{1,0}f_1}+\frac{fk\overline{M_0D_2'g}}{d_1D_{1,0}f_1}\right).
\end{align*}
 Since $D_{1,0}\mid (d_1f_1)^\infty$, we can evaluate $\mathcal{C}_{5,1}$ by rewriting $k\mapsto \beta+\gamma d_1f_1$. Writing $$Y_{4,1}:=D_{2,0}'D_2' +\epsilon_1 n_0D_{1,0}D_1gh_1^2h_2\alpha\beta,$$
we have \begin{align*}
    \mathcal{C}_{5,1}=&\mathop{\sumast\quad\sumast}_{\alpha,\beta\Mod{d_1f_1}} \sum_{\gamma\Mod{D_{1,0}}}e\left(\epsilon_2\eta_1\frac{\nu_0n_{0,0}'n_0D_{0,1}D_0'n'\overline{f_2g_0g(\beta+\gamma d_1f_1)}\left(\overline{D_{2,0}'D_2'}-\overline{Y_{4,1}}\right)}{D_{1,0}f_1}+\frac{f\overline{M_0D_2'g}(\beta+\gamma d_1f_1)}{d_1D_{1,0}f_1}\right)\\
    =&\mathop{\sumast\quad\sumast}_{\alpha,\beta\Mod{d_1f_1}} \sum_{\gamma\Mod{D_{1,0}}}e\left(\epsilon_2\eta_1\frac{\nu_0n_{0,0}'n_0D_{0,1}D_0'n'\overline{f_2g_0g \beta}\left(\overline{D_{2,0}'D_2'}-\overline{Y_{4,1}}\right)/D_{1,0}}{f_1}+\frac{f\overline{M_0D_2'g}(\beta+\gamma d_1f_1)}{d_1D_{1,0}f_1}\right).
\end{align*}
Summing over $\gamma\Mod{D_{1,0}}$, we see that this vanishes unless $D_{1, 0} \mid f$, as desired. 

Recalling also the first condition in \eqref{var}, we can therefore simplify
\begin{align}\label{sigmasharp}
    \Sigma^{(2)}_6 =&X^{4/3}\sum_{\epsilon_1,\epsilon_2,\eta_1=\pm1} \sum_{\nu_0}\sum_{g_0}\sum_{n_{0,0}'}\sum_{(f_0',n_{0,0}'g_0)=1} \sum_{D_{0,0}}\mathop{\sum\sum\sum\sum}_{\substack{n_{1,0}D_{0,1}D_{1,0}D_{2,0}'\mid (\nu_0n_{0,0}'D_{0,0}f_0'g_0)^\infty\\(n_{1,0}D_{1,0}D_{2,0}',D_{0,1})=1\\(D_{1,0},\nu_0n_{0,0}'D_{2,0}'g_0)=(D_{2,0}',\nu_0)=1}}\mathop{\sum\sum\sum\sum\sum\sum\sum\sum}_{\substack{(\nu n_0D_0'D_1D_2'gh_1h_2,\nu_0n_{0,0}'D_{0,0}f_0'g_0)=1\\
    (h_1h_2,D_2'g)=(\nu n_0D_1D_2'gh_1,D_0'h_2)=1}}\nonumber\\
    &\times \frac{ \mu(n_{0,0}'n_0D_{0,0}D_1g_0gh_1)\mu(g_0gh_1h_2)}{\nu_0^7\nu^3 n_{0,0}'^8n_{1,0}^3n_0^3D_{0,0}^7D_{0,1}^4D_{1,0}^2D_{2,0}'^2D_0'^3D_1^2D_2'^2f_0'^6g_0^7g^3h_1^4h_2^4}  W_0\Big(\frac{D_{1, 0}g}{\Xi}\Big) 
    \sum_{n'} \frac{B(D_{0,0}D_1n',\nu_0\nu n_{0,0}'n_{1,0}n_0h_1)}{n'} \mathcal{C}_5
    \nonumber\\
    &\times \tilde{F}_2^{\epsilon_1,\epsilon_2,\eta_1}\Bigg(\frac{\nu_0^3\nu n_{0,0}'^3n_{1,0}n_0D_{0,0}^3D_{0,1}^2D_{1,0}D_{2,0}'D_0'D_1D_2'f_0'^3g_0^3(gh_1h_2)^2}{X^{2/3-\varepsilon}},\frac{\nu_0^3\nu n_{0,0}'^3n_{1,0}n_0^2D_{0,0}^2D_{0,1}^2D_0'^2D_1f_0'g_0h_1^2h_2^2n'}{D_{2,0}'D_2'X^{2/3}},\nonumber\\
    &\quad \frac{\nu_0^5\nu^2n_{0,0}'^5n_{1,0}^2n_0^2D_{0,0}^4D_{0,1}^3D_0'^3D_1f_0'^3g_0^3h_1^2h_2^3n'}{X} 
    \Bigg)
\end{align}
where $\mathcal{C}_5=\mathcal{C}_5(0)$ and
\begin{align}\label{FTilde2Def}
    \tilde{F}_2^{\epsilon_1,\epsilon_2,\eta_1}(a,b,c)=&\int_0^\infty \int_{(0)}\int_0^\infty x^{-s-1}c^{-s}y^{3s+1} \mathcal{G}_{\mu_0}^{\eta_1}(s+1) U\left(a\Big(\frac{\sqrt{c}}{\sqrt{b}y^2} + \frac{b}{y^2} + \frac{b}{c \sqrt{y}}\Big)\right) \nonumber \\
    &\times \Phi_{\omega_6}\left(-\epsilon_2\frac{cx}{by^2},-\epsilon_1\frac{b^2my}{c^2}\right)V(x)\left(1-U\left( {\frac{X^{2/3}y}{\Delta^*}}\right)\right)e\left(-\epsilon_2\eta_1by^{-1}\right)\mathrm{d} x\, \frac{\mathrm{d}s}{2\pi i} \,\mathrm{d} y.
\end{align}

\subsection{Interlude IV: Removing the truncation}\label{sec57}

Before we proceed we would like to remove the factor 
\begin{equation}\label{U}
U\Big(a\Big(\frac{\sqrt{c}}{\sqrt{b}y^2} + \frac{b}{y^2} + \frac{b}{c \sqrt{y}}\Big)\Big)
\end{equation}
in the definition \eqref{FTilde2Def} of $ \tilde{F}_2^{\epsilon_1,\epsilon_2,\eta_1}$ which we only inserted artificially in order to force $f=0$ in the Poisson summation formula. To this end let us define $ \breve{\Sigma}^{\flat}_6$ to be the same quantity as \eqref{sigmasharp} except that $\tilde{F}_2^{\epsilon_1,\epsilon_2,\eta_1}$ is replaced with 
  \begin{align}\label{F2Defnew}
    \breve{F}_2^{\epsilon_1,\epsilon_2,\eta_1}(a,b,c)=&\int_0^\infty \int_{(0)}\int_0^\infty x^{-s-1}c^{-s}y^{3s+1} \mathcal{G}_{\mu_0}^{\eta_1}(s+1) \left(1-U\Big(a\Big(\frac{\sqrt{c}}{\sqrt{b}y^2} + \frac{b}{y^2} + \frac{b}{c \sqrt{y}}\Big)\Big)\right) \nonumber \\
    &\times \Phi_{\omega_6}\Big(-\epsilon_2\frac{cx}{by^2},-\epsilon_1\frac{b^2my}{c^2}\Big)V(x)\left(1-U\left( {\frac{X^{2/3}y}{\Delta^*}}\right)\right)e\left(-\epsilon_2\eta_1by^{-1}\right)\mathrm{d} x\, \frac{\mathrm{d}s}{2\pi i} \,\mathrm{d} y. 
\end{align}
This equals
\begin{equation}\label{f2alt}
\int_0^{\infty}  F_1^{\epsilon_1,\epsilon_2,\eta_1}\Big( \frac{c}{y^3}, \frac{c}{by^2}, \frac{b^2y}{c^2}\Big) \left(1-U\Big(a\Big(\frac{\sqrt{c}}{\sqrt{b}y^2} + \frac{b}{y^2} + \frac{b}{c \sqrt{y}}\Big)\Big)\right)\left(1-U\left( {\frac{X^{2/3}y}{\Delta^*}}\right)\right)e\left(-\epsilon_2\eta_1by^{-1}\right) y\, dy.
\end{equation}

\begin{lemma}\label{lem55}  {Let $\Delta^{\ast} \geq X^{11/21}$}. Then we have
$\breve{\Sigma}^{\flat}_6 \preccurlyeq   {X^{41/42 }}.$
\end{lemma}

\begin{proof} Both the result and the structure of the proof are similar to Lemma \ref{prop.2ndPreBound}. After some preparation, we use the Cauchy-Schwarz inequality where we move the $s$-integral in the definition of $\breve{F}_{2}$ outside to deal with some simple cases. For the hardest case, we keep $\breve{F}_{2}$ untouched and apply the asymptotic analysis of Lemma \ref{lemF1}b). The function $\breve{F}_{2}$ shares many properties of $F_2$. Again we split the $y$-integral into smooth dyadic ranges $y \asymp Y$ and call the corresponding piece $\breve{F}_{2, Y}$. This satisfies the bound \eqref{boundf2} with $d = c/b. $ In addition we now have the condition 
$$a\Big(\frac{\sqrt{c/b}}{Y^2} + \frac{b}{Y^2} + \frac{b}{c \sqrt{Y}}\Big)\succcurlyeq 1$$
and partial integration with respect to $y$ gives
$$\frac{b}{Y^2} \preccurlyeq \frac{\sqrt{c/b}}{Y^2}  + \frac{b}{c \sqrt{Y}}.$$
{Write \begin{align*}
    \breve{F}_{2, Y}^{\epsilon_1,\epsilon_2,\eta_1}(a,b,c)=\int_{y\asymp Y}\int_{(0)}\int_0^\infty\left(\frac{cx}{y^3}\right)^{-s}\mathcal{G}_{\mu_0}^{\eta_1}(s+1) G_2(a,b,c,  x, y)y\frac{\mathrm{d}x}{x}\frac{\mathrm{d}s}{2\pi i}\mathrm{d}y.
\end{align*}
We now apply the same  analysis as in \eqref{statphs}, based on Stirling's formula, to show that we can restrict the $s$-integral to 
$|\Im(s)| \preccurlyeq 1 + c^{1/3}/Y$. On the other hand, Lemma \ref{Kw6} gives us directly \begin{equation}\label{breveF}
       G_2(a,b,c, x,y) \preccurlyeq { \min \Big(\Big(\frac{c^{1/4}\sqrt{Y}}{b^{1/4}}, \frac{\sqrt{c}Y^{1/4}}{\sqrt{b}}\Big) + \delta(c \asymp bY) \Big(\frac{cY}{b}\Big)^{1/4} }
       \end{equation}
and   $ \breve{F}_{2, Y}^{\epsilon_1,\epsilon_2,\eta_1}(a,b,c)$ is negligible unless       
     $$   {Y + \frac{c}{b}}+ \frac{c}{(c/b)^{3/2}+Y^{3/2}} +  \frac{\Delta^{\ast}}{X^{2/3}Y} + \frac{1}{a[\sqrt{c}/(\sqrt{b}Y^2)+b/(c\sqrt{Y})]} + \frac{b}{\sqrt{c/b} + bY^{3/2}/c} \preccurlyeq  1.  $$}



Next we analyze the character sum (cf.\ \eqref{C5Def} with $f=0$ and \eqref{C4Def})
\begin{align*}
    \mathcal{C}_5
    =&\sumast_{k\Mod{M}}\sumast_{\substack{\alpha\Mod{\nu_0n_{0,0}'D_{0,0}D_{0,1}f_0'g_0}\\Y_4 \equiv 0 \Mod{\nu_0D_{0,1}}\\\left(\frac{Y_4}{\nu_0D_{0,1}},f_0'\right)=1}} e\left(-\epsilon_2\eta_1\frac{n_{0,0}'n_0D_0'\overline{\frac{Y_4}{\nu_0D_{0,1}}kg_0g} n'}{D_{1,0}f_0'}\right)\nonumber\\
    &\times e\left(\epsilon_2\eta_1\frac{\nu_0n_{0,0}'n_0D_{0,1}D_0'\overline{k}n'}{D_{1,0}D_{2,0}'D_2'f_0'g_0g}\right)S\left(m,\nu n_{1,0}\overline{D_0'h_2}\frac{\overline{\alpha} Y_4}{D_{0,1}};\nu_0\nu n_{0,0}'n_{1,0}D_{0,0}D_{2,0}'D_2'f_0'g_0\right),
\end{align*}
with $$Y_4 =  D_{2,0}'D_2' +\epsilon_1 n_0D_{1,0}D_1gh_1^2h_2k\alpha, \qquad M=\nu_0n_{0,0}'D_{0,0}D_{0,1}D_{1,0}D_{2,0}'D_2'f_0'g_0g.$$
{Since $n_0$ is squarefree, we can write $n_0=n_1n_2$ with $(n_1,D_2')=1$ and $n_2 \mid D'_2$. Write $n_2D_2=D_2'$. Similarly, we can write $g=g_1g'$ with $g_1=(g,D_2)$ and $(g',D_2)=1$ as $g$ is squarefree. Since $\left(\nu D_2'g,\frac{M}{D_2'g}\right)=1$, we can split the character sum \begin{align*}
    \mathcal{C}_5=\mathcal{C}_{5,1}\mathcal{C}_{5,2}(n')\mathcal{C}_{5,3}(n'),
\end{align*}
given by \begin{align*}
    \mathcal{C}_{5,1}=&\sumast_{k\Mod{\frac{M}{D_2'g}}}\sumast_{\substack{\alpha\Mod{\nu_0n_{0,0}'D_{0,0}D_{0,1}f_0'g_0}\\Y_4 \equiv 0 \Mod{\nu_0D_{0,1}}\\\left(\frac{Y_4}{\nu_0D_{0,1}},f_0'\right)=1}} e\left(-\epsilon_2\eta_1\frac{n_{0,0}'n_0D_0'\overline{\frac{Y_4}{\nu_0D_{0,1}}g_0gk} n'}{D_{1,0}f_0'}\right)\nonumber\\
    &\times e\left(\epsilon_2\eta_1\frac{\nu_0n_{0,0}'n_0D_{0,1}D_0'\overline{D_2'gk}n'}{D_{1,0}D_{2,0}'f_0'g_0}\right)S\left(m,n_{1,0}\overline{\nu D_0'D_2'^2h_2}\frac{\overline{\alpha} Y_4}{D_{0,1}};\nu_0n_{0,0}'n_{1,0}D_{0,0}D_{2,0}'f_0'g_0\right),
\end{align*}
\begin{align*}
    \mathcal{C}_{5,2}(n')=&\sumast_{k\Mod{n_2D_2g_1}}e\Big(\epsilon_2\eta_1\frac{\nu_0n_{0,0}'n_1D_{0,1}D_0'\overline{D_{1,0}D_{2,0}'f_0'g_0g'k}n'}{D_2g_1}\Big)\nonumber\\
    &\times S\left(m,\epsilon_1\nu n_{1,0}n_1n_2D_{1,0}D_1gh_1^2k\overline{(\nu_0n_{0,0}'n_{1,0}D_{0,0}D_{2,0}'f_0'g_0)^2D_{0,1}D_0'};\nu n_2D_2\right)
\end{align*}
and \begin{align*}
    \mathcal{C}_{5,3}(n')=&\sumast_{k\Mod{g'}}e\Big(\epsilon_2\eta_1\frac{\nu_0n_{0,0}'n_1D_{0,1}D_0'\overline{D_{1,0}D_{2,0}'D_2f_0'g_0g_1k}n'}{g'}\Big)= \sum_{g_2g_3=g'}\mu(g_2)g_3\delta(g_3\mid n').
\end{align*}
Here we have used $(g,n_1n_2)=1$. Note that the Weil bound on the Kloosterman sums gives us \begin{align*}
    \mathcal{C}_{5,1}\ll \nu_0^{3/2}n_{0,0}'^{5/2}n_{1,0}^{1/2}D_{0,0}^{5/2}D_{0,1}D_{1,0}D_{2,0}'^{3/2}f_0'^{5/2}g_0^{5/2}.
\end{align*}}

As in Section \ref{secondprelim} we split all variables in dyadic intervals:
\begin{displaymath}
\begin{split}
   & g_0 \asymp G_0, \quad n'_{0, 0} \asymp N_{0, 0}, \quad f_0' \asymp F_0, \quad D_{0, 0} \asymp \Delta_{0, 0}, \quad n_{1, 0} \asymp N_{1, 0}, \quad D_{0, 1} \asymp \Delta_{0, 1}, \quad D_{1, 0} \asymp \Delta_{1, 0}, \quad D_{2, 0}' \asymp \Delta_{2, 0}, \quad D_1 \asymp \Delta_1,\\ 
   &n' \asymp N,  \quad {n_1 \asymp N_1, \quad n_2 \asymp N_2}, \quad D_0' \asymp \Delta_0, \quad {D_2 \asymp \Delta_2}, \quad {g} \asymp G, \quad \nu \asymp M, \quad \nu_0 \asymp M_0, \quad h_1 \asymp H_1, \quad h_2 \asymp H_2.
    \end{split}
\end{displaymath}
In addition, we split the $y$-integral in the definition  \eqref{F2Defnew} into dyadic ranges 
$$y \asymp Y.$$
These conventions will from now on be implicit in our notation.

{Write $n'=g_3n$, \begin{align*}
    A=&\frac{\nu_0^3\nu n_{0,0}'^3n_{1,0}n_1n_2^2D_{0,0}^3D_{0,1}^2D_{1,0}D_{2,0}'D_0'D_1D_2f_0'^3g_0^3(gh_1h_2)^2}{X^{2/3-\varepsilon}}, \quad B=\frac{\nu_0^3\nu n_{0,0}'^3n_{1,0}n_1^2n_2D_{0,0}^2D_{0,1}^2D_0'^2D_1f_0'g_0h_1^2h_2^2}{D_{2,0}'D_2X^{2/3}},\\
    C=&\frac{\nu_0^5\nu^2n_{0,0}'^5n_{1,0}^2n_1^2n_2^2D_{0,0}^4D_{0,1}^3D_0'^3D_1f_0'^3g_0^3h_1^2h_2^3}{X}. 
\end{align*}
Let $\mathbf{B}$ and $\mathbf{C}$ denote the corresponding dyadic size of $B$ and $C$ respectively. Write $\sigma=-1+\varepsilon$ if $Y^3X^\varepsilon> \mathbf{C}N$ and $\sigma =0$ otherwise. Write $T=\max\{1, (\mathbf{C}N)^{1/3}/Y\}X^\varepsilon$. Shifting the contour of the $s$-integral to $\Re(s)=\sigma$, the earlier discussion implies that $\Im(s)\ll T$.}

Applying the Cauchy-Schwarz inequality to take out { the $s$-integral and every sum except $D_2, g_1$, together with Stirling's approximation on the Gamma factors}, the corresponding dyadic piece of $\breve{\Sigma}_6^\flat $ is bounded by  (recall \eqref{sigmasharp}) \begin{align}
 \preccurlyeq & X^{4/3}\frac{T^{2+3\sigma}}{(\mathbf{C}N)^\sigma}
\Bigg(\sum_{\nu_0}\sum_{g_0}\sum_{n_{0,0}'}\sum_{f_0'} \sum_{D_{0,0}}\mathop{\sum\sum\sum\sum}_{ n_{1,0}D_{0,1}D_{1,0}D_{2,0}'\mid (\nu_0n_{0,0}'D_{0,0}f_0'g_0)^\infty
    }\mathop{\sum\sum\sum\sum\sum\sum\sum\sum\sum}_{ \nu , n_1, n_2, D_0', D_1, g_2,g_3, h_1, h_2 }\int_\R U\left(\frac{t}{T}\right)\sum_{n }V\left(\frac{g_3n}{N}\right)\frac{1}{ n}   \nonumber\\
    &\ \times \Bigg|\mathop{\sum\sum}_{\substack{(D_2g_1,\nu_0n_{0,0}'n_1D_{0,0}D_0'f_0'g_0h_1h_2)=1\\ g=g_1g_2g_3}}\frac{{\mu(n_1n_2g)\mu(g)\mathcal{C}_{5,1}}}{\nu_0^{13/2}\nu^{5/2} n_{0,0}'^{15/2}n_{1,0}^3n_1^{5/2}n_2^{9/2}D_{0,0}^{13/2}D_{0,1}^4D_{1,0}^2D_{2,0}'^2D_0'^{5/2}D_1^{3/2}D_2^2f_0'^{11/2}g_0^{13/2}g^3h_1^{7/2}h_2^{7/2}}\nonumber\\
    &\ \times W_0\Big(\frac{D_{1, 0}g}{\Xi}\Big){\mathcal{C}_{5, 2}(g_3n) \int_{y\asymp Y}\int_0^\infty G_2(A,Bg_3n,Cg_3n,  x, y)y\left(\frac{y^3}{x}\right)^{\sigma+it}\frac{\mathrm{d}x}{x}\mathrm{d}y}\Bigg|^2{\mathrm{d}t}\Bigg)^{1/2}.
    \label{suffices}
\end{align}

We open the square and introduce {four} new variables $\tilde{D}_2, \tilde{g}_1{, \tilde{x}, \tilde{y}}$. Correspondingly, we define $\tilde{A}, \tilde{B}$, $\tilde{g}$, and we denote by $\tilde{k}$ the summation variable in the second character sum. By Poisson summation, the $n'$-sum becomes\begin{align*}
    &\sum_{n}\frac{1}{n}V\left(\frac{g_3n}{N}\right)\mathcal{C}_{5,2}(g_3n) \overline{\mathcal{C}_{5,2}(g_3n) }{G_2(A,Bg_3n,Cg_3n, x, y)\overline{G_2(\tilde{A},\tilde{B}g_3n,Cg_3n, \tilde{x}, \tilde{y})}} = \frac{1}{D_2\tilde{D}_2g_1\tilde{g}_1}\sum_{n}\mathcal{C}_5' \mathcal{J}
\end{align*}
where  \begin{align*}
    \mathcal{C}_5' =\sum_{\gamma\Mod{D_2\tilde{D}_2g_1\tilde{g}_1}}\mathcal{C}_{5,2}(g_3\gamma)\overline{\mathcal{C}_{5,2}(g_3\gamma)}e\left(\frac{n\gamma}{D_2\tilde{D}_2g_1\tilde{g}_1}\right) ,
\end{align*}
and \begin{align*}
    \mathcal{J}=\int_0^\infty \frac{V(w)}{w}{G_2(A,BNw,CNw,x, y)\overline{G_2(\tilde{A},\tilde{B}Nw,CNw,\tilde{x}, \tilde{y})}}e\left(-\frac{nNw}{D_2\tilde{D}_2g_1\tilde{g}_1g_3}\right)\mathrm{d}w.
\end{align*}
{Notice that the only appearance of $w$ in $G_2(A,BNw,CNw,x, y)$ is in $e\left(-\epsilon_2\eta_1BNw/y\right)\left(1-U\Big(A\Big(\frac{\sqrt{C}}{\sqrt{B}y^2} + \frac{BNw}{y^2} + \frac{B}{C \sqrt{y}}\Big)\Big)\right)$, and similarly for $\overline{G_2(\tilde{A},\tilde{B}Nw,CNw,\tilde{x}, \tilde{y})}$.} Repeated integration by parts gives us arbitrary savings unless \begin{align*}
    {n\preccurlyeq  \tilde{N} := \frac{\Delta_2^2G^2}{N}+\frac{B\Delta_2^2G^2}{Y}.}
\end{align*}

{Inserting this back into \eqref{suffices}, the whole integral is now given by \begin{align}\label{WholeIntegralinLemma5.5}
    \int_\R &\ U\left(\frac{t}{T}\right)\int_{y\asymp Y}\int_{\tilde{y}\asymp Y}\int_0^\infty \int_0^\infty \int_0^\infty V(w)G_2(A,BNw,CNw,  x, y)\nonumber\\
    &\ \times 
    \overline{G_2(\tilde{A},\tilde{B}Nw,CNw,   \tilde{x}, \tilde{y})}\frac{(y\tilde{y})^{1+3\sigma}}{(x\tilde{x})^{1+\sigma}}\left(\frac{\tilde{x}y^3}{x\tilde{y}^3}\right)^{it}e\left(-\frac{nNw}{D_2\tilde{D}_2g_1\tilde{g}_1g_3}\right)\,\frac{\mathrm{d}w}{w}\, \mathrm{d}x\, \mathrm{d}\tilde{x}\, \mathrm{d}y\, \mathrm{d}\tilde{y}\, \mathrm{d}t.
\end{align}
The $t$-integral gives us arbitrary savings unless $|\tilde{x}-\frac{\tilde{y}^3}{y^3}x|\preccurlyeq 1/T$. Then by \eqref{breveF}  this whole integral is bounded by \begin{align*}
    \preccurlyeq &\ Y^{4+6\sigma}\Big(\min \Big(\Big(\frac{\mathbf{C}^{1/4}\sqrt{Y}}{\mathbf{B}^{1/4}}, \frac{\sqrt{\mathbf{C}}Y^{1/4}}{\sqrt{\mathbf{B}}}\Big) + \delta(\mathbf{C} \asymp \mathbf{B}Y) \Big(\frac{\mathbf{C}Y}{\mathbf{B}}\Big)^{1/4}\Big)^2.\\
\end{align*} 
}

Now we evaluate the character sum $\mathcal{C}_5'$. Summing over $\gamma$, we have \begin{align*}
    &\sum_{\gamma\Mod{D_2\tilde{D}_2g_1\tilde{g}_1}}e\left(\epsilon_2\eta_1\frac{\nu_0n_{0,0}'n_1D_{0,1}D_0'\overline{D_{1,0}D_{2,0}'f_0'g_0g_2k}\gamma}{D_2g_1}-\epsilon_2\eta_1\frac{\nu_0n_{0,0}'n_1D_{0,1}D_0'\overline{D_{1,0}D_{2,0}'f_0'g_0g_2\tilde{k}}\gamma}{\tilde{D}_2\tilde{g}_1}+\frac{n\gamma}{D_2\tilde{D}_2g_1\tilde{g}_1}\right)\\
    =&D_2\tilde{D}_2g_1\tilde{g}_1\delta\left(\tilde{D}_2\tilde{g}_1\overline{k}\equiv D_2g_1\overline{\tilde{k}}-\epsilon_2\eta_1D_{1,0}D_{2,0}'f_0'g_0g_2\overline{\nu_0n_{0,0}'n_1D_{0,1}D_0'}n\Mod{D_2\tilde{D}_2g_1\tilde{g}_1}\right),
\end{align*}
so that  \begin{align*}
    \mathcal{C}_5' =&D_2\tilde{D}_2g_1\tilde{g}_1\sumast_{k\Mod{n_2D_2g_1}}\sumast_{\tilde{k}\Mod{n_2\tilde{D}_2\tilde{g}_1}} \delta\left(\tilde{D}_2\tilde{g}_1\overline{k}\equiv D_2g_1\overline{\tilde{k}}-\epsilon_2\eta_1D_{1,0}D_{2,0}'f_0'g_0g_2\overline{\nu_0n_{0,0}'n_1D_{0,1}D_0'}n\Mod{D_2\tilde{D}_2g_1\tilde{g}_1}\right)\\
    &\times S\left(m,\epsilon_1\nu n_{1,0}n_1n_2D_{1,0}D_1gh_1^2k\overline{(\nu_0n_{0,0}'n_{1,0}D_{0,0}D_{2,0}'f_0'g_0)^2D_{0,1}D_0'};\nu n_2D_2\right)\\
    &\times S\left(m,\epsilon_1\nu n_{1,0}n_1n_2D_{1,0}D_1\tilde{g}h_1^2\tilde{k}\overline{(\nu_0n_{0,0}'n_{1,0}D_{0,0}D_{2,0}'f_0'g_0)^2D_{0,1}D_0'};\nu n_2 \tilde{D}_2\right).
\end{align*}
Write $$c_0=(D_2g_1,\tilde{D}_2\tilde{g}_1), \quad  c_1=\Big(\frac{D_2g_1}{c_0},c_0^\infty\Big), \quad \tilde{c_1}=\Big(\frac{\tilde{D}_2\tilde{g}_1}{c_0},c_0^\infty\Big), \quad c=\frac{D_2g_1}{c_0c_1}, \quad   \tilde{c}=\frac{\tilde{D}_2\tilde{g}_1}{c_0\tilde{c_1}}$$ and $$Z=\epsilon_2\eta_1D_{1,0}D_{2,0}'f_0'g_0g_2\overline{\nu_0n_{0,0}'n_1D_{0,1}D_0'}.$$
Then the congruence condition decomposes into \begin{gather*}
    c_0\mid n, \quad \tilde{c_1}\tilde{c}\overline{k}\equiv c_1c\overline{\tilde{k}}-Z\frac{n}{c_0}\Mod{c_0c_1\tilde{c_1}}, \quad 
    k\equiv -\tilde{D}_2\tilde{g}_1\overline{Zn}\Mod{c}, \quad \tilde{k}\equiv D_2g_1\overline{Zn}\Mod{\tilde{c}}.
\end{gather*}
Inserting this back into $\mathcal{C}_5'$ and applying the Weil bound on  {the} Kloosterman sums, we have \begin{align*}
    \mathcal{C}_5'\ll \nu n_2^3D_2^3g_1^2\delta(c_0\mid n)c_0c_1.
\end{align*}

Combining all of the above analysis and noting that
  by the definition of $T$ and $\sigma$, we have \begin{align*}
  \frac{T^{2+3\sigma}}{(\mathbf{C}N)^\sigma}   Y^{2+3\sigma}\preccurlyeq \min\left(\frac{\mathbf{C}N}{Y}, (\mathbf{C}N)^{2/3}\right),
\end{align*} we conclude after careful book-keeping  that \eqref{suffices} is bounded by\begin{align}\label{5.5Case1}
   \preccurlyeq  &\frac{X^{4/3}}{M_0^{9/2}M^{3/2}N_{0,0}^{9/2}N_{1,0}^{{5/2}}{N_1^2N_2^{5/2}}\Delta_{0,0}^{7/2}\Delta_{0,1}^3\Delta_{1,0}\Delta_{2,0}^{1/2}\Delta_0^2\Delta_1\Delta_2^{1/2}F_0^{5/2}G_0^{7/2}G^2H_1^3H_2^3}
   {\left(1+\frac{
    \Delta_2G}{\sqrt{N}}+\frac{\sqrt{B}\Delta_2G}{\sqrt{Y}}\right)\min\Big(\frac{\mathbf{C}N}{Y}, (\mathbf{C}N)^{2/3}\Big)    \Big(\frac{\mathbf{C}Y}{\mathbf{B}}\Big)^{1/4}}
\end{align}
with the conditions \begin{align}\label{5.5conditions}
    Y\preccurlyeq  {1}, \quad  \frac{\mathbf{C}}{\mathbf{B}}\preccurlyeq  {1}, \quad \mathbf{C}N\ll \left(\frac{\mathbf{C}}{\mathbf{B}}\right)^{3/2}+Y^{3/2}, \quad  X^{2/3}Y\gg \Delta^{\ast},\quad \frac{\mathbf{B}N}{Y^2} \preccurlyeq \frac{\sqrt{\mathbf{C}}}{\sqrt{\mathbf{B}}Y^2}+\frac{\mathbf{B}}{\mathbf{C}\sqrt{Y}}, \quad 
    \mathbf{A}\left(\frac{\sqrt{\mathbf{C}}}{\sqrt{\mathbf{B}}Y^2}+\frac{\mathbf{B}}{\mathbf{C}\sqrt{Y}}\right)\succcurlyeq 1,
\end{align}
where also $\mathbf{A}$ denotes the dyadic size of $A$. 
{We apply this bound under the \emph{stronger} assumption  $C\preccurlyeq BY$ and run the following linear program:}

\noindent {\tt a := 3 nu0 + nu + 3 n00 + n10 + n1 + 2 n2 + 3 d00 + 2 d01 + d10 + 
   d20 + d0 + d1 + d2 + 3 f0 + 3 g0 + 2 g + 2 h1 + 2 h2 - 2/3;\\
b := 3 nu0 + nu + 3 n00 + n10 + 2 n1 + n2 + 2 d00 + 2 d01 + 2 d0 + 
   d1 + f0 + g0 + 2 h1 + 2 h2 - d20 - d2 - 2/3;\\
c := 5 nu0 + 2 nu + 5 n00 + 2 n10 + 2 (n1 + n2) + 4 d00 + 3 d01 + 
   3 d0 + d1 + 3 f0 + 3 g0 + 2 h1 + 3 h2 - 1;\\
}

\noindent {\tt Maximize[\{4/3 - 9/2 nu0 - 3/2 nu - 9/2 n00 - 5/2 n10 - 2 n1 - 
   5/2 n2 - 7/2 d00 - 3 d01 - d10 - 1/2 d20 - 2 d0 - d1 - 1/2 d2 - 
   5/2 f0 - 7/2 g0 - 2 g - 3 h1 - 3 h2 +  Max[0, d2 + g - 1/2 n, 1/2 b + d2 + g - 1/2 y] + 
   Min [c + n - y, 2/3 (c + n)] + 1/4 (c + y - b) , 2/3 + y >= 11/21,  
  y <= 0, c - b <= y, n + c <= 3/2 Max[c - b, y], 
  a + Max[1/2 (c - b) - 2 y, -c + b - y/2] >= 0, 
  n + b - 2 y <= Max[1/2 (c - b) - 2 y, -c + b - y/2], nu0 >= 0, 
  g0 >= 0, n00 >= 0, f0 >= 0, d00 >= 0, n10 >= 0, d01 >= 0, d10 >= 0, 
  d20 >= 0, nu >= 0, n1 >= 0, n2 >= 0, d0 >= 0, d1 >= 0, d2 >= 0, 
  g >= 0, h1 >= 0, h2 >= 0, n >= 0\}, \{nu0, g0, n00, f0, d00, n10, d01,
   d10, d20, nu, n1, n2, d0, d1, d2, g, h1, h2, n, y\}]}

   \noindent {\tt \{11/12, \{nu0 -> 0, g0 -> 0, n00 -> 0, f0 -> 0, d00 -> 0, n10 -> 0, 
  d01 -> 0, d10 -> 0, d20 -> 0, nu -> 0, n1 -> 0, n2 -> 0, d0 -> 0, 
  d1 -> 0, d2 -> 0, g -> 1/6, h1 -> 0, h2 -> 0, n -> 1, y -> 0\}\}}

{

From now on we assume $$C\gg BYX^\varepsilon.$$ 
This implies automatically $C \gg BY^2 X^{\varepsilon/2}$, hence recalling \eqref{f2alt}, we are in a position to use the decomposition \eqref{phidecomp} and Lemma \ref{lemF1}b) for our analysis (which is contained in $\breve{F}_2$ by \eqref{f2alt}).  
For the contribution of 
$\Phi_{\omega_6}^\flat$  we use Lemma \ref{Kw6}  together with the same analysis as above   except that the whole integral in \eqref{suffices}, namely \eqref{WholeIntegralinLemma5.5}, has its bound replaced by \begin{align*}
    \preccurlyeq  Y^{4+6\sigma} \left(\frac{BY^2}{C}\right)^2.
\end{align*} 
 Hence by the same argument leading to \eqref{5.5Case1}, its (dyadic) contribution to $\breve{\Sigma}_6^\flat$ is bounded by \begin{align}\label{5.5Case2Error}
    \preccurlyeq &\frac{X^{4/3}}{M_0^{9/2}M^{3/2}N_{0,0}^{9/2}N_{1,0}^{{5/2}}{N_1^2N_2^{5/2}}\Delta_{0,0}^{7/2}\Delta_{0,1}^3\Delta_{1,0}\Delta_{2,0}^{1/2}\Delta_0^2\Delta_1\Delta_2^{1/2}F_0^{5/2}G_0^{7/2}G^2H_1^3H_2^3}  \left(1+\frac{
    \Delta_2G}{\sqrt{N}}+\frac{\sqrt{B}\Delta_2G}{\sqrt{Y}}\right)\min\Big(\frac{\mathbf{C}N}{Y}, (\mathbf{C}N)^{2/3}\Big)    \frac{\mathbf{B}Y^2}{\mathbf{C}},
\end{align}
with the same conditions \eqref{5.5conditions} and we are left with the following linear program:

\noindent {\tt Maximize[\{4/3 - 9/2 nu0 - 3/2 nu - 9/2 n00 - 5/2 n10 - 2 n1 - 
   5/2 n2 - 7/2 d00 - 3 d01 - d10 - 1/2 d20 - 2 d0 - d1 - 1/2 d2 - 
   5/2 f0 - 7/2 g0 - 2 g - 3 h1 - 3 h2 +  Max[0, d2 + g - 1/2 n, 1/2 b + d2 + g - 1/2 y] + 
   Min [c + n - y, 2/3 (c + n)] + b + 2 y - c , 2/3 + y >= 11/21,  
  y <= 0, c - b <= 0, c >= b + y, n + c <= 3/2 Max[c - b, y], 
  a + Max[1/2 (c - b) - 2 y, -c + b - y/2] >= 0, 
  n + b - 2 y <= Max[1/2 (c - b) - 2 y, -c + b - y/2], nu0 >= 0, 
  g0 >= 0, n00 >= 0, f0 >= 0, d00 >= 0, n10 >= 0, d01 >= 0, d10 >= 0, 
  d20 >= 0, nu >= 0, n1 >= 0, n2 >= 0, d0 >= 0, d1 >= 0, d2 >= 0, 
  g >= 0, h1 >= 0, h2 >= 0, n >= 0\}, \{nu0, g0, n00, f0, d00, n10, d01,
   d10, d20, nu, n1, n2, d0, d1, d2, g, h1, h2, n, y\}]}

   \noindent {\tt \{5/6, \{nu0 -> 0, g0 -> 0, n00 -> 0, f0 -> 0, d00 -> 0, n10 -> 0, 
  d01 -> 0, d10 -> 0, d20 -> 1/3, nu -> 0, n1 -> 0, n2 -> 0, d0 -> 0, 
  d1 -> 0, d2 -> 0, g -> 1/6, h1 -> 0, h2 -> 0, n -> 1, y -> 0\}\}}

We now continue the analysis of the final case: $C\gg BYX^\varepsilon$ and $\Phi_{\omega_6}$ is replaced by $\Phi_{\omega_6}^\sharp$. 
We decompose
\begin{align*}
    \breve{F}_{2,Y}^{\epsilon_1,\epsilon_2,\eta_1}\left(A,Bg_3n,Cg_3n \right) = \breve{F}_{2,Y}^{\epsilon_1,\epsilon_2,\eta_1; \sharp}\left(A,Bg_3n,Cg_3n \right) + \breve{F}_{2,Y}^{\epsilon_1,\epsilon_2,\eta_1; \flat}\left(A,Bg_3n,Cg_3n \right)
\end{align*}
and recall from \eqref{F2Defnew} that \begin{align}\label{F2F1Relation}
    \breve{F}_{2,Y}^{\epsilon_1,\epsilon_2,\eta_1;\sharp}(A,Bg_3n,Cg_3n)=&\int_{y\asymp Y}  y\left(1-U\Big(A\Big(\frac{\sqrt{C}}{\sqrt{B}y^2} + \frac{Bg_3n}{y^2} + \frac{B}{C \sqrt{y}}\Big)\Big)\right)\left(1-U\left( {\frac{X^{2/3}y}{\Delta^{\ast}}}\right)\right) \nonumber \\
    &\times F_1^{\epsilon_1,\epsilon_2,\eta_1;\sharp}\left(\frac{Cg_3n}{y^3},\frac{C}{By^2},\frac{B^2y}{C^2}\right)e\left(-\epsilon_2\eta_1Bg_3ny^{-1}\right) \mathrm{d} y,
\end{align}
with $F_1^{\epsilon_1,\epsilon_2,\eta_1;\sharp}$ as in \eqref{F1AsympDecomposition}. 
We return to \eqref{sigmasharp}, apply the Cauchy-Schwarz inequality to take out everything   except $D_2, g_1$, but now keeping the full function $\breve{F}_{2,Y}^{\epsilon_1,\epsilon_2,\eta_1;\sharp}(A,Bg_3n,Cg_3n)$ inside, and obtain the bound \begin{align*}
    \preccurlyeq &X^{4/3}\sum_{\epsilon_1,\epsilon_2,\eta_1=\pm1} \Bigg(\sum_{\nu_0}\sum_{g_0}\sum_{n_{0,0}'}\sum_{f_0'} \sum_{D_{0,0}}\mathop{\sum\sum\sum\sum}_{\substack{n_{1,0}D_{0,1}D_{1,0}D_{2,0}'\mid (\nu_0n_{0,0}'D_{0,0}f_0'g_0)^\infty\\(n_{1,0}D_{1,0}D_{2,0}',D_{0,1})=1\\(D_{1,0},\nu_0n_{0,0}'D_{2,0}'g_0)=(D_{2,0}',\nu_0)=1}}\mathop{\sum\sum\sum\sum\sum\sum\sum\sum}_{\substack{(\nu n_1n_2D_0'D_1g_2g_3h_1h_2,\nu_0n_{0,0}'D_{0,0}f_0'g_0)=1}}\sum_n\frac{1}{ n}\nonumber\\
    &\times V\left(\frac{g_3n}{N}\right)\Bigg|\mathop{\sum\sum\sum\sum\sum}_{\substack{(D_2g_1,\nu_0n_{0,0}'D_{0,0}f_0'g_0)=1\\ g_1g_2g_3=g, \, g_1\mid D_2 \\ (n_1g_2g_3,D_2)=(h_1h_2,n_2D_2g)=1\\ (\nu n_1n_2D_1D_2gh_1,D_0'h_2)=1}}\frac{\mu(n_1n_2g)\mu(g)\mathcal{C}_{5,1}}{\nu_0^{13/2}\nu^{5/2} n_{0,0}'^{15/2}n_{1,0}^3n_1^{5/2}n_2^{9/2}D_{0,0}^{13/2}D_{0,1}^4D_{1,0}^2D_{2,0}'^2D_0'^{5/2}D_1^{3/2}D_2^2f_0'^{11/2}g_0^{13/2}g^3h_1^{7/2}h_2^{7/2}}  
    \nonumber\\
    &\times W_0\Big(\frac{D_{1, 0}g}{\Xi}\Big)\mathcal{C}_{5,2}(g_3n) \breve{F}_{2,Y}^{\epsilon_1,\epsilon_2,\eta_1;\sharp}(A,Bg_3n,Cg_3n)\Bigg|^2\Bigg)^{1/2}.
\end{align*}
Following the same analysis as in the previous cases, we open the square and introduce two new variables $\tilde{D}_2, \tilde{g_1}$. Correspondingly, we define $\tilde{A}, \tilde{B}$, $\tilde{g}$, and we denote by $\tilde{k}$ the summation variable in the second character sum. By Poisson summation, the $n$-sum becomes\begin{align*}
    &\sum_{n}\frac{1}{n}V\left(\frac{g_3n}{N}\right)\mathcal{C}_{5,2}(g_3n) \overline{\mathcal{C}_{5,2}(g_3n) }\breve{F}_{2,Y}^{\epsilon_1,\epsilon_2,\eta_1;\sharp}(A,Bg_3n,Cg_3n)\overline{\breve{F}_{2,Y}^{\epsilon_1,\epsilon_2,\eta_1;\sharp}(\tilde{A},\tilde{B}g_3n,\tilde{C}g_3n)} = \frac{1}{D_2\tilde{D}_2g_1\tilde{g}_1}\sum_{n}\mathcal{C}_5' \mathcal{J}'
\end{align*}
where \begin{align*}
    \mathcal{C}_5' =\sum_{\gamma\Mod{D_2\tilde{D}_2g_1\tilde{g}_1}}\mathcal{C}_{5,2}(g_3\gamma)\overline{\mathcal{C}_{5,2}(g_3\gamma)}e\left(\frac{n\gamma}{D_2\tilde{D}_2g_1\tilde{g}_1}\right)\ll \nu n_2^3D_2^3g_1^2\delta(c_0\mid n)c_0c_1
\end{align*}
with $c_0=(D_2g_1,\tilde{D}_2\tilde{g}_1)$ and $c_1=(D_2g_1/c_0,c_0^\infty)$ as before, and \begin{align*}
    \mathcal{J}'=\int_0^\infty \frac{V(w)}{w}\breve{F}_{2,Y}^{\epsilon_1,\epsilon_2,\eta_1;\sharp}(A,BNw,CNw)\overline{\breve{F}_{2,Y}^{\epsilon_1,\epsilon_2,\eta_1;\sharp}(\tilde{A},\tilde{B}Nw,\tilde{C}Nw)}e\left(-\frac{nNw}{D_2\tilde{D}_2g_1\tilde{g}_1g_3}\right)\mathrm{d}w.
\end{align*}
Repeated integration by parts gives us arbitrary savings unless \begin{align*}
    n\preccurlyeq \tilde{N}':= \frac{\Delta_2^2g_1^2g_3}{N}+\frac{\mathbf{B}\Delta_2^2g_1^2g_3}{Y}\ll \left(\frac{\Delta_2^2G^2}{N}+\frac{\mathbf{B}\Delta_2^2G^2}{Y}\right).
\end{align*}

Since $C\gg BYX^\varepsilon$, Lemma \ref{lemF1} applied on \eqref{F2F1Relation} together with the analysis in Lemma \ref{lemf2} gives us \begin{align*}
    \breve{F}_{2,Y}^{\epsilon_1,\epsilon_2,\eta_1; \sharp}\preccurlyeq &\ \frac{\mathbf{C}Y}{\mathbf{B}}\delta(N\asymp \sqrt{\mathbf{C}}\mathbf{B}^{-3/2}),\nonumber\\
\end{align*}
and 
as a result 
\begin{align*}
    \mathcal{J}'\preccurlyeq &\ \left(\frac{\mathbf{C}Y}{\mathbf{B}}\right)^2\delta(N\asymp \sqrt{\mathbf{C}}\mathbf{B}^{-3/2}),\nonumber\\
\end{align*}
keeping in mind the condition \eqref{5.5conditions}.  Combining the above analysis, we now  obtain the upper bound
\begin{align*}
    \breve{\Sigma}_6^{\flat; \sharp}\preccurlyeq &\frac{X^{4/3}}{M_0^{9/2}M^{3/2}N_{0,0}^{9/2}N_{1,0}^{{5/2}}{N_1^2N_2^{5/2}}\Delta_{0,0}^{7/2}\Delta_{0,1}^3\Delta_{1,0}\Delta_{2,0}^{1/2}\Delta_0^2\Delta_1\Delta_2^{1/2}F_0^{5/2}G_0^{7/2}G^2H_1^3H_2^3}\nonumber\\
    &\times \left(1+\frac{\Delta_2^2G^2}{N}+\frac{\mathbf{B}\Delta_2^2G^2}{Y}\right)^{1/2}\frac{\mathbf{C}Y}{\mathbf{B}}\delta(N\asymp \sqrt{\mathbf{C}}\mathbf{B}^{-3/2}).
\end{align*}
Our final linear program is 


}

\noindent {\tt Maximize[\{4/3 - 9/2 nu0 - 3/2 nu - 9/2 n00 - 5/2 n10 - 2 n1 - 5/2 n2 - 
   7/2 d00 - 3 d01 - d10 - 1/2 d20 - 2 d0 - d1 - 1/2 d2 - 5/2 f0 - 
   7/2 g0 - 2 g - 3 h1 - 3 h2 + 1/2 Max[0, 2d2 + 2g-n, b + 2d2 + 2g - y] +   c +  y - b , 
  n == 1/2 c - 3/2 b,   2/3 + y >= 11/21, y <= 0, c - b <= 0, 
  n + c <= 3/2 Max[c - b, y], 
  a + Max[1/2 (c - b) - 2 y, -c + b - y/2] >= 0, 
  n + b - 2 y <= Max[1/2 (c - b) - 2 y, -c + b - y/2], nu0 >= 0, 
  g0 >= 0, n00 >= 0, f0 >= 0, d00 >= 0, n10 >= 0, d01 >= 0, d10 >= 0, 
  d20 >= 0, nu >= 0, n1 >= 0, n2 >= 0, d0 >= 0, d1 >= 0, d2 >= 0, 
  g >= 0, h1 >= 0, h2 >= 0, n >= 0\}, \{nu0, g0, n00, f0, d00, n10, d01,
   d10, d20, nu, n1, n2, d0, d1, d2, g, h1, h2, n, y\}] }

\noindent {\tt  \{41/42, \{nu0 -> 0, g0 -> 0, n00 -> 0, f0 -> 0, d00 -> 0, n10 -> 0, d01 -> 0, 
 d10 -> 0, d20 -> 1/3, nu -> 0, n1 -> 0, n2 -> 0, d0 -> 0, d1 -> 0, 
 d2 -> 0, g -> 1/42, h1 -> 0, h2 -> 0, n -> 1, y -> -1/7\}\}  }
 \end{proof}

\subsection{Interlude V: Preparation of the second Voronoi summation}

We return to \eqref{sigmasharp} where the definition of $\tilde{F}_2$ is modified 
by removing the  factor \eqref{U}. We call the modified weight function $ \acute{F}_2$ (spelled out in \eqref{Facute2Def} below) and the modified quantity $\Sigma_6^{(3)}$. In order to prepare for the second application of Voronoi summation, we need to manipulate some of the sums.

Applying Lemma \ref{lem.Hecke} with $(D_{0, 0}, D_1) = 1$, we have \begin{align*}
    B(D_{0,0}D_1n',\nu_0\nu n_{0,0}'n_{1,0}n_0h_1)=\mathop{\sum\sum}_{\substack{r_0s_0't_0=D_{0,0}\\s_0'\mid \nu_0\nu n_{0,0}'n_{1,0}n_0h_1t_0\\t_0\mid n'}}\mathop{\sum\sum}_{\substack{rst=D_1\\s\mid \nu_0\nu n_{0,0}'n_{1,0}n_0h_1t\\t\mid n'}}\mu(s_0's)\mu(t_0t)B\left(\frac{n'}{t_0t},\frac{\nu_0\nu n_{0,0}'n_{1,0}n_0h_1t_0t}{s_0's}\right)B(r_0r,1).
\end{align*}
Note that since $n_{0,0}'n_0D_{0,0}D_1h_1$ is squarefree, pairwise coprime, and $(D_{0,0},\nu)=(D_1,\nu_0n_{1,0})=1$, the divisibility conditions for $s_0'$ and $s$ reduce to $s_0'\mid \nu_0n_{1,0}$ and $s\mid \nu$ respectively. Write $$s_0'=s_0s_{1,0}$$ with $s_0\mid \nu_0$ and $(s_{1,0},\nu_0/s_0)=1$. Writing $$\nu=\nu's, \quad \nu_0=\nu_0's_0, \quad n_{1,0}=n_{1,0}'s_{1,0}, \quad  n'=t_0tn,$$ we have \begin{align*}
    \Sigma_6^{(3)}=&\sum_{\epsilon_1,\epsilon_2,\eta_1=\pm1} \sum_{\nu_0'}\sum_{g_0}\sum_{r_0}\sum_{s_0}\sum_{(s_{1,0},\nu_0')=1}\sum_{t_0}\sum_{n_{0,0}'}\sum_{(f_0',n_{0,0}'g_0)=1}\mathop{\sum\sum\sum\sum}_{\substack{n_{1,0}'D_{0,1}D_{1,0}D_{2,0}'\mid (\nu_0'n_{0,0}'r_0s_0s_{1,0}t_0f_0'g_0)^\infty\\(n_{1,0}'s_{1,0}D_{1,0}D_{2,0}',D_{0,1})=1\\(D_{1,0},\nu_0'n_{0,0}'s_0D_{2,0}'g_0)=(D_{2,0}',\nu_0's_0)=1}}\mathop{\sum\sum\sum\sum\sum\sum\sum\sum\sum\sum}_{\substack{(\nu' n_0rstD_0'D_2'gh_1h_2,\nu_0'n_{0,0}'r_0s_0s_{1,0}t_0f_0'g_0)=1\\
    (h_1h_2,D_2'g)=(\nu' n_0rstD_2'gh_1,D_0'h_2)=1}}\\
    &\times \frac{X^{4/3}\mu (n_{0,0}'n_0r_0rs_0s_{1,0}st_0tg_0gh_1 )\mu(s_0s_{1,0}st_0tg_0gh_1h_2)B(r_0r,1)}{\nu_0'^7\nu'^3 n_{0,0}'^8n_{1,0}'^3n_0^3r_0^7r^2s_0^{14}s_{1,0}^{10}s^5t_0^8t^3D_{0,1}^4D_{1,0}^2D_{2,0}'^2D_0'^3D_2'^2f_0'^6g_0^7g^3h_1^4h_2^4} W_0\Big(\frac{D_{1, 0}g}{\Xi}\Big) 
    \sum_{n} \frac{B (n,\nu_0'\nu' n_{0,0}'n_{1,0}'n_0t_0th_1 )}{n}\\
    & \times \mathcal{C}_6
    \acute{F}_2^{\epsilon_1,\epsilon_2,\eta_1}(X_1n, X_2n)
    \end{align*}
    where \begin{align}\label{Facute2Def}
    &\acute{F}_2^{\epsilon_1,\epsilon_2,\eta_1}(b,c)\\
    =&\int_0^\infty \int_{(0)}\int_0^\infty x^{-s-1}c^{-s}y^{3s+1} \mathcal{G}_{\mu_0}^{\eta_1}(s+1)\Phi_{\omega_6}\left(-\epsilon_2\frac{cx}{by^2},-\epsilon_1\frac{b^2my}{c^2}\right)V(x)\left(1-U\left( {\frac{X^{2/3}y}{\Delta^{\ast}}}\right)\right)e\left(-\epsilon_2\eta_1by^{-1}\right)\mathrm{d} x\, \frac{\mathrm{d}s}{2\pi i} \mathrm{d} y, \nonumber
\end{align}
with 
\begin{displaymath}
     \begin{split}  
  & X_1 = \frac{\nu_0'^3\nu' n_{0,0}'^3n_{1,0}'n_0^2r_0^2rs_0^5s_{1,0}^3s^2t_0^3t^2D_{0,1}^2D_0'^2f_0'g_0h_1^2h_2^2}{D_{2,0}'D_2'X^{2/3}}
, \quad X_2 = \frac{\nu_0'^5\nu'^2n_{0,0}'^5n_{1,0}'^2n_0^2r_0^4rs_0^9s_{1,0}^6s^3t_0^5t^2D_{0,1}^3D_0'^3f_0'^3g_0^3h_1^2h_2^3}{X},
    \end{split}
    \end{displaymath}
and 
 \begin{align*}
 \mathcal{C}_6 =    &\sumast_{k\Mod{\nu_0'n_{0,0}'r_0s_0^2s_{1,0}t_0D_{0,1}D_{1,0}D_{2,0}'D_2'f_0'g_0g}}e\left(\epsilon_2\eta_1\frac{\nu_0'n_{0,0}'n_0s_0t_0tD_{0,1}D_0'\overline{k}n}{D_{1,0}D_{2,0}'D_2'f_0'g_0g}\right)\sumast_{\substack{\alpha\Mod{\nu_0'n_{0,0}'r_0s_0^2s_{1,0}t_0D_{0,1}f_0'g_0}\\Y_6 \equiv 0 \Mod{\nu_0's_0D_{0,1}}\\ \big(\frac{Y_6}{\nu_0's_0D_{0,1}},f_0'\big)=1}}\\
    &\times  e\left(-\epsilon_2\eta_1\frac{n_{0,0}'n_0t_0tD_0'\overline{\frac{Y_6}{\nu_0's_0D_{0,1}}k g_0g} n}{D_{1,0}f_0'}\right) S\left(m,\nu' n_{1,0}'s_{1,0}s\overline{D_0'h_2}\frac{Y_6\overline{\alpha}}{D_{0,1}};\nu_0'\nu' n_{0,0}'n_{1,0}'r_0s_0^2s_{1,0}^2st_0D_{2,0}'D_2'f_0'g_0\right)
\end{align*}
with
$$ Y_6 = D_{2,0}'D_2'+\epsilon_1n_0rstD_{1,0}gh_1^2h_2k\alpha.$$
To simplify the character sum, note that \begin{align*}    (n_{0,0}'n_0t_0tD_0',D_{1,0}D_{2,0}'D_2'f_0'g_0g)=(D_{2,0}',n_{0,0}')(D_{1,0}D_{2,0}'f_0',t_0)(D_2',n_0t)
\end{align*}
and \begin{align*}
    (\nu_0's_0D_{0,1},D_{2,0}'D_2'g_0g)=(\nu_0'D_{0,1},g_0).
\end{align*}
We first pull out the divisors $\delta:=(D_2',n_0)$, $\tau:=(t,D_2')$ and $g_{0,0}:=(g_0,\nu_0')$. Write \begin{align*}
    D_2'=\delta\tau D_2, \quad n_0=\delta n_0', \quad t=\tau t', \quad g_0=g_{0,0}g_0'' , \quad  \nu_0'=g_{0,0}\nu_0.
\end{align*}
We now pull out the divisor $\tau_0:=(f_0',t_0)$ and write $$f_0'=\tau_0f_0, \quad  t_0=\tau_0t_0''.$$
Note that we have $(D_{1,0},D_{2,0}')=(n_{0,0}',t_0)=1$. We pull out the divisors $\tau_{1,0}:=(D_{1,0},t_0'')$, $\tau_{2,0}:=(D_{2,0}',t_0'')$ and $\delta_{2,0}:=(D_{2,0}',n_{0,0}')$. Write \begin{align*}
    D_{1,0}=\tau_{1,0}D_{1,0}', \quad D_{2,0}'=\delta_{2,0}\tau_{2,0}D_{2,0}, \quad t_0''=\tau_{1,0}\tau_{2,0}t_0',  \quad  n_{0,0}'=\delta_{2,0}n_{0,0}.
\end{align*}
Finally, pull out the divisor $g_{0,1}:=(D_{0,1},g_0'')$ and write \begin{align*}
    D_{0,1}=g_{0,1}D_{0,1}' , \quad  g_0''=g_{0,1}g_0'.
\end{align*}
Writing \begin{align*}
     Y_7 = Y_6/(\delta\tau)=\delta_{2,0}\tau_{2,0} D_{2,0}D_2+\epsilon_1 \tau_{1,0} n_0'rst'D_{1,0}'gh_1^2h_2k\alpha,
\end{align*}
note that the two exponential factors in $\mathcal{C}_6$ combine into \begin{align*}
    e\left(\epsilon_2\eta_1\frac{n_{0,0}n_0't_0't'D_0'\overline{k}n(\nu_0s_0D_{0,1}'-\delta_{2,0}\tau_{2,0}D_{2,0}D_2g_0'g\overline{\frac{Y_7}{\nu_0s_0D_{0,1}'}g_0'g})}{D_{1,0}'D_{2,0}D_2f_0g_0'g}\right).
\end{align*}
Now we have \begin{align*}
    &\left(\nu_0s_0D_{0,1}'-\overline{\frac{Y_7}{\nu_0s_0D_{0,1}'}g_0'g}\delta_{2,0}\tau_{2,0} D_{2,0}D_2g_0'g,D_{1,0}'D_{2,0}D_2f_0g_0'g\right)\\
    =&\left(\delta_{2,0}\tau_{2,0} D_{2,0}D_2+\epsilon_1 \tau_{1,0} n_0'rst'D_{1,0}'gh_1^2h_2k\alpha-\delta_{2,0}\tau_{2,0} D_{2,0}D_2,D_{1,0}'f_0\right)\left(\nu_0s_0D_{0,1}',D_{2,0}D_2g_0'g\right)=D_{1,0}'.
\end{align*}
Inserting all the above relabeling and coprimality conditions back into $\Sigma_6^{(3)}$
, we can recast $\Sigma_6^{(3)}$ as 
\begin{equation}\label{sigma63new}
\begin{split}
    &\sum_{\epsilon_1,\epsilon_2,\eta_1=\pm1} \sum_{\nu_0}\sum_{g_{0,0}}\mathop{\sum\sum}_{(g_{0,1}g_0',\nu_0)=1}\sum_{r_0}\sum_{s_0}\sum_{(s_{1,0},\nu_0)=1}\sum_{\tau_0}\sum_{(\tau_{1,0},\nu_0)=1}\sum_{(\tau_{2,0},\nu_0)=1}\sum_{t_0'}\sum_{(\delta_{2,0},\nu_0)=1}\sum_{n_{0,0}}\sum_{(f_0,\delta_{2,0}\tau_{1,0}\tau_{2,0}n_{0,0}t_0'g_{0,0}g_{0,1}g_0')=1}\\
    &\times \mathop{\sum\sum\sum\sum}_{\substack{n_{1,0}'D_{0,1}'D_{1,0}'D_{2,0}\mid (\nu_0\delta_{2,0}\tau_0\tau_{1,0}\tau_{2,0}n_{0,0}r_0s_0s_{1,0}t_0'f_0g_{0,0}g_{0,1}g_0')^\infty\\(\delta_{2,0}\tau_{1,0}\tau_{2,0}n_{1,0}'s_{1,0}D_{1,0}'D_{2,0}g_0',D_{0,1}'g_{0,1})=1\\(D_{1,0}',\nu_0\delta_{2,0}\tau_{2,0}n_{0,0}s_0t_0'D_{2,0}g_{0,0}g_{0,1}g_0')=(D_{2,0},\nu_0\tau_{1,0}n_{0,0}s_0t_0'g_{0,0})=1}}\mathop{\sum\sum\sum\sum\sum\sum\sum\sum\sum\sum\sum\sum}_{\substack{(\nu' \delta\tau n_0'rst'D_0'D_2gh_1h_2,\nu_0\delta_{2,0}\tau_0\tau_{1,0}\tau_{2,0}n_{0,0}r_0s_0s_{1,0}t_0'f_0g_{0,0}g_{0,1}g_0')=1\\
    (h_1h_2,\delta\tau D_2g)=(\nu'\delta\tau n_0'rst'D_2gh_1,D_0'h_2)=1\\ (D_2,n_0't')=1 }}\\
    &\times \frac{X^{4/3}\mu(\delta_{2,0}\delta \tau_0\tau_{1,0}\tau_{2,0}\tau n_{0,0}n_0'r_0rs_0s_{1,0}st_0't'g_{0,0}g_{0,1}g_0'gh_1)\mu(\tau_0\tau_{1,0}\tau_{2,0}\tau s_0s_{1,0}st_0't'g_{0,0}g_{0,1}g_0'gh_1h_2)}{\nu_0^7\nu'^3\delta_{2,0}^{10} \delta^5\tau_0^{14}\tau_{1,0}^{10}\tau_{2,0}^{10}\tau^5n_{0,0}^8n_{1,0}'^3n_0'^3r_0^7r^2s_0^{14}s_{1,0}^{10}s^5t_0'^8t'^3D_{0,1}'^4D_{1,0}'^2D_{2,0}^2D_0'^3D_2^2f_0^6g_{0,0}^{14}g_{0,1}^{11}g_0'^7g^3h_1^4h_2^4}\\
    &\times W_0\Big(\frac{\tau_{1,0}D_{1, 0}'g}{\Xi}\Big) 
    B(r_0r,1)\sum_{n} \frac{B(n,\nu_0\nu'\delta_{2,0}\delta\tau_0\tau_{1,0}\tau_{2,0}\tau n_{0,0}n_{1,0}'n_0't_0't'g_{0,0}h_1)}{n}\mathcal{C}_7
    \acute{F}_2^{\epsilon_1,\epsilon_2,\eta_1}(X_1n,X_2n)
    \end{split}
\end{equation}
where $X_1, X_2$ are now updated to \begin{equation}
\begin{split}\label{x1x2x3}
    X_1=&\frac{\nu_0^3\nu'\delta_{2,0}^2\delta\tau_0^4\tau_{1,0}^3\tau_{2,0}^2\tau n_{0,0}^3n_{1,0}'n_0'^2r_0^2rs_0^5s_{1,0}^3s^2t_0'^3t'^2D_{0,1}'^2D_0'^2f_0g_{0,0}^4g_{0,1}^3g_0'h_1^2h_2^2}{D_{2,0}D_2X^{2/3}},    \\
    X_2=&\frac{\nu_0^5\nu'^2\delta_{2,0}^5\delta^2\tau_0^8\tau_{1,0}^5\tau_{2,0}^5\tau^2n_{0,0}^5n_{1,0}'^2n_0'^2r_0^4rs_0^9s_{1,0}^6s^3t_0'^5t'^2D_{0,1}'^3D_0'^3f_0^3g_{0,0}^8g_{0,1}^6g_0'^3h_1^2h_2^3}{X},\\
    \end{split}
\end{equation}
and $\mathcal{C}_7$ is given by 
\begin{align}
 & \sumast_{k\, \big(\text{mod }\nu_0\delta_{2,0}^2\delta\tau_0^2\tau_{1,0}^2\tau_{2,0}^2\tau n_{0,0}r_0s_0^2s_{1,0}t_0'D_{0,1}'D_{1,0}'D_{2,0}D_2f_0g_{0,0}^2g_{0,1}^2g_0'g\big)}\sumast_{\substack{\alpha \, \big(\text{mod }\nu_0\delta_{2,0}\tau_0^2\tau_{1,0}\tau_{2,0}n_{0,0}r_0s_0^2s_{1,0}t_0'D_{0,1}'f_0g_{0,0}^2g_{0,1}^2g_0'\big) \nonumber\\ Y_7 \equiv 0\Mod{\nu_0s_0D_{0,1}'g_{0,0}g_{0,1}}\\\Big(\frac{Y_7}{\nu_0s_0D_{0,1}'g_{0,0}g_{0,1}},\tau_0 f_0\Big)=1}}\\
    &\times   e\left(\epsilon_2\eta_1\frac{n_{0,0}n_0't_0't'D_0'Y\overline{k}n}{D_{2,0}D_2f_0g_0'g}\right) S\left(m,\nu'\delta\tau n_{1,0}'s_{1,0}s\overline{D_0'h_2}\frac{Y_7 \overline{\alpha}}{D_{0,1}'g_{0,1}}; 
    \nu_0\nu'\delta_{2,0}^2\delta\tau_0^2\tau_{1,0}\tau_{2,0}^2\tau n_{0,0}n_{1,0}'r_0s_0^2s_{1,0}^2st_0'D_{2,0}D_2f_0g_{0,0}^2g_{0,1}g_0'\right), \nonumber
\end{align}
with \begin{align*}
    Y_7 = \delta_{2,0}\tau_{2,0} D_{2,0}D_2+\epsilon_1 \tau_{1,0} n_0'rst'D_{1,0}'gh_1^2h_2k\alpha, \quad Y=\frac{1}{D_{1, 0}'}\Big(\nu_0s_0D_{0,1}'-\overline{\frac{Y_7}{\nu_0s_0D_{0,1}'}\delta\tau g_0'g}\delta_{2,0}\delta\tau_{2,0}\tau D_{2,0}D_2g_0'g\Big).
\end{align*}
Here the inverse is taken mod $
D_{1,0}'f_0$. 

\subsection{Step 4: Second Voronoi summation}

We are now ready to perform the final step,  a second application of Voronoi summation (Lemma \ref{lem.Voronoi}) on the $n$-sum in $\Sigma_6^{(3)}$. 
Write \begin{equation}\label{mu}
    M=D_{2,0}D_2f_0g_0'g \quad \text{ and } \quad u= \nu_0\nu'\delta_{2,0}\delta\tau_0\tau_{1,0}\tau_{2,0}\tau n_{0,0}n_{1,0}'n_0't_0't'g_{0,0}h_1.
\end{equation}
At the cost of a very small error, we add a factor $$W_0\Big(\frac{n'_0rt'h_1}{\Xi'}\Big)W_0\Big(\frac{1}{\Xi'n}\Big)$$ with $$\Xi' = X^{10^{10}},$$ which will be used to ensure absolute convergence.  As before we continue to call this slightly modified quantity $\Sigma_6^{(3)}$.

Singling out the $n$-sum, we have \begin{align}\label{vor2}
\begin{split}
    &\sum_{n} \frac{B\left(n,u\right)}{n}e\left(\epsilon_2\eta_1\frac{n_{0,0}n_0't_0't'D_0'Y\overline{k}n}{M}\right)\acute{F}_2^{\epsilon_1,\epsilon_2,\eta_1}\left(X_1n,X_2n\right) W_0\Big( \frac{1}{\Xi' n}\Big) 
    \\
    =&MX_2\sum_{\eta_2=\pm1}\sum_{n_1\mid Mu}\sum_n\frac{B(n_1,n)}{n_1n}S\Bigg(\epsilon_2\eta_1u\overline{n_{0,0}n_0't_0't'D_0'Y}k,\eta_2 n;\frac{Mu}{n_1}\Bigg)F_3^{\epsilon_1,\epsilon_2,\eta_1,\eta_2}\Bigg(\frac{n_1^2n}{M^3uX_2},\frac{X_1}{X_2},X_2 \Bigg),
\end{split}
\end{align}
where \begin{equation}\label{deff3}
\begin{split}
    F_3^{\epsilon_1,\epsilon_2,\eta_1,\eta_2}&(a,b,d)=\int_{(0)}\int_0^\infty\int_0^\infty\int_{(0)}\int_0^\infty a^{-s_2}x^{-s_1-1}y^{3s_1+1}z^{-s_1-s_2-2}\mathcal{G}_{\mu_0}^{\eta_1}(s_1+1)\mathcal{G}_{-\mu_0}^{\eta_2}(s_2+1)\\
    &\times  V(x)\Big(1-U\Big( {\frac{  X^{2/3}y }{\Delta^{\ast}}}\Big)\Big)
    W_0\Big(\frac{d}{\Xi' z}\Big)
   \Phi_{\omega_6}\left(-\epsilon_2\frac{x}{by^2},-\epsilon_1 b^2my\right)e\left(-\epsilon_2\eta_1\frac{bz}{y}\right)\mathrm{d} x \frac{\mathrm{d}s_1}{2\pi i} \mathrm{d} y\, \mathrm{d} z\frac{\mathrm{d}s_2}{2\pi i} ,
   \end{split}
\end{equation}
with $\mathcal{G}_\mu^\pm$ as defined in Lemma \ref{lem.Voronoi}, and   the inverse in the Kloosterman sum  is modulo $M = D_{2,0}D_2f_0g_0'g$.  Recalling \eqref{x1x2x3} and \eqref{mu}, the arguments   in \eqref{vor2} of $F_3$ are given by
\begin{equation}\label{x1x2x3prime}
    \begin{split}
        X_1' &= \frac{n_1^2n}{M^3uX_2} = \frac{n_1^2nX}{\nu_0^6\nu'^3\delta_{2,0}^6\delta^3\tau_0^9\tau_{1,0}^6\tau_{2,0}^6\tau^3 n_{0,0}^6n_{1,0}'^3n_0'^3r_0^4rs_0^9s_{1,0}^6s^3t_0'^6t'^3D_{0,1}'^3D_{2,0}^3D_0'^3D_2^3f_0^6g_{0,0}^9g_{0,1}^6g_0'^6g^3h_1^3h_2^3},\\
    X_2' & =  \frac{X_1}{X_2}= \frac{X^{1/3}}{\nu_0^2\nu'\delta_{2,0}^3\delta\tau_0^4\tau_{1,0}^2\tau_{2,0}^3\tau n_{0,0}^2n_{1,0}'r_0^2s_0^4s_{1,0}^3st_0'^2 D_{0,1}'D_{2,0}D_0'D_2f_0^2g_{0,0}^4g_{0,1}^3g_0'^2h_2},\\
    X_3'& = X_2   =\frac{\nu_0^5\nu'^2\delta_{2,0}^5\delta^2\tau_0^8\tau_{1,0}^5\tau_{2,0}^5\tau^2n_{0,0}^5n_{1,0}'^2n_0'^2r_0^4rs_0^9s_{1,0}^6s^3t_0'^5t'^2D_{0,1}'^3D_0'^3f_0^3g_{0,0}^8g_{0,1}^6g_0'^3h_1^2h_2^3}{X}.  
    \end{split}
\end{equation}

We proceed to analyze $F_3$ as a conditionally convergent integral in the same way as $F_2$, and to this end we again insert first  smooth partitions of unity to localize $y \asymp Y \gg \Delta^{\ast}/X^{2/3}$ and $z \asymp Z$. We call the truncated 5-fold integral $ F_{3, Y, Z}^{\epsilon_1,\epsilon_2,\eta_1,\eta_2}(a,b,d)$.

\begin{lemma}\label{lemf3} We have
\begin{equation}\label{f3bound}
    \begin{split}
F_{3, Y, Z}^{\epsilon_1,\epsilon_2,\eta_1,\eta_2}&(a,b, d) \preccurlyeq  \frac{Y^2}{Z}\min\big(aZ, (aZ)^{2/3}\big) \min\Big( \frac{Z}{Y^3}, \frac{Z^{2/3}}{Y^2}\Big) \Big(1 +  \frac{\Delta^{\ast}}{X^{2/3}Y} + \frac{d}{Z\Xi'} +  { Y + \frac{1}{b} }+ \frac{Z/Y^3}{T_1^3} + \frac{aZ}{(\frac{bZ}{Y} + T_1 )^3} \Big)^{-A} \\
& \times \begin{cases} (Y/b)^{1/4}, & \text{always},\\
((bY^2)^{-1} + b^2 Y)^{-1}, & (bY^2)^{-1/2} \geq X^{\varepsilon}(1 + (b^2Y)^{1/2} + Z^{1/3}/Y).\\
\end{cases}
\end{split}
\end{equation}
where
$$T_1 = \frac{1}{\sqrt{b}Y} +  \min\Big( \frac{1}{\sqrt{Y}}, \frac{bZ}{Y}    + b \sqrt{Y}\Big).$$
\end{lemma}

\begin{proof} The upper bound  $$\frac{Y^2}{Z}\min\big(aZ, (aZ)^{2/3}\big) \min\Big( \frac{Z}{Y^3}, \frac{Z^{2/3}}{Y^2}\Big) \Big(\frac{Y}{b}\Big)^{1/4}
$$
follows literally as in Lemma \ref{lemF1}, using a simple stationary phase argument for the integrals over $s_1$ and $s_2$ and Lemma \ref{Kw6}a). For later purposes, we recall that the $s_1$-integral yields a phase $e(\pm 3 z^{1/3} x^{1/3}/y)$. 
The support conditions $   X^{2/3}Y   \succcurlyeq \Delta^{\ast}$, $d /(Z\Xi') \preccurlyeq 1,  { Y \preccurlyeq 1\preccurlyeq b}$ are clear.  
\begin{itemize}
\item Integration by parts in $x$ together with Lemma \ref{lem.Phi6Truncation} gives $\Im s_1 \preccurlyeq  \frac{1}{\sqrt{b}Y} + \frac{1}{\sqrt{Y}} $;
\item Integration by parts in $y$ together with Lemma \ref{lem.Phi6Truncation}  gives $\Im s_1 \preccurlyeq  \frac{1}{\sqrt{b}Y} + b \sqrt{Y}  + \frac{bZ}{Y}$.
\end{itemize}
These two bounds together imply $\Im s_1 \preccurlyeq T_1$.
\begin{itemize}
\item Shifting the $s_1$-contour to the right gives $Z/Y^3 \preccurlyeq T_1^3$, in particular $Z\preccurlyeq Y^3\left(\frac{1}{\sqrt{b}Y} + \frac{1}{\sqrt{Y}}\right)^3\asymp b^{-3/2}+Y^{3/2} {\preccurlyeq 1}$;
\item Integration by parts in $z$ gives $\Im s_2 \preccurlyeq T_1 + bZ/Y$;
\item Shifting the $s_2$-contour to the right gives $aZ \preccurlyeq (T_1 + bZ/Y)^3$.
\end{itemize}
Finally assume that $(bY^2)^{-1/2} \geq X^{\varepsilon} (1+ (b^2Y)^{1/2})$. Under this assumption, we may apply the asymptotic formula in Lemma \ref{Kw6}b). The ``error term'' bounds $\Phi^{\flat}_{\omega_6}$ by $\preccurlyeq  ((bY^2)^{-1} + b^2 Y)^{-1}$, while the main term becomes negligible after integration over $x$, provided that $(bY^2)^{-1/2} \gg X^{\varepsilon}(1+(b^2Y)^{1/2} + Z^{1/3}/Y)$.  
This completes the proof.
\end{proof}

We conclude in particular that we can now remove the dyadic partition and as before view $F_3$ as a conditionally convergent multiple integral that is negligible unless   (note that $T_1 \preccurlyeq b/Y$)
\begin{equation}\label{boundsF}
b \succcurlyeq 1, \quad a \preccurlyeq   \frac{b^3}{ZY^3} \preccurlyeq \frac{\Xi' b^3}{d} \frac{X^2}{(\Delta^{\ast})^3}. 
\end{equation}
For later purposes, we observe that the factor  $W_0(d/(\Xi' z ))$ is important to obtain the second condition in \eqref{boundsF}, but unimportant for the convergence of the un-truncated five-fold integral \eqref{deff3}. Indeed, the previous lemma implies
$$F_{3, Y, Z}^{\epsilon_1,\epsilon_2,\eta_1,\eta_2}(a,b, d) \preccurlyeq  \frac{a}{b^{1/4}} Z^{2/3} Y^{1/4} (1 + b^{-1} + Y  +Z)^{-A},$$
independently of $\Xi'$, so that we can drop the dyadic partition in any case.

 Inserting \eqref{vor2} back into \eqref{sigma63new}, we recast $\Sigma_6^{(3)}$ as  \begin{equation}\label{abscon}
 \begin{split}
   &\sum_{\epsilon_1,\epsilon_2,\eta_1,\eta_2=\pm1} \sum_{\nu_0}\sum_{g_{0,0}}\mathop{\sum\sum}_{(g_{0,1}g_0',\nu_0)=1}\sum_{r_0}\sum_{s_0}\sum_{(s_{1,0},\nu_0)=1}\sum_{\tau_0}\sum_{(\tau_{1,0},\nu_0)=1}\sum_{(\tau_{2,0},\nu_0)=1}\sum_{t_0'}\sum_{(\delta_{2,0},\nu_0)=1}\sum_{n_{0,0}}\sum_{(f_0,\delta_{2,0}\tau_{1,0}\tau_{2,0}n_{0,0}t_0'g_{0,0}g_{0,1}g_0')=1}\\
    &\times \mathop{\sum\sum\sum\sum}_{\substack{n_{1,0}'D_{0,1}'D_{1,0}'D_{2,0}\mid(\nu_0\delta_{2,0}\tau_0\tau_{1,0}\tau_{2,0}n_{0,0}r_0s_0s_{1,0}t_0'f_0g_{0,0}g_{0,1}g_0')^\infty\\(\delta_{2,0}\tau_{1,0}\tau_{2,0}n_{1,0}'s_{1,0}D_{1,0}'D_{2,0}g_0',D_{0,1}'g_{0,1})=1\\(D_{1,0}',\nu_0\delta_{2,0}\tau_{2,0}n_{0,0}s_0t_0'D_{2,0}g_{0,0}g_{0,1}g_0')=(D_{2,0},\nu_0\tau_{1,0}n_{0,0}s_0t_0'g_{0,0})=1}}\mathop{\sum\sum\sum\sum\sum\sum\sum\sum\sum\sum\sum\sum}_{\substack{(\nu' \delta\tau n_0'rst'D_0'D_2gh_1h_2,\nu_0\delta_{2,0}\tau_0\tau_{1,0}\tau_{2,0}n_{0,0}r_0s_0s_{1,0}t_0'f_0g_{0,0}g_{0,1}g_0')=1\\
    (h_1h_2,\delta\tau D_2g)=(\nu'\delta\tau n_0'rst'D_2gh_1,D_0'h_2)=1\\ (D_2,n_0't')=1 }}\\
    &\times \frac{X^{1/3}\mu (\delta_{2,0}\delta \tau_0\tau_{1,0}\tau_{2,0}\tau n_{0,0}n_0'r_0rs_0s_{1,0}st_0't'g_{0,0}g_{0,1}g_0'gh_1)\mu(\tau_0\tau_{1,0}\tau_{2,0}\tau s_0s_{1,0}st_0't'g_{0,0}g_{0,1}g_0'gh_1h_2)}{\nu_0^2\nu'\delta_{2,0}^{5} \delta^3\tau_0^{6}\tau_{1,0}^{5}\tau_{2,0}^{5}\tau^3n_{0,0}^3n_{1,0}'n_0'r_0^{3}rs_0^{5}s_{1,0}^{4}s^2t_0'^{3}t'D_{0,1}'D_{1,0}'^2D_{2,0}D_2f_0^2g_{0,0}^{6}g_{0,1}^{5}g_0'^3g^2h_1^2h_2} W_0\Big(\frac{\tau_{1,0}D_{1,0}'g}{\Xi}\Big)\\
    &\times B(r_0r,1)\sum_{n_1\mid\nu_0\nu'\delta_{2,0}\delta\tau_0\tau_{1,0}\tau_{2,0}\tau n_{0,0}n_{1,0}'n_0't_0't'D_{2,0}D_2f_0g_{0,0}g_0'gh_1}\sum_n\frac{B(n_1,n)}{n_1n}  W_0\Big( \frac{n_0'rt'h_1}{\Xi'}\Big)\mathcal{C}_8
    F_3^{\epsilon_1,\epsilon_2,\eta_1,\eta_2}(X_1', X_2', X_3'),
    \end{split}
\end{equation}
with $(X_1', X_2', X_3')$ as in \eqref{x1x2x3prime} and $\mathcal{C}_8$ 
given by \begin{align*}
    &\sumast_{k\Mod{\nu_0\delta_{2,0}^2\delta\tau_0^2\tau_{1,0}^2\tau_{2,0}^2\tau n_{0,0}r_0s_0^2s_{1,0}t_0'D_{0,1}'D_{1,0}'D_{2,0}D_2f_0g_{0,0}^2g_{0,1}^2g_0'g}}\sumast_{\substack{\alpha\Mod{\nu_0\delta_{2,0}\tau_0^2\tau_{1,0}\tau_{2,0}n_{0,0}r_0s_0^2s_{1,0}t_0'D_{0,1}'f_0g_{0,0}^2g_{0,1}^2g_0'}\\ Y_7 \equiv 0\Mod{\nu_0s_0D_{0,1}'g_{0,0}g_{0,1}}\\\big(\frac{Y_7}{\nu_0s_0D_{0,1}'g_{0,0}g_{0,1}},\tau_0 f_0\big)=1}}\\
    &\times S\left(m,\nu'\delta\tau n_{1,0}'s_{1,0}s\overline{D_0'h_2}\frac{Y_7\overline{\alpha}}{D_{0,1}'g_{0,1}}; 
   \nu_0\nu'\delta_{2,0}^2\delta\tau_0^2\tau_{1,0}\tau_{2,0}^2\tau n_{0,0}n_{1,0}'r_0s_0^2s_{1,0}^2st_0'D_{2,0}D_2f_0g_{0,0}^2g_{0,1}g_0'\right)\\
    &\times S\left(\epsilon_2\eta_1\nu_0\nu'\delta_{2,0}\delta\tau_0\tau_{1,0}\tau_{2,0}\tau n_{0,0}n_{1,0}'n_0't_0't'g_{0,0}h_1\overline{n_{0,0}n_0't_0't'D_0'Y}k,\eta_2 n; \frac{\nu_0\nu'\delta_{2,0}\delta\tau_0\tau_{1,0}\tau_{2,0}\tau n_{0,0}n_{1,0}'n_0't_0't'D_{2,0}D_2f_0g_{0,0}g_0'gh_1}{n_1}\right).
\end{align*}
The first condition in \eqref{boundsF} implies that all variables are effectively bounded by some power of $X$, except perhaps  $n_0', r, t', g$, $h_1, D_{1,0}'$, which are bounded by $\Xi$ or $\Xi'$, and $n_1, n$ which are bounded by the second condition in \eqref{boundsF}.  Hence the entire expression is absolutely convergent.  

\subsection{Simplification I: Elimination of variables by M\"obius inversion}

We have now reached the point with the longest expression, and we proceed to simplify.  
Notice that by M\"obius inversion, we have \begin{align*}
    \sum_x\sum_y F(xy)\mu(x)=F(1)
\end{align*}
for any function $F$. Applying this equality to the pairs $(\delta_{2,0},\tau_{2,0})$, $(\delta,\tau)$, $(n_{0,0},t_0')$ and $(n_0',t')$, we get \begin{equation}\label{one}
    \delta_{2,0}=\delta=\tau_{2,0}=\tau=n_{0,0}=n_0'=t_0'=t'=1.
\end{equation}

To prepare ourselves for the evaluation of the character sum $\mathcal{C}_8$, we write 
\begin{equation}\label{new-var}
   D_{1,1}=(D_{1,0}',\tau_{1,0}^\infty), \quad D_{2,1}=(D_{2,0},g_0'^\infty),  \quad n_{1,1}=(n_{1,0}',  (\tau_{1,0}g_0')^\infty), \quad D_{2,0}=D_{2,0}'D_{2,1}, \quad D_{1,0}'=D_{1,0}D_{1,1}, \quad n_{1,0}'=n_{1,0}n_{1,1}.
   \end{equation}
   Then we have \begin{equation}\label{endvor2}
\begin{split}
    \Sigma_6^{(3)}=&\sum_{\epsilon_1,\epsilon_2,\eta_1,\eta_2=\pm1} \sum_{\nu_0}\sum_{g_{0,0}}\mathop{\sum\sum}_{(g_{0,1}g_0',\nu_0)=1}\sum_{r_0}\sum_{s_0}\sum_{(s_{1,0},\nu_0)=1}\sum_{\tau_0}\sum_{(\tau_{1,0},\nu_0)=1}\sum_{(f_0,\tau_{1,0}g_{0,0}g_{0,1}g_0')=1}\sum_{D_{1,1}\mid \tau_{1,0}^\infty}\sum_{D_{2,1}\mid g_0'^\infty}\sum_{n_{1,1}\mid (\tau_{1,0}g_0')^\infty}\\
    &\times \mathop{\sum\sum\sum\sum}_{\substack{n_{1,0}D_{0,1}'D_{1,0}D_{2,0}'\mid (\nu_0\tau_0\tau_{1,0}r_0s_0s_{1,0}f_0g_{0,0}g_{0,1}g_0')^\infty\\(\tau_{1,0}n_{1,0}s_{1,0}D_{1,0}D_{2,0}'g_0',D_{0,1}'g_{0,1})=1\\(D_{1,0},\nu_0\tau_{1,0}s_0D_{2,0}'g_{0,0}g_{0,1}g_0')=(D_{2,0}',\nu_0\tau_{1,0}s_0g_{0,0}g_0')=(n_{1,0},\tau_{1,0}g_0')=1}}\mathop{\sum\sum\sum\sum\sum\sum\sum\sum}_{\substack{(\nu'  rsD_0'D_2gh_1h_2,\nu_0\tau_0\tau_{1,0}r_0s_0s_{1,0}f_0g_{0,0}g_{0,1}g_0')=1\\
    (h_1h_2,D_2g)=(\nu'rsD_2gh_1,D_0'h_2)=1}}\\
    &\times \frac{X^{1/3}\mu(\tau_0\tau_{1,0}r_0rs_0s_{1,0}sg_{0,0}g_{0,1}g_0'gh_1h_2)\mu(\tau_0\tau_{1,0} s_0s_{1,0}sg_{0,0}g_{0,1}g_0'gh_1)}{\nu_0^2\nu' \tau_0^{6}\tau_{1,0}^{5}n_{1,0}n_{1,1}r_0^{3}rs_0^{5}s_{1,0}^{4}s^2D_{0,1}'D_{1,0}^2D_{1,1}^2D_{2,0}'D_{2,1}D_2f_0^2g_{0,0}^{6}g_{0,1}^{5}g_0'^3g^2h_1^2h_2}W_0\Big(\frac{\tau_{1,0}D_{1,0}D_{1,1}g}{\Xi}\Big)\\
    &\times B(r_0r,1)\sum_{n_1\mid \nu_0\nu'\tau_0\tau_{1,0}n_{1,0}n_{1,1}D_{2,0}'D_{2,1}D_2f_0g_{0,0}g_0'gh_1}\sum_n\frac{B(n_1,n)}{n_1n} W_0\Big( \frac{rh_1}{\Xi'}\Big)\mathcal{C}_8F_3^{\epsilon_1,\epsilon_2,\eta_1,\eta_2} (X_1', X_2', X_3')\\
    \end{split}
\end{equation}
where $X_1', X_2', X_3'$ undergo the simplification \eqref{one} and  now \begin{equation}\label{c8simpl}
\begin{split}
  \mathcal{C}_8 =  &\sumast_{k\Mod{\nu_0\tau_0^2\tau_{1,0}^2r_0s_0^2s_{1,0}D_{0,1}'D_{1,0}D_{1,1}D_{2,0}'D_{2,1}D_2f_0g_{0,0}^2g_{0,1}^2g_0'g}}\sumast_{\substack{\alpha\Mod{\nu_0\tau_0^2\tau_{1,0}r_0s_0^2s_{1,0}D_{0,1}'f_0g_{0,0}^2g_{0,1}^2g_0'}\\ Y_8 \equiv 0\Mod{\nu_0s_0D_{0,1}'g_{0,0}g_{0,1}}\\\left(\frac{Y_8}{\nu_0s_0D_{0,1}'g_{0,0}g_{0,1}},\tau_0 f_0\right)=1}}\\
    &\times S\Bigg(m,\nu' n_{1,0}n_{1,1}s_{1,0}s\overline{D_0'h_2}\frac{Y_8 \overline{\alpha}}{D_{0,1}'g_{0,1}};
\nu_0\nu'\tau_0^2\tau_{1,0}n_{1,0}n_{1,1}r_0s_0^2s_{1,0}^2sD_{2,0}'D_{2,1}D_2f_0g_{0,0}^2g_{0,1}g_0'\Bigg)\\
    &\times S\left(\epsilon_2\eta_1\nu_0\nu'\tau_0\tau_{1,0}n_{1,0}n_{1,1}g_{0,0}h_1\overline{D_0'Y}k,\eta_2 n;\frac{\nu_0\nu'\tau_0\tau_{1,0}n_{1,0}n_{1,1}D_{2,0}'D_{2,1}D_2f_0g_{0,0}g_0'gh_1}{n_1}\right)
    \end{split}
\end{equation}
with 
\begin{equation}\label{y8}
    Y_8 = D_{2,0}'D_{2,1}D_2+\epsilon_1 \tau_{1,0} rsD_{1,0}D_{1,1}gh_1^2h_2k\alpha, \quad Y= \frac{1}{D_{1,0}D_{1,1}}\Big(\nu_0s_0D_{0,1}'-\overline{\frac{Y_8}{\nu_0s_0D_{0,1}'} g_0'g}D_{2,0}'D_{2,1}D_2g_0'g\Big), 
\end{equation}
the  inverse in $Y$ being defined mod $
D_{1,0}D_{1,1}f_0$. 

\subsection{Simplification II: Evaluating the character sum}
We now evaluate the character sum $\mathcal{C}_8$. The final expression is given in Lemma \ref{finalc8} below.   Write \begin{equation}\label{new-var1}
    \begin{matrix*}[l]
        n_1':=(n_1,\nu'\tau_{1,0}n_{1,1}D_{2,1}D_2g_0'gh_1), & n_2'=n_1/n_1',\\
        K_1:= \tau_{1,0}^2D_{1,1}D_{2,1}D_2g_0'g, & K_2:=\nu_0\tau_0^2r_0s_0^2s_{1,0}D_{0,1}'D_{1,0}D_{2,0}'f_0g_{0,0}^2g_{0,1}^2,\\
        \mathcal{A}_1:=\tau_{1,0}g_0', & \mathcal{A}_2:=\nu_0\tau_0^2r_0s_0^2s_{1,0}D_{0,1}'f_0g_{0,0}^2g_{0,1}^2,\\
        A_1:=\nu'\tau_{1,0} n_{1,1}sD_{2,1}D_2g_0', & A_2:=\nu_0\tau_0^2 n_{1,0}r_0s_0^2s_{1,0}^2D_{2,0}'f_0g_{0,0}^2g_{0,1},\\
        B_1:=\nu'\tau_{1,0}n_{1,1}D_{2,1}D_2g_0'gh_1/n_1', & B_2:=\nu_0\tau_0n_{1,0}D_{2,0}'f_0g_{0,0}/n_2'.
    \end{matrix*}
\end{equation}
Then $(n_1',n_2')=(K_1,K_2)=(\mathcal{A}_1,\mathcal{A}_2)=(A_1,A_2)=(B_1,B_2)=1$. Splitting $\mathcal{C}_8$ 
into their respective moduli as above, we have 
    $\mathcal{C}_8 = \mathcal{C}_{8, 1} \mathcal{C}_{8, 2}$ 
with \begin{align*}
    \mathcal{C}_{8,1}
    =&\sumast_{\beta_1\Mod{A_1}}\sumast_{\gamma_1\Mod{B_1}}e\left(\frac{\beta_1 \overline{A_2}m}{A_1}+\eta_2\frac{\overline{B_2\gamma_1}n}{B_1}\right)\sumast_{\alpha\Mod{\tau_{1,0}g_0'}}e\left(\frac{\alpha}{\tau_{1,0}g_0'}\right)\sumast_{k\Mod{\tau_{1,0}^2D_{1,1}D_{2,1}D_2g_0'g}}\\
    & e\left(\epsilon_1\frac{ rsD_{1,0}D_{1,1}gh_1^2k\overline{\nu_0\tau_0^2r_0s_0^2s_{1,0}D_{0,1}'D_{2,0}'D_0'f_0g_{0,0}^2g_{0,1}^2\beta_1}}{ D_{2,1}D_2g_0'}\right)  e\left(\epsilon_2\eta_1\frac{n_1'n_2'D_{1,0}D_{1,1}\overline{\nu_0s_0D_{0,1}'D_{2,0}'D_0'f_0}\gamma_1 k}{D_{2,1}D_2g_0'g}\right)
\end{align*}
and 
\begin{align*}
    \mathcal{C}_{8,2}
    =&\sumast_{k\Mod{K_2}}\sumast_{\substack{\alpha\Mod{\nu_0\tau_0^2r_0s_0^2s_{1,0}D_{0,1}'f_0g_{0,0}^2g_{0,1}^2}\\ Y_8 \equiv 0\Mod{\nu_0s_0D_{0,1}'g_{0,0}g_{0,1}}\\\left(\frac{Y_8}{\nu_0s_0D_{0,1}'g_{0,0}g_{0,1}},\tau_0 f_0\right)=1}}\sumast_{\beta_2\Mod{A_2}}\sumast_{\gamma_2\Mod{B_2}}\\
    &\times e\left(\frac{\beta_2 \overline{A_1}m}{A_2}+\eta_2\frac{\overline{B_1\gamma_2}n}{B_2}+\frac{\frac{Y_8 \overline{\alpha}}{\nu_0s_0D_{0,1}'g_{0,0}g_{0,1}}\overline{\tau_{1,0}D_{2,1}D_0'D_2g_0'h_2\beta_2}}{\tau_0^2r_0s_0s_{1,0}D_{2,0}'f_0g_{0,0}g_{0,1}}+\epsilon_2\eta_1\frac{n_1'n_2'\overline{D_{2,1}D_0'D_2g_0'gY}\gamma_2 k}{D_{2,0}'f_0}\right).
\end{align*}
Here we used 
    $Y\equiv \nu_0s_0D_{0,1}'\overline{D_{1,0}D_{1,1}}\Mod{D_{2,1}D_2g_0'g}.$

We expect the generic case to lie in the first set of moduli. However, we will first analyse the second set of moduli, for the reason that we will observe some simplification to the whole structure by computing $\mathcal{C}_{8,2}$.

\begin{lemma} We have
\begin{displaymath}
\begin{split}
    \mathcal{C}_{8,2}=&\delta\big(\tau_0=s_0=g_{0,0}=g_{0,1}=(r_0,\nu_0D_{0,1}')=(r_0s_{1,0},D_{1,0},f_0)=(n_2'r_0s_{1,0},D_{2,0}')=1\big)\nu_0(r_0s_{1,0},D_{1,0})D_{0,1}'D_{1,0}D_{2,0}'f_0\nonumber\\
    &\times \sum_{\substack{c_0\mid f_0\\(c_0,\nu_0D_{0,1}'D_{1,0})=1}}\sum_{\substack{k_0\mid \nu_0r_0s_{1,0}D_{0,1}'f_0\\ (k_0,\frac{c_0r_0s_{1,0}D_{2,0}'}{(r_0s_{1,0},D_{1,0})})=1}}\frac{\mu(c_0r_0s_{1,0})\mu(k_0)}{c_0k_0}\sum_{\substack{\omega_0\mid f_0\\D_{2,0}'\omega_0\mid (A_2,B_2)\\(\omega_0,n_2'r_0s_{1,0})=1}}\omega_0\mu\left(\frac{r_0s_{1,0}f_0}{(r_0s_{1,0},D_{1,0})\omega_0}\right)\\
    &\times \mathop{\sumast_{\beta_2\Mod{A_2}}\sumast_{\gamma_2\Mod{B_2}}}_{rsg^2h_1^2\overline{\beta_2}\equiv r_0s_{1,0}n_1'n_2'\overline{\gamma_2}\Mod{D_{2,0}'\omega_0}} e\left(\frac{\beta_2 \overline{A_1}m}{A_2}-\epsilon_1\epsilon_2\eta_1\eta_2\frac{\gamma_2\overline{B_1}n}{B_2}\right).
    \end{split}
    \end{displaymath}
\end{lemma}

\begin{proof}  Recall   the definition of $Y$ in \eqref{y8}. 
We claim that 
\begin{align*}
    e\left(\epsilon_2\eta_1\frac{n_1'n_2'
    \overline{D_{2,1}D_0'D_2g_0'gY}\gamma_2 k}{D_{2,0}'f_0}\right)=e\left(\epsilon_1\epsilon_2\eta_1\frac{n_1'n_2' \frac{Y_8 \overline{\alpha}}{\nu_0s_0D_{0,1}'}\overline{\tau_{1,0}rsD_{2,1}D_0'D_2g_0'g^2h_1^2h_2}\gamma_2}{D_{2,0}'f_0}\right).
\end{align*}
Indeed, since 
  $$\Big(\frac{Y_8}{\nu_0s_0D_{0,1}'} g_0'g, D_{1, 0}D_{1, 1}\Big)= (\nu_0s_0D_{0,1}',D_{1,0}D_{1,1}D_{2,0}')=(D_{1,0}D_{1,1},D_{2,0}') = 1,$$ we have
\begin{displaymath}
\begin{split}
\overline{Y} & =   \overline{\frac{1}{D_{1,0}D_{1,1}}\left(\nu_0s_0D_{0,1}'-\overline{\frac{Y_8}{\nu_0s_0D_{0,1}'} g_0'g}D_{2,0}'D_{2,1}D_2g_0'g\right)} =\frac{Y_8}{\nu_0s_0D_{0,1}'} g_0'g \cdot  \overline{\frac{ Y_8  g_0'g   - D_{2,0}'D_{2,1}D_2g_0'g }{D_{1,0}D_{1,1}} } \\
& =\frac{Y_8}{\nu_0s_0D_{0,1}'} g_0'g \cdot \overline{ \epsilon_1 \tau_{1,0}rsg_0'g^2h_1^2h_2k \alpha } = \epsilon_1 Y_8 \frac{\overline{\alpha} }{ \nu_0s_0D_{0,1}'} \overline{ \tau_{1,0}rsgh_1^2h_2k }
\end{split}
    \end{displaymath}
    in the exponential on the left hand side. 
 
Changing $\alpha$ to $D_{2,1}D_2\overline{\alpha}$ and $k$ to $\epsilon_1 \overline{\tau_{1,0}rsD_{1,1}gh_1^2h_2}k$, we recast $\mathcal{C}_{8,2}$ as 
\begin{align*}
   \sumast_{k\Mod{K_2}}\sumast_{\substack{\alpha\Mod{\nu_0\tau_0^2r_0s_0^2s_{1,0}D_{0,1}'f_0g_{0,0}^2g_{0,1}^2}\\ D_{2,0}'\alpha\equiv -D_{1,0}k\Mod{\nu_0s_0D_{0,1}'g_{0,0}g_{0,1}}\\ \left(\frac{D_{2,0}'\alpha+D_{1,0}k}{\nu_0s_0D_{0,1}'g_{0,0}g_{0,1}},\tau_0f_0\right)=1}}&\sumast_{\beta_2\Mod{A_2}}\sumast_{\gamma_2\Mod{B_2}} e\left(\frac{\beta_2 \overline{A_1}m}{A_2}+\eta_2\frac{\overline{B_1\gamma_2}n}{B_2}+\frac{\frac{D_{2,0}'\alpha+D_{1,0}k}{\nu_0s_0D_{0,1}'g_{0,0}g_{0,1}}\overline{\tau_{1,0}D_{2,1}D_0'D_2g_0'h_2\beta_2}}{\tau_0^2r_0s_0s_{1,0}D_{2,0}'f_0g_{0,0}g_{0,1}}\right)\\
    &\times e\left(\epsilon_1\epsilon_2\eta_1\frac{n_1'n_2'\frac{D_{2,0}'\alpha+D_{1,0}k}{\nu_0s_0D_{0,1}'}\overline{\tau_{1,0}rsD_{2,1}D_0'D_2g_0'g^2h_1^2h_2}\gamma_2 }{D_{2,0}'f_0}\right).
\end{align*}
Removing the coprimality condition $(\alpha, \tau_0  r_0s_{1, 0}f_0)=1$  
by M\"obius inversion, the $\alpha$-sum is equal to \begin{align*}
    &\sum_{\substack{c_0\mid \tau_0  r_0s_{1, 0}f_0\\(c_0,\nu_0s_0D_{0,1}'g_{0,0}g_{0,1})=1}}\mu(c_0)\sum_{\substack{\alpha\Mod{\nu_0\tau_0^2r_0s_0^2s_{1,0}D_{0,1}'f_0g_{0,0}^2g_{0,1}^2}\\c_0\mid \alpha\\ D_{2,0}'\alpha\equiv -D_{1,0} k\Mod{\nu_0s_0D_{0,1}'g_{0,0}g_{0,1}}\\ \big(\frac{D_{2,0}'\alpha+D_{1,0}k}{\nu_0s_0D_{0,1}'g_{0,0}g_{0,1}},\tau_0f_0\big)=1}}  e\left(\frac{\frac{D_{2,0}'\alpha+D_{1,0}k}{\nu_0s_0D_{0,1}'g_{0,0}g_{0,1}}\overline{\tau_{1,0}D_{2,1}D_0'D_2g_0'h_2}Z}{\tau_0^2r_0s_0s_{1,0}D_{2,0}'f_0g_{0,0}g_{0,1}}\right)\\
    =&\sum_{\substack{c_0\mid \tau_0  r_0s_{1, 0}f_0\\(c_0,\nu_0s_0D_{0,1}')=1}}\mu(c_0)\sum_{\substack{\alpha\Mod{\nu_0\tau_0^2r_0s_0^2s_{1,0}D_{0,1}'D_{2,0}'f_0g_{0,0}^2g_{0,1}^2}\\c_0D_{2,0}'\mid \alpha\\ \alpha\equiv -D_{1,0} k\Mod{\nu_0s_0D_{0,1}'g_{0,0}g_{0,1}}\\ \big(\frac{\alpha+D_{1,0}k}{\nu_0s_0D_{0,1}'},\tau_0f_0\big)=1}} e\left(\frac{\frac{\alpha+D_{1,0}k}{\nu_0s_0D_{0,1}'g_{0,0}g_{0,1}}\overline{\tau_{1,0}D_{2,1}D_0'D_2g_0'h_2} Z }{\tau_0^2r_0s_0s_{1,0}D_{2,0}'f_0g_{0,0}g_{0,1}}\right)
\end{align*}
where $Z$ is shorthand for
$$Z = \overline{\beta_2}+\epsilon_1\epsilon_2\eta_1n_1'n_2'\tau_0^2r_0s_0s_{1,0}g_{0,0}^2g_{0,1}^2\overline{rsg^2h_1^2}\gamma_2.$$
Introducing  the new variable $\theta=\frac{\alpha+D_{1,0}k}{\nu_0s_0D_{0,1}'g_{0,0}g_{0,1}}$, this is equal to \begin{align*}
    &\sum_{\substack{c_0\mid \tau_0  r_0s_{1, 0}f_0\\(c_0,\nu_0s_0D_{0,1}')=1}}\mu(c_0)\sum_{\substack{\theta\Mod{\tau_0^2r_0s_0s_{1,0}D_{2,0}'f_0g_{0,0}g_{0,1}}\\\nu_0s_0D_{0,1}'g_{0,0}g_{0,1}\theta\equiv D_{1,0}k\Mod{c_0D_{2,0}'}\\ \left(\theta,\tau_0f_0\right)=1}}e\left(\frac{\theta\overline{\tau_{1,0}D_{2,1}D_0'D_2g_0'h_2}Z}{\tau_0^2r_0s_0s_{1,0}D_{2,0}'f_0g_{0,0}g_{0,1}}\right).
\end{align*}
Since $c_0\mid \tau_0  r_0s_{1, 0}f_0$, $(D_{2,0}',D_{1,0})=1$ and $(\theta,f_0)=(c_0,\nu_0s_0D_{0,1}'g_{0,0}g_{0,1})=1$, the congruence condition implies $(c_0,D_{1,0})\mid \theta$ and $(\theta,D_{2,0}')=1$. Substituting back into $\mathcal{C}_{8,2}$, changing the variable $\theta\mapsto (c_0,D_{1,0})\theta$ and swapping the order of the $\theta-$ and $k$-sum, we recast $\mathcal{C}_{8,2}$ as 
\begin{align*}
   &\sum_{\substack{c_0\mid \tau_0  r_0s_{1, 0}f_0\\(c_0,\nu_0s_0D_{0,1}')=1\\(c_0,D_{1,0},\tau_0f_0)=1}}\mu(c_0)\sumast_{\beta_2\Mod{A_2}}\sumast_{\gamma_2\Mod{B_2}}e\left(\frac{\beta_2 \overline{A_1}m}{A_2}+\eta_2\frac{\overline{B_1\gamma_2}n}{B_2}\right) \sum_{\substack{\theta\Mod{\tau_0^2r_0s_0s_{1,0}D_{2,0}'f_0g_{0,0}g_{0,1}/(c_0,D_{1,0})}\\(\theta,\tau_0c_0D_{2,0}'f_0/(c_0,D_{1,0}))=1}}\\
   & \times \sumast_{\substack{k\Mod{\nu_0\tau_0^2r_0s_0^2s_{1,0}D_{0,1}'D_{1,0}D_{2,0}'f_0g_{0,0}^2g_{0,1}^2}\\\nu_0s_0D_{0,1}'g_{0,0}g_{0,1}\theta\equiv D_{1,0}k/(c_0,D_{1,0})\Mod{c_0D_{2,0}'/(c_0,D_{1,0})}}}  e\left(\frac{\theta\overline{\tau_{1,0}D_{2,1}D_0'D_2g_0'h_2}(c_0,D_{1,0})Z}{\tau_0^2r_0s_0s_{1,0}D_{2,0}'f_0g_{0,0}g_{0,1}}\right).
\end{align*}
Since $$\Big(\frac{\nu_0s_0D_{0,1}'D_{1,0}g_{0,0}g_{0,1}\theta}{(c_0,D_{1,0})},\frac{c_0D_{2,0}'}{(c_0,D_{1,0})}\Big)=1,$$ the $k$-sum is \begin{align*}
    \sumast_{\substack{k\Mod{\nu_0\tau_0^2r_0s_0^2s_{1,0}D_{0,1}'D_{1,0}D_{2,0}'f_0g_{0,0}^2g_{0,1}^2}\\\nu_0s_0D_{0,1}'g_{0,0}g_{0,1}\theta\equiv D_{1,0}k/(c_0,D_{1,0})\Mod{c_0D_{2,0}'/(c_0,D_{1,0})}}}1=& \sum_{\substack{k_0\mid \nu_0\tau_0r_0s_0s_{1,0}D_{0,1}'D_{1,0}f_0g_{0,0}g_{0,1}\\(k_0,c_0D_{2,0}'/(c_0,D_{1,0}))=1}}\mu(k_0)\sum_{\substack{k\Mod{\nu_0\tau_0^2r_0s_0^2s_{1,0}D_{0,1}'D_{1,0}D_{2,0}'f_0g_{0,0}^2g_{0,1}^2/k_0}\\ k\equiv \overline{\frac{D_{1,0}}{(c_0,D_{1,0})}}\nu_0s_0D_{0,1}'g_{0,0}g_{0,1}\theta\Mod{\frac{c_0D_{2,0}'}{(c_0,D_{1,0})} }}}1\\
    =&\frac{\nu_0\tau_0^2r_0s_0^2s_{1,0}(c_0,D_{1,0})D_{0,1}'D_{1,0}f_0g_{0,0}^2g_{0,1}^2}{c_0}\sum_{\substack{k_0\mid \nu_0\tau_0r_0s_0s_{1,0}D_{0,1}'f_0g_{0,0}g_{0,1}\\(k_0,c_0D_{2,0}'/(c_0,D_{1,0}))=1}}\frac{\mu(k_0)}{k_0}.
\end{align*}
This yields \begin{align*}
    \mathcal{C}_{8,2}&=\nu_0\tau_0^2r_0s_0^2s_{1,0}D_{0,1}'D_{1,0}f_0g_{0,0}^2g_{0,1}^2\sum_{\substack{c_0\mid \tau_0r_0s_{1,0}f_0\\(c_0,\nu_0s_0D_{0,1}')=1\\(c_0,D_{1,0},\tau_0f_0)=1}}\frac{\mu(c_0)(c_0,D_{1,0})}{c_0}\sum_{\substack{k_0\mid \nu_0\tau_0r_0s_0s_{1,0}D_{0,1}'f_0g_{0,0}g_{0,1}\\(k_0,c_0D_{2,0}'/(c_0,D_{1,0}))=1}}\frac{\mu(k_0)}{k_0}\\
    &\times \sumast_{\beta_2\Mod{A_2}}\sumast_{\gamma_2\Mod{B_2}}e\left(\frac{\beta_2 \overline{A_1}m}{A_2}+\eta_2\frac{\overline{B_1\gamma_2}n}{B_2}\right)\sum_{\substack{\theta\Mod{\tau_0^2r_0s_0s_{1,0}D_{2,0}'f_0g_{0,0}g_{0,1}/(c_0,D_{1,0})}\\(\theta,\tau_0c_0D_{2,0}'f_0/(c_0,D_{1,0}))=1}} e\left(\frac{\theta\overline{\tau_{1,0}D_{2,1}D_0'D_2g_0'h_2}(c_0,D_{1,0})Z}{\tau_0^2r_0s_0s_{1,0}D_{2,0}'f_0g_{0,0}g_{0,1}}\right) .
\end{align*}

On the other hand, summing over $\theta$ yields \begin{align*}
    &\sum_{\substack{\theta\Mod{\tau_0^2r_0s_0s_{1,0}D_{2,0}'f_0g_{0,0}g_{0,1}/(c_0,D_{1,0})}\\(\theta,\tau_0c_0D_{2,0}'f_0/(c_0,D_{1,0}))=1}} e\left(\frac{\theta\overline{\tau_{1,0}D_{2,1}D_0'D_2g_0'h_2}Z}{\tau_0^2r_0s_0s_{1,0}D_{2,0}'f_0g_{0,0}g_{0,1}/(c_0,D_{1,0})}\right)\\
    =&\frac{\tau_0^2r_0s_0s_{1,0}D_{2,0}'f_0g_{0,0}g_{0,1}}{(c_0,D_{1,0})}\sum_{\omega_0\mid \tau_0D_{2,0}'f_0\frac{(c_0,r_0s_{1,0})}{(c_0,D_{1,0})}}\frac{\mu(\omega_0)}{\omega_0}\delta\left(\overline{\beta_2}\equiv-G\gamma_2\Mod{\frac{\tau_0^2r_0s_0s_{1,0}D_{2,0}'f_0g_{0,0}g_{0,1}}{\omega_0(c_0,D_{1,0})}}\right),
\end{align*}
where \begin{align*}
    G=\epsilon_1\epsilon_2\eta_1n_1'n_2'\tau_0^2r_0s_0s_{1,0}g_{0,0}^2g_{0,1}^2\overline{rsg^2h_1^2}.
\end{align*}
Since $(\beta_2,\tau_0r_0s_0s_{1,0}g_{0,0}g_{0,1})=1$, the congruence condition evaluates to $0$ unless $$\frac{\tau_0^2r_0s_0s_{1,0}g_{0,0}g_{0,1}}{(c_0,D_{1,0})}\mid \omega_0.$$ Since $\mu$ is supported on squarefree numbers, $c_0\mid \tau_0r_0s_{1,0}f_0$, $(c_0,D_{1,0})\mid r_0s_{1,0}$ and $(g_{0,0}g_{0,1},D_{2,0}'f_0)=1$, the above sum is equal to \begin{align*}
    \delta(\tau_0=g_{0,0}=g_{0,1}=1)\sum_{\frac{r_0s_{1,0}}{(c_0,D_{1,0})}s_0\omega_0\mid\frac{(c_0,r_0s_{1,0})D_{2,0}'f_0}{(c_0,D_{1,0})}}\mu\left(\frac{r_0s_{1,0}}{(c_0,D_{1,0})}s_0\omega_0\right)\frac{D_{2,0}'f_0}{\omega_0}\delta\left(\overline{\beta_2}\equiv-G\gamma_2\Mod{D_{2,0}'f_0/\omega_0}\right).
\end{align*}
Note that the condition of $\omega_0$ implies $r_0s_0s_{1,0}/(c_0,r_0s_{1,0})\mid D_{2,0}'f_0/\omega_0$. At this stage, we again see that this vanishes  unless $r_0s_0s_{1,0}/(c_0,r_0s_{1,0})\mid \omega_0$. Together with $\mu$ being supported on squarefree numbers, we have \begin{align*}
    s_0=1 \quad \text{ and } \quad r_0s_{1,0}\mid c_0
\end{align*}
as well. Changing $c_0$ to $r_0s_{1,0}c_0$, $\gamma_2$ to $-\epsilon_1\epsilon_2\eta_1\overline{\gamma_2}$ and $\omega_0$ to $D_{2,0}'f_0/\omega_0$, we have \begin{align*}
    \mathcal{C}_{8,2}=&\delta(\tau_0=s_0=g_{0,0}=g_{0,1}=(r_0,\nu_0D_{0,1}')=(r_0s_{1,0},D_{1,0},f_0)=1)\nu_0(r_0s_{1,0},D_{1,0})D_{0,1}'D_{1,0}f_0\sum_{\substack{c_0\mid f_0\\(c_0,\nu_0D_{0,1}'D_{1,0})=1}}\sum_{\substack{k_0\mid \nu_0r_0s_{1,0}D_{0,1}'f_0\\ (k_0,\frac{c_0r_0s_{1,0}D_{2,0}'}{(r_0s_{1,0},D_{1,0})})=1}}\\
    &\times \frac{\mu(c_0r_0s_{1,0})\mu(k_0)}{c_0k_0}\sum_{\omega_0\mid D_{2,0}'f_0}\omega_0\mu\left(\frac{r_0s_{1,0}D_{2,0}'f_0}{(r_0s_{1,0},D_{1,0})\omega_0}\right)\mathop{\sumast_{\beta_2\Mod{A_2}}\sumast_{\gamma_2\Mod{B_2}}}_{rsg^2h_1^2\overline{\beta_2}\equiv r_0s_{1,0}n_1'n_2'\overline{\gamma_2}\Mod{\omega_0}} e\left(\frac{\beta_2 \overline{A_1}m}{A_2}-\epsilon_1\epsilon_2\eta_1\eta_2\frac{\gamma_2\overline{B_1}n}{B_2}\right).
\end{align*}
Since $D_{2,0}'\mid (\tau_0r_0s_{1,0}f_0g_{0,1})^\infty=(r_0s_{1,0}f_0)^\infty$ and $(D_{1,0},D_{2,0}')=1$, the fact that $\mu$ is supported on squarefree numbers yields \begin{align}
    \mathcal{C}_{8,2}=&\delta(\tau_0=s_0=g_{0,0}=g_{0,1}=(r_0,\nu_0D_{0,1}')=(r_0s_{1,0},D_{1,0},f_0)=1)\nu_0(r_0s_{1,0},D_{1,0})D_{0,1}'D_{1,0}D_{2,0}'f_0\sum_{\substack{c_0\mid f_0\\(c_0,\nu_0D_{0,1}'D_{1,0})=1}}\sum_{\substack{k_0\mid \nu_0r_0s_{1,0}D_{0,1}'f_0\\ (k_0,\frac{c_0r_0s_{1,0}D_{2,0}'}{(r_0s_{1,0},D_{1,0})})=1}}\nonumber\\
    &\times \frac{\mu(c_0r_0s_{1,0})\mu(k_0)}{c_0k_0}\sum_{\omega_0\mid f_0}\omega_0\mu\left(\frac{r_0s_{1,0}f_0}{(r_0s_{1,0},D_{1,0})\omega_0}\right)\mathop{\sumast_{\beta_2\Mod{A_2}}\sumast_{\gamma_2\Mod{B_2}}}_{rsg^2h_1^2\overline{\beta_2}\equiv r_0s_{1,0}n_1'n_2'\overline{\gamma_2}\Mod{D_{2,0}'\omega_0}} e\left(\frac{\beta_2 \overline{A_1}m}{A_2}-\epsilon_1\epsilon_2\eta_1\eta_2\frac{\gamma_2\overline{B_1}n}{B_2}\right).\nonumber
\end{align}
Finally note that as $(D_{2,0}'\omega_0,rsg^2h_1^2)=1$, and $D_{2,0}'\omega_0\mid (A_2,n_2'B_2)$, the congruence condition implies $D_{2,0}'\omega_0 \mid (A_2,B_2)$. Moreover, $n_2'r_0s_{1,0}\mid A_2$ and $(rsgh_1,A_2)=1$ implies that $(D_{2,0}'\omega_0,n_2'r_0s_{1,0})=1$, which concludes the proof. 
\end{proof}

 \begin{lemma} For  $ \tau_0=s_0=g_{0,0}=g_{0,1}=1$ we have
\begin{displaymath}
    \begin{split}
    \mathcal{C}_{8,1}=&\tau_{1,0}D_{1,1}D_{2,1}\varphi(\tau_{1,0})\mu(\tau_{1,0}g_0')\delta\big((D_{2,1},n_1')=1\big)\sum_{\alpha\mid (D_2g,rsg^2,n_1')}\alpha \\
    &\times \sum_{\substack{\alpha\tilde{\omega}\mid D_2g_0'g\\D_{2,1}\tilde{\omega}\mid (A_1,B_1)\\(\alpha^2\tilde{\omega},n_1'rsg^2)=\alpha^2}}\tilde{\omega}\mu\left(\frac{D_2g_0'g}{\alpha\tilde{\omega}}\right)\mathop{\sumast_{\beta_1\Mod{A_1}}\sumast_{\gamma_1\Mod{B_1}}}_{\frac{rsg^2}{\alpha}h_1^2\gamma_1\equiv r_0s_{1,0}\frac{n_1'}{\alpha}n_2'\beta_1\Mod{D_{2,1}\tilde{\omega}}}e\left(\frac{\beta_1 \overline{A_2}m}{A_1}-\epsilon_1\epsilon_2\eta_1\eta_2\frac{\gamma_1\overline{B_2} n}{B_1}\right).
    \end{split}
\end{displaymath}
 \end{lemma}  

\begin{proof} 
Recall that with the   restriction $ \tau_0=s_0=g_{0,0}=g_{0,1}=1$ we have \begin{align*}
    \mathcal{C}_{8,1}=&\sumast_{\beta_1\Mod{A_1}}\sumast_{\gamma_1\Mod{B_1}}e\left(\frac{\beta_1 \overline{A_2}m}{A_1}+\eta_2\frac{\overline{B_2\gamma_1}n}{B_1}\right)\sumast_{\alpha\Mod{\tau_{1,0}g_0'}}e\left(\frac{\alpha}{\tau_{1,0}g_0'}\right)\\
    &\times \sumast_{k\Mod{\tau_{1,0}^2D_{1,1}D_{2,1}D_2g_0'g}}e\left(\epsilon_1\frac{ rsD_{1,0}D_{1,1}gh_1^2k\overline{\nu_0r_0s_{1,0}D_{0,1}'D_{2,0}'D_0'f_0\beta_1}}{D_{2,1}D_2g_0'}\right)  e\left(\epsilon_2\eta_1\frac{n_1'n_2'D_{1,0}D_{1,1}\overline{\nu_0D_{0,1}'D_{2,0}'D_0'f_0}\gamma_1 k}{D_{2,1}D_2g_0'g}\right).
\end{align*}
Summing over the $\alpha$-sum, we have \begin{align*}
    \sumast_{\alpha\Mod{\tau_{1,0}g_0'}}e\left(\frac{\alpha}{\tau_{1,0}g_0'}\right)=\mu(\tau_{1,0}g_0').
\end{align*}
Summing over the $k$-sum, we have \begin{align*}
    \mathcal{C}_{8,1}&=\mu(\tau_{1,0}g_0')\sum_{\omega\mid \tau_{1,0}^2D_{1,1}D_{2,1}D_2g_0'g}\omega\mu\left(\frac{\tau_{1,0}^2D_{1,1}D_{2,1}D_2g_0'g}{\omega}\right)\sumast_{\beta_1\Mod{A_1}}\sumast_{\gamma_1\Mod{B_1}}e\left(\frac{\beta_1 \overline{A_2}m}{A_1}+\eta_2\frac{\overline{B_2\gamma_1}n}{B_1}\right)\nonumber\\
    &\times \delta\left(\tau_{1,0}^2D_{1,1}^2(rsg^2h_1^2\overline{r_0s_{1,0}\beta_1}+\epsilon_1\epsilon_2\eta_1n_1'n_2'\gamma_1)\equiv 0\Mod{\omega}\right)\nonumber\\
    =&\tau_{1,0}D_{1,1}D_{2,1}\mu(\tau_{1,0}g_0')\sum_{\omega_{1,0}\mid \tau_{1,0}}\omega_{1,0}\mu\left(\frac{\tau_{1,0}}{\omega_{1,0}}\right)\sum_{\omega\mid D_2g_0'g}\omega\mu\left(\frac{D_2g_0'g}{\omega}\right)\sumast_{\beta_1\Mod{A_1}}\sumast_{\gamma_1\Mod{B_1}}e\left(\frac{\beta_1 \overline{A_2}m}{A_1}+\eta_2\frac{\overline{B_2\gamma_1}n}{B_1}\right)\nonumber\\
    &\times \delta\left(rsg^2h_1^2\overline{\beta_1}+\epsilon_1\epsilon_2\eta_1r_0s_{1,0}n_1'n_2'\gamma_1\equiv 0\Mod{D_{2,1}\omega}\right)\nonumber\\
    =&\tau_{1,0}D_{1,1}D_{2,1}\varphi(\tau_{1,0})\mu(\tau_{1,0}g_0')\sum_{\omega\mid D_2g_0'g}\omega\mu\left(\frac{D_2g_0'g}{\omega}\right)\mathop{\sumast_{\beta_1\Mod{A_1}}\sumast_{\gamma_1\Mod{B_1}}}_{rsg^2h_1^2\gamma_1\equiv r_0s_{1,0}n_1'n_2'\beta_1\Mod{D_{2,1}\omega}}e\left(\frac{\beta_1 \overline{A_2}m}{A_1}-\epsilon_1\epsilon_2\eta_1\eta_2\frac{\gamma_1\overline{B_2} n}{B_1}\right).
\end{align*}
Here we used $D_{1,1}\mid\tau_{1,0}^\infty$ and $D_{2,1}\mid g_0'^\infty$. On the other hand, note that $(r_0s_{1,0},D_{2,1}\omega)=(D_{2,1},rsg^2h_1^2)=1$. Let $$\alpha=(rsg^2h_1^2,r_0s_{1,0}n_1'n_2',D_{2,1}\omega)=(rsg^2,n_1',\omega)$$ and write $\omega=\alpha\tilde{\omega}$, then we have $(\tilde{\omega},\frac{rsg^2}{\alpha}h_1^2,\frac{n_1'}{\alpha}n_2')=1$. Since $D_{2,1}\omega\mid(A_1rsg^2h_1^2,B_1n_1'n_2')$, the congruence condition forces $(\frac{rsg^2}{\alpha}h_1^2\frac{n_1'}{\alpha}n_2',D_{2,1}\tilde{\omega})=1$ and $D_{2,1}\tilde{\omega}\mid (A_1,B_1)$. This concludes the proof. 
\end{proof}

As a combination of the previous two lemmata we obtain
\begin{lemma}\label{finalc8} Let 
$$D_{2,0}=D_{2,0}'D_{2,1}, \quad D_{1,0}'=D_{1,0}D_{1,1}, \quad n_{1,0}'=n_{1,0}n_{1,1}, \quad n_1=n_1'n_2'$$
and 
 $$A=\nu_0\nu'\tau_{1,0}n_{1,0}' r_0s_{1,0}^2sD_{2,0}D_2f_0g_0' , \quad B = 
        \frac{\nu_0\nu'\tau_{1,0}n_{1,0}'D_{2,0}D_2f_0g_0'gh_1  }{n_1},$$
with all the variables satisfying divisibility and coprimality conditions as in \eqref{endvor2}.
For $a \mid A$ write 
$\Delta = \frac{D_{2,0}\nu}{(a,D_{2,0}\nu)}.$
Then the character sum $\mathcal{C}_8$ given in \eqref{c8simpl} equals
\begin{align*}
    \mathcal{C}_8 =& \delta(\tau_0=s_0=g_{0,0}=g_{0,1}=(r_0,\nu_0D_{0,1}')=(r_0s_{1,0},D_{1,0}',f_0)=(n_1r_0s_{1,0},D_{2,0})=1) \\ &\times \nu_0\tau_{1,0}(r_0s_{1,0},D_{1,0}')D_{0,1}'D_{1,0}'f_0AB\varphi(\tau_{1,0})\mu(\tau_{1,0}g_0')  \mathcal{C}_9
\end{align*}
where 
\begin{align}\label{C9}
  \mathcal{C}_9 =  &   
    \sum_{\substack{c_0\mid f_0 \\(c_0,\nu_0D_{0,1}'D_{1,0}')=1}}\sum_{\substack{k_0\mid \nu_0r_0s_{1,0}D_{0,1}'f_0\\(k_0,\frac{c_0r_0s_{1,0}D_{2,0}}{(r_0s_{1,0},D_{1,0}')})=1}}\frac{\mu(c_0r_0s_{1,0})\mu(k_0)}{c_0k_0}  \sum_{\alpha\mid (D_2g,rsg^2,n_1)}\alpha\\ \nonumber
    & \times \sum_{\substack{\alpha\nu\mid D_2f_0g_0'g\\n_1\nu\mid\nu_0\nu'\tau_{1,0}n_{1,0}'D_2f_0g_0'(n_1,gh_1)\\(\alpha^2\nu,n_1r_0rs_{1,0}sg^2)=\alpha^2}}\mu\left(\frac{r_0s_{1,0}D_2f_0g_0'g}{(r_0s_{1,0},D_{1,0}')\alpha\nu}\right)\sum_{\substack{b\mid B\\(b,D_{2,0}\nu)=1}}\frac{\mu(b)}{b}\sum_{a\mid A}\frac{\mu(a)}{a}\  \delta\Bigg(n=\frac{B}{bD_{2,0}\nu}n'\Bigg)\\
    & \times \delta\left( m=\frac{A/a}{ (\frac{A}{a},\Delta^\infty )}m'\right)\delta\left(b\frac{rsg^2}{\alpha}h_1^2m'\equiv \epsilon_1\epsilon_2\eta_1\eta_2r_0s_{1,0}\frac{n_1}{\alpha}\frac{a (\frac{A}{a},\Delta^\infty )}{D_{2,0}\nu}n'\, \left(\text{{\rm mod }} \Big(\frac{A}{a},\Delta^\infty\Big)\right)\right). \nonumber
\end{align}
Here the first $\delta$-condition  should be understood as ``$n$ is divisible by $B/(bD_{2, 0} \nu)$ and the codivisor is called $n'$'' and similarly for $m$. 
\end{lemma}

\begin{proof}
We recall the notation \eqref{new-var} and \eqref{new-var1}, in particular
$$D_{2,0}=D_{2,0}'D_{2,1}, \quad D_{1,0}'=D_{1,0}D_{1,1}, \quad n_{1,0}'=n_{1,0}n_{1,1}, \quad n_1=n_1'n_2'.$$
With $\tau_0=s_0=g_{0,0}=g_{0,1}=1$, we now have \begin{align*}
    \begin{matrix*}[l]
        A_1=\nu'\tau_{1,0}n_{1,1} sD_{2,1}D_2g_0', & A_2=\nu_0r_0s_{1,0}^2n_{1,0}D_{2,0}'f_0\\
        B_1=\nu'\tau_{1,0}n_{1,1}D_{2,1}D_2g_0'gh_1/n_1', & B_2=\nu_0n_{1,0}D_{2,0}'f_0/n_2',
    \end{matrix*}
\end{align*}
so that $A = A_1A_2$ and $B = B_1B_2$. 
Combining the previous two lemmata   
gives us  \begin{align*}
    \mathcal{C}_8=&\delta(\tau_0=s_0=g_{0,0}=g_{0,1}=(r_0,\nu_0D_{0,1}')=(r_0s_{1,0},D_{1,0}',f_0)=(n_1r_0s_{1,0},D_{2,0})=1)\nu_0\tau_{1,0}(r_0s_{1,0},D_{1,0}')D_{0,1}'D_{1,0}'D_{2,0}f_0\varphi(\tau_{1,0})\\
    &\times \mu(\tau_{1,0}g_0')\sum_{\substack{c_0\mid f_0\\(c_0,\nu_0D_{0,1}'D_{1,0}')=1}}\sum_{\substack{k_0\mid \nu_0r_0s_{1,0}D_{0,1}'f_0\\(k_0,\frac{c_0r_0s_{1,0}D_{2,0}}{(r_0s_{1,0},D_{1,0}')})=1}}\frac{\mu(c_0r_0s_{1,0})\mu(k_0)}{c_0k_0} \sum_{\alpha\mid (D_2g,rsg^2,n_1)}\alpha\\ &\sum_{\substack{\alpha\nu\mid D_2f_0g_0'g\\ n_1\nu\mid \nu_0\nu'\tau_{1,0}n_{1,0}'D_2f_0g_0'(n_1,gh_1)\\(\alpha^2\nu,n_1r_0rs_{1,0}sg^2)=\alpha^2}}\nu\mu\left(\frac{r_0s_{1,0}D_2f_0g_0'g}{(r_0s_{1,0},D_{1,0}')\alpha\nu}\right)  \mathop{\sumast_{\beta\Mod{A_1A_2}}\sumast_{\gamma\Mod{B_1B_2}}}_{\frac{rsg^2}{\alpha}h_1^2\gamma\equiv r_0s_{1,0}\frac{n_1}{\alpha}\beta\Mod{D_{2,0}\nu}}e\left(\frac{\beta m}{A_1A_2}-\epsilon_1\epsilon_2\eta_1\eta_2\frac{\gamma n}{B_1B_2}\right).
\end{align*}
Write $$x=rsg^2h_1^2/\alpha, \quad  y=r_0s_{1,0}n_1/\alpha, \quad  z=D_{2,0}\nu,$$ so that $\Delta = z/(a, z)$. Note that we have $(xy,z)=1$. The $\beta, \gamma$-sum is equal to \begin{align*}
    \mathop{\sumast_{\beta\Mod{A}}\sumast_{\gamma\Mod{B}}}_{x\gamma\equiv y\beta\Mod{z}}e\left(\frac{\beta m}{A}-\epsilon_1\epsilon_2\eta_1\eta_2\frac{\gamma n}{B}\right)=&\sumast_{\beta\Mod{A}}\sum_{\substack{b\mid B\\(b,z)=1}}\mu(b)\sum_{\substack{\gamma\Mod{B/b}\\ \gamma\equiv \overline{bx}y\beta\Mod{z}}}e\left(\frac{\beta m}{A}-\epsilon_1\epsilon_2\eta_1\eta_2\frac{\gamma n}{B/b}\right)\\
    =&\sumast_{\beta\Mod{A}}e\left(\frac{\beta m}{A}\right)\sum_{\substack{b\mid B\\(b,z)=1}}\mu(b)\frac{B}{bz}\delta\left(\frac{B}{bz}\bigg| n\right)  e\left(-\epsilon_1\epsilon_2\eta_1\eta_2\frac{\overline{bx}y\beta\frac{n bz}{B}}{z}\right).
\end{align*}
Note that $z\mid A$. Writing $n=\frac{B}{bz}n'$, this is equal to \begin{align*}
    &\sum_{\substack{b\mid B\\(b,z)=1}}\mu(b)\frac{B}{bz}\sumast_{\beta\Mod{A}}e\left(\frac{\beta m}{A}-\epsilon_1\epsilon_2\eta_1\eta_2\frac{\overline{bx}y\beta n'}{z}\right)=\sum_{\substack{b\mid B\\(b,z)=1}}\mu(b)\frac{B}{bz}\sum_{a\mid A}\mu(a)\sum_{\beta\Mod{A/a}}e\left(\frac{\beta m}{A/a}-\epsilon_1\epsilon_2\eta_1\eta_2\frac{\overline{bx}ya\beta n'}{z}\right)\\
    =&\sum_{\substack{b\mid B\\(b,z)=1}}\mu(b)\frac{B}{bz}\sum_{a\mid A}\mu(a)\sum_{\beta_1\left(\text{mod }\left(\frac{A}{a},\Delta^\infty\right)\right)}e\left(\frac{\beta_1 \overline{\frac{A/a}{ (A/a,\Delta^\infty )}}\left(m-\epsilon_1\epsilon_2\eta_1\eta_2\overline{bx}y n' A/z \right)}{ (\frac{A}{a},\Delta ^\infty )}\right)  \sum_{\beta_2\left(\text{mod}\frac{A/a}{ (A/a,\Delta^\infty )}\right)}e\left(\frac{\beta_2 \overline{ (\frac{A}{a},\Delta^\infty )}m}{\frac{A/a}{ (A/a,\Delta^\infty )}}\right)\\
    =&\sum_{\substack{b\mid B\\(b,z)=1}}\mu(b)\frac{B}{bz}\sum_{a\mid A}\mu(a)\frac{A}{a}\delta\left(\frac{A/a}{ (A/a,\Delta^\infty)}\, \Bigg|\,  m\right) \delta\left( bxm\equiv \epsilon_1\epsilon_2\eta_1\eta_2y\frac{A}{z}n'\, \Big(\text{mod } \Big(\frac{A}{a},\Delta^\infty\Big)\Big)\right).
\end{align*}
Writing $m=\frac{A/a}{(A/a, \Delta^{\infty})}m'$, the congruence condition is equivalent to \begin{align*}
    bxm'\equiv\epsilon_1\epsilon_2\eta_1\eta_2y\frac{a (\frac{A}{a},\Delta^\infty )}{z}n'\, \Big(\text{mod }\Big(\frac{A}{a},\Delta^\infty\Big)\Big).
\end{align*}
This completes the proof. 
\end{proof}

Inserting the character sum analysis back into \eqref{endvor2} yields
\begin{align*}
    \Sigma_6^{(3)}=&\sum_{\epsilon_1,\epsilon_2,\eta_1,\eta_2=\pm1} \sum_{\nu_0}\mathop{\sum\sum\sum\sum}_{(\tau_{1,0}r_0s_{1,0}g_0',\nu_0)=1}\sum_{(f_0,\tau_{1,0}g_0')=1}\mathop{\sum\sum\sum\sum}_{\substack{n_{1,0}'D_{0,1}'D_{1,0}'D_{2,0}\mid (\nu_0\tau_{1,0}r_0s_{1,0}f_0g_0')^\infty\\(D_{0,1}',\tau_{1,0}n_{1,0}'r_0s_{1,0}D_{1,0}'D_{2,0}g_0')=1\\(D_{1,0}',\nu_0D_{2,0}g_0')=(D_{2,0},\nu_0\tau_{1,0}r_0s_{1,0})=1\\(D_{1,0}',r_0s_{1,0},f_0)=1}}\mathop{\sum\sum\sum\sum\sum\sum\sum\sum}_{\substack{(\nu'  rsD_0'D_2gh_1h_2,\nu_0\tau_{1,0}r_0s_{1,0}f_0g_0')=1\\
    (h_1h_2,D_2g)=(\nu'rsD_2gh_1,D_0'h_2)=1}}\nonumber\\
    &\times \frac{\nu_0\nu'n_{1,0}'(r_0s_{1,0},D_{1,0}')D_{2,0}D_2f_0\varphi(\tau_{1,0})X^{1/3}}{\tau_{1,0}^2r_0^2rs_{1,0}^2sD_{1,0}'g_0'gh_1h_2}\mu\left(\tau_{1,0}r_0rs_{1,0}sg_0'gh_1h_2\right)\mu(s_{1,0}sgh_1)\nonumber\\
    &\times W_0\Big(\frac{\tau_{1,0}D_{1,0}'g}{\Xi}\Big)B(r_0r,1)\sum_{\substack{n_1\mid \nu_0\nu'\tau_{1,0}n_{1,0}'D_2f_0g_0'gh_1\\(n_1,D_{2,0})=1}}\sum_n   \frac{B(n_1,n)}{n_1^2n}W_0\Big( \frac{rh_1}{\Xi'}\Big) \mathcal{C}_9 F_3^{\epsilon_1,\epsilon_2,\eta_1,\eta_2}(X_1,X_2,  X_3),
\end{align*}
where $A, B$ are as in Lemma \ref{finalc8} and 
\begin{equation}\label{x4}
\begin{split}
   & X_1=\frac{n_1^2nX}{\nu_0^6\nu'^3\tau_{1,0}^6 n_{1,0}'^3r_0^4rs_{1,0}^6s^3D_{0,1}'^3D_{2,0}^3D_0'^3D_2^3f_0^6g_0'^6g^3h_1^3h_2^3},\quad X_2=\frac{X^{1/3}}{\nu_0^2\nu'\tau_{1,0}^2 n_{1,0}'r_0^2s_{1,0}^3s D_{0,1}'D_{2,0}D_0'D_2f_0^2g_0'^2h_2}, \\
 & X_3  = \frac{\nu_0^5\nu'^2  \tau_{1,0}^5   n_{1,0}'^2 r_0^4 r s_{1,0}^6 s^3 D_{0,1}'^3D_0'^3f_0^3 g_0'^3h_1^2h_2^3}{X}. 
    \end{split}
\end{equation}

\subsection{Simplification III: Equality in the congruence condition}

The character sum $\mathcal{C}_9$, defined in \eqref{C9},  features the congruence
$$b\frac{rsg^2}{\alpha}h_1^2m'\equiv \epsilon_1\epsilon_2\eta_1\eta_2r_0s_{1,0}\frac{n_1}{\alpha}\frac{a (\frac{A}{a},\Delta^\infty )}{D_{2,0}\nu}n'\, \left(\text{{\rm mod }} \Big(\frac{A}{a},\Delta^\infty\Big)\right).$$
Since $m$ is bounded, but also divisible by $\frac{A/a}{(A/a, \Delta^{\infty})}$, we conclude that the modulus of the previous congruence is $\gg A/a$ which is typically large, but generically we think of both sides of the congruence as small. Therefore the main contribution should come from the situation when the congruence is an equality, as alluded to in the introduction. Let $\Sigma_6^{(4)}$ denote the contribution to $\Sigma_6^{(3)}$ where we restrict the congruence to an equality. We show that this is a good approximation:
\begin{lemma}\label{lem510}  For $\Delta^{\ast} \geq X^{11/21}$ we have 
     $\Sigma_6^{(4)} - \Sigma_6^{(3)} \preccurlyeq  X^{{139/150}}$. 
\end{lemma}

\begin{proof} This is again a direct estimation using the bounds on $F_3$ from Lemma \ref{lemf3} and the evaluation of $\mathcal{C}_9$ in \eqref{C9}. The key observation is the following: if the congruence is \emph{not} an equality, then at least one of the two sides must  {be} bigger than the modulus $(A/a, \Delta^{\infty}) \gg A/a$ and will be restricted to an arithmetic progression modulo $(A/a, \Delta^{\infty}) $. 

Precisely, notice that $(b\frac{rsg^2}{\alpha}h_1^2 r_0s_{1,0}\frac{n_1}{\alpha}\frac{a}{(a,D_{2,0}\nu)},D_{2,0}\nu)=1$. Hence the congruence implies $(A/a,\Delta^\infty)/\Delta|m'$ and \begin{align*}
    b\frac{rsg^2}{\alpha}h_1^2\frac{m'}{(A/a,\Delta^\infty)/\Delta}\equiv \epsilon_1\epsilon_2\eta_1\eta_2r_0s_{1,0}\frac{n_1}{\alpha}\frac{a}{(a,D_{2,0}\nu)}n'\, \left(\text{{\rm mod }} \Delta\right).
\end{align*}
This determines $n'$ mod $\Delta$. Since $\frac{(A/a,\Delta^\infty)}{\Delta}\mid m'$ and $\frac{A/a}{(A/a, \Delta^{\infty})}|m$, we have $\Delta \gg A/a$. Hence the number of $n'$ with $n'\asymp N'$ is $\ll 1+aN'/A$. Moreover, the congruence being not an equality implies that \begin{align*}
    \left(\frac{A}{a},\Delta^\infty\right) \leq b\frac{rsg^2}{\alpha}h_1^2m'+r_0s_{1,0}\frac{n_1}{\alpha}\frac{a (\frac{A}{a},\Delta^\infty )}{D_{2,0}\nu}n' \quad \Longleftrightarrow& \quad \frac{A}{a}\leq b\frac{rsg^2}{\alpha}h_1^2m+r_0s_{1,0}\frac{n_1}{\alpha}b\frac{A}{B}n\\
    \Longleftrightarrow& \quad \frac{A}{a}\leq \frac{bs}{\alpha gh_1}\left(rg^3h_1^3m+r_0^2s_{1,0}^3n_1^2n\right).
\end{align*}
Thus our effective saving is $\min(A/a, Nb D_{2,0} \nu/B)$ and we record the previous inequality.  We estimate rather coarsely, e.g.\ $(r_0 s_{1, 0}, D_{1, 0}') \leq r_0s_{1, 0}$. Recalling the last factor in \eqref{f3bound}, we distinguish the two cases   $(X_2'Y^2)^{-1/2} \geq X^{1/100}(1 + (X_2'^2 Y)^{1/2} + Z^{1/3}/Y)$ and its opposite and run the   following two linear programs:


\noindent {\tt A := d2 + d20 + f0 + g0 + n10 + nu0 + nuprime + r0 + s + 2s10 + tau10;\\
B := d2 + d20 + f0 + g + g0 + h1 - n1 + n10 + nu0 + nuprime + tau10;\\
x1 := 1 - 3d0 - 3d01 - 3d2 - 3d20 - 6f0 - 3g - 6g0 - 3h1 - 
 3h2 + n + 2n1 - 3n10 - 6nu0 - 3nuprime - r - 4r0 - 3s - 
 6s10 - 6tau10;\\
x2 := 1/3 - d0 - d01 - d2 - d20 - 2f0 - 2g0 - h2 - n10 - 
 2nu0 - nuprime - 2r0 - s - 3s10 - 2tau10;
}\\

\noindent {\tt {Maximize[\{1/3 - a + alpha - b - c0 + d0 - d10 + 2d2 + d20 + 2f0 - k0 - 
 2n1 + n10 + 2nu0 + 2nuprime - (B - b - d20   - nu) - Min[A - a, n - B + b + d20 + nu] + 2y - z + Min[x1 + z , 2/3(x1 + z)] + Min[z
- 3y, 2/3(z - 3y)] - Max[2 x2 + y, - x2 - 2 y], z <= 3y + 3 Min[Max[- x2/2 - y, - 1/2 y], Max[x2 + z - y, - 1/2
x2 - y, x2 + y/2]], x1 + z <= 3 Max[ Min[Max[- x2/2 - y, - 1/2 y], Max[x2 + z - y, - 1/2 x2 - y, x2
+ y/2]], x2 + z - y] ,  {y <= 0, x2 >= 0}, 2/3 + y  >= 11/21, A - a + alpha + g + h1 <= 
 b + s + Max[r + 3 g + 3 h1, 2 r0 + 3 s10 + 2 n1 + n], n1 <= nuprime + tau10  + d2 + g0
+ g + h1 +  nu0 + n10   + f0, c0 <= f0, k0 <= nu0 + r0 + s10 + d01 + f0, alpha <= Min[d2 + g, r + s + 2 g, n1], alpha + nu
<= d2 + f0 + g0 + g, n1  + nu <= nu0 + nuprime + tau10 + n10  + d2 + f0 + g0 + Min[n1, g
+ h1], b <= B, a <= A, n >= B - b - d20  - nu, nu0 >= 0, g0 >= 0, tau10 >= 0, f0 >= 0,    n10 >= 0, d01 >= 0, d10 >= 0, d20 >= 0, nu >= 0, r >= 0, s >= 0, d0 >= 0, d2
>= 0, g >= 0, h1 >= 0, h2 >= 0, n1 >= 0,  n >= 0, c0 >= 0, k0 >= 0, alpha >= 0,  b >= 0, a
>= 0, nuprime >= 0, r0 >= 0, s10 >= 0\}, \{nu0, g0, tau10, f0,   n10, d01, d10, d20, nu, r, s,  {r0, s10}, d0, d2, g,
h1, h2, n1, n, c0, k0, alpha,   b, a, y, z, nuprime\}]}}

 \noindent {\tt \{17/21, \{nu0 -> 0, g0 -> 0, tau10 -> 0, f0 -> 0, n10 -> 0, d01 -> 0, 
  d10 -> 0, d20 -> 0, nu -> 0, r -> 0, s -> 0, r0 -> 0, s10 -> 0, 
  d0 -> 1/3, d2 -> 0, g -> 0, h1 -> 0, h2 -> 0, n1 -> 0, n -> 6/7, 
  c0 -> 0, k0 -> 0, alpha -> 0, b -> 0, a -> 0, y -> -1/7, 
  z -> -3/7, nuprime -> 0\}\}}\\

 \noindent {\tt Maximize[\{1/3 - a + alpha - b - c0 + d0 - d10 + 2d2 + d20 + 2f0 - k0 - 
 2n1 + n10 + 2nu0 + 2nuprime - (B - b - d20   - nu) - Min[A - a, n - B + b + d20 + nu] + 2y - z + Min[x1 + z , 2/3(x1 + z)] + Min[z
- 3y, 2/3(z - 3y)] - 1/4(x2 - y), -1/2(x2 + 2y) <= 1/100 + Max[0, 1/2(2x2 + y), 1/3 z - y] , z <= 3y + 3 Min[Max[- x2/2 - y, - 1/2 y], Max[x2 + z - y, - 1/2
x2 - y, x2 + y/2]], x1 + z <= 3 Max[ Min[Max[- x2/2 - y, - 1/2 y], Max[x2 + z - y, - 1/2 x2 - y, x2
+ y/2]], x2 + z - y] ,  {y <= 0, x2 >= 0},  2/3 + y  >= 11/21, A - a + alpha + g + h1 <= 
 b + s + Max[r + 3 g + 3 h1, 2 r0 + 3 s10 + 2 n1 + n], n1 <= nuprime + tau10  + d2 + g0
+ g + h1 +  nu0 + n10  + f0, c0 <= f0, k0 <= nu0 + r0 + s10 + d01 + f0, alpha <= Min[d2 + g, r + s + 2 g, n1], alpha + nu
<= d2 + f0 + g0 + g, n1  + nu <= nu0 + nuprime + tau10 + n10  + d2 + f0 + g0 + Min[n1, g
+ h1], b <= B, a <= A, n >= B - b - d20  - nu, nu0 >= 0, g0 >= 0, tau10 >= 0, f0 >= 0,    n10 >= 0, d01 >= 0, d10 >= 0, d20 >= 0, nu >= 0, r >= 0, s >= 0, d0 >= 0, d2
>= 0, g >= 0, h1 >= 0, h2 >= 0, n1 >= 0,  n >= 0, c0 >= 0, k0 >= 0, alpha >= 0,   b >= 0, a
>= 0, nuprime >= 0, r0 >= 0, s10 >= 0\}, \{nu0, g0, tau10, f0,   n10, d01, d10, d20, nu, r, s,  {r0, s10}, d0, d2, g,
h1, h2, n1, n, c0, k0, alpha,   b, a, y, z, nuprime\}]}

\noindent {\tt \{139/150, \{nu0 -> 0, g0 -> 0, tau10 -> 0, f0 -> 0, n10 -> 0, d01 -> 0,
   d10 -> 0, d20 -> 0, nu -> 0, r -> 0, s -> 0, r0 -> 0, s10 -> 0, 
  d0 -> 1/3, d2 -> 0, g -> 0, h1 -> 0, h2 -> 0, n1 -> 0, n -> 321/700,
   c0 -> 0, k0 -> 0, alpha -> 0, b -> 0, a -> 0, y -> -1/7, 
  z -> -3/100, nuprime -> 0\}\}}
\end{proof}

Now we continue with the equality in the congruence condition, i.e.\ with the analysis of $\Sigma_6^{(4)}$.  
This gives us \begin{align*}
    \epsilon_1\epsilon_2\eta_1\eta_2=1 \quad \text{ and } \quad rg^3h_1^3m = r_0^2s_{1,0}^3n_1^2n.
\end{align*}
Since $(s_{1,0},rgh_1)=1$ and $m$ is cube-free by our general assumption, we have $s_{1,0}=1$. Moreover, the equality and the definition of $m'$ and $n'$ give us \begin{align*}
    r_0n_1\frac{a (\frac{A}{a}, \Delta^\infty)}{D_{2,0}\nu}n'=rsg^2h_1^2bm',
\end{align*}
and hence \begin{align*}
    &\frac{r_0n_1\frac{a (A/a,\Delta^\infty)}{D_{2,0}\nu}}{ (r_0n_1\frac{a (A/a,\Delta^\infty )}{D_{2,0}\nu},rsg^2h_1^2b )}\frac{A/a}{ (\frac{A}{a},\Delta^\infty )}\mid m  \quad \Longleftrightarrow  \quad \frac{\nu_0\nu'\tau_{1,0}n_{1,0}'n_1r_0^2sD_2f_0g_0'}{\nu (r_0n_1\frac{a (A/a,\Delta^\infty )}{D_{2,0}\nu},rsg^2h_1^2b )}\mid m.
\end{align*}
Notice that by the divisibility conditions of $b, \nu$, $n_{1,0}'\mid (\nu_0\tau_{1,0}r_0f_0g_0')^\infty$, $(r_0,b)\mid n_{1,0}'f_0$, $b$   squarefree and $(b,\nu)=1$, we have \begin{align*}
    \nu\Big(r_0n_1\frac{a (A/a,\Delta ^\infty )}{D_{2,0}\nu},rsg^2h_1^2b\Big)\mid \nu_0\nu'\tau_{1,0}(r_0,n_{1,0}')n_1sD_2f_0g_0'.
\end{align*}
Hence the equality conditions on $m',n',m,n$ are equivalent to \begin{align*}
    \epsilon_1\epsilon_2\eta_1\eta_2=1, \quad r_0^2n_1^2n=rg^3h_1^3m, \quad \text{ and } \quad \frac{r_0^2n_{1,0}'}{(r_0,n_{1,0}')}\frac{\nu_0\nu'\tau_{1,0}(r_0,n_{1,0}')n_1sD_2f_0g_0'}{\nu (r_0n_1\frac{a (A/a,\Delta^\infty )}{D_{2,0}\nu},rsg^2h_1^2b )}\mid m.
\end{align*}

Substituting everything back into $\Sigma_6^{(4)}$ and writing $\omega=\frac{D_2f_0g_0'g}{\alpha\nu}$, we have \begin{equation}\label{sigma64}
\begin{split}
    \Sigma_6^{(4)}&=\sum_{\substack{\epsilon_1,\epsilon_2,\eta_1,\eta_2=\pm1\\\epsilon_1\epsilon_2\eta_1\eta_2=1}} \sum_{\nu_0}\mathop{\sum\sum\sum}_{(\tau_{1,0}r_0g_0',\nu_0)=1}\sum_{(f_0,\tau_{1,0}g_0')=1}\mathop{\sum\sum\sum\sum}_{\substack{n_{1,0}'D_{0,1}'D_{1,0}'D_{2,0}\mid (\nu_0\tau_{1,0}r_0f_0g_0')^\infty\\(D_{0,1}',\tau_{1,0}n_{1,0}'r_0D_{1,0}'D_{2,0}g_0')=1\\(D_{1,0}',\nu_0D_{2,0}g_0')=(D_{2,0},\nu_0\tau_{1,0}r_0)=1\\(D_{1,0}',r_0,f_0)=1}}\mathop{\sum\sum\sum\sum\sum\sum\sum\sum}_{\substack{(\nu'  rsD_0'D_2gh_1h_2,\nu_0\tau_{1,0}r_0f_0g_0')=1\\
    (h_1h_2,D_2g)=(\nu'rsD_2gh_1,D_0'h_2)=1}}\\
    &\times \frac{\nu_0\nu'n_{1,0}'(r_0,D_{1,0}')D_{2,0}D_2f_0\varphi(\tau_{1,0})X^{1/3}}{\tau_{1,0}^2r^2sD_{1,0}'g_0'g^4h_1^4h_2m}\mu\left(\tau_{1,0}r_0rsg_0'gh_1h_2\right)\mu(sgh_1)W_0\Big(\frac{\tau_{1,0}D_{1,0}'g}{\Xi}\Big) \\
    &\times W_0\Big(\frac{rh_1}{\Xi'}\Big)B(r_0r,1)\sum_{\substack{n_1\mid \nu_0\nu'\tau_{1,0}n_{1,0}'D_2f_0g_0'gh_1\\(n_1,D_{2,0})=1}}\sum_{\substack{c_0\mid f_0 \\(c_0,\nu_0D_{0,1}'D_{1,0}')=1}}\sum_{\substack{k_0\mid \nu_0r_0D_{0,1}'f_0\\(k_0,\frac{c_0r_0D_{2,0}}{(r_0,D_{1,0}')})=1}}\frac{\mu(c_0r_0)\mu(k_0)}{c_0k_0} \\
              \end{split}
\end{equation}
\begin{equation*} 
\begin{split}
         &\times \sum_{\alpha\mid (D_2g,rsg^2,n_1)}\alpha\sum_{\substack{\alpha\omega\mid D_2f_0g_0'g\\n_1g\mid \alpha\omega\nu_0\nu'\tau_{1,0}n_{1,0}'(n_1,gh_1)\\(\alpha\frac{D_2f_0g_0'g}{\omega},n_1r_0rsg^2)=\alpha^2}}\mu\left(\frac{r_0\omega}{(r_0,D_{1,0}')}\right)\sum_{a\mid A}\frac{\mu(a)}{a}\sum_{\substack{b\mid B\\(b,D_{2,0}\frac{D_2f_0g_0'g}{\alpha\omega})=1}}\frac{\mu(b)}{b} \\
    &\times \delta\left(\frac{r_0^2n_{1,0}'}{(r_0,n_{1,0}')}\frac{\alpha\omega\nu_0\nu'\tau_{1,0}(r_0,n_{1,0}')n_1s}{g\left(r_0n_1\frac{\alpha\omega\left(A,a\left(\frac{D_{2,0}D_2f_0g_0'g}{(\alpha\omega a,D_{2,0}D_2f_0g_0'g)}\right)^\infty\right)}{D_{2,0}D_2f_0g_0'g},rsg^2h_1^2b\right)}\bigg|m\right)B\left(n_1,\frac{rg^3h_1^3m}{r_0^2n_1^2}\right) F_3^{\epsilon_1,\epsilon_2,\eta_1,\eta_2}\Bigg(\frac{mZ^3}{D_2^3},\frac{Z}{D_2},  X_3\Bigg),
    \end{split}
\end{equation*}
with \begin{equation}\label{para}
   \begin{split}
    &A=\nu_0\nu'\tau_{1,0}n_{1,0}'r_0sD_{2,0}D_2f_0g_0', \quad B=\frac{\nu_0\nu'\tau_{1,0}n_{1,0}'D_{2,0}D_2f_0g_0'gh_1}{n_1}, \quad 
   Z=\frac{X^{1/3}}{\nu_0^2\nu'\tau_{1,0}^2 n_{1,0}'r_0^2s D_{0,1}'D_{2,0}D_0'f_0^2g_0'^2h_2} 
   \end{split}
\end{equation}
and  
$X_3$ as in \eqref{x4} with $s_{1,0}=1$. 

We end this section with one last piece of simplification. The definition \eqref{deff3} of $F_3^{\epsilon_1,\epsilon_2,\eta_1,\eta_2}(a, b,   d)$ contains a factor $W_0(d/(\Xi' z))$ which provides a very weak lower bound for $z$. We used this for the second condition in \eqref{boundsF}, which was finally responsible to truncate the $n_1, n$-sums in \eqref{abscon}. After restriction to the equality in the congruence condition imposed by the character $\mathcal{C}_9$, these sums have gone, and we do not need this factor any longer. We therefore drop this factor from $F_3$ and observe that by Lemma \ref{lemf3} the resulting integral is still convergent as discussed immediately after Lemma \ref{lemf3}. 

At the same time we need to remember that the second entry of $F_3$ is bounded by $Z \leq X^{1/3}$, so at no cost we can include a redundant factor $W_0(b/X^{1/3})$. With this in mind, we define   
\begin{equation}\label{curlyf3}
\begin{split}
    \mathcal{F}_3^{\epsilon_1,\epsilon_2,\eta_1,\eta_2}(a,b)=&\int_{(0)}\int_0^\infty\int_0^\infty\int_{(0)}\int_0^\infty a^{-s_2}x^{-s_1-1}y^{3s_1+1}z^{-s_1-s_2-2}\mathcal{G}_{\mu_0}^{\eta_1}(s_1+1)\mathcal{G}_{-\mu_0}^{\eta_2}(s_2+1)\\
    &\times  V(x)\Big(1-U\Big( {\frac{  X^{2/3}y}{\Delta^{\ast}}}\Big)\Big)
   \Phi_{\omega_6}\left(-\epsilon_2\frac{x}{by^2},-\epsilon_1 b^2my\right)e\left(-\epsilon_2\eta_1\frac{bz}{y}\right)\mathrm{d} x \frac{\mathrm{d}s_1}{2\pi i} \mathrm{d} y\, \mathrm{d} z\frac{\mathrm{d}s_2}{2\pi i}  W_0\Big(\frac{b}{X^{1/3}}\Big),
   \end{split}
\end{equation}
and replace $F_3^{\epsilon_1,\epsilon_2,\eta_1,\eta_2} ( mZ^3/D_2^3, Z/D_2, X_3 )$ with $\mathcal{F}_3^{\epsilon_1,\epsilon_2,\eta_1,\eta_2} ( mZ^3/D_2^3, Z/D_2)$ at the cost of a negligible error. We continue to call the resulting expression $\Sigma_6^{(4)}$. The key point is that we can now evaluate the $z$-integral explicitly. 
Again we momentarily restrict to dyadic pieces $y\asymp Y$ and $z\asymp Z$
and move the $z$-integral inside. As long as $-2< \Re (s_1 +s_2) < -1 $, we can re-assemble the dyadic pieces and compute the $z$-integral as a conditionally convergent integral:
$$\int_0^\infty z^{-s_1-s_2-2}e\left(-\epsilon_2\eta_1\frac{bz}{y}\right)\mathrm{d}z=\left(\frac{y}{2\pi b}\right)^{-s_1-s_2-1}e\left(\epsilon_2\eta_1\frac{s_1+s_2+1}{4}\right)\Gamma\left(-s_1-s_2-1\right), $$
so that 
\begin{align*}
  \mathcal{F}_3^{\epsilon_1,\epsilon_2,\eta_1,\eta_2}(a,b) =   &\int_{(-7/8)}\int_{(-7/8)}\int_{0}^{\infty} \int_0^\infty a^{-s_2}(2\pi b)^{s_1+s_2+1}x^{-s_1-1}y^{2s_1-s_2}\mathcal{G}_{\mu_0}^{\eta_1}(s_1+1)\mathcal{G}_{-\mu_0}^{\eta_2}(s_2+1)\Big(1-U\Big( {\frac{  X^{2/3}y }{\Delta^{\ast}}}\Big)\Big) \nonumber\\
    &\times e\Big(\epsilon_2\eta_1\frac{s_1+s_2+1}{4}\Big)\Gamma\left(-s_1-s_2-1\right) V(x)
    \Phi_{\omega_6}\Big(-\epsilon_2\frac{x}{by^2},-\epsilon_1 b^2my\Big)\, \mathrm{d} x \,\mathrm{d} y\, \frac{\mathrm{d}s_1}{2\pi i}  \frac{\mathrm{d}s_2}{2\pi i} W_0\Big(\frac{b}{X^{1/3}}\Big).
\end{align*}
 This integral is not absolutely convergent as a  fourfold integral, but as an iterated integral in the above order: the $x, y$-integrals are absolutely convergent, and integration by parts in $x$ or $y$ produces sufficient decay for the $s_1, s_2$-integrals.

\subsection{Step 5: Poisson summation}

Finally, we want to apply Poisson summation on the $D_2$-sum in $\Sigma_6^{(4)}$. Again we shall see that only the central term will survive. This manoeuvre needs some preparation. For convergence reasons we insert into \eqref{sigma64} an essentially redundant factor $$W_0\Big(\frac{ {\nu_0\nu'sf_0D_0'}D_{2, 0}}{\Xi'}\Big)
$$ with $\Xi' = X^{10^{10}}$ at the cost of a very small error (and continue to call the modified expression $\Sigma_6^{(4)}$).  
The difficulty is that $D_2$ occurs in \eqref{sigma64} in the divisibility conditions involving $n_1, \alpha, \omega, a, b$ (since $D_2$ is contained in $A, B$), as well as in the $\delta(...)$-condition. With this in mind, we first consider a general expression of the shape 
\begin{align*}
    &\sum_{(D_2,\nu_0\tau_{1,0}r_0D_0'f_0g_0'h_1h_2)=1}\sum_{\substack{n_1\mid \nu_0\nu'\tau_{1,0}n_{1,0}'D_2f_0g_0'gh_1\\(n_1,D_{2,0})=1}}\sum_{\alpha\mid (D_2g,rsg^2,n_1)}\sum_{\substack{\alpha\omega\mid D_2f_0g_0'g\\n_1g\mid \alpha\omega\nu_0\nu'\tau_{1,0}n_{1,0}'(n_1,gh_1)\\(\alpha\frac{D_2f_0g_0'g}{\omega},n_1r_0rsg^2)=\alpha^2}}\sum_{a\mid A}\sum_{\substack{b\mid B\\(b,D_{2,0}\frac{D_2f_0g_0'g}{\alpha\omega})=1}}H(D_2,n_1,\alpha,\omega,a,b)
\end{align*}
for a function $H$. Pulling out gcds and using $$\sum_{x\mid yz}F(x)=\sum_{x_1\mid y}\sum_{\substack{x_2\mid z\\ (x_2,y/x_1)=1}}F(x_1x_2)$$ for any $y,z\in\N$ and any function $F$, we can re-write this 
 as
\begin{align*}    
    &\sum_{\substack{n_1\mid \nu_0\nu'\tau_{1,0}n_{1,0}'f_0g_0'gh_1\\(n_1,D_{2,0})=1}}\sum_{\left(m_0,\frac{\nu'g}{(n_1,\nu'g)}\nu_0\tau_{1,0}r_0D_0'f_0g_0'h_1h_2\right)=1}\sum_{\alpha_0\mid (rsg^2,m_0n_1)}\sum_{\alpha'\mid \left(m_0g,\frac{(rsg^2,m_0n_1)}{\alpha_0}\right)}\sum_{\substack{\alpha'\omega'\mid f_0g_0'm_0g\\ \left(\frac{f_0g_0'm_0g}{\alpha'\omega'},\frac{m_0n_1r_0rsg^2}{\alpha_0^2\alpha'^2}\right)=1}}\\
    &\times \sum_{\substack{\left(\omega_0,\alpha'\nu_0\tau_{1,0}r_0D_0'f_0g_0'h_1h_2\right)=1\\ m_0n_1g\mid \alpha_0\alpha'\omega_0\omega'\nu_0\nu'\tau_{1,0}n_{1,0}'(m_0n_1,gh_1)}}\sum_{\left(a_0,\frac{m_0n_1rsg^2}{\alpha_0^2\alpha'}\omega'\nu_0\tau_{1,0}r_0D_0'f_0g_0'h_1h_2\right)=1}\sum_{\substack{a'\mid m_0\alpha_0\omega_0A'}}\\
    &\times \sum_{\substack{b\mid \alpha_0\omega_0B'\\ \left(b,a_0D_{2,0}\frac{f_0g_0'm_0g}{\alpha'\omega'}\right)=1}}\sum_{\substack{\left(D_2',a'b\frac{m_0n_1rsg^2}{\alpha_0^2\alpha'}\omega'\nu_0\tau_{1,0}r_0D_0'f_0g_0'h_1h_2\right)=1}}H(m_0\alpha_0\omega_0a_0D_2',m_0n_1,\alpha_0\alpha',\omega_0\omega',a_0a',b)
    \end{align*}
    with $A' = A/D_2$ 
and $B'= m_0B/D_2$. This equals
    \begin{align*}
    &\sum_{\substack{n_1\mid \nu_0\nu'\tau_{1,0}n_{1,0}'f_0g_0'gh_1\\(n_1,D_{2,0})=1}}\sum_{\left(m_0,\frac{\nu'g}{(n_1,\nu'g)}\nu_0\tau_{1,0}r_0D_0'f_0g_0'h_1h_2\right)=1}\sum_{\alpha_0\mid (rsg^2,m_0n_1)}\sum_{\alpha'\mid \left(m_0g,\frac{(rsg^2,m_0n_1)}{\alpha_0}\right)}\sum_{\substack{\alpha'\omega'\mid f_0g_0'm_0g\\ \left(\frac{f_0g_0'm_0g}{\alpha'\omega'},\frac{m_0n_1r_0rsg^2}{\alpha_0^2\alpha'^2}\right)=1}}\\
    &\times \sum_{\substack{\left(\omega_0,\alpha'\nu_0\tau_{1,0}r_0D_0'f_0g_0'h_1h_2\right)=1\\ m_0n_1g\mid \alpha_0\alpha'\omega_0\omega'\nu_0\nu'\tau_{1,0}n_{1,0}'(m_0n_1,gh_1)}}\sum_{\left(a_0,\frac{m_0n_1rsg^2}{\alpha_0^2\alpha'}\omega'\nu_0\tau_{1,0}r_0D_0'f_0g_0'h_1h_2\right)=1}\sum_{\substack{a'\mid m_0\alpha_0\omega_0A'}}\sum_{\substack{b\mid \alpha_0\omega_0B'\\ \left(b,a_0D_{2,0}\frac{f_0g_0'm_0g}{\alpha'\omega'}\right)=1}}\\
   & \times \sum_{\substack{d_0\mid \frac{\alpha_0\omega_0\nu's}{(g,\alpha_0\omega_0\nu')}\\ \left(d_0,a'b\frac{m_0n_1rsg^2}{\alpha_0^2\alpha'}\omega'h_1\right)=1 }}\sum_{\substack{\left(D_2',a'b\frac{m_0n_1rsg^2}{\alpha_0^2\alpha'}\omega'\frac{\alpha_0\omega_0\nu's}{d_0(g,\alpha_0\omega_0\nu')}\nu_0\tau_{1,0}r_0D_0'f_0g_0'h_1h_2\right)=1}}  H(m_0\alpha_0\omega_0a_0d_0D_2',m_0n_1,\alpha_0\alpha',\omega_0\omega',a_0a',b),
\end{align*}

Returning to \eqref{sigma64}, we apply this with
$$H(D_2, n_1, \alpha, \omega, a, b) = D_2 \alpha \mu\left(\frac{r_0\omega}{r_0,D_{1,0}'}\right) \frac{\mu(a)\mu(b)}{ab} \delta\big(... \mid m) 
B\left(n_1,\frac{rg^3h_1^3m}{r_0^2n_1^2}\right) \mathcal{F}_3^{\epsilon_1,\epsilon_2,\eta_1,\eta_2}\Bigg(\frac{mZ^3}{D_2^3},\frac{Z}{D_2}\Bigg).$$
 {At this point, we also insert a redundant factor $W_0(a_0 \omega'/\Xi')$. }

With the above notations and reordering, notice that the congruence condition in the $\delta$-term is now given by \begin{align}\label{MDef}
    \mathcal{M}:=\frac{r_0^2n_{1,0}'}{(r_0,n_{1,0}')}\frac{\alpha_0\alpha'\omega_0\omega'\nu_0\nu'\tau_{1,0}(r_0,n_{1,0}')m_0n_1s}{\left(r_0m_0n_1\left(\alpha_0\alpha'\omega_0\omega'\nu_0\nu'\tau_{1,0}n_{1,0}'r_0s,\frac{\alpha'\omega'a'g}{(\alpha'\omega'a',D_{2,0}f_0g_0'm_0g)}\left(\frac{d_0D_{2,0}f_0g_0'm_0g}{(\alpha'\omega'a',D_{2,0}f_0g_0'm_0g)}\right)^\infty\right),rsg^3h_1^2b\right)}\Bigg|m, 
\end{align}
and in particular free of $D_2'$. Hence the remaining $D_2'$-sum in \eqref{sigma64} (with $F_3$ changed to $\mathcal{F}_3$) is now of the form \begin{align*}
    &\sum_{\substack{\left(D_2',a'b\frac{m_0n_1rsg^2}{\alpha_0^2\alpha'}\omega'\frac{\alpha_0\omega_0\nu's}{d_0(g,\alpha_0\omega_0\nu')}\nu_0\tau_{1,0}r_0D_0'f_0g_0'h_1h_2\right)=1}}D_2'  \mathcal{F}_3^{\epsilon_1,\epsilon_2,\eta_1,\eta_2}\Bigg(\frac{mZ^3}{(m_0\alpha_0\omega_0a_0d_0D_2')^3},\frac{Z}{m_0\alpha_0\omega_0a_0d_0D_2'}\Bigg).
\end{align*}
Applying M\"obius inversion to remove the coprimality condition, this is equal to \begin{equation*}
\begin{split}
    &\sum_{d\mid a'b\frac{m_0n_1rsg^2}{\alpha_0^2\alpha'}\omega'\frac{\alpha_0\omega_0\nu's}{d_0(g,\alpha_0\omega_0\nu')}\nu_0\tau_{1,0}r_0D_0'f_0g_0'h_1h_2}\mu(d)d  \sum_{D_2'}D_2'F_3^{\epsilon_1,\epsilon_2,\eta_1,\eta_2}\Bigg(\frac{mZ^3}{(m_0\alpha_0\omega_0a_0d_0dD_2')^3},\frac{Z}{m_0\alpha_0\omega_0a_0d_0dD_2'}\Bigg),
    \end{split}
\end{equation*}
The inner sum is finally in shape for Poisson summation and equals
\begin{align*}
    \left(\frac{Z}{m_0\alpha_0\omega_0a_0d_0}\right)^2\sum_{d\mid a'b\frac{m_0n_1rsg^2}{\alpha_0^2\alpha'}\omega'\frac{\alpha_0\omega_0\nu's}{d_0(g,\alpha_0\omega_0\nu')}\nu_0\tau_{1,0}r_0D_0'f_0g_0'h_1h_2}\frac{\mu(d)}{d}\sum_{D_2'}F_4^{\epsilon_1,\epsilon_2,\eta_1,\eta_2}\left(\frac{D_2'Z}{m_0\alpha_0\omega_0a_0d_0d}\right),
\end{align*}
where \begin{align}\label{F4Def}
    F_4^{\epsilon_1,\epsilon_2,\eta_1,\eta_2}&\left(a\right)=\int_0^\infty  W_0\Big(\frac{1}{w X^{1/3}}\Big)\int_{(-7/8)}\int_0^\infty\int_{(-7/8)} \int_0^\infty(2\pi)^{s_1+s_2+1}m^{-s_2}w^{-s_1+2s_2}x^{-s_1-1}y^{2s_1-s_2}e\left(-aw\right) e\left(\epsilon_2\eta_1\frac{s_1+s_2+1}{4}\right) \nonumber\\
    &\times  \Gamma\left(-s_1-s_2-1\right)\mathcal{G}_{\mu_0}^{\eta_1}(s_1+1)\mathcal{G}_{-\mu_0}^{\eta_2}(s_2+1)V(x)\Big(1-U\Big( {\frac{  X^{2/3}y }{\Delta^{\ast}}}\Big)\Big)
    \Phi_{\omega_6}\left(-\epsilon_2\frac{wx}{y^2},-\epsilon_1\frac{my}{w^2}\right) \, \mathrm{d} x \, \mathrm{d} y\,\frac{\mathrm{d}s_1}{2\pi i}\frac{\mathrm{d}s_2}{2\pi i}\mathrm{d} w.
\end{align}
We denote by 
 $\tilde{F}_4^{\epsilon_1,\epsilon_2,\eta_1,\eta_2}\left(a\right)$
the same multiple integral without the factor $W_0(1/wX^{1/3})$. 


\begin{lemma}\label{F4lem} For any $\varepsilon > 0$ we have
   {$$ F_{4}^{\epsilon_1,\epsilon_2,\eta_1,\eta_2}\left(a\right) \ll  (1 + |a|)^{-1-\varepsilon} X^{\varepsilon} \Big(\frac{X^{2/3}}{\Delta^{\ast}} \Big)^2.$$}
   Moreover, we have
   $$\tilde{F}_4^{\epsilon_1,\epsilon_2,\eta_1,\eta_2}\left(0\right) - F_4^{\epsilon_1,\epsilon_2,\eta_1,\eta_2}\left(0\right) \preccurlyeq X^{-2/9}  \frac{X^{2/3}}{\Delta^{\ast}} .$$
\end{lemma}

\begin{proof}
As usual we begin by splitting the $y, w$-integrals into smooth dyadic ranges $y \asymp Y$, $w \asymp W$ and call the corresponding pieces $F_{4, Y, W}^{\epsilon_1,\epsilon_2,\eta_1,\eta_2}\left(a\right)$. 

The decay of $\Phi_{\omega_6}$ implies  $W, Y \preccurlyeq 1$, while the truncation factors imply $$Y \gg Y_0 := \Delta^{\ast}/X^{2/3}, \quad W \gg X^{-1/3}.$$  With $s_j = \sigma_j + it_j$ we have by Stirling's formula for $\sigma_1, \sigma_2 > -1$, $\sigma_1 + \sigma_2 < -1$ that 
\begin{equation}\label{gamma}
\Gamma(-s_1-s_2 - 1) \mathcal{G}^{\eta_1}_{\mu_0}(s_1+1) \mathcal{G}^{\eta_2}_{-\mu_0}(s_2+1) \ll (1 + |t_1 + t_2|)^{-\sigma_1 -\sigma_2 - \frac{3}{2}} (1 + |t_1|)^{3\sigma_1 + \frac{3}{2}}(1 + |t_2|)^{3\sigma_2 + \frac{3}{2}}.
\end{equation}
Integration by parts in $x, y$ implies $\Im s_1, \Im s_2 \preccurlyeq \frac{\sqrt{W}}{Y} + \frac{\sqrt{Y}}{W}$, while integration in $w$ implies
\begin{equation}\label{abound}
a \preccurlyeq \frac{1}{W} \Big(|\Im s_1| + |\Im s_2| + \frac{\sqrt{W}}{Y} + \frac{\sqrt{Y}}{W}\Big) \ll \frac{1}{W} \Big( \frac{\sqrt{W}}{Y} + \frac{\sqrt{Y}}{W}\Big). 
\end{equation}

We start with the contribution of the pieces  $W \geq X^{-\varepsilon} Y$. In this case we shift the contours to $\sigma_1 = \sigma_2 = -1/2 - \varepsilon$ and estimate trivially: by Lemma \ref{Kw6}a) we bound $\Phi_{\omega_6}$ by $(YW)^{1/4}$. Integration by parts in $x, y$ implies $\Im s_1, \Im s_2 \preccurlyeq \frac{\sqrt{W}}{Y} + \frac{\sqrt{Y}}{W}$, so that 
\begin{equation}\label{f1}
F_{4, Y, W}^{\epsilon_1,\epsilon_2,\eta_1,\eta_2}\left(a\right) \preccurlyeq \int_{y \asymp Y} \int_{w \asymp W} \int_{x \asymp 1} \int_{t_1, t_2 \preccurlyeq \sqrt{W}/Y}  \frac{(WY)^{-1/2} (WY)^{1/4}}{(1 + |t_1 + t_2|)^{1/2}} {\rm d}x\, {\rm d}y\, {\rm d}w\, {\rm d}t_1\, {\rm d}t_2 \ll \frac{W^{3/2}}{Y^{3/4}}. 
\end{equation} 
If $|a| \geq 1$, by \eqref{abound} this is at most
$$\preccurlyeq \frac{1}{|a|^{5/4}} \frac{W^{7/8}}{Y^{2}} \ll \frac{1}{|a|^{5/4}}Y_0^{-2} $$
and so in all cases
$$F_{4, Y, W}^{\epsilon_1,\epsilon_2,\eta_1,\eta_2}\left(a\right) \preccurlyeq 
\frac{1}{(1+|a|)^{5/4}}Y_0^{-2}$$
if $W \geq X^{-\varepsilon} Y$.

Let us now assume $W \leq X^{-\varepsilon} Y$. In this case, we first analyse the $t_1, t_2$-integral 
$$\mathcal{I} (y, w) = 
\int_{(-\frac{5}{6})} \int_{(-\frac{5}{6})} \Big( \frac{2\pi y^2}{xw}\Big)^{s_1+\frac{5}{6}}  \Big(\frac{2\pi w^2}{my}\Big)^{s_2+\frac{5}{6}} e\left(\epsilon_2\eta_1\frac{s_1+s_2+1}{4}\right)   \Gamma\left(-s_1-s_2-1\right)\mathcal{G}_{\mu_0}^{\eta_1}(s_1+1)\mathcal{G}_{-\mu_0}^{\eta_2}(s_2+1) \frac{\mathrm{d}s_1}{2\pi i}\frac{\mathrm{d}s_2}{2\pi i}  $$
and its derivatives in $y, w$ more carefully. For $j_1, j_2 \in \Bbb{N}_0$ we define
$$\mathcal{I}_{j_1, j_2}(y, w) := Y^{j_1} W^{j_2} \frac{\partial^{j_1}}{\partial y^{j_1}}\frac{\partial^{j_2}}{\partial w^{j_2}}\mathcal{I} (y, w), \quad \mathcal{I}^{\ast} _{j_1, j_2}(y, w) := Y^{j_1} W^{j_2} \frac{\partial^{j_1}}{\partial y^{j_1}}\frac{\partial^{j_2}}{\partial w^{j_2}}\mathcal{I}^{\ast} (y, w)$$
where we define the modified integral $\mathcal{I}^{\ast} (y, w) $ as 
$$  \int_{(-\frac{5}{6})} \int_{(-\frac{5}{6})} e\Big(  \text{sgn}(\Im s_2) \frac{2\sqrt{my}}{w}\Big) \Big( \frac{2\pi y^2}{xw}\Big)^{s_1+\frac{5}{6}}  \Big(\frac{2\pi w^2}{my}\Big)^{s_2+\frac{5}{6}} e\left(\epsilon_2\eta_1\frac{s_1+s_2+1}{4}\right)   \Gamma\left(-s_1-s_2-1\right)\mathcal{G}_{\mu_0}^{\eta_1}(s_1+1)\mathcal{G}_{-\mu_0}^{\eta_2}(s_2+1) \frac{\mathrm{d}s_1}{2\pi i}\frac{\mathrm{d}s_2}{2\pi i}.  $$

We cut the integral smoothly into various pieces and note that for $\sigma_1 = \sigma_2 = -5/6$ the right hand side of \eqref{gamma} reads $(1 + |t_1+t_2|)^{1/6} /((1+|t_1)(1+|t_2|)) $. The exact choice of $\sigma_1$ and $\sigma_2$ is not relevant. Let $T_0 > X^{\varepsilon}$ be a parameter. We start with an analysis of $\mathcal{I}(y, w)$ in the following cases. 

1) The portion $|t_1|, |t_2| \leq T_0$ is simply bounded by $\preccurlyeq T_0^{1/6+j_1+j_2}$.

2) Consider next the portion $|t_2| \leq T_2$, $|t_1 |\asymp T_1 \geq  \max(T_2X^{\varepsilon}, T_0)$.   Note that this automatically implies $|t_1 + t_2| \asymp |t_1| \asymp T_1$. Stirling's formula gives a phase in $t_1$ of the form
\begin{equation}\label{case2}
\phi(t_1) = t_1 \log \frac{  |t_1|^3y^2}{(2\pi e)^2|t_1+t_2|xw} + t_2 \log \frac{  e}{(2\pi)^2|t_1+t_2|}, \quad \phi'(t_1) = \log \frac{ |t_1|^3y^2}{(2\pi)^2|t_1+t_2|xw}, \quad \phi^{(j)}(t_1) \ll \frac{1}{T_1^{j-1}}
\end{equation}
for $j \geq 2$. Applying \cite[Lemma 8.1]{BKY} with ${\tt R} = 1, {\tt Q} = {\tt U} = {\tt Y} = T_1$, we see that the $t_1$-integral is negligible unless
$T_1 \asymp \sqrt{W}/Y,$ (which by assumption needs to be $\geq 1$) and so this portion is bounded by $\preccurlyeq (\sqrt{W}/Y)^{1/6+j_1+j_2}$. 

3) Consider next the portion $|t_1| \asymp T_1 \geq T_0$, $|t_2| \asymp T_2 \geq T_0$ and assume that $X^{-\varepsilon} T_2 \leq T_1 \leq X^{\varepsilon} T_2$ and in addition $|t_1 + t_2| \asymp \max(T_1, T_2)$ throughout the entire range. Here the $t_1, t_2$-integral features the phase
$$\phi(t_1, t_2) = t_1 \log\frac{  |t_1|^3y^2}{(2\pi e)^2|t_1+t_2|xw} + t_2   \log\frac{  |t_2|^3w^2}{(2\pi e)^2|t_1+t_2|my}$$ 
with
$$\nabla\phi(t_1, t_2) = \left(\begin{matrix} \log \frac{  |t_1|^3y^2}{(2\pi)^2|t_1+t_2|xw}\\ \log \frac{  |t_2|^3w^2}{(2\pi)^2 |t_1+t_2|my} \end{matrix}\right), \quad \frac{\partial^j}{\partial t_1^j} \phi(t_1, t_2) \ll \frac{1}{|t_1|^{j-1}} +\frac{1}{|t_1+t_2|^{j-1}}, \quad \frac{\partial^j}{\partial t_2^j} \phi(t_1, t_2) \ll \frac{1}{|t_2|^{j-1}} +\frac{1}{|t_1+t_2|^{j-1}} $$
for $j\geq 2$. By the same argument as above  we see that the contribution is negligible unless
$$\frac{T_1^3 Y^2}{\max(T_1, T_2)W} \asymp \frac{T_2^3 W^2}{\max(T_1, T_2)Y},$$
but this is excluded under the present assumption that $W$ and $Y$ are of different order of magnitude.  Consider similarly the portion $|t_1| \asymp T_1 \geq T_0$, $|t_2| \asymp T_2 \geq T_0$ and assume that $|t_1+t_2| \leq c \min(|t_1|, |t_2|)$ throughout for some sufficiently small constant $c > 0$; this implies in particular $T_1 \asymp T_2$. We write $t_2 = u - t_1$. We have
$$\Big|\frac{d}{dt_1} \phi(t_1, u-t_1)\Big| = \Big|\log \frac{|t_1|^3 m y^3}{|u-t_1|^3xw^3}\Big| \gg 1$$
and the same argument as above shows that this portion is negligible.

4) It remains to treat the portion $|t_1| \leq T_1$, $|t_2 |\asymp T_2 \geq  \max(T_1X^{\varepsilon}, T_0)$.   Note that this automatically implies $|t_1 + t_2| \asymp |t_2| \asymp T_2$. The same argument as above with reversed indices and the phase
$$\phi'(t_2) = \log \frac{ |t_2|^3 w^2}{(2\pi)^2 |t_1+t_2|my}$$ implies that we can assume $T_2 \asymp \sqrt{Y}/W$. A more careful stationary phase argument based on \cite[Proposition 8.2]{BKY} and a Taylor expansion shows in fact square-root cancellation in the $t_2$-integral and a phase of the shape $$e\Big(- \text{sgn}(t_2) \frac{2\sqrt{my}}{w} + O(X^{\varepsilon}|t_1|)\Big).$$
This phase explains the definition $\mathcal{I}^{\ast}(y, w)$, and we conclude that the corresponding portion $|t_1| \asymp T_1$, $|t_2| \asymp T_2$ and its derivatives under the above assumptions is $\preccurlyeq T_2^{-1/3} T_1^{j_1+j_2} \asymp (Y^{1/2}/W)^{-1/3} T_1^{j_1+j_2}$.

We are now ready to estimate $F_{4, Y, W}^{\epsilon_1,\epsilon_2,\eta_1,\eta_2}\left(a\right)$ under the assumption $W \leq X^{-\varepsilon} Y$. In this case we can decompose $\Phi_{\omega_6}$ as in \eqref{phidecomp}. We start with the contribution of the ``error term'' $\Phi_{\omega_6}^{\flat}$ which we bound by $\preccurlyeq W^2/Y$. We apply the bounds from the above cases 1), 2), 4) with $j_1 = j_2 = 0$ and $T_0 = X^{\varepsilon}$ (case 3) is negligible). In this way we obtain
\begin{equation}\label{f2}
F_{4, Y, W}^{\epsilon_1,\epsilon_2,\eta_1,\eta_2; \flat}\left(a\right) \preccurlyeq \int_{y \asymp Y} \int_{w \asymp W} (WY)^{-5/6} \frac{W^2}{Y} \Big(1 + \Big(\frac{\sqrt{W}}{Y}\Big)^{1/6} + \Big(\frac{Y^{1/2}}{W}\Big)^{-1/3}\Big) \ll \frac{W^{2 + \frac{1}{6}}}{Y} \preccurlyeq \frac{1}{(1 + |a|)^{13/12}}Y_0^{-1/2}
\end{equation}
using \eqref{abound}. 

We continue with the ``main term'' where we replace $\Phi_{\omega_6}$ with 
$$ e\Big(\pm 2\Big|\Big(\frac{wx}{y^2}\Big)^{1/3}  \pm \Big(\frac{my}{w^2}\Big)^{1/3}\Big|^{3/2}\Big) y^{1/4} w^{1/2} \Lambda(y, w, x) 
$$
with $Y^{j_1}W^{j_2}\Lambda^{(j_1, j_2, j_3)}(y, w, x) \preccurlyeq 1.$ By a Taylor expansion, we can write this as 
$$   e\Big(\pm 2  \Big(\frac{my}{w^2}\Big)^{1/2}\Big) y^{1/4} w^{1/2} \tilde{\Lambda}(y, w, x),  \quad Y^{j_1}W^{j_2}\tilde{\Lambda}^{(j_1, j_2, j_3)}(y, w, x) \preccurlyeq Y^{-\frac{1}{2}(j_1 + j_2+j_3)}.$$
Partial integration in the $x$-integral shows immediately $\Im s_1 \preccurlyeq Y^{-1/2}$. Applying  cases 1) and 2) above with $T_0 = \max(X^{\varepsilon}, Y^{-1/2})$, we are left with bounding integrals of the form
\begin{equation}\label{f3}
\int_{y \asymp Y} \int_{w \asymp W}  e(-aw) e\Big(\pm \frac{2(m y)^{1/2}}{w}\Big) \Psi(y, w) {\rm d}w \, {\rm d} y
\end{equation}
where $$Y^{j_1}W^{j_2}\Psi^{(j_1, j_2)}(y, w) \preccurlyeq (WY)^{-5/6} Y^{1/4} W^{1/2} Y_0^{\frac{1}{2}(-\frac{1}{6} - j_1 - j_2)}$$
Integration by parts in $y$ shows that this integral is negligible unless $W \gg X^{-\varepsilon} Y_0$, so the integral is bounded by
$$ W^{2/3}Y^{5/12} Y_0^{-1/12} \preccurlyeq W^{13/6} Y_0^{-19/12} \preccurlyeq \frac{1}{(1 + |a|)^{13/12}} Y_0^{-19/12}$$
using \eqref{abound}. Case 3) is again negligible, and in case 4) we argue similarly with an additional phase $e(\pm 2\sqrt{my}/w)$. The same argument as before works if the combined phase becomes $e(\pm 4\sqrt{my}/w)$. If the two phases cancel each other,  our integral is
$$\int_{y \asymp Y} \int_{w \asymp W}  e(-aw)  \Psi(y, w) {\rm d}w \, {\rm d} y$$
where now 
$$Y^{j_1}W^{j_2}\Psi^{(j_1, j_2)}(y, w) \preccurlyeq (WY)^{-5/6} Y^{1/4} W^{1/2}\Big(\frac{W}{Y^{1/2}} \Big)^{1/3} Y_0^{-\frac{1}{2}(j_1 + j_2)}.$$ 
 Integration by parts shows that we can assume $a \preccurlyeq Y_0^{-1/2} W^{-1}$, in which case the integral is at most 
\begin{equation}\label{f4}
YW(WY)^{-5/6} Y^{1/4} W^{1/2} \Big(\frac{W}{Y^{1/2}}\Big)^{1/3} = WY^{1/4} \preccurlyeq \frac{1}{(1+|a|)^{1+\rho}} X^{\rho/3} Y_0^{-1}
\end{equation}
for any $0 < \rho < 1/10$ using $W \gg X^{-1/3}$. Combining all previous results, we complete the proof of the first part of the lemma. For the second part, we recall \eqref{f1}, \eqref{f2}, \eqref{f3}, \eqref{f4} to conclude that the dyadic pieces
$$\tilde{F}_{4, Y, W}^{\epsilon_1,\epsilon_2,\eta_1,\eta_2}\left(0\right) - F_{4, Y, W}^{\epsilon_1,\epsilon_2,\eta_1,\eta_2}\left(0\right) \preccurlyeq \frac{W^{3/2}}{Y^{3/4}} + \frac{W^{2+\frac{1}{6}}}{Y} + W^{2/3}Y^{5/12} Y_0^{-1/12} + WY^{1/4} \ll X^{-2/9}Y_0^{-1}.$$
This completes the proof. 
\end{proof}

Inserting this back into $\Sigma_6^{(4)}$, we have 
\begin{equation}\label{eulerproduct}
\begin{split}
    \Sigma_6^{(4)}=&X\sum_{\substack{\epsilon_1,\epsilon_2,\eta_1,\eta_2=\pm1\\\epsilon_1\epsilon_2\eta_1\eta_2=1}}\sum_{\nu_0}\mathop{\sum\sum\sum}_{(\tau_{1,0}r_0g_0',\nu_0)=1}\sum_{(f_0,\tau_{1,0}g_0')=1}\mathop{\sum\sum\sum\sum}_{\substack{n_{1,0}'D_{0,1}'D_{1,0}'D_{2,0}\mid (\nu_0\tau_{1,0}r_0f_0g_0')^\infty\\(D_{0,1}',\tau_{1,0}n_{1,0}'r_0D_{1,0}'D_{2,0}g_0')=1\\(D_{1,0}',\nu_0D_{2,0}g_0')=(D_{2,0},\nu_0\tau_{1,0}r_0)=1\\(D_{1,0}',r_0,f_0)=1}}\mathop{\sum\sum\sum\sum\sum\sum\sum}_{\substack{(\nu'  rsD_0'gh_1h_2,\nu_0\tau_{1,0}r_0f_0g_0')=1\\
    (\nu'rsgh_1,D_0'h_2)=1}}\\
    &\times \frac{\mu\left(\tau_{1,0}r_0rsg_0'gh_1h_2\right)\mu(sgh_1)\varphi(\tau_{1,0})(r_0,D_{1,0}')}{\nu_0^3\nu'\tau_{1,0}^6n_{1,0}'r_0^4r^2s^3D_{0,1}'^2D_{1,0}'D_{2,0}D_0'^2f_0^3g_0'^5g^4h_1^4h_2^3m}W_0\Big(\frac{\tau_{1,0}D_{1,0}'g}{\Xi}\Big)W_0\Big(\frac{rh_1}{\Xi'}\Big)W_0\Big(\frac{ {\nu_0\nu'sf_0D_0'}D_{2, 0}}{\Xi'}\Big) {W_0\Big(\frac{\omega'a_0}{\Xi'}\Big)}
    B(r_0r,1)\\
    &\times \sum_{\substack{c_0\mid f_0 \\(c_0,\nu_0D_{0,1}'D_{1,0}')=1}}\sum_{\substack{k_0\mid \nu_0r_0D_{0,1}'f_0\\(k_0,\frac{c_0r_0D_{2,0}}{(r_0,D_{1,0}')})=1}}\frac{\mu(c_0r_0)\mu(k_0)}{c_0k_0} \sum_{\substack{n_1\mid \nu_0\nu'\tau_{1,0}n_{1,0}'f_0g_0'gh_1\\(n_1,D_{2,0})=1}}\sum_{\left(m_0,\frac{\nu'g}{(n_1,\nu'g)}\nu_0\tau_{1,0}r_0D_0'f_0g_0'h_1h_2\right)=1}\sum_{\alpha_0\mid (rsg^2,m_0n_1)}\\
    &\times \sum_{\alpha'\mid \left(m_0g,\frac{(rsg^2,m_0n_1)}{\alpha_0}\right)}\sum_{\substack{\alpha'\omega'\mid f_0g_0'm_0g\\ \left(\frac{f_0g_0'm_0g}{\alpha'\omega'},\frac{m_0n_1r_0rsg^2}{\alpha_0^2\alpha'^2}\right)=1}}\sum_{\substack{\left(\omega_0,\alpha'\nu_0\tau_{1,0}r_0D_0'f_0g_0'h_1h_2\right)=1\\ m_0n_1g\mid \alpha_0\alpha'\omega_0\omega'\nu_0\nu'\tau_{1,0}n_{1,0}'(m_0n_1,gh_1)}}\sum_{\left(a_0,\frac{m_0n_1rsg^2}{\alpha_0^2\alpha'}\omega'\nu_0\tau_{1,0}r_0D_0'f_0g_0'h_1h_2\right)=1}\\
   & {\times} \sum_{\substack{a'\mid m_0\alpha_0\omega_0A'}}\sum_{\substack{b\mid \alpha_0\omega_0B'\\ \left(b,a_0D_{2,0}\frac{f_0g_0'm_0g}{\alpha'\omega'}\right)=1}}
   \sum_{\substack{d_0\mid \frac{\alpha_0\omega_0\nu's}{(g,\alpha_0\omega_0\nu')}\\ \left(d_0,a'b\frac{m_0n_1rsg^2}{\alpha_0^2\alpha'}\omega'h_1\right)=1 }}\frac{\mu\left(\frac{r_0\omega_0\omega'}{(r_0,D_{1,0}')}\right)\mu(a_0a')\mu(b)\alpha'}{m_0\omega_0a_0^2a'bd_0}\delta\left(\mathcal{M}\mid m\right)B\left(m_0n_1,\frac{rg^3h_1^3m}{(r_0m_0n_1)^2}\right)\\
   \end{split}
   \end{equation}
   \begin{equation*}
   \begin{split}
    &\times \sum_{d\mid a'b\frac{m_0n_1rsg^2}{\alpha_0^2\alpha'}\omega'\frac{\alpha_0\omega_0\nu's}{d_0(g,\alpha_0\omega_0\nu')}\nu_0\tau_{1,0}r_0D_0'f_0g_0'h_1h_2}\frac{\mu(d)}{d}\sum_{D_2'\in \Bbb{Z}}F_4^{\epsilon_1,\epsilon_2,\eta_1,\eta_2}\left(\frac{D_2'Z}{m_0\alpha_0\omega_0a_0d_0d}\right),
    \end{split}
\end{equation*}
where 
\begin{equation}\label{AB}
A'=\nu_0\nu'\tau_{1,0}n_{1,0}'r_0sD_{2,0}f_0g_0', \quad B'=\nu_0\nu'\tau_{1,0}n_{1,0}'D_{2,0}f_0g_0'gh_1/n_1,
\end{equation}
with $Z$ as in \eqref{para} 
and $\mathcal{M}$ as in \eqref{MDef}.

\emph{Remark 1:} Although clear from the derivation, let us justify again that the divisibility condition in the $\delta$-symbol makes sense, i.e.\ the fraction in 
$$\mathcal{M}=\frac{r_0^2n_{1,0}'}{(r_0,n_{1,0}')}\frac{\alpha_0\alpha'\omega_0\omega'\nu_0\nu'\tau_{1,0}(r_0,n_{1,0}')m_0n_1s}{\left(r_0m_0n_1\left(\alpha_0\alpha'\omega_0\omega'\nu_0\nu'\tau_{1,0}n_{1,0}'r_0s,\frac{\alpha'\omega'a'g}{(\alpha'\omega'a',D_{2,0}f_0g_0'm_0g)}\left(\frac{d_0D_{2,0}f_0g_0'm_0g}{(\alpha'\omega'a',D_{2,0}f_0g_0'm_0g)}\right)^\infty\right),rsg^3h_1^2b\right)}$$
is an integer. Indeed, $\mathcal{M} \in \Bbb{N}$ is obvious. To see  {that} we can factor out $r_0^2n_{1, 0}'/(r_0,n_{1,0}')$,  we observe that $rsgh_1$ is coprime to $r_0n_{1, 0}'$ and $b$ is squarefree, so only divisors of $(\text{rad}(r_0n_{1, 0}'),b) \mid (\nu_0\tau_{1, 0}r_0f_0g_0' ,b)$ are  {relevant} in the denominator. Moreover, each prime  {factor} of $\text{rad}(r_0n_{1, 0}')$ has multiplicities at most 1 in the denominator. Since $\left(b,\frac{f_0g_0'm_0g}{\alpha'\omega'}\right)=1$ and $(r_0,b)\mid (r_0,n_{1,0}'f_0)$, we have \begin{align*}
    (\text{rad}(r_0n_{1, 0}'),b) \mid (\nu_0\tau_{1, 0}r_0f_0g_0' ,b)\mid \alpha'\omega'\nu_0\tau_{1,0}(r_0,n_{1,0}').
\end{align*}
Together with such prime factors being multiplicity one in the denominator, we have justified that $\frac{r_0^2n_{1,0}'}{(r_0,n_{1,0}')}\mid \mathcal{M}$.

 {\emph{Remark 2:} Let us now discuss the convergence of the term \eqref{eulerproduct}. Recall that the condition $(r_0m_0n_1)^2\mid rg^3h_1^3m$ is recorded in the presence of $B\left(m_0n_1,\frac{rg^3h_1^3m}{(r_0m_0n_1)^2}\right)$. Moreover,  the condition $\mathcal{M}\mid m$ implies that $$\omega_0\nu'\mid rsg^3h_1^3bm\mid \nu'rsg^3h_1^3m \alpha_0\omega_0 \nu_0\nu'\tau_{1, 0} n_{1, 0}' D_{2, 0} f_0 g_0' gh_1$$ 
and $r_{0}, n_{1, 0}' \leq m \ll 1$. Finally, we have  $\alpha'\alpha_0 \leq rsg^2$. Using  divisor estimates for the variables $c_0, k_0, n_1, m_0, \alpha_0, \alpha', \omega', \omega_0, \nu', a', d_0, d$ and the bound $F_{4}^{\epsilon_1,\epsilon_2,\eta_1,\eta_2}\left(a\right) \preccurlyeq  (1 + |a|)^{-1-\varepsilon}(X^{2/3}/\Delta^{\ast})^2$ from Lemma \ref{F4lem}, we need to show the convergence of
\begin{equation}\label{conv}
\begin{split}
 &\underset{\nu_0, \tau_{1, 0}, g_0', f_0, r, s, D_0', g, h_1, h_2, a_0}{\sum\sum\sum\sum\sum\sum\sum\sum\sum\sum\sum}\frac{W_0 (\frac{rh_1}{\Xi'} ) }{(\nu_0 \tau_{1, 0}^3 rsD_0' f_0g_0'^3 g^2h_1^4h_2^2a_0b)^{1-\varepsilon} }\\
 & \sum_{D_{0, 1}' D_{1, 0}', D_{2, 0} \mid (\nu_0 \tau_{1, 0}   f_0 g_0' )^{\infty}} \frac{1 }{(D_{0, 1}'D_{1, 0}')^{1-\varepsilon} }\underset{\omega_0\nu' \mid b \mid \omega_0\nu'}{\sum\sum\sum} \frac{W_0 (\frac{ {\nu_0\nu'sf_0D_0'}D_{2, 0}}{\Xi'} )  W_0 (\frac{ \omega'a_0}{\Xi'} )}{  b^{1-\varepsilon}},  
 \end{split}
\end{equation}
which is  $\preccurlyeq 1$. In fact this computation shows that the contribution of $D_2' \not= 0$ in \eqref{eulerproduct} is $\preccurlyeq X^{2/3}(X^{2/3}/\Delta^{\ast})^2$ (since $Z$ contains a factor $X^{1/3}$), which for $\Delta^{\ast} = X^{11/21}$ is certainly acceptable.}

\subsection{Postludium: analyzing the weight function}\label{endgame}

We are left with the analysis of $\Sigma_6^{(5)}$, say, where $D'_2 = 0$. In this subsection we take a closer look at the integral $F_4^{\epsilon_1,\epsilon_2,\eta_1,\eta_2}(0)$ in \eqref{F4Def}. As a first step we will replace $F_4^{\epsilon_1,\epsilon_2,\eta_1,\eta_2}(0)$ with $\tilde{F}_4^{\epsilon_1,\epsilon_2,\eta_1,\eta_2}(0)$ using the second bound of Lemma \ref{F4lem} at the cost of an admissible error, and then we remove the truncation introduced at the very beginning in Subsection \ref{smalld1}. After this,  we will apply Bessel orthogonality in the form of Lemma \ref{ortho}. 


We recall the definition  of $G_{\mu_0}^{\epsilon_1,\epsilon_2}(s_1,s_2)$ in \eqref{vrvkernel}. Together with a change of variables $s_j\mapsto s_j-1$, we have \begin{align*}
    \sum_{\epsilon_1\epsilon_2\eta_1\eta_2=1}\tilde{F}_4^{\epsilon_1,\epsilon_2,\eta_1,\eta_2}&\left(0\right)=\sum_{\epsilon_1, \epsilon_2 = \pm 1} 2\pi\int_0^\infty \int_{(1/4)}\int_0^\infty\int_{(1/4)} \int_0^\infty \left(\frac{wx}{2\pi y^2}\right)^{1-s_1}\left(\frac{my}{2\pi w^2}\right)^{1-s_2}  \nonumber\\
    &\times  \frac{V(x)}{x}\Big(1-U\Big( {\frac{  X^{2/3}y
    }{\Delta^{\ast}}}\Big)\Big)
    G_{\mu_0}^{\epsilon_1,\epsilon_2}(s_1,s_2)\Phi_{\omega_6}\left(-\epsilon_2\frac{wx}{y^2},-\epsilon_1\frac{my}{w^2}\right)  \mathrm{d} x \frac{\mathrm{d}s_1}{2\pi i}\mathrm{d} y\frac{\mathrm{d}s_2}{2\pi i}\mathrm{d} w.
\end{align*}
By Lemma \ref{kernelequal} this equals
\begin{align*}
     &\sum_{\epsilon_1, \epsilon_2 = \pm 1}256\pi\int_0^\infty \int_{(1/4)}\int_0^\infty\int_{(1/4)} \int_0^\infty \left(\frac{\pi^2wx}{  y^2}\right)^{1-s_1}\left(\frac{\pi^2my}{ w^2}\right)^{1-s_2}  \nonumber\\
    &\times  \frac{V(x)}{x}\Big(1-U\Big( {\frac{  X^{2/3}y 
    }{\Delta^{\ast}}}\Big)\Big)
    G_{\text{sym}}^{-\epsilon_2,-\epsilon_1}((s_1,s_2), -\mu_0)\Phi_{\omega_6}\left(-\epsilon_2\frac{wx}{y^2},-\epsilon_1\frac{my}{w^2}\right)  \mathrm{d} x \frac{\mathrm{d}s_1}{2\pi i}\mathrm{d} y\frac{\mathrm{d}s_2}{2\pi i}\mathrm{d} w.
\end{align*}
We now apply Mellin inversion to the $s_1, s_2$-integral. To justify this, we apply the same reasoning as in the discussion after \eqref{F4Def}: we first insert suitable truncations of the $y, w$-integral, move the $s_1, s_2$-integrals inside, apply Mellin inversion to the formula \eqref{SymKernelMellin}, and remove the truncations. In this way we see that the previous display equals
\begin{align*}
     &\sum_{\epsilon_1, \epsilon_2 = \pm 1}256\pi^5m\int_0^\infty \ \int_0^\infty  \int_0^\infty    V(x)\Big(1-U\Big( {\frac{X^{2/3}y
    }{\Delta^{\ast}}}\Big)\Big)
     K_{\omega_6}^{\text{sym}}\Big(\Big(-\epsilon_2\frac{wx}{y^2}, -\epsilon_1\frac{my}{w^2}\Big);-\mu_0\Big)\Phi_{\omega_6}\left(-\epsilon_2\frac{wx}{y^2},-\epsilon_1\frac{my}{w^2}\right)  \mathrm{d} x   \frac{\mathrm{d} y \,\mathrm{d} w}{yw}
\end{align*}
which converges absolutely  {by the bounds of Lemma \ref{lem.Phi6Truncation}  and  \ref{Kw6}a). In fact, at this point we can remove the truncation factor $1 - U$ at the cost of a small error. 
The values $y$ that need to be inserted satisfy $X^{2/3}y \ll \Delta^{\ast}$.  
Thus by Lemma \ref{lem.Phi6Truncation}  and  \ref{Kw6}a) (the latter of which holds also for $K_{\omega_6}$) the error is at most
$$\int_{w>0} \int_{y \ll \Delta^{\ast}/X^{2/3}} \frac{(wy)^{1/2}}{(1 + y)^A(1 + w)^A} \frac{dy\, dw}{yw} \ll \Big(\frac{\Delta^{\ast}}{X^{2/3}}\Big)^{1/2}. $$}
Now we change variables $\xi =  wx/y^2$, $\eta =  my/w^2$ 
getting
\begin{align*}
     \sum_{\epsilon_1, \epsilon_2 = \pm 1}256\pi^5m\int_0^\infty & \int_0^\infty  \int_0^\infty    V(x)
     K_{\omega_6}^{\text{sym}} ( (- {\epsilon_2}\xi, -\epsilon_1\eta );-\mu_0 )\Phi_{\omega_6} (-\epsilon_2\xi,-\epsilon_1\eta)  \mathrm{d} x   \frac{\mathrm{d} \xi \,\mathrm{d} \eta}{3\xi \eta} { + O\Big(\Big(\frac{\Delta^{\ast}}{X^{2/3}}\Big)^{1/2}\Big)}.
\end{align*}
We can now compute the $x$-integral. Applying orthogonality of the Bessel kernel in the form of Lemma \ref{ortho}, we obtain finally
\begin{equation}\label{upsdelta}
\sum_{\epsilon_1\epsilon_2\eta_1\eta_2=1}F_4^{\epsilon_1,\epsilon_2,\eta_1,\eta_2}(0) = c m\tilde{V}(1) h(\mu_0)+  {O(X^{- \eta})}
\end{equation}
for some absolute constants $\eta > 0$,  $c \not= 0$.

With this information we return to \eqref{eulerproduct}. It is easy to see  that the removal of the factors
$$ {W_0\Big(\frac{\tau_{1,0}D_{1,0}'g}{\Xi}\Big)W_0\Big(\frac{rh_1}{\Xi'}\Big)W_0\Big(\frac{\nu_0 \nu'sf_0D_0'D_{2, 0}}{\Xi'}\Big)  W_0\Big(\frac{\omega' a_0}{\Xi'}\Big)}$$
incorporates a  small error term.  {This is a similar computation as in \eqref{conv}, except that we are now in the case $D_2' = 0$, so we can bound $F_{4}^{\epsilon_1,\epsilon_2,\eta_1,\eta_2}\left(a\right) \preccurlyeq  1$, rather than $\preccurlyeq(1 + |a|)^{-1-\varepsilon}$, which saves a lot of auxiliary variables. The expression corresponding to \eqref{conv} is now (using $\alpha' \leq m_0g$)
\begin{equation*}
\begin{split}
 &\underset{\nu_0, \tau_{1, 0}, g_0', f_0, r, s, D_0', g, h_1, h_2, a_0}{\sum\sum\sum\sum\sum\sum\sum\sum\sum\sum\sum}\frac{W_0 (\frac{rh_1}{\Xi'} ) }{(\nu_0^3 \tau_{1, 0}^5 r^2s^3D_0'^2 f_0^3g_0'^5 g^3h_1^4h_2^3a_0^2)^{1-\varepsilon} }\\
 &\quad\quad\sum_{D_{0, 1}' D_{1, 0}', D_{2, 0} \mid (\nu_0 \tau_{1, 0}   f_0 g_0' )^{\infty}} \frac{1 }{(D_{0, 1}'D_{1, 0}'D_{2, 0})^{1-\varepsilon} }\underset{\omega_0\nu' \mid b \mid \omega_0\nu'}{\sum\sum\sum} \frac{W_0 (\frac{ {\nu_0\nu'sf_0D_0'}D_{2, 0}}{\Xi'} )  W_0 (\frac{\omega'a_0}{\Xi'} )}{ (\nu'\omega_0 b)^{1-\varepsilon}},  
 \end{split}
\end{equation*}
which is conveniently convergent without the weight functions $W_0$. }

 {This implies that} also the error term contributed by the evaluation \eqref{upsdelta} is negligible. 

Summarizing everything we did in this section, we have proved that 
$$\Sigma_6 = c \tilde{V}(1)h(\mu_0) \mathcal{E} X + O(X^{1- \rho})$$
for some $\rho > 0$, $c \not= 0$, where $\mathcal{E}$ is the expression
\begin{displaymath}
\begin{split}
 &  \sum_{\nu_0}\mathop{\sum\sum\sum}_{(\tau_{1,0}r_0g_0',\nu_0)=1}\sum_{(f_0,\tau_{1,0}g_0')=1}\mathop{\sum\sum\sum\sum}_{\substack{n_{1,0}'D_{0,1}'D_{1,0}'D_{2,0}\mid (\nu_0\tau_{1,0}r_0f_0g_0')^\infty\\(D_{0,1}',\tau_{1,0}n_{1,0}'r_0D_{1,0}'D_{2,0}g_0')=1\\(D_{1,0}',\nu_0D_{2,0}g_0')=(D_{2,0},\nu_0\tau_{1,0}r_0)=1\\(D_{1,0}',r_0,f_0)=1}}\mathop{\sum\sum\sum\sum\sum\sum\sum}_{\substack{(\nu'  rsD_0'gh_1h_2,\nu_0\tau_{1,0}r_0f_0g_0')=1\\
    (\nu'rsgh_1,D_0'h_2)=1}}\\
    &\times \frac{\mu\left(\tau_{1,0}r_0rsg_0'gh_1h_2\right)\mu(sgh_1)\varphi(\tau_{1,0})(r_0,D_{1,0}')}{\nu_0^3\nu'\tau_{1,0}^6n_{1,0}'r_0^4r^2s^3D_{0,1}'^2D_{1,0}'D_{2,0}D_0'^2f_0^3g_0'^5g^4h_1^4h_2^3}   B(r_0r,1)\sum_{\substack{c_0\mid f_0 \\(c_0,\nu_0D_{0,1}'D_{1,0}')=1}}\sum_{\substack{k_0\mid \nu_0r_0D_{0,1}'f_0\\(k_0,\frac{c_0r_0D_{2,0}}{(r_0,D_{1,0}')})=1}}\frac{\mu(c_0r_0)\mu(k_0)}{c_0k_0}\\
    \end{split}
\end{displaymath}
 \begin{displaymath}
\begin{split}  
    &\times\sum_{\substack{n_1\mid \nu_0\nu'\tau_{1,0}n_{1,0}'f_0g_0'gh_1\\(n_1,D_{2,0})=1}}\sum_{\left(m_0,\frac{\nu'g}{(n_1,\nu'g)}\nu_0\tau_{1,0}r_0D_0'f_0g_0'h_1h_2\right)=1}\sum_{\alpha_0\mid (rsg^2,m_0n_1)}\sum_{\alpha'\mid \left(m_0g,\frac{(rsg^2,m_0n_1)}{\alpha_0}\right)}\\
        &\times \sum_{\substack{\alpha'\omega'\mid f_0g_0'm_0g\\ \left(\frac{f_0g_0'm_0g}{\alpha'\omega'},\frac{m_0n_1r_0rsg^2}{\alpha_0^2\alpha'^2}\right)=1}}\sum_{\substack{ (\omega_0,\alpha'\nu_0\tau_{1,0}r_0D_0'f_0g_0'h_1h_2 )=1\\ m_0n_1g\mid \alpha_0\alpha'\omega_0\omega'\nu_0\nu'\tau_{1,0}n_{1,0}'(m_0n_1,gh_1)}}\sum_{\left(a_0,\frac{m_0n_1rsg^2}{\alpha_0^2\alpha'}\omega'\nu_0\tau_{1,0}r_0D_0'f_0g_0'h_1h_2\right)=1}\sum_{\substack{a'\mid m_0\alpha_0\omega_0A'}}\\
   & \times \sum_{\substack{b\mid \alpha_0\omega_0B'\\ \left(b,a_0D_{2,0}\frac{f_0g_0'm_0g}{\alpha'\omega'}\right)=1}}\sum_{\substack{d_0\mid \frac{\alpha_0\omega_0\nu's}{(g,\alpha_0\omega_0\nu')}\\ \left(d_0,a'b\frac{m_0n_1rsg^2}{\alpha_0^2\alpha'}\omega'h_1\right)=1 }}\frac{\mu (\frac{r_0\omega_0\omega'}{(r_0,D_{1,0}')} )\mu(a_0a')\mu(b)\alpha'}{m_0\omega_0a_0^2a'bd_0}\delta\left(\mathcal{M}\mid m\right)B\left(m_0n_1,\frac{rg^3h_1^3m}{(r_0m_0n_1)^2}\right)\sum_{d\mid xD_0'}\frac{\mu(d)}{d}
    \end{split}
\end{displaymath}
with $\mathcal{M}$ as in \eqref{MDef}, $A', B'$ as in \eqref{AB} and 
$$x = a'b\frac{m_0n_1rsg^2}{\alpha_0^2\alpha'}\omega'\frac{\alpha_0\omega_0\nu's}{d_0(g,\alpha_0\omega_0\nu')}\nu_0\tau_{1,0}r_0 f_0g_0'h_1h_2.$$

\subsection{Endgame: the Euler product}\label{eulerprod}
Our final task is to compute  $\mathcal{E}$ as a function of $m$. For our application it is not enough to know that it converges absolutely (which is easy to see), but we need to know that it factors off $B(1, m)$ (which is not easy to see at all). To this end we compute $\mathcal{E}$ as an Euler product. 
To begin with we put
$$y = \nu_0  \tau_{1, 0}   r_0  f_0  g'_0  \nu' r  s  g  h_1  m_0   \omega_0 a_0$$
and compute the $D'_0$-sum
$$\sum_{\substack{d \mid x D_0'\\ (D_0', y) = 1}} \frac{1}{D_0^{'2}} = \frac{(d, x)^2}{d^2} \delta\Big(\Big(\frac{d}{(d, x)}, y\Big) = 1\Big) \prod_{p \nmid y} \Big(1 - \frac{1}{p^2}\Big)^{-1}.$$
In addition we write the  $k_0$-sum as an Euler product and recast $\mathcal{E}$ as
\begin{displaymath}
\begin{split}
 &  \sum_{\nu_0}\mathop{\sum\sum\sum}_{(\tau_{1,0}r_0g_0',\nu_0)=1}\sum_{(f_0,\tau_{1,0}g_0')=1}\mathop{\sum\sum\sum\sum}_{\substack{n_{1,0}'D_{0,1}'D_{1,0}'D_{2,0}\mid (\nu_0\tau_{1,0}r_0f_0g_0')^\infty\\(D_{0,1}',\tau_{1,0}n_{1,0}'r_0D_{1,0}'D_{2,0}g_0')=1\\(D_{1,0}',\nu_0D_{2,0}g_0')=(D_{2,0},\nu_0\tau_{1,0}r_0)=1\\(D_{1,0}',r_0,f_0)=1}}\mathop{\sum\sum\sum\sum\sum\sum }_{\substack{(\nu'  rs gh_1h_2,\nu_0\tau_{1,0}r_0f_0g_0')=1\\
    (\nu'rsgh_1, h_2)=1}} \frac{\mu\left(\tau_{1,0}r_0rsg_0'gh_1h_2\right)\mu(sgh_1)\varphi(\tau_{1,0})(r_0,D_{1,0}')}{\nu_0^3\nu'\tau_{1,0}^6n_{1,0}'r_0^4r^2s^3D_{0,1}'^2D_{1,0}'D_{2,0} f_0^3g_0'^5g^4h_1^4h_2^3}  \\
    & \times B(r_0r,1)\sum_{\substack{c_0\mid f_0 \\(c_0,\nu_0D_{0,1}'D_{1,0}')=1}} \frac{\mu(c_0r_0) }{c_0 }  \prod_{\substack{ p \mid \nu_0r_0D_{0,1}'f_0\\ p \nmid \frac{c_0r_0D_{2,0}}{(r_0,D_{1,0}')}}} \Big(1 - \frac{1}{p}\Big)  \sum_{\substack{n_1\mid \nu_0\nu'\tau_{1,0}n_{1,0}'f_0g_0'gh_1\\(n_1,D_{2,0})=1}}\sum_{\left(m_0,\frac{\nu'g}{(n_1,\nu'g)}\nu_0\tau_{1,0}r_0 f_0g_0'h_1h_2\right)=1}\sum_{\alpha_0\mid (rsg^2,m_0n_1)}\\
    &\times\sum_{\alpha'\mid \left(m_0g,\frac{(rsg^2,m_0n_1)}{\alpha_0}\right)} \sum_{\substack{\alpha'\omega'\mid f_0g_0'm_0g\\ \left(\frac{f_0g_0'm_0g}{\alpha'\omega'},\frac{m_0n_1r_0rsg^2}{\alpha_0^2\alpha'^2}\right)=1}}\sum_{\substack{ (\omega_0,\alpha'\nu_0\tau_{1,0}r_0 f_0g_0'h_1h_2 )=1\\ m_0n_1g\mid \alpha_0\alpha'\omega_0\omega'\nu_0\nu'\tau_{1,0}n_{1,0}'(m_0n_1,gh_1)}}\sum_{\left(a_0,\frac{m_0n_1rsg^2}{\alpha_0^2\alpha'}\omega'\nu_0\tau_{1,0}r_0 f_0g_0'h_1h_2\right)=1}\sum_{\substack{a'\mid m_0\alpha_0\omega_0A'}}\\
   & \times \sum_{\substack{b\mid \alpha_0\omega_0B'\\ \left(b,a_0D_{2,0}\frac{f_0g_0'm_0g}{\alpha'\omega'}\right)=1}} \sum_{\substack{d_0\mid \frac{\alpha_0\omega_0\nu's}{(g,\alpha_0\omega_0\nu')}\\ \left(d_0,a'b\frac{m_0n_1rsg^2}{\alpha_0^2\alpha'}\omega'h_1\right)=1 }}\frac{\mu (\frac{r_0\omega_0\omega'}{(r_0,D_{1,0}')} )\mu(a_0a')\mu(b)\alpha'  }{m_0\omega_0a_0^2a'bd_0 }\delta\left(\mathcal{M}\mid m\right)B\left(m_0n_1,\frac{rg^3h_1^3m}{(r_0m_0n_1)^2}\right)\sum_{(\frac{d}{(d, x)}, y) = 1 }\frac{\mu(d)(d, x)^2}{d^3}\prod_{p \nmid y} \Big(1 - \frac{1}{p^2}\Big)^{-1}.
    \end{split}
\end{displaymath}

To write this as an Euler product, we can assume that every variable is a power of a fixed prime $p$. Since 
$m$ is squarefree, 
it suffices to analyze the cases $m = 1$ and $m=p$. 

Computing this Euler product is actually a finite computational problem due to the following observations: by multiplicativity, we can restrict all variables to be powers of a fixed prime $p$. The variables
$$\tau_{1, 0}, r_0, r, s, g_0', g, h_1, h_2, c_0, \omega_0, \omega', a_0, a', b, d$$
are squarefree because of the various M\"obius functions. In fact, even the product 
\begin{equation}
\label{squarefree}
\tau_{1, 0} r_0 r s g_0' g h_1 h_2
\end{equation}
is squarefree. This implies that
$$v_p(\alpha_0), v_p(\alpha') \leq v_p(rsg^2) \leq 2.$$
Since the arguments in $B(., .)$ are integers, we must have
$(m_0n_1)^2 \mid rg^3h_1^3 m$, so that
$$v_p(n_1), v_p(m_0) \leq \frac{1}{2}v_p(rg^3h_1^3m) \leq \frac{1}{2}(v_p(m) + 3).$$
The condition $\mathcal{M} \mid m$ together with $(\nu_0, rsgh_1) = 1 = (d_0, bh_1) = 1$,  implies 
$$v_p(n_{1, 0}') \leq v_p(m), \quad v_p(\nu_0) \leq v_p(m) + 1, \quad v_p(d_0) \leq v_p(rsg^3) + v_p(m), \quad v_p(\nu') \leq v_p(rsg^3h_1^2) + 1 + v_p(m).$$
The only variables that can possibly have an arbitrarily large $p$-power are the four remaining variables $f_0, D_{0, 1}', D'_{1, 0}, D_{2, 0}$. Here we have
$$(D_{0, 1}', D'_{1, 0}) = (D_{0, 1}', D_{2, 0}) = (D_{1, 0}, D_{2, 0}) = 1$$
and moreover $$(D_{0, 1}', \tau_{1, 0} r_0 r s g_0' g h_1 h_2) = (D_{1, 0}',    r s g_0' g h_1 h_2) = (D_{2, 0},  \tau_{1, 0} r_0 r s   g h_1 h_2) = (f_0,  \tau_{1, 0}   r s g_0' g h_1 h_2) = 1.$$
We observe that as soon as the $p$-adic valuation of the three $D_{*}$- variables is $\geq 1$ or $v_p(f_0) \geq 3$, they do not interact with the other variables and stabilize in the following sense: if all other variables are kept fixed, the contribution of $v_p(f_0) \geq 3$ is 
$$1 + \frac{1}{p^3} + \frac{1}{p^6} + \ldots = \Big(1 - \frac{1}{p^3}\Big)^{-1}$$
times the contribution of $v_p(f_0) = 3$ and similarly for the variables $D_{0, 1}',D_{1, 0}', D_{2, 0}$. Thus we have reduced   the evaluation to a finite problem that can be checked by computer. At the cost of a slightly longer code, it's easiest to split the computation according to whether one of the variables in \eqref{squarefree} is divisible by $p$. It is easy to see that by the above restrictions the term $(...)^{\infty}$ in the definition \eqref{MDef} of $\mathcal{M}$ can be replaced with $(...)^{20}$. 

\noindent {\tt A := m0 + alpha0 + omega0 + nu0 + nu + tau10 + n10 + r0 + s + d20 + 
  f0 + g0;\\
BB := alpha0 + omega0 + nu0 + nu + tau10 + n10 + d20 + f0 + g0 + g + 
  h1 - n1;\\
  x := a + b + m0 + n1 + r + s + 2g - 2alpha0 - alpha + omega + 
  alpha0 + omega0 + nu + s - d0 - Min[g, alpha0 + omega0 + nu] + 
  nu0 + tau10 + r0  + f0 + g0 + h1 + h2;\\
M := 2r0 + n10  + alpha0 + alpha + omega0 + omega + nu0 + nu + 
  tau10  + m0 + n1 + s - 
  Min[r + s + 3g + 2h1 + b, 
   r0 + m0 + n1 + 
    Min[alpha0 + alpha + omega0 + omega + nu0 + nu + tau10 + n10 + 
      r0 + s, 
     alpha + omega + a + g - 
      Min[alpha + omega + a, d20 + f0 + g0 + m0 + g] + 
      20(d0 + d20 + f0 + g0 + m0 + g - 
         Min[alpha + omega + a, d20 + f0 + g0 + m0 + g])]];\\
y := nu0 + tau10 + r0 + f0 + g0 + nu + r + s + g + h1 + m0 +  
  omega0 + a0;\\
delta[a\_, b\_] := If[b == 0 \&\& a > 0, 0, 1];\\
P[n\_] := (1 - 1/p\^{}n)\^{}(-1);\\

\noindent summand[m\_] := 
 If[tau10 + r0 + r + s + g0 + g + h1 + h2 <= 1 \&\& c0 + r0 <= 1 \&\& 
   a0 + a <= 1 \&\& 
   r0 + omega0 + omega - Min[r0, d10] <= 1 \&\& (d - Min[d, x]) y == 
    0 \&\& (tau10 + r0 + g0)nu0  == 0 \&\& f0(tau10 + g0) == 0 \&\& 
   d01(tau10 + n10 + r0 + d10 + d20 + g0) == 0 \&\& 
   d10(nu0 + d20 + g0) == 0 \&\& d20(nu0 + tau10 + r0) == 0 \&\& 
   Min[d10, r0, f0] == 
    0 \&\& (nu + r + s   + g + h1 + h2)(nu0 + tau10 + r0 + f0 + g0) == 
    0 \&\& (nu + r + s + g + h1)h2 == 0 \&\& 
   c0(nu0 + d01 + d10) == 0 \&\&   n1 d20  == 0 \&\& 
   m0(nu + g - Min[n1, nu + g] + nu0 + 
       tau10 + r0 + f0 + g0 + h1 + h2) == 0 \&\&
   (f0 + g0 + m0 + g - alpha - omega)(m0 + n1 + r + r0 + s + 2 g - 2alpha0 - 2alpha) == 0 \&\& 
   omega0(alpha + nu0 + tau10 + r0 + f0 + g0 + h1 + h2) == 0 \&\& 
   a0(m0 + n1 + r + s + 2g - 2alpha0 - alpha + omega + nu0 + 
       tau10 + r0 + f0 + g0 + h1 + h2) == 0 \&\& 
   b(a0 + d20 + f0 + g0 + m0 + g - alpha - omega) == 0 \&\& 
   d0(a + b + m0 + n1 + r + s + 2g - 2alpha0 - alpha + omega + 
       h1) == 0 \&\& M <= m \&\& c0 <= f0 \&\&   
   n1 <= nu0 +  nu + tau10 + n10 + f0 + g0 + g + h1 \&\& 
   alpha0 <= Min[r + s + 2g, m0 + n1] \&\& 
   alpha <= Min[m0 + g, r + s + 2g - alpha0, m0 + n1 - alpha0] \&\& 
   alpha + omega <= f0 + g0 + m0 + g \&\& 
   m0 + n1 + g <= 
    alpha0 + alpha + omega0 + omega + nu0 + nu + tau10 + n10 + 
     Min[m0 + n1, g + h1] \&\& a <= A \&\&b <= BB \&\& 
   d0 <= alpha0 + omega0 + nu + s - Min[g, alpha0 + omega0 + nu] \&\&
   delta[n10 + d01 + d10 + d20, nu0 + tau10 + r0 + f0 + g0] == 
    1, (-1)\^{}(tau10 + r  + g0  + h2 + c0  + a0 + a + b + d + r0 + 
      omega0 + omega - Min[r0, d10])   p\^{}(-3nu0 - nu - 6tau10 - 
      n10 - 4r0 - 2r - 3s - 2d01 - d10 - d20 - 3f0 - 5g0 - 
      4g - 4h1 - 3h2 + Min[r0, d10] - c0  + alpha - m0 - omega0 - 
      2a0 - a - b - d0 - 3d + 2Min[d, x]) B[m0 + n1, 
    r + 3g + 3h1 - 2(r0 + m0 + n1) + m] B[r0 + r, 0] If[
    tau10 == 0, 1, p\^{}(tau10 - 1) (p - 1)]If[
    y == 0, P[2], 1] If[f0 == 3, P[3], 1] If[d01 == 1, 
    P[2] , 1] If[d10 == 1, P[1] , 1] If[d20 == 1, P[1], 1] If[
    nu0 + r0 + d01 + f0 >= 1 \&\& c0 + r0 + d20   - Min[r0, d10] == 0, 
    1 - 1/p, 1], 0];\\

\noindent Expand[Simplify[Sum[summand[m], \{tau10, 0, 0\}, \{r0, 0, 0\}, \{r, 0, 0\}, \{s, 0, 
      0\}, \{g0, 0, 0\}, \{g, 0, 0\}, \{h1, 0, 0\}, \{h2, 1, 1\},    \{c0, 0, 
      0\},   \{omega, 0, 1\}, \{omega0, 0, 1\}, \{a, 0, 1\}, \{a0, 0, 1\}, \{b, 
      0, 1\}, \{d, 0, 1\},  \{n1, 0, 1\}, \{m0, 0, 1\},  \{n10, 0, 2\}, \{nu0, 
      0, 3\}, \{alpha0, 0, 0\}, \{alpha, 0, 0\}, \{nu, 0, 3\}, \{d0, 0, 
      2\},   \{f0, 0, 0\}, \{d01, 0, 0\}, \{d10, 0, 0\}, \{d20, 0, 0\} ]]  + \\
   Simplify[
    Sum[summand[m], \{tau10, 0, 0\}, \{r0, 0, 0\}, \{r, 0, 0\}, \{s, 0, 
      0\}, \{g0, 0, 0\}, \{g, 0, 0\}, \{h1, 1, 1\}, \{h2, 0, 0\},    \{c0, 0, 
      0\},   \{omega, 0, 1\}, \{omega0, 0, 1\}, \{a, 0, 1\}, \{a0, 0, 1\}, \{b, 
      0, 1\}, \{d, 0, 1\},  \{n1, 0, 2\}, \{m0, 0, 2\},  \{n10, 0, 2\}, \{nu0, 
      0, 3\}, \{alpha0, 0, 0\}, \{alpha, 0, 0\}, \{nu, 0, 5\}, \{d0, 0, 
      2\},   \{f0, 0, 0\}, \{d01, 0, 0\}, \{d10, 0, 0\}, \{d20, 0, 0\} ]] + \\
   Simplify[
    Sum[summand[m], \{tau10, 0, 0\}, \{r0, 0, 0\}, \{r, 0, 0\}, \{s, 0, 
      0\}, \{g0, 0, 0\}, \{g, 1, 1\}, \{h1, 0, 0\}, \{h2, 0, 0\},    \{c0, 0, 
      0\},   \{omega, 0, 1\}, \{omega0, 0, 1\}, \{a, 0, 1\}, \{a0, 0, 1\}, \{b, 
      0, 1\}, \{d, 0, 1\},  \{n1, 0, 2\}, \{m0, 0, 2\},  \{n10, 0, 2\}, \{nu0, 
      0, 3\}, \{alpha0, 0, 2\}, \{alpha, 0, 2\}, \{nu, 0, 6\}, \{d0, 0, 
      5\},   \{f0, 0, 0\}, \{d01, 0, 0\}, \{d10, 0, 0\}, \{d20, 0, 0\} ]] + \\
   Simplify[
    Sum[summand[m], \{tau10, 0, 0\}, \{r0, 0, 0\}, \{r, 0, 0\}, \{s, 0, 
      0\}, \{g0, 1, 1\}, \{g, 0, 0\}, \{h1, 0, 0\}, \{h2, 0, 0\},    \{c0, 0, 
      0\},   \{omega, 0, 1\}, \{omega0, 0, 1\}, \{a, 0, 1\}, \{a0, 0, 1\}, \{b, 
      0, 1\}, \{d, 0, 1\},  \{n1, 0, 1\}, \{m0, 0, 1\},  \{n10, 0, 2\}, \{nu0, 
      0, 3\}, \{alpha0, 0, 0\}, \{alpha, 0, 0\}, \{nu, 0, 3\}, \{d0, 0, 
      2\},   \{f0, 0, 0\}, \{d01, 0, 0\}, \{d10, 0, 0\}, \{d20, 0, 1\} ]]    + \\
   Simplify[
    Sum[summand[m],  \{tau10, 0, 0\}, \{r0, 0, 0\}, \{r, 0, 0\}, \{s, 1, 
      1\}, \{g0, 0, 0\}, \{g, 0, 0\}, \{h1, 0, 0\}, \{h2, 0, 0\},    \{c0, 0, 
      0\},   \{omega, 0, 1\}, \{omega0, 0, 1\}, \{a, 0, 1\}, \{a0, 0, 1\}, \{b, 
      0, 1\}, \{d, 0, 1\},  \{n1, 0, 1\}, \{m0, 0, 1\},  \{n10, 0, 2\}, \{nu0, 
      0, 3\}, \{alpha0, 0, 1\}, \{alpha, 0, 1\}, \{nu, 0, 4\}, \{d0, 0, 
      3\},   \{f0, 0, 0\}, \{d01, 0, 0\}, \{d10, 0, 0\}, \{d20, 0, 0\} ]]   + \\
   Simplify[
    Sum[summand[m],   \{tau10, 0, 0\}, \{r0, 0, 0\}, \{r, 1, 1\}, \{s, 0, 
      0\}, \{g0, 0, 0\}, \{g, 0, 0\}, \{h1, 0, 0\}, \{h2, 0, 0\},    \{c0, 0, 
      0\},   \{omega, 0, 1\}, \{omega0, 0, 1\}, \{a, 0, 1\}, \{a0, 0, 1\}, \{b, 
      0, 1\}, \{d, 0, 1\},  \{n1, 0, 2\}, \{m0, 0, 2\},  \{n10, 0, 2\}, \{nu0, 
      0, 3\}, \{alpha0, 0, 1\}, \{alpha, 0, 1\}, \{nu, 0, 4\}, \{d0, 0, 
      3\},   \{f0, 0, 0\}, \{d01, 0, 0\}, \{d10, 0, 0\}, \{d20, 0, 0\} ]]   + \\
   Simplify[
    Sum[summand[m],    \{tau10, 0, 0\}, \{r0, 1, 1\}, \{r, 0, 0\}, \{s, 0, 
      0\}, \{g0, 0, 0\}, \{g, 0, 0\}, \{h1, 0, 0\}, \{h2, 0, 0\},    \{c0, 0, 
      1\},   \{omega, 0, 1\}, \{omega0, 0, 1\}, \{a, 0, 1\}, \{a0, 0, 1\}, \{b, 
      0, 1\}, \{d, 0, 1\},  \{n1, 0, 1\}, \{m0, 0, 1\},  \{n10, 0, 2\}, \{nu0, 
      0, 3\}, \{alpha0, 0, 0\}, \{alpha, 0, 0\}, \{nu, 0, 3\}, \{d0, 0, 
      2\},   \{f0, 0, 0\}, \{d01, 0, 0\}, \{d10, 0, 1\}, \{d20, 0, 0\} ]]  + \\
   Simplify[
    Sum[summand[m],  \{tau10, 1, 1\}, \{r0, 0, 0\}, \{r, 0, 0\}, \{s, 0, 
      0\}, \{g0, 0, 0\}, \{g, 0, 0\}, \{h1, 0, 0\}, \{h2, 0, 0\},    \{c0, 0, 
      0\},   \{omega, 0, 1\}, \{omega0, 0, 1\}, \{a, 0, 1\}, \{a0, 0, 1\}, \{b, 
      0, 1\}, \{d, 0, 1\},  \{n1, 0, 1\}, \{m0, 0, 1\},  \{n10, 0, 2\}, \{nu0, 
      0, 3\}, \{alpha0, 0, 0\}, \{alpha, 0, 0\}, \{nu, 0, 3\}, \{d0, 0, 
      2\},   \{f0, 0, 0\}, \{d01, 0, 0\}, \{d10, 0, 1\}, \{d20, 0, 0\} ]]  + \\
   Simplify[
    Sum[summand[m],   \{tau10, 0, 0\}, \{r0, 0, 0\}, \{r, 0, 0\}, \{s, 0, 
      0\}, \{g0, 0, 0\}, \{g, 0, 0\}, \{h1, 0, 0\}, \{h2, 0, 0\},    \{c0, 0, 
      1\},   \{omega, 0, 1\}, \{omega0, 0, 1\}, \{a, 0, 1\}, \{a0, 0, 1\}, \{b, 
      0, 1\}, \{d, 0, 1\},  \{n1, 0, 1\}, \{m0, 0, 1\},  \{n10, 0, 2\}, \{nu0, 
      0, 3\}, \{alpha0, 0, 0\}, \{alpha, 0, 0\}, \{nu, 0, 3\}, \{d0, 0, 
      2\},   \{f0, 0, 3\}, \{d01, 0, 1\}, \{d10, 0, 1\}, \{d20, 0, 1\}]]]
}


In this way we obtain for $m=1$ the $p$-th Euler factor
$$    \Big(1 - \frac{1}{p^3} - \frac{1}{p^4} + \frac{1}{p^6} \Big)B(1, 1)^2  -\frac{(p-1)B(1, p)B(p, 1)}{p^3}+ \frac{(p-1)B(p, p)}{p^4}  =    \Big(1 + \frac{1}{p^2}\Big)\Big(1 - \frac{1}{p^2}\Big)^2 - \frac{1}{p^2} \Big(1 - \frac{1}{p}\Big)^2  B(p, p) $$
by the Hecke relation $B(1, p) B(p, 1) = 1 + B(p, p)$, $B(1, 1) = 1$. 

Similarly, for $m=p$ we obtain the $p$-th Euler factor
\begin{displaymath}
    \begin{split}
&  B(1, 1)B(1, p)\Big(1 - \frac{2}{p^4} + \frac{1}{p^6}\Big) +\big( B(1, p^2)B(p, 1)  + B(p, 1)^2\big)\Big( \frac{1}{p^3} - \frac{1}{p^2} \Big)   + B(1, 1)\big(B(p, p^2) + B(p^2, 1)\big)  \Big( \frac{1}{p^3} - \frac{1}{p^4} \Big)  \\
& =   B(1, p) \Big( \Big(1 + \frac{1}{p^2}\Big)\Big(1 - \frac{1}{p^2}\Big)^2 - \frac{1}{p^2} \Big(1 - \frac{1}{p}\Big)^2 B(p, p)\Big)
\end{split}
\end{displaymath}
using the Hecke relations $$B(p, 1)^2 = B(p^2, 1) + B(1, p), \quad B(p, p^2) = B(1, p) B(p, p) - B(p^2, 1) - B(1, p), \quad B(1, p^2) B(p, 1) + B(p^2, 1) = B(1, p) B(p, p).$$

(One can check that for $m = p^2$ we obtain the $p$-th Euler factor
$$\Big(1 + \frac{1}{p^6} - \frac{2}{p^4}\Big)B(1, p^2) + \Big(\frac{1}{p^3} - \frac{1}{p^4}\Big) B(p, 1) + \Big(\frac{1}{p^3} - \frac{1}{p^2}\Big) (B(1, p^3)B(p, 1) + B(p, 1) B(p, p))+ \Big(\frac{1}{p^3} - \frac{1}{p^4}\Big) (B(p^2, p) + B(p, p^3)).$$
Using the Hecke relations $B(p, 1)(B(1, p^3) +  B(p, p)) = B(1, p^2)(B(p, p) + 1)$, $B(p, p^3) + B(p^2, p) = B(1, p^2)(B(p, p) - 1) - B(p, 1)$,  this equals
$$  B(1, p^2) \Big( \Big(1 + \frac{1}{p^2}\Big)\Big(1 - \frac{1}{p^2}\Big)^2 - \frac{1}{p^2} \Big(1 - \frac{1}{p}\Big)^2 B(p, p)\Big)$$
as desired.) 

This coincides with \eqref{euler-c} and is an excellent way to double-check the long and tedious computation (and needless to say, we found a number of minor errors in this way). This completes the proof of Theorem \ref{thm2}. 

\section{Proof of Theorem \ref{thm1}}

By a routine argument we can deduce Theorem \ref{thm1} from Theorem \ref{thm2}, cf.\ \cite[Section 3]{Ve-thesis}, although we need to be a little careful.  The key point is to use the flexibility of the test function $h$ and the parameter $m$ to single out a specific form $\pi_0 \in \mathcal{B}$ in the sum \eqref{sm}. To this end, given a spectral parameter $\mu_0 \in (i\Bbb{R})^3$, we choose $h = h_Z$ as in \eqref{hz}.  
From the discreteness of the discrete spectrum, as well as  the Rankin-Selberg bound for $A_{\pi}(1, n)$ and the estimate $|B(1, n)| \leq d_3(n) $ from the Ramanujan bound which imply
$$\sum_n B(1,n)A_\pi(1,n) V\left(\frac{n}{X}\right)\frac{1}{L(1, \text{Ad}, \pi)} \ll X  (\log X)^{O(1)} \| \mu_{\pi} \|^{O(1)},$$
 we conclude
\begin{displaymath}
    \begin{split}
&\sum_{\substack{\pi \in \mathcal{B} \\\mu_{\pi} = \mu_0}} \sum_n B(1,n)A_\pi(1,n)\overline{A_\pi(1,m)}V\left(\frac{n}{X}\right)\frac{1}{L(1, \text{Ad}, \pi)} \\
&= \sum_{\pi \in \mathcal{B} } \sum_n B(1,n)A_\pi(1,n)\overline{A_\pi(1,m)}V\left(\frac{n}{X}\right)\frac{h_Z(\mu_\pi)}{L(1, \text{Ad}, \pi)} + O\Big( X (\log X)^{O(1)}\exp(-\delta Z)\Big) 
\end{split}
\end{displaymath}
for some $\delta > 0$ depending on $\mu_0$. By Theorem \ref{thm2} this implies 
$$\sum_{\substack{\pi \in \mathcal{B} \\\mu_{\pi} = \mu_0}} \sum_n B(1,n)A_\pi(1,n)\overline{A_\pi(1,m)}V\left(\frac{n}{X}\right)\frac{1}{L(1, \text{Ad}, \pi)} = C B(1, m)X + O_{Z}(X^{1-\delta}) + O\Big(  X(\log X)^{O(1)}\exp(-\delta Z)\Big)$$
for squarefree $m$ and some $\delta > 0$, maybe different from the preceding $\delta$, and some constant $C \not= 0$, with $V$ chosen to satisfy $\tilde{V}(1)\neq 0$. The dependence on $Z$ in the first error term is of the form $\exp(BZ)$ for some constant $B$, as follows from Lemmas \ref{lem.Phi4Truncation} and \ref{lem.Phi6Truncation}. Choosing $Z = \eta \log X$ for a sufficiently small $\eta = \eta (\delta, B)$,  we obtain
\begin{equation}\label{C}
\lim_{X \rightarrow \infty} \frac{1}{X} \sum_{\substack{\pi \in \mathcal{B} \\\mu_{\pi} = \mu_0}} \sum_n B(1,n)A_\pi(1,n)\overline{A_\pi(1,m)}V\left(\frac{n}{X}\right)\frac{1}{L(1, \text{Ad}, \pi)} = C B(1, m)
\end{equation}
for fixed squarefree $m \in \Bbb{N}$. 
The set $\mathcal{B}(\mu_0)$ of forms $\pi$ with $\mu_{\pi} = \mu_0$ spans a finite dimensional vector space. We can therefore use the parameter $m$ to single out a specific form in this vector space. To this end recall that by multiplicity one and the Hecke relations, the sequence   $(A_{\pi}(1, m))_{\mu^2(m)=1}$ characterizes a newform $\pi \in \mathcal{B}(\mu_0)$, as described after Theorem \ref{thm2}. 
Following \cite{MR1669479}, we call two multiplicative functions $f(n), g(n)$ equivalent, if $f(p^n) = g(p^n)$ for all $n\in \Bbb{N}_0$ and all but finitely many primes, otherwise inequivalent. From \cite[Theorem 2]{MR1669479} we know that pairwise inequivalent functions are linearly independent. We conclude that  the matrix $$\big(A_{\pi}(1, m)\big)_{\substack{\mu^2(m) = 1\\\mu_{\pi} = \mu_0}} $$  (with infinitely many rows and finitely many columns) has full rank, and  
for each $\pi_0 \in \mathcal{B}(\mu_0)$ we can find a finite sequence of $\alpha_m \in \Bbb{C}$, $\mu^2(m) = 1$,  
with
$$\sum_m \alpha_m \overline{A_{\pi}(1, m)} = \delta_{\pi = \pi_0}$$
for every   $\pi  \in \mathcal{B}(\mu_0)$. We conclude that
$$\lim_{X \rightarrow \infty}\frac{1}{X}   \sum_n B(1,n)A_{\pi_0}(1,n) V\left(\frac{n}{X}\right)\frac{1}{L(1, \text{Ad}, \pi_0)} = C\sum_{m} \alpha_m B(1, m)=: c(\pi_0)  \quad \text{(say)}$$
exists for every  $\pi_0 \in \mathcal{B}(\mu_0)$. 

Now let $m$ be a squarefree  number. Then the $(1, m)$-th
Fourier coefficient 
of the finite linear combination 
$$\Pi := \sum_{ \pi \in \mathcal{B}( \mu_0)}  \frac{c(\pi)}{C} \tilde{\pi}$$
with $C$ as in \eqref{C} is given by
\begin{equation}\label{linear}
   \sum_{ \pi \in \mathcal{B}(\mu_0)} \frac{c(\pi) }{C}  \overline{ A_{\pi}(1, m) }= B(1, m).
   \end{equation}
 As seen above, the sequences
$$(A_{\pi}(1, m))_{m \text{ squarefree}}$$
for $\pi \in \mathcal{B}(\mu_0)$ 
are pairwise inequivalent. Hence the linear combination \eqref{linear} can only hold if there is only one non-zero term on the left hand side. (Indeed, if the sequence $(B(1, m))_{\mu^2(m) = 1}$ is inequivalent to all $(A_{\pi}(1, m))_{\mu^2(m) = 1}$, there is an immediate contradiction, if it is equivalent to one of them, say, $\pi_0$,  restrict away from the finitely many exceptional primes and conclude that $c(\pi_0)/C = 1$.) In other words 
$\overline{A_{\pi_0}(1, m)} = B(1, m)$
for some $\pi_0 \in \mathcal{B}(\mu_0)$ and all squarefree $m$. As explained above, since both sides satisfy the Hecke relations, 
we must have $$\overline{A_{\pi_0}(n, m)} = B(n, m)$$ for all $n, m$, hence $\bar{\pi}_0$ is the desired automorphic form.

\printbibliography

@article{kiral2022parametrization,
  title={Parametrization of Kloosterman sets and SL3-Kloosterman sums},
  author={K{\i}ral, Eren Mehmet and Nakasuji, Maki},
  journal={Advances in Mathematics},
  volume={403},
  pages={108392},
  year={2022},
  publisher={Elsevier}
}

@article {MR1669479,
    AUTHOR = {Kaczorowski, J. and Molteni, G. and Perelli, A.},
     TITLE = {Linear independence in the {S}elberg class},
   JOURNAL = {C. R. Math. Acad. Sci. Soc. R. Can.},
  FJOURNAL = {Comptes Rendus Math\'{e}matiques de l'Acad\'{e}mie des
              Sciences. La Soci\'{e}t\'{e} Royale du Canada. Mathematical
              Reports of the Academy of Science. The Royal Society of
              Canada},
    VOLUME = {21},
      YEAR = {1999},
    NUMBER = {1},
     PAGES = {28--32},
    %  ISSN = {0706-1994,2816-5810},
   MRCLASS = {11M36 (11M41)},
  MRNUMBER = {1669479},
MRREVIEWER = {M.\ Ram\ Murty},
}

@article {KS,
    AUTHOR = {Kim, Henry H.},
     TITLE = {Functoriality for the exterior square of {${\rm GL}_4$} and
              the symmetric fourth of {${\rm GL}_2$}},
      NOTE = {With appendix 1 by Dinakar Ramakrishnan and appendix 2 by Kim
              and Peter Sarnak},
   JOURNAL = {J. Amer. Math. Soc.},
  FJOURNAL = {Journal of the American Mathematical Society},
    VOLUME = {16},
      YEAR = {2003},
    NUMBER = {1},
     PAGES = {139--183},
  %    ISSN = {0894-0347,1088-6834},
%   MRCLASS = {11F70 (11R39 22E46)},
%  MRNUMBER = {1937203},
%MRREVIEWER = {Mahdi\ Asgari},
%       DOI = {10.1090/S0894-0347-02-00410-1},
 %      URL = {https://doi.org/10.1090/S0894-0347-02-00410-1},
}

@letter{Sa-letter, 
  author = {Sarnak, Peter}, 
title = {Comments on Langland's Lecture ``Endoscopy and Beyond''}, 
        URL = {https://publications.ias.edu/sarnak/paper/487}}

@article {Sa2,
    AUTHOR = {Sakellaridis, Yiannis},
     TITLE = {Beyond endoscopy for the relative trace formula {II}: {G}lobal
              theory},
   JOURNAL = {J. Inst. Math. Jussieu},
  FJOURNAL = {Journal of the Institute of Mathematics of Jussieu. JIMJ.
              Journal de l'Institut de Math\'{e}matiques de Jussieu},
    VOLUME = {18},
      YEAR = {2019},
    NUMBER = {2},
     PAGES = {347--447},
 %     ISSN = {1474-7480,1475-3030},
%   MRCLASS = {11F72 (11F30 11F67 22E50)},
  MRNUMBER = {3915291},
MRREVIEWER = {Anton\ Deitmar},
 %      DOI = {10.1017/s1474748017000032},
    %   URL = {https://doi.org/10.1017/s1474748017000032},
}

@incollection {Sa1,
    AUTHOR = {Sakellaridis, Yiannis},
     TITLE = {Beyond endoscopy for the relative trace formula {I}: {L}ocal
              theory},
 BOOKTITLE = {Automorphic representations and {$L$}-functions},
    SERIES = {Tata Inst. Fundam. Res. Stud. Math.},
    VOLUME = {22},
     PAGES = {521--590},
 PUBLISHER = {Tata Inst. Fund. Res., Mumbai},
      YEAR = {2013},
%      ISBN = {978-93-80250-49-6},
   MRCLASS = {11F72 (11F30 11F67 22E50)},
%  MRNUMBER = {3156863},
MRREVIEWER = {Anton\ Deitmar},
}

@article {GM,
    AUTHOR = {Ganguly, Satadal and Mawia, Ramdin},
     TITLE = {Rankin-{S}elberg {$L$}-functions and ``beyond endoscopy''},
   JOURNAL = {Math. Z.},
  FJOURNAL = {Mathematische Zeitschrift},
    VOLUME = {296},
      YEAR = {2020},
    NUMBER = {1-2},
     PAGES = {175--184},
 %     ISSN = {0025-5874,1432-1823},
   MRCLASS = {11F66 (11F12 11F30)},
  MRNUMBER = {4140737},
MRREVIEWER = {Ravi\ Raghunathan},
    %   DOI = {10.1007/s00209-019-02431-5},
   %    URL = {https://doi.org/10.1007/s00209-019-02431-5},
}

@article {Al2,
    AUTHOR = {Altu\u{g}, S. Ali},
     TITLE = {Beyond endoscopy via the trace formula, {II}: {A}symptotic
              expansions of {F}ourier transforms and bounds towards the
              {R}amanujan conjecture},
   JOURNAL = {Amer. J. Math.},
  FJOURNAL = {American Journal of Mathematics},
    VOLUME = {139},
      YEAR = {2017},
    NUMBER = {4},
     PAGES = {863--913},
  %    ISSN = {0002-9327,1080-6377},
   MRCLASS = {11F72},
  MRNUMBER = {3689319},
MRREVIEWER = {Wen-Wei\ Li},
   %    DOI = {10.1353/ajm.2017.0023},
  %     URL = {https://doi.org/10.1353/ajm.2017.0023},
}

@article {He1,
    AUTHOR = {Herman, P. Edward},
     TITLE = {Quadratic base change and the analytic continuation of the
              {A}sai {$L$}-function: a new trace formula approach},
   JOURNAL = {Amer. J. Math.},
  FJOURNAL = {American Journal of Mathematics},
    VOLUME = {138},
      YEAR = {2016},
    NUMBER = {6},
     PAGES = {1669--1729},
   %   ISSN = {0002-9327,1080-6377},
   MRCLASS = {11F72 (11F30 11F41 11N37)},
  MRNUMBER = {3595498},
MRREVIEWER = {Fan\ Zhou},
    %   DOI = {10.1353/ajm.2016.0050},
     %  URL = {https://doi.org/10.1353/ajm.2016.0050},
}

@book {Ve-thesis,
    AUTHOR = {Venkatesh, Akshay},
     TITLE = {Limiting forms of the trace formula},
      NOTE = {Thesis (Ph.D.)--Princeton University},
 PUBLISHER = {ProQuest LLC, Ann Arbor, MI, 168pp.},
      YEAR = {2002},
%Comments = {168pp.}
%     PAGES = {168},
    %  ISBN = {978-0493-78369-7},
   MRCLASS = {99-05},
  MRNUMBER = {2703729},
  %     URL =
        %      {http://gateway.proquest.com/openurl?url_ver=Z39.88-2004&rft_val_fmt=info:ofi/fmt:kev:mtx:dissertation&res_dat=xri:pqdiss&rft_dat=xri:pqdiss:3062525},
}

@article {Ve,
    AUTHOR = {Venkatesh, Akshay},
     TITLE = {``{B}eyond endoscopy'' and special forms on {GL}(2)},
   JOURNAL = {J. Reine Angew. Math.},
  FJOURNAL = {Journal f\"{u}r die Reine und Angewandte Mathematik. [Crelle's
              Journal]},
    VOLUME = {577},
      YEAR = {2004},
     PAGES = {23--80},
  %    ISSN = {0075-4102,1435-5345},
   MRCLASS = {22E55 (11F72)},
  MRNUMBER = {2108212},
MRREVIEWER = {A.\ Raghuram},
   %    DOI = {10.1515/crll.2004.2004.577.23},
   %    URL = {https://doi.org/10.1515/crll.2004.2004.577.23},
}

@incollection {La,
    AUTHOR = {Langlands, Robert P.},
     TITLE = {Beyond endoscopy},
 BOOKTITLE = {Contributions to automorphic forms, geometry, and number
              theory},
     PAGES = {611--697},
 PUBLISHER = {Johns Hopkins Univ. Press, Baltimore, MD},
      YEAR = {2004},
  %    ISBN = {0-8018-7860-8},
   MRCLASS = {11F70 (11F72 22E55)},
  MRNUMBER = {2058622},
MRREVIEWER = {Volker\ J.\ Heiermann},
}

@book {goldfeld2006automorphic,
    AUTHOR = {Goldfeld, Dorian},
     TITLE = {Automorphic forms and {$L$}-functions for the group {${\rm
              GL}(n,\bold R)$}},
    SERIES = {Cambridge Studies in Advanced Mathematics},
    VOLUME = {99},
      NOTE = {With an appendix by Kevin A. Broughan},
 PUBLISHER = {Cambridge University Press, Cambridge},
      YEAR = {2006},
     PAGES = {xiv+493},
 %     ISBN = {978-0-521-83771-2; 0-521-83771-5},
   MRCLASS = {11F55 (11F66 11F70 11R39)},
  MRNUMBER = {2254662},
MRREVIEWER = {Emmanuel\ P.\ Royer},
   %    DOI = {10.1017/CBO9780511542923},
   %    URL = {https://doi.org/10.1017/CBO9780511542923},
}

@article {MS,
    AUTHOR = {Miller, Stephen D. and Schmid, Wilfried},
     TITLE = {Automorphic distributions, {$L$}-functions, and {V}oronoi
              summation for {${\rm GL}(3)$}},
   JOURNAL = {Ann. of Math. (2)},
  FJOURNAL = {Annals of Mathematics. Second Series},
    VOLUME = {164},
      YEAR = {2006},
    NUMBER = {2},
     PAGES = {423--488},
   %   ISSN = {0003-486X,1939-8980},
   MRCLASS = {11F66 (11F70 11M41)},
  MRNUMBER = {2247965},
MRREVIEWER = {Andre\ Reznikov},
  %     DOI = {10.4007/annals.2006.164.423},
  %     URL = {https://doi.org/10.4007/annals.2006.164.423},
}

@article {JPSS,
    AUTHOR = {Jacquet, Herv\'{e} and Piatetski-Shapiro, Ilja Iosifovitch and
              Shalika, Joseph},
     TITLE = {Automorphic forms on {${\rm GL}(3)$}. {II}},
   JOURNAL = {Ann. of Math. (2)},
  FJOURNAL = {Annals of Mathematics. Second Series},
    VOLUME = {109},
      YEAR = {1979},
    NUMBER = {2},
     PAGES = {213--258},
  %    ISSN = {0003-486X},
   MRCLASS = {10D40 (22E50 22E55)},
  MRNUMBER = {528964},
MRREVIEWER = {Stephen\ Gelbart},
  %     DOI = {10.2307/1971112},
  %     URL = {https://doi.org/10.2307/1971112},
}

@article {He,
    AUTHOR = {Hecke, E.},
     TITLE = {\"{U}ber die {B}estimmung {D}irichletscher {R}eihen durch ihre
              {F}unktionalgleichung},
   JOURNAL = {Math. Ann.},
  FJOURNAL = {Mathematische Annalen},
    VOLUME = {112},
      YEAR = {1936},
    NUMBER = {1},
     PAGES = {664--699},
 %     ISSN = {0025-5831,1432-1807},
   MRCLASS = {99-04},
  MRNUMBER = {1513069},
    %   DOI = {10.1007/BF01565437},
  %     URL = {https://doi.org/10.1007/BF01565437},
}

@article {We,
    AUTHOR = {Weil, Andr\'{e}},
     TITLE = {\"{U}ber die {B}estimmung {D}irichletscher {R}eihen durch
              {F}unktionalgleichungen},
   JOURNAL = {Math. Ann.},
  FJOURNAL = {Mathematische Annalen},
    VOLUME = {168},
      YEAR = {1967},
     PAGES = {149--156},
 %     ISSN = {0025-5831,1432-1807},
   MRCLASS = {10.20 (14.40)},
  MRNUMBER = {207658},
MRREVIEWER = {A.\ P.\ Ogg},
 %      DOI = {10.1007/BF01361551},
 %      URL = {https://doi.org/10.1007/BF01361551},
}

@article {Blomer2015OnTS,
    AUTHOR = {Blomer, Valentin and Buttcane, Jack},
     TITLE = {On the subconvexity problem for {$L$}-functions on {$\rm
              GL(3)$}},
   JOURNAL = {Ann. Sci. \'{E}c. Norm. Sup\'{e}r. (4)},
  FJOURNAL = {Annales Scientifiques de l'\'{E}cole Normale Sup\'{e}rieure.
              Quatri\`eme S\'{e}rie},
    VOLUME = {53},
      YEAR = {2020},
    NUMBER = {6},
     PAGES = {1441--1500},
%      ISSN = {0012-9593,1873-2151},
   MRCLASS = {11F66},
  MRNUMBER = {4203038},
MRREVIEWER = {Saurabh\ Kumar\ Singh},
  %     DOI = {10.24033/asens.2451},
  %     URL = {https://doi.org/10.24033/asens.2451},
}

@article{Buttcane2022,
  title={The arithmetic Kuznetsov formula on GL(3), II: The general case},  
  author={Buttcane, Jack},
  journal={Algebra and Number Theory},
  volume={16},
  number={3},
  pages={567–646},
  year={2022}
}

@incollection {Jut,
    AUTHOR = {Jutila, M.},
     TITLE = {A variant of the circle method},
 BOOKTITLE = {Sieve methods, exponential sums, and their applications in
              number theory ({C}ardiff, 1995)},
    SERIES = {London Math. Soc. Lecture Note Ser.},
    VOLUME = {237},
     PAGES = {245--254},
 PUBLISHER = {Cambridge Univ. Press, Cambridge},
      YEAR = {1997},
%      ISBN = {0-521-58957-6},
   MRCLASS = {11P55 (11F30)},
  MRNUMBER = {1635766},
MRREVIEWER = {G.\ Greaves},
 %      DOI = {10.1017/CBO9780511526091.016},
 %      URL = {https://doi.org/10.1017/CBO9780511526091.016},
}

@article {Blo12,
    AUTHOR = {Blomer, Valentin},
     TITLE = {Subconvexity for twisted {$L$}-functions on {${\rm GL}(3)$}},
   JOURNAL = {Amer. J. Math.},
  FJOURNAL = {American Journal of Mathematics},
    VOLUME = {134},
      YEAR = {2012},
    NUMBER = {5},
     PAGES = {1385--1421},
%      ISSN = {0002-9327},
   MRCLASS = {11F67 (11F70)},
  MRNUMBER = {2975240},
MRREVIEWER = {Jannis A. Antoniadis},
 %      DOI = {10.1353/ajm.2012.0032},
 %      URL = {http://dx.doi.org/10.1353/ajm.2012.0032},
}

@article{KMV,
  title={{Rankin-Selberg $L$-functions in the level aspect}},
  author={Kowalski, Emmanuel and Michel, Philippe and VanderKam, Jeffrey },
  journal={Duke Mathematical Journal},
  volume={114},
  number={1},
  pages={123--191},
  year={2002},
  publisher={Duke University Press}
}

@book {IK,
    AUTHOR = {Iwaniec, Henryk and Kowalski, Emmanuel},
     TITLE = {Analytic number theory},
    SERIES = {American Mathematical Society Colloquium Publications},
    VOLUME = {53},
 PUBLISHER = {American Mathematical Society, Providence, RI},
      YEAR = {2004},
     PAGES = {xii+615},
 %     ISBN = {0-8218-3633-1},
   MRCLASS = {11-02 (11Fxx 11Lxx 11Mxx 11Nxx)},
  MRNUMBER = {2061214},
MRREVIEWER = {K. Soundararajan},
}

@article{kiral2016voronoi,
  title={The Voronoi formula and double Dirichlet series},
  author={K{\i}ral, Eren and Zhou, Fan},
  journal={Algebra \& Number Theory},
  volume={10},
  number={10},
  pages={2267--2286},
  year={2016},
  publisher={Mathematical Sciences Publishers}
}

@article {BlMi,
    AUTHOR = {Blomer, Valentin and Mili\'{c}evi\'{c}, Djordje},
     TITLE = {The second moment of twisted modular {$L$}-functions},
   JOURNAL = {Geom. Funct. Anal.},
  FJOURNAL = {Geometric and Functional Analysis},
    VOLUME = {25},
      YEAR = {2015},
    NUMBER = {2},
     PAGES = {453--516},
%      ISSN = {1016-443X,1420-8970},
   MRCLASS = {11F66 (11F72 11L07)},
  MRNUMBER = {3334233},
MRREVIEWER = {Ravi\ Raghunathan},
  %     DOI = {10.1007/s00039-015-0318-7},
 %      URL = {https://doi.org/10.1007/s00039-015-0318-7},
}

@article {BKY,
    AUTHOR = {Blomer, Valentin and Khan, Rizwanur and Young, Matthew},
     TITLE = {Distribution of mass of holomorphic cusp forms},
   JOURNAL = {Duke Math. J.},
  FJOURNAL = {Duke Mathematical Journal},
    VOLUME = {162},
      YEAR = {2013},
    NUMBER = {14},
     PAGES = {2609--2644},
 %     ISSN = {0012-7094,1547-7398},
   MRCLASS = {11F66 (11F11)},
  MRNUMBER = {3127809},
MRREVIEWER = {Hidenori\ Katsurada},
 %      DOI = {10.1215/00127094-2380967},
  %     URL = {https://doi.org/10.1215/00127094-2380967},
}

@article{buttcane2020plancherel,
  title={Plancherel distribution of Satake parameters of Maass cusp forms on ${\rm GL}_3$},
  author={Buttcane, Jack and Zhou, Fan},
  journal={International Mathematics Research Notices},
  volume={2020},
  number={5},
  pages={1417--1444},
  year={2020},
  publisher={Oxford University Press}
}

@article {Blomer2013,
    AUTHOR = {Blomer, Valentin},
     TITLE = {Applications of the {K}uznetsov formula on {$GL(3)$}},
   JOURNAL = {Invent. Math.},
  FJOURNAL = {Inventiones Mathematicae},
    VOLUME = {194},
      YEAR = {2013},
    NUMBER = {3},
     PAGES = {673--729},
     
}

\end{document}